\documentclass[10pt]{amsart}
\usepackage[USenglish,activeacute]{babel}
\usepackage[T1]{fontenc}
\usepackage[a4paper, total={6in, 9in}]{geometry}
\usepackage{tikz-cd}
\usepackage[utf8]{inputenc}
\usepackage{fancyhdr}
\usepackage{yhmath}
\usepackage{mathrsfs}
\usepackage{graphicx}
\usepackage{wrapfig}
\usepackage{caption}
\usepackage{dsfont}
\usepackage[normalem]{ulem}
\usepackage{enumitem}   
\usepackage{enumitem}   
\usepackage{cutwin}
\usepackage{amsfonts}
\usepackage{mathtools}
\usepackage{subcaption}
\usepackage[colorlinks=true]{hyperref}
\usepackage{multirow}
\usepackage{amsmath}
\usepackage{amssymb}
\usepackage{accents}
\usepackage{amsthm}
\usepackage{textcomp}
\usepackage{adjustbox}
\makeatletter
\newtheorem*{not*}{{ Notation}}
\newtheorem{defi}{ Definition}[subsection]
\newtheorem*{defi*}{Definition}
\newtheorem{teo}[defi]{ Theorem}
\newtheorem*{teo*}{{ Theorem}}
\newtheorem{intteo}{Theorem}

\newtheorem{prop}[defi]{ Proposition}
\newtheorem*{prop*}{{ Proposition}}
\newtheorem{obs}[defi]{{Remark}}
\newtheorem*{obs*}{{Remark}}

\newtheorem*{Lemma*}{Lemma}
\newtheorem{Lemma}[defi]{Lemma}
\newtheorem*{coro*}{Corollary}
\newtheorem{coro}[defi]{Corollary}

\newcommand{\op}{\operatorname{op}}
\newcommand{\Sp}{\operatorname{Sp}}

\newcommand{\End}{\operatorname{End}}
\newcommand{\Hom}{\operatorname{Hom}}
\newcommand{\Der}{\underline{\operatorname{Der}}}
\newcommand{\Mod}{\operatorname{Mod}}
\newcommand{\D}{\mathcal{D}}
\newcommand{\OX}{\mathcal{O}}

\newcommand{\Spec}{\operatorname{Spec}}

\newcommand{\Indban}{\operatorname{Ind(Ban}_K)}
\newcommand{\Inn}{\underline{\operatorname{Inn}}}
\newcommand{\Outder}{\underline{\operatorname{Out}}}
\newcommand{\ncd}{\Omega^1_{\operatorname{nc}}}
\title{Hochschild (Co)-Homology of D-modules on Rigid Analytic Spaces II}
\author{Fernando Peña Vázquez}
\begin{document}
\begin{abstract}
Let $X$ be a smooth $p$-adic Stein space with free tangent sheaf. We use the notion of Hochschild cohomology for sheaves of Ind-Banach algebras developed in our previous work to study the Hochschild cohomology of the algebra of infinite order differential operators $\wideparen{\D}_X(X)$. In particular, we show that $\operatorname{HH}^{\bullet}(\wideparen{\D}_X(X))$ is a strict complex of nuclear Fréchet spaces which is quasi-isomorphic to the de Rham complex of $X$. We then use this to compare $\operatorname{HH}^{\bullet}(\wideparen{\D}_X(X))$ with a wide array of Ext functors. Finally, we investigate the relation of $\operatorname{HH}^{\bullet}(\wideparen{\D}_X(X))$ with the deformation theory of $\wideparen{\D}_X(X)$. Assuming some finiteness conditions on the de Rham cohomology of $X$, we define explicit isomorphisms between 
$\operatorname{HH}^{1}(\wideparen{\D}_X(X))$ and the space of bounded outer derivations of $\wideparen{\D}_X(X)$, and between $\operatorname{HH}^{2}(\wideparen{\D}_X(X))$ and the space of infinitesimal deformations of $\wideparen{\D}_X(X)$.
\end{abstract}
\maketitle
\tableofcontents
\section{ Introduction}
Let $K$ be a complete non-archimedean extension of $\mathbb{Q}_p$. Let  $\mathcal{R}$ be the ring of integers of $K$, and let $X$ be a smooth and separated rigid analytic $K$-variety. In \cite{HochDmod} we developed a formalism of Hochschild (co)-homology for sheaves of Ind-Banach algebras on $X$, putting special emphasis on the sheaf of infinite order differential operators $\wideparen{\D}_X$. The goal of the present paper is showcasing the relevance of the Hochschild cohomology complex $\operatorname{HH}^{\bullet}(\wideparen{\D}_X)$ as a tool in $p$-adic geometry. We will do so by showing some applications of $\operatorname{HH}^{\bullet}(\wideparen{\D}_X)$ to the  study of smooth rigid analytic spaces, and to the category of co-admissible $\wideparen{\D}$-modules attached to them. By construction, the complex 
$\operatorname{HH}^{\bullet}(\wideparen{\D}_X)$ is a geometric invariant of $X$. Hence, as with any geometric invariant, a good starting point is finding a class of spaces for which the invariant is well-behaved. In analogy with the algebraic setting, one could think that the best candidate for such a class would be that of affinoid spaces. However, this is not the correct choice. Indeed, 
let $\Indban$ be the quasi-abelian category of Ind-Banach spaces over $K$, and let $\operatorname{Shv}(X,\Indban)$ be the category of sheaves of Ind-Banach spaces on $X$. Recall that the main result of \emph{loc. cit.} is the existence of the following isomorphism in $\operatorname{D}(\operatorname{Shv}(X,\Indban))$:
\begin{equation*}
    \mathcal{HH}^{\bullet}(\wideparen{\D}_X):=R\underline{\mathcal{H}om}_{\wideparen{\D}_X^e}(\wideparen{\D}_X,\wideparen{\D}_X)=\Omega_{X/K}^{\bullet}.
\end{equation*}
Let $\widehat{\mathcal{B}}c_K$ be the quasi-abelian category of complete bornological $K$-vector spaces. As a consequence of the previous isomorphism, we obtain the following identity in $\operatorname{D}(\widehat{\mathcal{B}}c_K)$:
\begin{equation}\label{HH to dR}
    \operatorname{HH}^{\bullet}(\wideparen{\D}_X):=R\Gamma(X,\mathcal{HH}^{\bullet}(\wideparen{\D}_X))=R\Gamma(X,\Omega_{X/K}^{\bullet}),
\end{equation}
which shows that $\operatorname{HH}^{\bullet}(\wideparen{\D}_X)$ is intimately related with the de Rham cohomology of $X$. It is precisely for this reason that affinoid spaces are not the class of spaces we are looking for. Namely, for an arbitrary smooth affinoid space, the de Rham cohomology groups can showcase pathological behavior and, in general, they do not give the expected results. For example,  let $\mathbb{B}^1_K=\Sp(K\langle t\rangle)$ be the closed unit ball. In analogy with the complex analytic setting, one would expect that $\mathbb{B}^1_K=\Sp(K\langle t\rangle)$ has trivial de Rham cohomology. However, it can be shown that $\operatorname{H}^1_{\operatorname{dR}}(\mathbb{B}^1_K)$ is an infinite dimensional $K$-vector space. Even worse, if we regard $\operatorname{H}^1_{\operatorname{dR}}(\mathbb{B}^1_K)$ as a topological vector space via the surjection $\Omega_{\mathbb{B}^1_K/K}(\mathbb{B}^1_K)\rightarrow \operatorname{H}^1_{\operatorname{dR}}(\mathbb{B}^1_K)$, then $\operatorname{H}^1_{\operatorname{dR}}(\mathbb{B}^1_K)$ has the trivial topology. This stems from the fact that the de Rham complex $\Omega_{X/K}^{\bullet}(X)$ is not strict for an arbitrary smooth affinoid space $X$. Thus, even if affinoid spaces have the properties we are looking for with respect to sheaf cohomology, their analytic properties are not good enough to suit our needs.\\

Hence, we need to find a class of spaces in which co-admissible $\wideparen{\D}$-modules are well-behaved, and such that their de Rham complex has good analytic properties. This role will be played by the smooth Stein spaces. A Stein space is a rigid analytic space $X$ admitting an admissible cover:
\begin{equation*}
    U_0\subset U_1\subset \cdots \subset X,
\end{equation*}
satisfying that each $U_i$ is an affinoid space, and each map $U_i\rightarrow U_{i+1}$ is the immersion of a relatively compact Weierstrass subdomain (\emph{cf.} Definition \ref{defi Stein spaces}). The existence of this particular cover has powerful implications in the theory of Stein spaces, both in the analytic and algebraic aspects.  As a first example of this, E. Grosse-Kl{\"o}nne showed in  \cite[Corollary 3.2]{grosse2004rham} that if $X$ is a smooth Stein space, then the de Rham complex:
\begin{equation}\label{equation Elmars result}
    0\rightarrow \OX_X(X)\rightarrow \Omega_{X/K}^1(X)\rightarrow \cdots \rightarrow \Omega_{X/K}^{\operatorname{dim}(X)}(X)\rightarrow 0,
\end{equation}
is a strict complex of nuclear Fréchet spaces whose cohomology computes the de Rham cohomology of $X$. Hence, we can regard $\operatorname{H}_{\operatorname{dR}}^{\bullet}(X)$ as a strict complex of nuclear Fréchet spaces. Let $\mathcal{B}c_K$ be the category of bornological $K$-vector spaces, and $LCS_K$ be the category of locally convex $K$-vector spaces. As shown in \cite[Section 5.1]{bode2021operations}, there is an adjunction:
    \begin{equation*}
    (-)^t:\mathcal{B}c_K\leftrightarrows LCS_K:(-)^b,
\end{equation*}
which sends strict complexes of nuclear Fréchet spaces to strict complexes of complete bornological spaces. Hence, we let $\operatorname{H}_{\operatorname{dR}}^{\bullet}(X)^b$ be the associated strict complex of complete bornological spaces.\\

Thus, smooth Stein spaces satisfy the first of the properties we are looking for. Next, we need to show that the category of Ind-Banach $\wideparen{\D}$-modules on a smooth Stein space is well-behaved. Namely, let $X$ be a smooth Stein space. We need to show that $\mathcal{C}(\wideparen{\D}_X)$, the category of co-admissible $\wideparen{\D}_X$-modules, behaves like the category of co-admissible $\wideparen{\D}$-modules on an affinoid space. We may condense our main results into the following theorem:
\begin{intteo}\label{teo A}
Assume $K$ is either discretely valued or algebraically closed, and let $X$ be a smooth Stein space with an étale map $X\rightarrow \mathbb{A}^r_K$. The following hold:  
\begin{enumerate}[label=(\roman*)]
    \item $\wideparen{\D}_X(X)$ is a Fréchet-Stein algebra and a nuclear Fréchet space.
    \item Any co-admissible $\wideparen{\D}_X(X)$-module is a nuclear Fréchet space.
    \item There is an equivalence of abelian categories:
    \begin{equation*}
        \Gamma(X,-):\mathcal{C}(\wideparen{\D}_X)\rightarrow \mathcal{C}(\wideparen{\D}_X(X)).
    \end{equation*}
    \item Co-admissible $\wideparen{\D}_X$-modules are acyclic for the functor:
    \begin{equation*}
        \Gamma(X,-):\Mod_{\Indban}(\wideparen{\D}_X)\rightarrow \Mod_{\Indban}(\wideparen{\D}_X(X)).
    \end{equation*}
\end{enumerate}
\end{intteo}
\begin{proof}
Statements $(i)$ and $(ii)$ are a part of Theorem \ref{teo global sections of co-admissible modules on Stein spaces}.  Statement $(iii)$ is shown Proposition \ref{prop characterization of co-admissible modules on a smooth Stein space}, and $(iv)$ is Proposition \ref{prop acyclicity of co-admissible modules on Stein spaces}.
\end{proof}
We point out that our assumptions on $K$ are a byproduct of the methods we use to show Theorem \ref{teo A}, and that it is unknown to us if the theorem holds in greater generality. Combining this theorem with equations (\ref{HH to dR}), and (\ref{equation Elmars result}), we can deduce the following:
\begin{intteo}\label{teo B}
Let $X$ be a smooth Stein space. We have the following identities in $\operatorname{D}(\widehat{\mathcal{B}}c_K):$
\begin{equation*}
    \operatorname{HH}^{\bullet}(\wideparen{\D}_X)=R\Gamma(X,\Omega_{X/K}^{\bullet})=\Gamma(X,\Omega_{X/K}^{\bullet})=\operatorname{H}_{\operatorname{dR}}^{\bullet}(X)^b.
\end{equation*}
Thus, $\operatorname{HH}^{\bullet}(\wideparen{\D}_X)$ is a strict complex, and $\operatorname{HH}^{n}(\wideparen{\D}_X)=\operatorname{H}_{\operatorname{dR}}^{n}(X)^b$ is a nuclear Fréchet space for $n\geq 0$.    
\end{intteo}
\begin{proof}
This is Theorem \ref{teo hochschild cohomology groups as de rham cohomology groups}.
\end{proof}
Notice that Theorem \ref{teo B} holds without our previous assumptions on $K$. As an upshot, if $K$ is discretely valued or algebraically closed, we obtain a Künneth formula for the Hochschild cohomology of $\wideparen{\D}$. Namely, if  $X$ and $Y$ are smooth Stein spaces, then the following identity holds in $\operatorname{D}(\widehat{\mathcal{B}}c_K):$
\begin{equation*}
\operatorname{HH}^{\bullet}(\wideparen{\D}_{X\times Y})=\operatorname{HH}^{\bullet}(\wideparen{\D}_{X})\widehat{\otimes}^{\mathbb{L}}_K\operatorname{HH}^{\bullet}(\wideparen{\D}_{Y}).    
\end{equation*}
In particular, for any $n\in \mathbb{Z}$ we have the following Künneth formula in $\widehat{\mathcal{B}}c_K:$
\begin{equation*}
    \operatorname{HH}^{n}(\wideparen{\D}_{X\times Y})=\bigoplus_{r+s=n}\operatorname{HH}^{r}(\wideparen{\D}_{X})\widehat{\otimes}_K\operatorname{HH}^{s}(\wideparen{\D}_{Y}).
\end{equation*}   
Using the identification  from Theorem \ref{teo B}, we obtain a Künneth formula for the de Rham cohomology of the product of two smooth Stein spaces. See
Corollary \ref{coro kunneth formula for Hochschild cohomology} for details.\\

By now, the reader should be fairly convinced that the category of smooth Stein spaces is the correct one. Indeed, the behavior of the de Rham complex is optimal, and the category of co-admissible modules is completely determined by their global sections. Furthermore, Stein spaces include many interesting spaces, such as analytifications of affine varieties, and Drinfeld's $p$-adic symmetric spaces.\\

In the algebraic setting, the fact that quasi-coherent modules on an affine variety have trivial sheaf cohomology can be used to show that the Hochschild cohomology of the sheaf of differential operators agrees with the Hochschild cohomology of its algebra of global sections. That is, if $X$ is a smooth affine $K$-variety, then there are quasi-isomorphisms:
\begin{equation*}
    \operatorname{HH}^{\bullet}(\D_X):=R\Gamma(X,R\mathcal{H}om_{\D_X^e}(\D_X,\D_X))   =R\Hom_{\D_X(X)^e}(\D_X(X),\D_X(X))=:\operatorname{HH}^{\bullet}(\D_X(X)).
\end{equation*}
This is based upon the fact that, in the algebraic setting, all the relevant algebras are noetherian, so it is possible to construct resolutions by finite-projective modules. We would like to show that this phenomena is also present in the setting of co-admissible $\wideparen{\D}$-modules on $p$-adic Stein spaces.\\ 

Let $X$ be a smooth $p$-adic Stein space equipped with an étale map $X\rightarrow \mathbb{A}^r_K$. In order to show that the Hochschild cohomology of $\wideparen{\D}_X$ can be computed at the global sections, we will make extensive use of the sheaf of bi-enveloping algebras $\wideparen{E}_X$. This auxiliary sheaf of Ind-Banach algebras on $X^2$ was introduced in \cite{HochDmod}, and is
defined via the following formula:
\begin{equation*}
    \wideparen{E}_X:=p_1^{-1}(\wideparen{\D}_X)\overrightarrow{\otimes}_Kp_2^{-1}(\wideparen{\D}_X^{op}).
\end{equation*}
The sheaf of bi-enveloping algebras $\wideparen{E}_X$ is a key ingredient in most of  the calculations in \cite{HochDmod}.  Its relevance stems from the two  following facts: First, the following formula holds in $\operatorname{D}(\widehat{\mathcal{B}}c_K)$:
\begin{equation*}
    \operatorname{HH}^{\bullet}(\wideparen{\D}_X)=R\Gamma(X^2,R\underline{\mathcal{H}om}_{\wideparen{E}_X}(\Delta_*\wideparen{\D}_X,\Delta_*\wideparen{\D}_X)),
\end{equation*}
which allows  to calculate Hochschild cohomology via $\wideparen{E}_X$. Second, there is a side-switching equivalence for bimodules. That is, there are mutually inverse equivalences of quasi-abelian categories:
\begin{equation*}
    \operatorname{S}:\Mod_{\Indban}(\wideparen{\D}_{X^2})\leftrightarrows \Mod_{\Indban}(\wideparen{E}_X):\operatorname{S}^{-1}.
\end{equation*}
Even better, our assumption that the tangent sheaf of $X$ is free implies that this adjunction corresponds to the extension-restriction of scalars adjunction of an isomorphism of sheaves of Ind-Banach algebras $\wideparen{\mathbb{T}}:\wideparen{\D}_{X^2}\rightarrow \wideparen{E}_{X}$. In particular, $\wideparen{E}_{X}$ is locally a Fréchet-Stein algebra, and the side-switching formalism allows us to define a category of co-admissible $\wideparen{E}_X$-modules. Furthermore, as $X^2$ is a smooth Stein space, the statements of Theorem \ref{teo A} also hold for co-admissible $\wideparen{E}_X$-modules.\\

Even if $\wideparen{E}_X$ is locally Fréchet-Stein, it is still not a sheaf of noetherian algebras. Hence, it is  not clear if $\Delta_*\wideparen{\D}_X$  admits a strict resolution by finite-projective $\wideparen{E}_X$-modules. However, we can see that a slightly weaker statement does hold:
\begin{intteo}\label{intteo theorem C}
    Let $X$ be a smooth and separated rigid space with an étale map $X\rightarrow \mathbb{A}^r_K$. There is a co-admissible $\wideparen{E}_X$-module $\mathcal{S}_X$ such that there is a strict exact complex:
\begin{equation*}
   0\rightarrow P^{2r}\rightarrow \cdots\rightarrow P^0\rightarrow \mathcal{S}_X\oplus \Delta_*\wideparen{\D}_X\rightarrow 0,
\end{equation*} 
where each $P^i$ is a finite-projective $\wideparen{E}_X$-module and $P^0=\wideparen{E}_X$. Furthermore, if $X$ is affinoid or $X$ is Stein and $K$ is either discretely valued or algebraically closed, then this complex can be used to show that for any co-admissible module $\mathcal{M}\in \mathcal{C}(\wideparen{E}_X)$ we have the following identity in $\operatorname{D}(\widehat{\mathcal{B}}c_K)$:
\begin{equation*}
    R\Gamma(X^2,R\underline{\mathcal{H}om}_{\wideparen{E}_X}(\Delta_*\wideparen{\D}_X,\mathcal{M}))=R\underline{\Hom}_{\wideparen{\D}_X(X)^e}(\wideparen{\D}_X(X),\Gamma(X^2,\mathcal{M})).
\end{equation*}
Thus, setting $\mathcal{M}=\Delta_*\wideparen{\D}_X$, we obtain the following identities in $\operatorname{D}(\widehat{\mathcal{B}}c_K)$: 
\begin{equation*}
    \operatorname{HH}^{\bullet}(\wideparen{\D}_X)=R\underline{\Hom}_{\wideparen{\D}_X(X)^e}(\wideparen{\D}_X(X),\wideparen{\D}_X(X))=:\operatorname{HH}^{\bullet}(\wideparen{\D}_X(X)).
\end{equation*}
\end{intteo}
\begin{proof}
This is shown in Proposition \ref{prop resolution of differential operators on smooth Stein spaces} and Theorem \ref{teo hochschild cohomology in terms of extendions in DXe 1}.
\end{proof}
We remark that the existence of the étale map $X\rightarrow \mathbb{A}^r_K$ is an essential part of the proof, and that it is unknown to us if the result holds without this assumption. This theorem has deep implications for the theory of $\wideparen{\D}_X$-modules. For instance, it allows us to show the following structural result: 
\begin{intteo}\label{intteo fin global dimension}
 Let $X=\Sp(A)$ be an affinoid space and $\mathcal{A}\subset A$ be an affine formal model of $A$. Assume $\mathcal{T}_{X/K}$ admits a free $(\mathcal{R},\mathcal{A})$-Lie lattice $\mathcal{L}$ satisfying the following two properties:
 \begin{equation*}
     \mathcal{L}(\mathcal{A})\subset \pi\mathcal{A},\quad [\mathcal{L},\mathcal{L}]\subset \pi^2\mathcal{L}.
 \end{equation*}
 In this situation, the following inequality holds for each $n\geq 0$:
 \begin{equation*}
     \operatorname{gl.dim.}(\widehat{U}(\pi^n\mathcal{L})_K)\leq 2\operatorname{dim}(X).
 \end{equation*}
 Furthermore, if $M\in \Mod_{\Indban}(\widehat{U}(\pi^n\mathcal{L})_K)$ satisfies that it is projective as an Ind-Banach space, then $M$ admits a  resolution  by projective $\widehat{U}(\pi^n\mathcal{L})_K$-modules of length $2\operatorname{dim}(X)$.
\end{intteo}
\begin{proof}
This is Corollary \ref{coro global dim}.
\end{proof}
We point that a similar result has been recently shown by A. Bode in \cite[Theorem 1.1]{bodeauslander}. However, our methods are completely different. Combining Theorem \ref{intteo theorem C} with our previous results, we arrive at the following identity in $\operatorname{D}(\widehat{\mathcal{B}}c_K)$:
\begin{equation*}
    \operatorname{HH}^{\bullet}(\wideparen{\D}_X(X))=\operatorname{H}_{\operatorname{dR}}^{\bullet}(X)^b.
\end{equation*}
This identity shows that there is an intimate relation between the geometry of $X$ and the algebraic structure of $\wideparen{\D}_X(X)$. Following this line of thought, the next step is showing how the Hochschild cohomology groups can be used to study certain algebraic invariants of $\wideparen{\D}_X(X)$.\\

From a historical perspective, one of the most important features of the theory of Hochschild cohomology of associative algebras is its applications in deformation theory. Indeed, let $A$ be an associative $K$-algebra, and consider its Hochschild cohomology groups:
\begin{equation*}
    \operatorname{HH}^{\bullet}(A):=R\Hom_{A^e}(A,A).
\end{equation*}
As shown by M. Gerstenhaber in \cite{deformations}, the Hochschild cohomology groups of $A$ have explicit interpretations in low degrees. For instance, we have the following identities of $K$-vector spaces:
\begin{equation*}
    \operatorname{HH}^0(A)=\operatorname{Z}(A), \quad  \operatorname{HH}^1(A)=\operatorname{Out}(A), 
\end{equation*}
where $\operatorname{Z}(A)$ and $\operatorname{Out}(A)$ are the center of $A$, and the space of $K$-linear outer derivations of $A$ respectively. Working a bit more, it is possible to obtain interpretations of the Hochschild cohomology groups of $A$ in degrees $2$, and $3$. Namely, let us make the following definition:
\begin{defi*}
We define the following concepts:
    \begin{enumerate}[label=(\roman*)]
        \item An infinitesimal deformation of $A$ is a $K[t]/t^2$-algebra structure on $A':=A\otimes_K K[t]/t^2$ which is isomorphic to $A$ after reducing modulo $t$.
        \item Two infinitesimal deformations $B,C$ or $A$ are isomorphic if there is an isomorphism of $K[t]/t^2$-algebras $\varphi:B\rightarrow C$ which is the identity modulo $t$.
        \item We let $\mathscr{E}xt(A)$ be the set of isomorphism classes of infinitesimal extensions of $A$.
    \end{enumerate}
\end{defi*}
\begin{obs*}
Notice that these concepts also make sense for bornological algebras making the obvious modifications, and we will use them in this generality in this introduction.
\end{obs*}
As shown by Gerstenhaber, the space $\operatorname{HH}^{2}(A)$ parametrizes the isomorphism classes of infinitesimal deformations of $A$.  Furthermore, the obstructions to lifting an infinitesimal deformation of $A$ to a second order deformation (\emph{i.e.} $K[t]/t^3$-algebra structure on $A\otimes_K K[t]/t^3$ which is isomorphic to $A$ after reducing modulo $t$) are parameterized by $\operatorname{HH}^{3}(A)$. In the last part of the paper, we will see how some aspects of this theory can be generalized to the setting of complete bornological algebras.\\

Let $\mathscr{A}$ be a complete bornological algebra. The strategy is establishing a new cohomology theory based upon the following complex of complete bornological spaces:
\begin{equation}\label{equation intro complex}
    \mathcal{L}(\mathscr{A})^{\bullet}:= \left(0\rightarrow \mathscr{A}\xrightarrow[]{\delta_0} \underline{\Hom}_{\widehat{\mathcal{B}}c_K}(\mathscr{A},\mathscr{A})\xrightarrow[]{\delta_1} \underline{\Hom}_{\widehat{\mathcal{B}}c_K}(\widehat{\otimes}_K^2\mathscr{A},\mathscr{A})\xrightarrow[]{\delta_2} \cdots\right),
\end{equation}
obtained by applying the functor $\underline{\Hom}_{\widehat{\mathcal{B}}c_K}(-,\mathscr{A})$ to the bar resolution of $\mathscr{A}$  (\emph{cf}. Section \ref{taylor HA}).  We remark that this approach was first pursued by J.L Taylor in \cite{taylor1972homology} (although he works in the category of locally convex spaces). Notice that the complex in (\ref{equation intro complex}) is just a bornological version of the classical Hochschild cohomology complex of an associative $K$-algebra.\\

Our main  motivation for studying the complex $\mathcal{L}(\mathscr{A})^{\bullet}$
is that, unlike our previous definition of Hochschild cohomology, $\mathcal{L}(\mathscr{A})^{\bullet}$ has an explicit description as a chain complex, and we know the formulas of the differentials. Hence, we can perform explicit calculations with $\mathcal{L}(\mathscr{A})^{\bullet}$. In order to preserve this explicit nature, we will make use of the canonical inclusion functor:
\begin{equation*}
    J:\widehat{\mathcal{B}}c_K\rightarrow \mathcal{B}c_K.
\end{equation*}
This is due to the fact that the structure of cokernels in $\mathcal{B}c_K$ is simpler. See Section \ref{Section background HochDmod} for a more complete treatment of the functor $J$, and the discussion below Definition \ref{defi TH cohomology groups}
for a description of the cokernels in $\mathcal{B}c_K$ and $\widehat{\mathcal{B}}c_K$. Let us now define the TH cohomology groups of $\mathscr{A}$:
\begin{intteo}
 The \emph{TH} cohomology spaces of $\mathscr{A}$ are the following bornological spaces:
\begin{equation*}
\operatorname{HH}_{\operatorname{T}}^n(\mathscr{A})=\operatorname{Coker}\left(J(\mathcal{L}(\mathscr{A})^{n-1})\rightarrow J(\operatorname{Ker}(\delta_n))\right),
\end{equation*}
The \emph{TH} cohomology spaces satisfy the following properties:
\begin{enumerate}[label=(\roman*)]
    \item $\operatorname{HH}_{\operatorname{T}}^0(\mathscr{A})=\operatorname{Z}(\mathscr{A})$.
    \item $\operatorname{HH}_{\operatorname{T}}^1(\mathscr{A})=\Outder(\mathscr{A}):= \Der_K(\mathscr{A})/\Inn(\mathscr{A})$.
    \item $\operatorname{HH}_{\operatorname{T}}^2(\mathscr{A})=\mathscr{E}xt(\mathscr{A})$.
\end{enumerate}     
\end{intteo}
\begin{proof}
Statements  $(i)$ and  $(ii)$ are shown in Proposition \ref{prop interpretation of Taylor cohomology groups for n=0,1}.  Statement $(iii)$ is Theorem \ref{teo interpretation of Taylor cohomology groups for n=2}.
\end{proof}
If $\mathscr{A}$ is a Fréchet $K$-algebra, then it is possible to construct a first quadrant spectral sequence in the left heart $LH(\widehat{\mathcal{B}}c_K)$ connecting these two cohomology theories (\emph{cf.} Theorem \ref{teo existence TH SS}). Evaluating this sequence in the special case  $\mathscr{A}=\wideparen{\D}_X(X)$, we obtain the following theorem:
\begin{intteo}
Assume $K$ is either discretely valued or algebraically closed, and let $X$ be a smooth Stein space with an étale map $X\rightarrow \mathbb{A}^r_K$. Then the following hold:
\begin{enumerate}[label=(\roman*)]
    \item There are canonical isomorphisms of complete bornological spaces:
    \begin{equation*}
        \operatorname{HH}^{0}(\wideparen{\D}_X(X))=\operatorname{H}_{\operatorname{dR}}^{0}(X)^b=\operatorname{Z}(\wideparen{\D}_X(X)).
    \end{equation*}
    \item There are isomorphisms of complete bornological spaces:
    \begin{equation*}
        \operatorname{HH}^{1}(\wideparen{\D}_X(X))=\operatorname{H}_{\operatorname{dR}}^{1}(X)^b=\Outder(\wideparen{\D}_X(X)).
        \end{equation*}
        Furthermore,  we can construct an explicit isomorphism:
        \begin{equation*}
            c_1:\operatorname{H}_{\operatorname{dR}}^{1}(X)^b\rightarrow\Outder(\wideparen{\D}_X(X)).
        \end{equation*}
        \item Assume $\operatorname{H}_{\operatorname{dR}}^{2}(X)^b$ is finite-dimensional. Then there is an explicit isomorphism:
        \begin{equation*}
            c_2:\operatorname{H}_{\operatorname{dR}}^{2}(X)^b\rightarrow \mathscr{E}xt(\wideparen{\D}_X(X)).
        \end{equation*}
\end{enumerate}
\end{intteo}
\begin{proof}
Statement $(i)$ is Proposition \ref{prop computation 0th taylor cohomology group}. The first part of $(ii)$ is shown in Proposition \ref{prop comparison map first cohomology group}, and the construction of $c_1$ is done in Proposition \ref{prop explicit derivations}. Statement $(iii)$ is  Theorem \ref{teo comparison map on degree 2}.
\end{proof}
We remark that the results we obtain in the body of the text are slightly more general than the version presented here. In particular, the map $c_2$ can be defined even if the finiteness condition on de Rham cohomology fails. One can check Chapter \ref{Chapter deformation theory} for the specific statements.\\

Hence, at least in some situations, we have completely clarified the structure of the infinitesimal deformations of $\wideparen{\D}_X(X)$. We expect that the methods of this paper can be generalized to hold even when the de Rham cohomology is not finite (\emph{cf.} the discussion below Corollary \ref{coro invartiants and analytification}), and that these techniques open the way for new advances in deformation theory of rigid analytic spaces. For instance, a theory of formal deformations of $\wideparen{\D}_X(X)$ could be within reach. On the other hand we also expect that similar results hold in the $G$-equivariant setting, where $G$ is a finite group acting on our smooth Stein space $X$. In analogy with the algebraic situation, we hope that the $p$-adic Cherednik algebras defined in \cite{p-adicCheralg} can be used to classify all infinitesimal deformations of $G\ltimes \wideparen{\D}_X(X)$, in the same way that algebras of twisted differential operators classify the deformations of $\wideparen{\D}_X(X)$ (\emph{cf}. Definition \ref{defi deformation of algebras} and the discussion above). Although similar in nature, the $G$-equivariant case presents some technical challenges, so we will postpone this study to a future paper.\\

Although the focus of most of this paper is centered around the Hochschild cohomology groups of $\wideparen{\D}_X$ and their deformation-theoretic applications, we will also see that many of the calculations above can also be performed for the Hochschild homology groups. In fact, the results for Hochschild homology can be deduced from those of Hochschild cohomology. This is due to the fact that the Hochschild (co)-homology of $\wideparen{\D}_X$ satisfies 
Van den Bergh duality:
\begin{intteo}\label{intteo vdb duality}
Let $X$ be a smooth and separated rigid space with an étale map $X\rightarrow \mathbb{A}^r_K$. For any $\mathscr{M}\in \operatorname{D}(\wideparen{E}_X)$ there is a canonical isomorphism in $\operatorname{D}(\operatorname{Shv}(X^2,\Indban)$:
\begin{equation*}
    R\underline{\mathcal{H}om}_{\wideparen{E}_X}(\Delta_*\wideparen{\D}_X,\mathscr{M})=\Delta_*\wideparen{\D}_X\overrightarrow{\otimes}^{\mathbb{L}}_{\wideparen{E}_X}\mathscr{M}[-2\operatorname{dim}(X)].
\end{equation*}
Furthermore, if $X$ is affinoid or $X$ is Stein and $K$ is either discretely valued or algebraically closed, then for any $\mathscr{M}\in \operatorname{D}(\wideparen{\D}_X(X)^e)$  there is a canonical isomorphism in $\operatorname{D}(\widehat{\mathcal{B}}c_K)$:
\begin{equation*}
    \operatorname{H}^{\bullet}(\wideparen{\D}_X(X),\mathscr{M})=\operatorname{H}_{\bullet}(\wideparen{\D}_X(X),\mathscr{M})[-2\operatorname{dim}(X)].
    \end{equation*} 
\end{intteo}
\begin{proof}
    This is Theorem \ref{teo Van den Bergh duality}.
\end{proof}
This can be seen as a version of Poincaré duality for Hochschild (co)-homology. Notice in particular that Theorem 
\ref{intteo vdb duality} does not impose any finiteness assumptions on $\mathscr{M}$. In particular, we do not require that $\mathscr{M}$ is a $\mathcal{C}$-complex. As a corollary of this theorem, it follows that the Theorem \ref{intteo theorem C}
holds for Hochschild homology \emph{mutatis mutandis}.\\

 Although the Hochschild homology groups do not have a direct deformation-theoretic meaning, they offer the advantage of being explicitly computable. In particular, we will see in the body of the text that $\operatorname{HH}_{\bullet}(\wideparen{\D}_X)$ can be represented by a complex:
 \begin{equation*}
     \mathcal{T}(\wideparen{\D}_X(X))^{\bullet}:=\left(   
 \cdots \xrightarrow[]{\delta^{-2}}\widehat{\otimes}_K^2\wideparen{\D}_X(X)\widehat{\otimes}_K\wideparen{\D}_X(X)\xrightarrow[]{\delta^{-1}}\wideparen{\D}_X(X)\widehat{\otimes}_K\wideparen{\D}_X(X)  \xrightarrow[]{\delta^0}\wideparen{\D}_X(X)   \right),
 \end{equation*}
with differentials given by explicit formulas. In fact, $\mathcal{T}(\wideparen{\D}_X(X))^{\bullet}$ is an analytic version of the usual Hochschild homology complex of an associative algebra. An upshot of this is that  it allows us to define analytic versions of some of the well-known operations in Hochschild homology, and then we can use the Van den Bergh isomorphism to obtain the corresponding operations on Hochschild cohomology. For instance, we will construct analytic versions of the Connes differential, the Batalin-Vilkovisky operator, the cup product, and the Gerstenhaber bracket.
\subsection*{Related work}
This paper is a part of a series of papers \cite{HochDmod}, \cite{p-adicCheralg}, \cite{p-adicCatO}, in which we start a study of the deformation theory of the algebras $G\ltimes \wideparen{\D}_X(X)$, where $G$ is a finite group acting on a smooth Stein space $X$. In particular, in \cite{HochDmod} we develop a formalism of Hochschild cohomology and homology for $\wideparen{\D}_X$-modules on smooth and separated rigid analytic spaces, and give explicit calculations of the Hochschild cohomology groups $\operatorname{HH}^{\bullet}(\wideparen{\D}_X)$ in terms of the de Rham complex of $X$. This is done in the setting of sheaves of Ind-Banach spaces and quasi-abelian categories, as showcased in \cite{bode2021operations}.\\

In \cite{p-adicCheralg} we develop a theory of $p$-adic Cherednik algebras on smooth rigid analytic spaces with a finite group action. These algebras can be regarded as a $p$-adic analytic version of the sheaves of Cherednik algebras constructed by P. Etingof in \cite{etingof2004cherednik}. In particular, let $X$ be a smooth rigid analytic space with an action of a finite group $G$ satisfying some mild technical conditions. The family of $p$-adic Cherednik algebras associated to the action of $G$ on $X$ is a family of sheaves of complete bornological algebras on the quotient $X/G$  which generalizes the skew group algebra $G\ltimes \wideparen{\D}_X$. In particular, given a $p$-adic Cherednik algebra $\mathcal{H}_{t,c,\omega}$, its restriction to the smooth locus of $X/G$ is canonically isomorphic to $G\ltimes \wideparen{\D}_X$. Furthermore,  all $p$-adic Cherednik algebras are locally Fréchet-Stein, and their isomorphism classes are parameterized by certain geometric invariants associated to the action of $G$ on $X$.\\

In the algebraic setting, this parameter space has a deep deformation-theoretic meaning. Namely, if $X$ is a smooth affine $K$-variety with an action of a finite group $G$, then it can be shown that the family of Cherednik algebras associated to the action of $G$ on $X$  parametrizes all the infinitesimal deformations of the skew-group algebra $G\ltimes \D_X(X)$. We hope that our methods can be extended to show that this phenomena is also present in the rigid analytic setting. The paper \cite{p-adicCheralg} is devoted to laying the foundations of the theory of $p$-adic Cherednik algebras. In particular, we provide a construction of these algebras, study their algebraic properties, and obtain a classification theorem.\\

The study of $p$-adic Cherednik algebras is continued in  \cite{p-adicCatO}, 
where we study the representation theory of $p$-adic Cherednik algebras, focusing on the case where $X$ is the analytification of an affine space, and $G$ acts linearly on $X$. The algebras arising from such an action are called $p$-adic rational Cherednik algebras, and their representation theory shares many features with that of the  $p$-adic Arens-Michael envelope $\wideparen{U}(\mathfrak{g})$ of a finite dimensional reductive Lie algebra $\mathfrak{g}$ studied in \cite{Schmidt2010VermaMO}. In particular, to any $p$-adic rational Cherednik algebra one can attach a category $\wideparen{\OX}_c$, which is a highest weight category whose simple objects are in bijection with the irreducible representations of $G$. 
\subsection*{Structure of the paper}We will now give an overview of the contents of the paper:\newline
In Chapter \ref{background} we provide the reader with some basic material needed for the rest of the paper. Namely, we devote Sections \ref{section Ind-Ban spaces} and \ref{section sheaes of ind} to give a brief remainder on the closed symmetric monoidal structures on $\Indban$ and $\widehat{\mathcal{B}}c_K$, Section \ref{Section background HochDmod} to give a brief outline of the main results of \cite{HochDmod}, and Section \ref{section external hom functors} to define the sheaves of external homomorphisms.\\

In Chapter \ref{Chapter Stein spaces} we introduce the notion of Stein space, and show some of their most important properties. In particular, basic results on the geometry of Stein spaces are explained in Section \ref{section group act on Stein spaces}, and used in Section \ref{section co-ad D mod stein spaces} to study the properties of co-admissible $\wideparen{\D}_X$-modules on smooth Stein spaces.\\

In Chapter \ref{Chapter applications of HC} we start giving some applications of Hochschild cohomology of $\wideparen{\D}$-modules on a smooth Stein space. In Section \ref{Section dR cohomology} we establish a comparison theorem with de Rham cohomology, and use it to show that $\operatorname{HH}^{\bullet}(\wideparen{\D}_X)$ is a strict complex of nuclear Fréchet spaces. This is then used to show a Künneth formula for Hochschild cohomology. In Sections \ref{Section Hochschild cohomology and extension I} and \ref{Section HC and ext 2}, we compare the complex $\operatorname{HH}^{\bullet}(\wideparen{\D}_X)$ with different \emph{Ext} functors, allowing us to show that the Hochschild cohomology groups of $\wideparen{\D}_X$  parameterize certain Yoneda extensions. We also show the important formula:
\begin{equation*}
    \operatorname{HH}^{\bullet}(\wideparen{\D}_X)=\operatorname{HH}^{\bullet}(\wideparen{\D}_X(X)),
\end{equation*}
showing that the Hochschild cohomology can be calculated at the level of global sections. Finally, we show in Section \ref{Secion LR cohomology} that Hochschild cohomology computes Lie-Rinehart cohomology.\\

In Chapter \ref{section duality results} we will show Van den Bergh duality. We start with Section \ref{subsection review operations}, where we review the notion of a weakly holonomic $\wideparen{\D}_X$-module and show some properties regarding direct image and extraordinary direct image along a closed immersion. We finish with Section \ref{section vdb} where we show Theorem \ref{intteo vdb duality}, and Section \ref{section HHom}, where we show an analog of Theorem \ref{intteo theorem C} for Hochschild homology.\\

We finish the paper with Chapter \ref{Chapter deformation theory}, where we study the deformation theory of $\wideparen{\D}_X(X)$. In particular, we use Section \ref{taylor HA} to give an alternative definition of the Hochschild (co)-homology of a complete bornological algebra in terms of the bar resolution, and give explicit interpretations of the associated cohomology groups in low degrees. In Section \ref{comparison HC TH} we establish a spectral sequence connecting this new definition with the one we have been using so far. Next, we devote Sections \ref{TH groups of D} and \ref{section TH homology} and to study the particular case of $\wideparen{\D}_X(X)$, and use our previous results to obtain explicit calculations. Finally, we will dedicate Section \ref{section operations} to defining some operations on $\operatorname{HH}^{\bullet}(\wideparen{\D}_X(X))$ and $\operatorname{HH}_{\bullet}(\wideparen{\D}_X(X))$, and Sections \ref{section ncdf1} and \ref{section ncdf2} to define the module of non-commutative differential forms and use it to obtain some upgrades of the computations from Section \ref{TH groups of D}.
\subsection*{Notation and conventions}
For the rest of this text,  $K$ denotes a complete non-archimedean
field of mixed characteristic $(0,p)$. We let $\mathcal{R}$ be its valuation ring, $k$ be its residue field, and fix  
a pseudo-uniformizer  $\pi\in \mathcal{R}$  with $0< \vert \pi\vert <1$. If $A$ is a Banach $K$-algebra, we let $A^{\circ}$ be the $\mathcal{R}$-algebra of power-bounded elements in $A$, and let $A^{\circ\circ}$ be the ideal of $A^{\circ}$ given by the topologically nilpotent elements in $A$. Notice that, if $K$ is either discretely valued or algebraically closed, every affinoid $K$-algebra $A$ satisfies that $A^{\circ}$
is an affine formal model of $A$ (\emph{cf.} \cite[Corollary 6.4.1.5/6]{BGR}). We will use this fact in the sequel without further mention. All rigid analytic $K$-varieties are assumed to be quasi-separated and quasi-paracompact. For a $\mathcal{R}$-module $\mathcal{M}$, we will often times write $\mathcal{M}_K=\mathcal{M}\otimes_{\mathcal{R}}K$.  If $A=\varprojlim_n A_n$ is a Fréchet-Stein algebra, we let $\mathcal{C}(A)$ denote its category of co-admissible modules.\\

Throughout the text, we will make extensive use of the theory of $\wideparen{\D}$-modules on smooth rigid analytic spaces, as developed in \cite{ardakov2019}, and of the homological machinery of sheaves of Ind-Banach algebras from \cite{bode2021operations}. In order to keep the preliminaries to a minimum, we will not give an introduction to the theory of co-admissible $\wideparen{\D}$-modules here. All the necessary material is available in the background sections of our companion paper \cite{HochDmod}, as well as in the original reference \cite{ardakov2019}. In order to avoid set-theoretic issues, we will assume that we are working within a fixed Grothendieck universe.
\subsection*{Acknowledgments}
This paper was written as a part of a PhD thesis at the Humboldt-Universität zu Berlin under the supervision of Prof. Dr. Elmar Große-Klönne. I would like to thank Prof. Große-Klönne
for pointing me towards this exciting topic and reading an early draft of the paper.

\bigskip
Funded by the Deutsche Forschungsgemeinschaft (DFG, German Research
Foundation) under Germany´s Excellence Strategy – The Berlin Mathematics
Research Center MATH+ (EXC-2046/1, project ID: 390685689).
\section{Background}\label{background}
We will devote this chapter to laying some of the foundational material needed for the rest of the paper. In particular, we will give a brief summary of the contents of \cite{HochDmod}, and introduce the sheaf of external \emph{Hom} functors associated to a sheaf of Ind-Banach algebras on a rigid analytic space.
\subsection{\texorpdfstring{Generalities on $\Indban$ and $\widehat{\mathcal{B}}c_K$}{}}\label{section Ind-Ban spaces}
For the benefit of the reader, and in order to fix some notation, we will now give a brief outline of the properties of the categories of Ind-Banach and (complete) bornological spaces. In order to keep the preliminaries to a minimum, we will not include any proofs. A complete discussion on these topics can be found in Chapters 2 and 3 of our companion paper \cite{HochDmod}, and the references given therein. Let us start by giving some definitions:
    \begin{defi}\label{defi bornology}
A  bornology on a $K$-vector space $V$ is a family $\mathcal{B}$ of subsets
of $V$ such that:
\begin{enumerate}[label=(\roman*)]
    \item If $B_1\in \mathcal{B}$ and $B_2\subset B_1$, then $B_2\in\mathcal{B}$.
    \item $\{v\}\in\mathcal{B}$ for all $v\in V$.
    \item $\mathcal{B}$ is closed under taking finite unions.
    \item If $B\in\mathcal{B}$, then $\lambda\cdot B\in\mathcal{B}$ for all $\lambda\in K$.
    \item Let $B\in \mathcal{B}$, and $\mathcal{R}\cdot B$ be the $\mathcal{R}$-module spawned by $B$. Then $\mathcal{R}\cdot B\in\mathcal{B}$.
\end{enumerate}
We call the pair $(V,\mathcal{B})$ a bornological space, and the elements of $\mathcal{B}$ bounded subsets. We  refer to the pair $(V,\mathcal{B})$ as $V$. A $K$-linear map of bornological spaces $f:V\rightarrow W$ is bounded if it maps bounded subsets to bounded subsets. We let $\mathcal{B}c_K$
be the category of bornological spaces and bounded maps.
\end{defi}
Before giving further details on the properties of $\mathcal{B}c_K$, let us motivate the definition. First, let us point out that, when dealing with problems in non-archimedean functional analysis, the most commonly used category is the category of locally convex spaces over $K$, which we denote $LCS_K$. This category has been thoroughly studied in previous years, and there are plenty of texts detailing its properties. See, for example, P. Schneider's monograph \cite{schneider2013nonarchimedean}. The main problem that arises in $LCS_K$ is that the (projective) tensor product of locally convex spaces and the space of continuous maps with the convergence on bounded subsets topology do not endow $LCS_K$ with a closed symmetric monoidal structure. This pathological behavior makes it very  hard to study categories of modules in $LCS_K$. It is for this reason that bornological spaces need to be brought into the picture. Nonetheless, we would like to point out that there is an adjunction:
    \begin{equation*}
    (-)^t:\mathcal{B}c_K\leftrightarrows LCS_K:(-)^b.
\end{equation*}
The functor $(-)^b$ is called the bornologification functor, and it is fully faithful on metrizable locally convex spaces. See \cite[pp. 93,  Proposition 3]{houzel2006seminaire} for a complete discussion on this matter.\\

Back to our setting, let $V\in \mathcal{B}c_K$ be a bornological space. It can be shown that $V$ can be expressed as a filtered colimit in $\mathcal{B}c_K$:
\begin{equation*}
    V=\varinjlim_{i\in I}V_i,
\end{equation*}
satisfying that each $V_i$ is a normed space.  We say that $V$ is a complete bornological space if all the $V_i$ in the above filtered colimit can be taken to be Banach spaces.
\begin{defi}
Let $\widehat{\mathcal{B}}c_K$ be the full subcategory of $\mathcal{B}c_K$ given by the complete bornological spaces.
\end{defi}
We point out that there is an adjunction:
\begin{equation*}
    J:\widehat{\mathcal{B}}c_K\leftrightarrows \mathcal{B}c_K:\widehat{\operatorname{Cpl}},
\end{equation*}
where $J:\widehat{\mathcal{B}}c_K\rightarrow \mathcal{B}c_K$ is the inclusion, and $\widehat{\operatorname{Cpl}}:\mathcal{B}c_K\rightarrow \widehat{\mathcal{B}}c_K$ is the completion functor. This adjunction will be further explored in Section \ref{Section background HochDmod}. Let us now describe the closed symmetric monoidal structure on $\widehat{\mathcal{B}}c_K$: Consider a pair of bornological spaces $V,W\in \widehat{\mathcal{B}}c_K$. We may regard the $K$-vector space $V\otimes_KW$ as a bornological space with respect to the bornology induced by the family of subsets:
\begin{equation*}
    \{B_1\otimes_{\mathcal{R}}B_2\textnormal{ } \vert \textnormal{ } B_1\subset V, B_2\subset W \textnormal{ are bounded }\mathcal{R}\textnormal{-modules} \}.
\end{equation*}
This is called the (projective) bornological tensor product of $V$ and $W$. We define $V\widehat{\otimes}_KW$ as the completion of $V\otimes_KW$ with respect to this bornology. On the other hand, the $K$-vector space $\Hom_{\widehat{\mathcal{B}}c_K}(V,W)$ has a natural bornology with bounded subsets:
\begin{equation*}
    \{U\subset \Hom_{\widehat{\mathcal{B}}c_K}(V,W) \textnormal{ }\vert \textnormal{ for all bounded }B\subset V,\textnormal{ } \cup_{f\in U}f(B) \textnormal{ is bounded in } W \}.
\end{equation*}
We let $\underline{\Hom}_{\widehat{\mathcal{B}}c_K}(V,W)$ denote $\Hom_{\widehat{\mathcal{B}}c_K}(V,W)$ equipped with this bornology.  We remark that the space $\underline{\Hom}_{\widehat{\mathcal{B}}c_K}(V,W)$ is complete with respect to this bornology whenever $W$ is complete.
\begin{prop}
The functors $\widehat{\otimes}_K$, and $\underline{\Hom}_{\widehat{\mathcal{B}}c_K}$  make $\widehat{\mathcal{B}}c_K$   a closed symmetric monoidal category. In particular, for $U,V,W\in \widehat{\mathcal{B}}c_K$, we have an adjunction:
\begin{equation*}
    \Hom_{\widehat{\mathcal{B}}c_K}(U\widehat{\otimes}_KV,W)=\Hom_{\widehat{\mathcal{B}}c_K}(U,\underline{\Hom}_{\widehat{\mathcal{B}}c_K}(V,W)).
\end{equation*}
The unit of the closed symmetric monoidal structure of $\widehat{\mathcal{B}}c_K$ is $K$.
\end{prop}
\begin{proof}
This is \cite[Theorem 4.10]{bode2021operations}.
\end{proof}
Similarly,  $\otimes_K$ and $\underline{\Hom}_{\mathcal{B}c_K}$ give a closed symmetric monoidal structure to $\mathcal{B}c_K$. The next category of interest is the category of Ind-Banach spaces. Let us start by giving the definition:
\begin{defi}
 We define the category $\Indban$  of Ind-Banach spaces as the category with objects given by functors $V:I\rightarrow \operatorname{Ban}_K$, where $I$ is a small filtered category. We call these objects the Ind-Banach spaces. For simplicity, we will write Ind-Banach spaces as formal direct limits:
 \begin{equation*}
      ``\varinjlim"V_i.
 \end{equation*}
Morphisms in $\Indban$ are defined as follows: Given two objects $``\varinjlim"V_i$, $``\varinjlim"W_j$, we set:
\begin{equation*}
    \Hom_{\Indban}(``\varinjlim"V_i,``\varinjlim"W_j)=\varprojlim_i\varinjlim_j \Hom_{\operatorname{Ban}_K}(V_i,W_j).
\end{equation*}
\end{defi}
As before, $\Indban$ has a canonical closed symmetric monoidal structure. Let us describe the corresponding functors: Let $``\varinjlim"V_i$, and  $``\varinjlim"W_j$ be two Ind-Banach spaces. We define the tensor product of Ind-Banach spaces as follows:
\begin{equation*}
    ``\varinjlim"V_i\overrightarrow{\otimes}_K``\varinjlim"W_j= ``\varinjlim"V_i\widehat{\otimes}_KW_j,
\end{equation*}
where $V_i\widehat{\otimes}_KW_j$ denotes the complete projective tensor product 
of Banach spaces (\emph{cf}. Similarly, we define an inner hom functor via the following formula:
\begin{equation*}
    \underline{\Hom}_{\Indban}(``\varinjlim" V_i,``\varinjlim"W_j)=\varprojlim_i``\varinjlim_j " \underline{\Hom}_{\operatorname{Ban}_K}(V_i,W_j),
\end{equation*}
where $\underline{\Hom}_{\operatorname{Ban}_K}(V_i,W_j)$ represents the Banach space of continuous linear operators with the operator norm (\emph{cf}. \cite[Section 6]{schneider2013nonarchimedean}). Similar to the case of complete bornological spaces, we have the following:
\begin{prop}
The functors $\overrightarrow{\otimes}_K$, and $\underline{\Hom}_{\Indban}$  make $\Indban$ into a closed symmetric monoidal category. For $U,V,W\in \Indban$, we have an adjunction:
\begin{equation*}
    \Hom_{\Indban}(U\overrightarrow{\otimes}_KV,W)=\Hom_{\Indban}(U,\underline{\Hom}_{\Indban}(V,W)).
\end{equation*}
We point out that the unit of the closed symmetric monoidal structure of $\Indban$ is $K$.
\end{prop}
\begin{proof}
This is shown in \cite[Theorem 4.10]{bode2021operations}.
\end{proof}
As usual, we can use these closed symmetric monoidal structures to define monoids. In particular, a monoid in $\widehat{\mathcal{B}}c_K$ is a complete bornological space $\mathscr{A}\in \widehat{\mathcal{B}}c_K$, together with a pair of morphisms:
\begin{equation*}
    K\rightarrow \mathscr{A}, \, \mathscr{A}\widehat{\otimes}_K\mathscr{A}\rightarrow \mathscr{A},
\end{equation*}
satisfying the usual axioms of an algebra. We will call monoids in $\widehat{\mathcal{B}}c_K$ complete bornological algebras. Given a complete bornological algebra $\mathscr{A}$, an $\mathscr{A}$-module is a complete bornological space $\mathcal{M}$, together with a bounded map:
\begin{equation*}
    \mathscr{A}\widehat{\otimes}_K\mathcal{M}\rightarrow \mathcal{M},
\end{equation*}
satisfying the usual axioms of a module. We define $\Mod_{\widehat{\mathcal{B}}c_K}(\mathscr{A})$ as the category of all $\mathscr{A}$-modules with $\mathscr{A}$-linear maps.  There are also $\mathscr{A}$-module versions of the tensor product and inner homomorphism functor, which satisfy the usual adjunctions (\emph{cf.} \cite[Sections 2.3/2.5]{HochDmod}). We will denote the tensor product of $\mathscr{A}$-modules by:
\begin{equation*}
    -\widehat{\otimes}_{\mathscr{A}}-:\Mod_{\widehat{\mathcal{B}}c_K}(\mathscr{A}^{\op})\times \Mod_{\widehat{\mathcal{B}}c_K}(\mathscr{A})\rightarrow \widehat{\mathcal{B}}c_K,
\end{equation*}
and the inner homomorphism functor of $\mathscr{A}$-modules by:
\begin{equation*}
    \underline{\Hom}_{\mathscr{A}}(-,-):\Mod_{\widehat{\mathcal{B}}c_K}(\mathscr{A})\times \Mod_{\widehat{\mathcal{B}}c_K}(\mathscr{A})\rightarrow \widehat{\mathcal{B}}c_K.
\end{equation*}
We will give an explicit description of these operations in the following section.
Let us remark that the constructions so far make sense in any closed symmetric monoidal category. In particular, there are also monoids in $\mathcal{B}c_K$ and $\Indban$, which we call the bornological algebras and Ind-Banach algebras respectively. Given an Ind-Banach algebra $\mathscr{B}$, we denote its category of modules by $\Mod_{\Indban}(\mathscr{B})$, the tensor product of $\mathscr{B}$-modules by $\overrightarrow{\otimes}_{\mathscr{B}}$, and the inner homomorphism functor by $\underline{\Hom}_{\mathscr{B}}(-,-)$. We follow the analogous convention for bornological algebras.\\

The most attractive feature of these categories is that they have a canonical quasi-abelian category structure (\emph{cf}. \cite[Section 2.1]{HochDmod}). In practice, this means that, although they are not abelian, they are still good enough from the perspective of homological algebra. Indeed, let $\mathcal{C}$ be a category with kernels and cokernels. We say a morphism $f:X\rightarrow Y$ in $\mathcal{C}$ is strict if the induced map:
\begin{equation*}
   \operatorname{CoIm}(f) \rightarrow \operatorname{Im}(f),
\end{equation*}
is an isomorphism. A quasi-abelian category $\mathcal{C}$ is an additive category with kernels and
cokernels such that strict epimorphisms are stable under pullback, and strict monomorphisms are stable under pushout. In these conditions, we can define the derived category $\operatorname{D}(\mathcal{C})$ of $\mathcal{C}$ as the localization of the homotopy category $\operatorname{K}(\mathcal{C})$ of $\mathcal{C}$ at the full triangulated subcategory given by the strict exact complexes.
Furthermore, if $\mathcal{C}$ is a quasi-abelian category, then there is an abelian category $LH(\mathcal{C})$, which we call the left heart of $\mathcal{C}$, equipped with a fully faithful functor $I:\mathcal{C}\rightarrow LH(\mathcal{C})$ satisfying that $I$ admits a right adjoint $C:LH(\mathcal{C})\rightarrow \mathcal{C}$, and induces an equivalence of triangulated categories:
\begin{equation*}
    I:\operatorname{D}(\mathcal{C})\rightarrow \operatorname{D}(LH(\mathcal{C})),
\end{equation*}
where $\operatorname{D}(LH(\mathcal{C}))$ denotes the classical (unbounded) derived category of an abelian category. 
\begin{prop}\label{prop properties of categories of modules}
Let $\mathcal{C}$ be $\Indban$ or $\widehat{\mathcal{B}}c_K$, and let $\mathscr{A}$ be a monoid in $\mathcal{C}$. Then $\Mod_{\mathcal{C}}(\mathcal{A})$ is a quasi-abelian category. Furthermore, $LH(\Mod_{\mathcal{C}}(\mathcal{A}))$ is a Grothendieck abelian category.
\end{prop}
\begin{proof}
See \cite[Section 2.4]{HochDmod}.
\end{proof}
The idea now is extending the closed symmetric monoidal structures to the left hearts.  We will focus on the cases of $\widehat{\mathcal{B}}c_K$, and $\Indban$, which are the most relevant to our setting:
\begin{prop}\label{prop extension of closed structure}
Let $\mathcal{C}$ be $\Indban$ or $\widehat{\mathcal{B}}c_K$. The closed symmetric monoidal structure in $\mathcal{C}$ extends uniquely to a closed symmetric monoidal structure on $LH(\mathcal{C})$. We denote the corresponding functors by $\widetilde{\otimes}_{LH(\mathcal{C})}$ and $\underline{\Hom}_{LH(\mathcal{C})}$. Furthermore, the functor $I:\mathcal{C}\rightarrow LH(\mathcal{C})$ is lax symmetric monoidal, and we have the following identification of functors:
\begin{equation*}
    I(\underline{\Hom}_{\mathcal{C}}(-,-))=\underline{\Hom}_{LH(\mathcal{C})}(I(-),I(-)).
\end{equation*}
Furthermore, if $\mathcal{C}=\Indban$, then the functor $I:\Indban\rightarrow LH(\Indban)$ is strong symmetric monoidal. That is, there is an identification of functors:
\begin{equation*}
    I(-\overrightarrow{\otimes}_K-)=I(-)\widetilde{\otimes}_{LH(\Indban)}I(-).
\end{equation*}
\end{prop}
\begin{proof}
This is shown in \cite[Section 3.3]{bode2021operations}.
\end{proof}
The categories $\Indban$ and $\widehat{\mathcal{B}}c_K$ are related via the following proposition:
\begin{prop}
There is an adjunction of quasi-abelian categories:
\begin{equation*}
    \widehat{L}:\Indban \leftrightarrows \widehat{\mathcal{B}}c_K:\operatorname{diss}(-).
\end{equation*}
We call $\operatorname{diss}(-)$, the dissection functor. This adjunction satisfies the identity: $\widehat{L}\circ  \operatorname{diss}(-)=\operatorname{Id}_{\widehat{\mathcal{B}}c_K}$, and induces mutually inverse equivalences of abelian categories:
\begin{equation*} 
\widetilde{L}:LH(\Indban)\leftrightarrows LH(\widehat{\mathcal{B}}c_K):\widetilde{\operatorname{diss}}(-).
\end{equation*} 
Furthermore, $\widetilde{L}$ and $\widetilde{\operatorname{diss}}(-)$ are strong symmetric monoidal functors.
\end{prop}
\begin{proof}
See \cite[Proposition 5.15]{prosmans2000homological}, and \cite[Proposition 5.16.]{prosmans2000homological}.
\end{proof}
In order to simplify the notation, we will denote the functors of the closed symmetric monoidal structure of $LH(\widehat{\mathcal{B}}c_K)$ obtained in Proposition \ref{prop extension of closed structure} by $\widetilde{\otimes}_K$ and $\underline{\Hom}_{LH(\widehat{\mathcal{B}}c_K)}(-,-)$.\\

For some technical reasons regarding exactness of filtered colimits, it is often times better to work in the category $\Indban$. For instance, given a rigid space $X$, the category of sheaves $\operatorname{Shv}(X,\widehat{\mathcal{B}}c_K)$ is not necessarily a quasi-abelian category, but $\operatorname{Shv}(X,\Indban)$ is always a quasi-abelian category. However, $\widehat{\mathcal{B}}c_K$ is a much more explicit category. Furthermore, objects in $\widehat{\mathcal{B}}c_K$ have an underlying $K$-vector space, and this makes some of the constructions simpler. Thus, we will use this proposition to use whichever category is better for our purposes. Furthermore, as we are interested in algebras of infinite order differential operators, most of the algebras we will deal with are (bornologifications of) Fréchet algebras. In this setting, the results above are even stronger:
\begin{prop}\label{prop comparison of tensor products metrizable spaces}
Let $V,W$ be  (bornologifications of) metrizable locally convex  spaces. Then there is a canonical isomorphism in $\Indban$:
\begin{equation*}
  \operatorname{diss}(V)\overrightarrow{\otimes}_K\operatorname{diss}(W)\rightarrow \operatorname{diss}(V\widehat{\otimes}_KW).
\end{equation*}
In particular, if  $\mathscr{A}$ is the bornologification of a Fréchet K-algebra, then the dissection functor induces an exact and fully faithful functor:
     \begin{equation*}
         \operatorname{diss}_A: \Mod_{\widehat{\mathcal{B}}c_K}(\mathscr{A})\rightarrow \Mod_{\Indban}(\operatorname{diss}(\mathscr{A})).
     \end{equation*}
Furthermore, this yields an equivalence of
abelian categories:
\begin{equation*}
    LH(\Mod_{\widehat{\mathcal{B}}c_K}(\mathscr{A})) \cong LH(\Mod_{\Indban}(\operatorname{diss}(\mathscr{A}))) \cong \Mod_{LH(\widehat{\mathcal{B}}c_K)}(I(\mathscr
    A)).
\end{equation*}
\end{prop}
\begin{proof}
This is \cite[Proposition 4.25]{bode2021operations} and \cite[
Proposition 4.3]{bode2021operations}.
\end{proof}
\subsection{Sheaves of Ind-Banach spaces}\label{section sheaes of ind}
Let $X$ be a rigid space, and let $\operatorname{Shv}(X,\Indban)$ be the category of sheaves of Ind-Banach spaces on $X$. As mentioned above, this is a quasi-abelian category. Furthermore, we have an identification of abelian categories:
\begin{equation*}
    LH(\operatorname{Shv}(X,\Indban))=\operatorname{Shv}(X,LH(\Indban))=\operatorname{Shv}(X,LH(\widehat{\mathcal{B}}c_K)).
\end{equation*}
The functor $I:\operatorname{Shv}(X,\Indban)\rightarrow \operatorname{Shv}(X,LH(\widehat{\mathcal{B}}c_K))$ is defined for every sheaf of Ind-Banach spaces $\mathcal{F}$ and every open subspace $U\subset X$ by the rule:
\begin{equation*}
    I(\mathcal{F})(U)=I(\mathcal{F}(U)).
\end{equation*}
As shown in Proposition \ref{prop properties of categories of modules}, $LH(\widehat{\mathcal{B}}c_K)$ is a Grothendieck category. Hence, $\operatorname{Shv}(X,LH(\widehat{\mathcal{B}}c_K))$ is also a Grothendieck category. In particular, it has enough injectives and exact filtered colimits.\\

The closed symmetric monoidal structure of $\Indban$ extends naturally to $\operatorname{Shv}(X,\Indban)$. Namely, for $\mathcal{F},\mathcal{G}\in \operatorname{Shv}(X,\Indban)$, we let $\mathcal{F}\overrightarrow{\otimes}_K\mathcal{G}$ be the sheafification of the following presheaf:
\begin{equation*}
    U\mapsto \mathcal{F}(U)\overrightarrow{\otimes}_K\mathcal{G}(U).
\end{equation*}
In this way, we obtain a tensor product of sheaves of Ind-Banach spaces:
\begin{equation*}
    -\overrightarrow{\otimes}_K-:\operatorname{Shv}(X,\Indban)\times \operatorname{Shv}(X,\Indban)\rightarrow \operatorname{Shv}(X,\Indban),
\end{equation*}
which can be shown to be exact. Analogously, there is a tensor product: 
\begin{equation*}
 -\widetilde{\otimes}_K-:\operatorname{Shv}(X,LH(\widehat{\mathcal{B}}c_K))\times \operatorname{Shv}(X,LH(\widehat{\mathcal{B}}c_K))\rightarrow \operatorname{Shv}(X,LH(\widehat{\mathcal{B}}c_K)),  
\end{equation*}
and it follows by Proposition \ref{prop extension of closed structure} that $I(-\overrightarrow{\otimes}_K-)=I(-)\widetilde{\otimes}_KI(-)$.\\

We now construct the inner homomorphism functor. For $\mathcal{F},\mathcal{G}\in \operatorname{Shv}(X,\Indban)$, the sheaf of inner homomorphisms $\underline{\mathcal{H}om}_{\Indban}(\mathcal{F},\mathcal{G})$ is defined by setting for every admissible open $U\subset X$:
\begin{equation*}
    \underline{\mathcal{H}om}_{\Indban}(\mathcal{F},\mathcal{G})(U)=\operatorname{Eq}\left(\prod_{V\subset U}\underline{\Hom}_{\Indban}(\mathcal{F}(V),\mathcal{G}(V))\rightrightarrows \prod_{W\subset V}\underline{\Hom}_{\Indban}(\mathcal{F}(V),\mathcal{G}(W)) \right).
\end{equation*}
 The maps in the equalizer are the following: For admissible open 
subspaces $W\subset V\subset U$, a homomorphism $f\in \underline{\Hom}_{\Indban}(\mathcal{F}(V),\mathcal{G}(V))$ is mapped to the composition:
\begin{equation*}
    \mathcal{F}(V)\xrightarrow[]{f}\mathcal{G}(V)\rightarrow\mathcal{G}(W).
\end{equation*}
Similarly, every $g\in \underline{\Hom}_{\Indban}(\mathcal{F}(W),\mathcal{G}(W))$ is mapped to:
\begin{equation*}
    \mathcal{F}(V)\rightarrow\mathcal{F}(W)\xrightarrow[]{g}\mathcal{G}(W).
\end{equation*}
 We may define $\underline{\mathcal{H}om}_{LH(\widehat{\mathcal{B}}c_K)}(-,-)$ analogously, and the fact that $I$ commutes with equalizers, together with Proposition \ref{prop extension of closed structure} show that $I(\underline{\mathcal{H}om}_{\Indban}(-,-))=\underline{\mathcal{H}om}_{LH(\widehat{\mathcal{B}}c_K)}(I(-),I(-))$.\\

As above, the fact that $\operatorname{Shv}(X,\Indban)$ is a closed symmetric monoidal category allows us to define monoids in $\operatorname{Shv}(X,\Indban)$, which we call sheaves of Ind-Banach algebras on $X$. Given a sheaf of Ind-Banach algebras $\mathscr{A}$, we let $\Mod_{\Indban}(\mathscr{A})$ be the sheaf of Ind-Banach $\mathscr{A}$-modules. Analogously, if $\mathscr{B}$ is a monoid in $\operatorname{Shv}(X,LH(\widehat{\mathcal{B}}c_K))$, we let $\Mod_{LH(\widehat{\mathcal{B}}c_K)}(\mathscr{B})$ be its category of modules. We may condense the main properties of these categories into the following proposition:
\begin{prop}\label{prop properties cat of modules}
  Let $\mathscr{A}$ be a sheaf of Ind-Banach algebras on $X$. The following hold:
  \begin{enumerate}[label=(\roman*)]
      \item $\Mod_{\Indban}(\mathscr{A})$ is a quasi-abelian category.
      \item There is a canonical isomorphism of abelian categories:
      \begin{equation*}
          LH(\Mod_{\Indban}(\mathscr{A}))\cong \Mod_{LH(\widehat{\mathcal{B}}c_K)}(I(\mathscr{A})).
      \end{equation*}
      \item $\Mod_{LH(\widehat{\mathcal{B}}c_K)}(I(\mathscr{A}))$ is a Grothendieck abelian category.
  \end{enumerate}
\end{prop}
\begin{proof}
 See \cite[Proposition 2.6]{bode2021operations} and \cite[Lemma 3.7]{bode2021operations} for details.   
\end{proof} 
Fix a sheaf of Ind-Banach algebras $\mathscr{A}$ on $X$. There are relative versions of the tensor product and sheaf of inner homomorphisms for $\mathscr{A}$-modules. The tensor product of sheaves of Ind-Banach $\mathscr{A}$-modules, which we will denote by:
\begin{equation*}
    -\overrightarrow{\otimes}_{\mathscr{A}}-:\Mod_{\Indban}(\mathscr{A}^{\op})\times \Mod_{\Indban}(\mathscr{A})\rightarrow \operatorname{Shv}(X,\Indban),
\end{equation*}
is defined for $\mathcal{M}\in \Mod_{\Indban}(\mathscr{A}^{\op})$, and $\mathcal{N}\in \Mod_{\Indban}(\mathscr{A})$ by the formula:
\begin{equation*}
\mathcal{M}\overrightarrow{\otimes}_{\mathscr{A}}\mathcal{N}:=\operatorname{Coeq}\left( \mathcal{M}\overrightarrow{\otimes}_K\mathscr{A}\overrightarrow{\otimes}_K\mathcal{N}\rightrightarrows \mathcal{M}\overrightarrow{\otimes}_K\mathcal{N}\right),
\end{equation*}
where the maps in the coequalizer are given by the maps:
\begin{equation*}
T_{\mathcal{M}}:\mathcal{M}\overrightarrow{\otimes}_K\mathscr{A}\rightarrow \mathcal{M}, \quad T_{\mathcal{N}}:\mathscr{A}\overrightarrow{\otimes}_K\mathcal{N}\rightarrow \mathcal{N},
\end{equation*}
induced by the action of $\mathscr{A}$ on $\mathcal{M}$ and $\mathcal{N}$ respectively. The inner homomorphism functor for $\mathscr{A}$-modules is denoted by the following expression:
\begin{equation*}
    \underline{\mathcal{H}om}_{\mathscr{A}}(-,-):\Mod_{\Indban}(\mathscr{A})\times \Mod_{\Indban}(\mathscr{A})\rightarrow \operatorname{Shv}(X,\Indban),
\end{equation*}
and is
defined for $\mathcal{M},\mathcal{N}\in \Mod_{\Indban}(\mathscr{A})$ by the following equalizer:
\begin{equation*}
    \underline{\mathcal{H}om}_{\mathscr{A}}(\mathcal{M},\mathcal{N})=\operatorname{Eq}\left(\underline{\mathcal{H}om}_{\Indban}(\mathcal{M},\mathcal{N})\rightrightarrows \underline{\mathcal{H}om}_{\Indban}(\mathscr{A}\overrightarrow{\otimes}_K\mathcal{M},\mathcal{N})\right).
\end{equation*}
The first map in the equalizer is obtained by applying $\underline{\mathcal{H}om}_{\Indban}(-,\mathcal{N})$ to the action:
\begin{equation*}
    \mathscr{A}\overrightarrow{\otimes}_K\mathcal{M}\rightarrow \mathcal{M},
\end{equation*}
and the second is obtained by applying $\underline{\mathcal{H}om}_{\Indban}(\mathcal{M},-)$ to the canonical morphism:
\begin{equation*}
    \mathcal{N}\rightarrow \underline{\mathcal{H}om}_{\Indban}(\mathscr{A},\mathcal{N}).
\end{equation*}
Again, these constructions make sense in any additive closed symmetric monoidal category. Hence, there are relative tensor products and sheaves of inner homomorphisms for modules over monoids in $\operatorname{Shv}(X,LH(\widehat{\mathcal{B}}c_K))$, and we will use analogous naming conventions. Furthermore, we have the following identification of functors:
\begin{equation}\label{equation identity we use to derive}
    I(\underline{\mathcal{H}om}_{\mathscr{A}}(-,-))=\underline{\mathcal{H}om}_{I(\mathscr{A})}(I(-),I(-)).
\end{equation}
As mentioned above, this paper is devoted to the study of Hochschild cohomology. Thus, we  will be interested in the derived functors of $\underline{\mathcal{H}om}_{\mathscr{A}}(-,-)$. Now, the theory of derived functors in quasi-abelian categories is rather involved. In particular, the category $\Mod_{\Indban}(\mathscr{A})$ will almost never have enough injective objects, so it is difficult to formulate a good theory of derived functors. However, as seen in Proposition \ref{prop properties cat of modules}, its left heart $\Mod_{LH(\widehat{\mathcal{B}}c_K)}(I(\mathscr{A}))$ is a Grothendieck abelian category. Thus, it has enough injective objects, and left exact functors admit (unbounded) derived functors. Hence, we will use equation (\ref{equation identity we use to derive}), together with the identification of triangulated categories:
\begin{equation*}
    \operatorname{D}(\mathscr{A}):=\operatorname{D}(\Mod_{\Indban}(\mathscr{A})))\xrightarrow[]{I}\operatorname{D}(\Mod_{LH(\widehat{\mathcal{B}}c_K)}(I(\mathscr{A})))=:\operatorname{D}(I(\mathscr{A}))
\end{equation*}
to make the following definition:
\begin{equation*}
    R\underline{\mathcal{H}om}_{\mathscr{A}}(-,-):=R\underline{\mathcal{H}om}_{I(\mathscr{A})}(I(-),I(-)):\operatorname{D}(\mathscr{A})\times \operatorname{D}(\mathscr{A})\rightarrow \operatorname{D}(\operatorname{Shv}(X,\Indban)).
\end{equation*}
We remark that if $\underline{\mathcal{H}om}_{\mathscr{A}}(-,-)$ has a derived functor, then the identification of triangulated categories $I:D(\mathscr{A})\rightarrow \operatorname{D}(I(\mathscr{A}))$ and equation (\ref{equation identity we use to derive}) show that it must agree with $R\underline{\mathcal{H}om}_{I(\mathscr{A})}(I(-),I(-))$. Hence, there really is no loss in making this identification, and we will follow the same convention for the global sections functor $\Gamma(X,-)$. A more detailed discussion of these topics can be found in \cite[Section 3.6]{bode2021operations}, as well as in our companion paper \cite[Section 3.3]{HochDmod}.

\subsection{\texorpdfstring{Hochschild (co)-homology of $\wideparen{\D}_X$-modules}{}}\label{Section background HochDmod}
We will now provide an overview of the contents of \cite{HochDmod}. Let $X$ be a rigid space, and $\mathscr{A}$ be a sheaf of Ind-Banach algebras on $X$. The sheaf of enveloping algebras of $\mathscr{A}$ is the following sheaf of Ind-Banach algebras:
\begin{equation*}
    \mathscr{A}^e=\mathscr{A}\overrightarrow{\otimes}_K\mathscr{A}^{\op}.
\end{equation*}
Let us now introduce the definition of the Hochschild cohomology of $\mathscr{A}$:
\begin{defi}[{\cite[Definition 5.4.4]{HochDmod}}]\label{defi hichschild cohomology}
Let $\mathscr{A}$ be a sheaf of Ind-Banach algebras on $X$ and consider a complex $\mathscr{M}\in \operatorname{D}(\mathscr{A})$. We define the following objects: 
\begin{enumerate}[label=(\roman*)]
    \item The inner Hochschild cohomology complex of $\mathscr{A}$  with coefficients in $\mathscr{M}$  is the following complex in $\operatorname{D}(\operatorname{Shv}(X,\Indban)):$
\begin{equation*}
  \mathcal{HH}^{\bullet}(\mathscr{A},\mathscr{M})=R\underline{\mathcal{H}om}_{\mathscr{A}^e}(\mathscr{A},\mathscr{M}).
\end{equation*}
The $n$-th Hochschild cohomology sheaf of $\mathscr{A}$ with coefficients in $\mathscr{M}$ is the following sheaf:
\begin{equation*}
\mathcal{HH}^n(\mathscr{A}):=\operatorname{H}^n\left(\mathcal{HH}^{\bullet}(\mathscr{A},\mathscr{M})\right).
\end{equation*}

\item The Hochschild cohomology complex of $\mathscr{A}$ with coefficients in $\mathscr{M}$ is the following object in the derived category $\operatorname{D}(\widehat{\mathcal{B}}c_K)$:
\begin{equation*}
    \operatorname{HH}^{\bullet}(\mathscr{A},\mathscr{M})=R\Gamma(X,\mathcal{HH}^{\bullet}(\mathscr{A},\mathscr{M})).
\end{equation*}
The $n$-th Hochschild cohomology group of $\mathscr{A}$ with coefficients in $\mathscr{M}$ is the following object:
\begin{equation*}
    \operatorname{HH}^{n}(\mathscr{A},\mathscr{M})=\operatorname{H}^n\left(\operatorname{HH}^{\bullet}(\mathscr{A},\mathscr{M})\right).
\end{equation*}

\end{enumerate}
If the choice of $\mathscr{A}$ is clear, we will sometimes drop it from the notation and simply write $\mathcal{HH}^{\bullet}(\mathscr{M})$, and $\operatorname{HH}^{\bullet}(\mathscr{M})$. If $\mathscr{\mathscr{M}}=\mathscr{A}$, we will call $\mathcal{HH}^{\bullet}(\mathscr{A})$ the inner Hochschild cohomology complex of $\mathscr{A}$, and $\operatorname{HH}^{\bullet}(\mathscr{A})$ the Hochschild cohomology complex of $\mathscr{A}$. Notice that if $X=\Sp(K)$ then we have the following canonical identifications:
\begin{equation*}
    \mathcal{HH}^{\bullet}(\mathscr{A},\mathscr{\mathscr{M}})=\operatorname{HH}^{\bullet}(\mathscr{A},\mathscr{\mathscr{M}})=R\underline{\operatorname{Hom}}_{\mathscr{A}^e}(\mathscr{A},\mathscr{\mathscr{M}}),
\end{equation*}
and we call this complex the Hochschild cohomology complex of $\mathscr{A}$ with coefficients in $\mathscr{\mathscr{M}}$. 
\end{defi}
Similarly, we can make the following definition:
\begin{defi}[{\cite[Definition 5.4.7]{HochDmod}}]\label{defi Hochschild homology}
  Let $\mathscr{A}$ be a sheaf of Ind-Banach algebras on $X$, and consider a complex $\mathscr{M}\in \operatorname{D}(\mathscr{A})$. We define the following objects:  
\begin{enumerate}[label=(\roman*)]
    \item The inner Hochschild cohomology complex of $\mathscr{A}$ with coefficients in $\mathscr{M}$ is the following complex in $\operatorname{D}(\operatorname{Shv}(X,\Indban)):$
\begin{equation*}
  \mathcal{HH}_{\bullet}(\mathscr{A},\mathscr{M})=\mathscr{A}\overrightarrow{\otimes}^{\mathbb{L}}_{\mathscr{A}^e}\mathscr{M}.
\end{equation*}
The $n$-th Hochschild homology sheaf of $\mathscr{A}$ with coefficients in $\mathscr{M}$ is the following sheaf:
\begin{equation*}
    \mathcal{HH}_n(\mathscr{A},\mathscr{M}):=\operatorname{H}^{-n}\left(\mathcal{HH}_{\bullet}(\mathscr{A},\mathscr{M})\right).
\end{equation*}

\item The Hochschild homology complex of $\mathscr{A}$ with coefficients in $\mathscr{M}$ is the following complex in the derived category $\operatorname{D}(\widehat{\mathcal{B}}c_K)$:
\begin{equation*}
    \operatorname{HH}_{\bullet}(\mathscr{A},\mathscr{M})=R\Gamma(X,\mathcal{HH}_{\bullet}(\mathscr{A},\mathscr{M})).
\end{equation*}
The $n$-th Hochschild homology group of $\mathscr{A}$ with coefficients in $\mathscr{M}$ is the following object:
\begin{equation*}
    \operatorname{HH}_{n}(\mathscr{A},\mathscr{M})=\operatorname{H}^{-n}\left(\operatorname{HH}_{\bullet}(\mathscr{A},\mathscr{M})\right).
\end{equation*}
\end{enumerate}
We fill follow the analogous notational conventions as with Hochschild cohomology.
\end{defi}
Let $X$ be a smooth and separated rigid analytic space. Most of the contents of \cite{HochDmod} are aimed at calculating the inner Hochschild cohomology complex of the sheaf of infinite order differential operators $\wideparen{\D}_X$. These efforts can be summarized into the following theorem:
\begin{teo}[{\cite[Theorem 8.1.6]{HochDmod}}]\label{teo HochDmod main teo}
 There is an isomorphism in $\operatorname{D}(\operatorname{Shv}(X,\Indban)):$
 \begin{equation*}
     \mathcal{HH}^{\bullet}(\wideparen{\D}_X)=\Omega_{X/K}^{\bullet}.
 \end{equation*}
 In particular, we have $\operatorname{HH}^{\bullet}(\wideparen{\D}_X)=R\Gamma(X,\Omega_{X/K}^{\bullet})$ as objects in $\operatorname{D}(\widehat{\mathcal{B}}c_K)$.
\end{teo}
Notice that the above expression does not mean that 
$\operatorname{HH}^{\bullet}(\wideparen{\D}_X)$ agrees with the de Rham cohomology of $X$. Indeed, the previous theorem shows that $\operatorname{HH}^{\bullet}(\wideparen{\D}_X)$ agrees with the hypercohomology of $\Omega_{X/K}^{\bullet}$ viewed as an object in $\operatorname{D}(\operatorname{Shv}(X,\Indban))$. In particular, the cohomology groups of $\operatorname{HH}^{\bullet}(\wideparen{\D}_X)$ are objects in  $LH(\widehat{\mathcal{B}}c_K)$, and it is not even clear that they are contained in $\widehat{\mathcal{B}}c_K$.\\

On the other hand, the de Rham cohomology $\operatorname{H}^{\bullet}_{\operatorname{dR}}(X)$ is defined as the hypercohomology of $\Omega_{X/K}^{\bullet}$ viewed as a complex of sheaves of $K$-vector spaces. Hence, the de Rham cohomology groups of $X$ are $K$-vector spaces. This discussion shows that these two notions are different. However, the similarity in the formulas suggest that both cohomologies should be related. In order to understand this relation, we will need some technical results from \cite[Section 8.1]{HochDmod}.\\

First, recall that there is an adjunction:
\begin{equation*}
    J:\widehat{\mathcal{B}}c_K\leftrightarrows \mathcal{B}c_K:\widehat{\operatorname{Cpl}},
\end{equation*}
where $J:\widehat{\mathcal{B}}c_K\rightarrow \mathcal{B}c_K$ denotes the inclusion, and $\widehat{\operatorname{Cpl}}:\mathcal{B}c_K\rightarrow \widehat{\mathcal{B}}c_K$ denotes the completion of bornological spaces. This extends to an adjunction between the associated left hearts:
\begin{equation*}
    J:LH(\widehat{\mathcal{B}}c_K)\leftrightarrows LH(\mathcal{B}c_K): \operatorname{H}^0(L\widehat{\operatorname{Cpl}}).
\end{equation*}
Furthermore, the functor $J:LH(\widehat{\mathcal{B}}c_K)\rightarrow LH(\mathcal{B}c_K)$
is exact, and the previous adjunction showcases that $LH(\widehat{\mathcal{B}}c_K)$ is a reflexive abelian subcategory of $LH(\mathcal{B}c_K)$. Similarly, we have a forgetful functor:
\begin{equation*}
    \operatorname{Forget}(-):\mathcal{B}c_K\rightarrow \operatorname{Vect}_K.
\end{equation*}
By construction of the quasi-abelian structure on $\mathcal{B}c_K$, this functor is strongly exact. In particular, it preserves arbitrary kernels and cokernels. As $\operatorname{Vect}_K$ is an abelian category, it is canonically isomorphic to its left heart. Thus, we get a unique extension:
\begin{equation*}  
\widetilde{\operatorname{Forget}}(-):LH(\mathcal{B}c_K)\rightarrow LH(\operatorname{Vect}_K)=\operatorname{Vect}_K.
\end{equation*}
The fact that $\operatorname{Forget}(-)$ is strongly exact implies that $\widetilde{\operatorname{Forget}}(-)$ is exact. Hence, in order to simplify notation, we will denote $\widetilde{\operatorname{Forget}}(-)$ simply by $\operatorname{Forget}(-)$. In \cite{HochDmod}, we discussed an extension of these functors to the corresponding categories of sheaves. In particular, we have the following definition:
\begin{defi}[{\cite[Definition 8.1.8]{HochDmod}}]\label{defi J and Forget for sheaves}
We define the following functor:
    \begin{equation*}
        J:\operatorname{Shv}(X,LH(\widehat{\mathcal{B}}c_K))\rightarrow \operatorname{Shv}(X,LH(\mathcal{B}c_K)), \quad \mathcal{F}\mapsto J(\mathcal{F}),
    \end{equation*}
    where $J(\mathcal{F})$ satisfies $J(\mathcal{F})(U)=J(\mathcal{F}(U))$ for every admissible open $U\subset X$. Analogously, we define:
    \begin{equation*}
       \operatorname{Forget}(-):\operatorname{Shv}(X,LH(\mathcal{B}c_K))\rightarrow \operatorname{Shv}(X,\operatorname{Vect}_K), \quad \mathcal{G}\mapsto \operatorname{Forget}(\mathcal{G}),
    \end{equation*}
where $\operatorname{Forget}(\mathcal{G})$ satisfies $\operatorname{Forget}(\mathcal{G})(U)=\operatorname{Forget}(\mathcal{G}(U))$ for every admissible open $U\subset X$.
\end{defi}
Notice that these functors are well defined because both $J$ and $\operatorname{Forget}(-)$ are exact. Furthermore, by \cite[Lemma 8.1.10]{HochDmod}, these functors are in fact exact. This can be used to show the following:
\begin{prop}[{\cite[Corollary 8.1.12]{HochDmod}}]
We have the following identity in $\operatorname{D}(\operatorname{Vect}_K)$:
\begin{equation*}
    \operatorname{H}_{\operatorname{dR}}^{\bullet}(X)=\operatorname{Forget}(J\operatorname{HH}^{\bullet}(\wideparen{\D}_X)).
\end{equation*}
\end{prop}
Besides this result, the functors $J$ and $\operatorname{Forget}(-)$ will also be used in Chapter \ref{Chapter applications of HC}, as a tool to compare Hochschild cohomology of $\wideparen{\D}_X$ with different types of Hom functors.\\

Before moving on, let us recall one of the most important technical tools of \cite{HochDmod}. Namely, the sheaf of complete bi-enveloping algebras of $X$. This sheaf will be of the utmost importance in Section \ref{Section HC and ext 2}. Let $X^2=X\times_{\Sp(K)}X$, and denote the projections by $p_i:X^2\rightarrow X$, for $i=1,2$. We define the sheaf of bi-enveloping algebras of $X$ as the following sheaf of Ind-Banach algebras on $X^2$:
\begin{equation*}
    \wideparen{E}_X:=p_1^{-1}(\wideparen{\D}_X)\overrightarrow{\otimes}_Kp_2^{-1}(\wideparen{\D}_X^{\op}).
\end{equation*}
Our interest in $\wideparen{E}_X$ stems from the following properties:
\begin{teo}[{\cite[Theorem A]{HochDmod}}]\label{teo  properties of EX}
The sheaf $\wideparen{E}_X$ satisfies the following properties:
\begin{enumerate}[label=(\roman*)]
    \item $\Delta^{-1}\wideparen{E}_X=\wideparen{\D}_X^e$.
    \item If $\mathcal{T}_{X/K}$ is free, there is an isomorphism of sheaves of  Ind-Banach algebras:
    \begin{equation*}
        \wideparen{\mathbb{T}}:\wideparen{\D}_{X^2}\rightarrow \wideparen{E}_X.
    \end{equation*}
    \item There are mutually inverse equivalences of quasi-abelian categories:
    \begin{equation*}
        \operatorname{S}:\Mod_{\Indban}(\wideparen{\D}_{X^2})\leftrightarrows \Mod_{\Indban}(\wideparen{E}_{X}):\operatorname{S}^{-1},
    \end{equation*}
    which we call the side-switching equivalence for bimodules.  
\end{enumerate}
If the conditions of $(iii)$ hold, then $\operatorname{S}$ agrees with the extension of scalars along $\wideparen{\mathbb{T}}$. 
\end{teo}
Statement $(i)$ in the preceding theorem allows us to regard the derived category $\operatorname{D}(\wideparen{\D}_X^e)$ as a full subcategory of  $\operatorname{D}(\wideparen{E}_X)$, which is in turn equivalent to $\operatorname{D}(\wideparen{\D}_{X^2})$ by $(iii)$. Roughly speaking, this implies that we can reduce problems about $\wideparen{\D}_X^e$-modules to problems about $\wideparen{E}_X$-modules, and then use the side-switching equivalence for bimodules to treat the problem in the category of $\wideparen{\D}_{X^2}$-modules. Finally, let us also point out that statement $(ii)$ can be used to define a category of co-admissible $\wideparen{E}_X$-modules, which we denote $\mathcal{C}(\wideparen{E}_X)$, and that the side-switching equivalence for bimodules restricts to an equivalence:
\begin{equation*}
    \operatorname{S}:\mathcal{C}(\wideparen{\D}_{X^2})\leftrightarrows \mathcal{C}(\wideparen{E}_{X}):\operatorname{S}^{-1}.
\end{equation*}
In particular, the contents of Section \ref{section co-ad D mod stein spaces} will also hold in the category of co-admissible $\wideparen{E}_X$-modules.
\subsection{The external sheaf of homomorphisms}\label{section external hom functors}
In most of our previous developments, we have worked with the inner \emph{Hom} functors, $\underline{\mathcal{H}om}_{\wideparen{\D}_X}(-,-)$, and $\underline{\Hom}_{\wideparen{\D}_X(X)}(-,-)$. These are obtained from the closed symmetric monoidal structures on $\operatorname{Shv}(X,\Indban)$ and $\Indban$ respectively. However, we can use the fact that $\Indban$ and $\operatorname{Shv}(X,\Indban)$ are additive categories to obtain a different notion of \emph{Hom} functors, which only take into account the additive (in fact $K$-linear) structures on the aforementioned categories. Let us start with the following definition:
\begin{defi}
A category $\mathcal{C}$ is said to be $K$-linear, if there is a factorization of its \emph{Hom} functor:
\begin{equation*}
    \Hom_{\mathcal{C}}(-,-):\mathcal{C}^{\op}\times \mathcal{C}\rightarrow \operatorname{Vect}_K.
\end{equation*}
\end{defi}
Let $``\varinjlim"V_i,"\varinjlim"W_j\in \Indban$ be a pair of Ind-Banach spaces. By definition, morphisms in Ind-Ban are given by the following expression:
\begin{equation*}
    \Hom_{\Indban}(``\varinjlim"V_i,``\varinjlim"W_j)=\varprojlim_i\varinjlim_j \Hom_{\operatorname{Ban}_K}(V_i,W_j),
\end{equation*}
which is clearly a $K$-vector space. In particular, we have a well-defined $K$-linear structure:
\begin{equation*}
    \Hom_{\Indban}(-,-):\Indban^{\op}\times \Indban \rightarrow \operatorname{Vect}_K.
\end{equation*}
Similarly, let $\mathscr{A}$ be an Ind-Banach algebra, and $\mathcal{M},\mathcal{N}$ be a pair of Ind-Banach $\mathscr{A}$-modules. By definition, we have the following identity:
\begin{equation*}
    \Hom_{\mathscr{A}}(\mathcal{M},\mathcal{N}):=\operatorname{Eq}\left(\Hom_{\Indban}(\mathcal{M},\mathcal{N})\rightrightarrows \Hom_{\Indban}(\mathscr{A}\overrightarrow{\otimes}_K\mathcal{M},\mathcal{N}) \right),
\end{equation*}
which shows that $\Hom_{\mathscr{A}}(\mathcal{M},\mathcal{N})$ is canonically a $K$-vector space. Again, this allows us to obtain a canonical $K$-linear structure on $\Mod_{\Indban}(\mathscr{A})$:
\begin{equation*}
    \Hom_{\mathscr{A}}(-,-):\Mod_{\Indban}(\mathscr{A})^{\op}\times \Mod_{\Indban}(\mathscr{A})\rightarrow \operatorname{Vect}_K.
\end{equation*}
The idea is extending this structure to the left heart $LH(\Mod_{\Indban}(\mathscr{A}))$:
\begin{Lemma}
Let $\mathcal{E}$ be a $K$-linear quasi-abelian category with enough projectives. Then $LH(\mathcal{E})$ admits a unique $K$-linear structure extending the one on $\mathcal{E}$.
\end{Lemma}
\begin{proof}
Denote the inclusion of $\mathcal{E}$ into its left heart by $I:\mathcal{E}\rightarrow LH(\mathcal{E})$. By construction, $LH(\mathcal{E})$ is an abelian category. In particular, it is additive. Hence, its \emph{Hom} functor satisfies: 
\begin{equation*}
    \Hom_{LH(\mathcal{E})}(-,-):LH(\mathcal{E})^{\op}\times LH(\mathcal{E})\rightarrow \operatorname{Ab}.
\end{equation*}
Furthermore, $\mathcal{E}$ is also additive, and by construction of $LH(\mathcal{E})$ we have:
\begin{equation*}
    \Hom_{\mathcal{E}}(-,-)=\Hom_{LH(\mathcal{E})}(I(-),I(-)),
\end{equation*}
as functors with values on abelian groups. Let $V,W\in LH(\mathcal{E})$. As $\mathcal{E}$ has enough projective objects, it follows by \cite[Proposition 1.3.24]{schneiders1999quasi} that there
are projective objects $P_0,P_1\in\mathcal{E}$ such that we have the following right exact sequence in $LH(\mathcal{E})$:
\begin{equation*}
    I(P_1)\rightarrow I(P_0)\rightarrow V.
\end{equation*}
Furthermore, all the  $I(P_i)$ is projective in $LH(\mathcal{E})$ for $i=0,1$. In particular, we obtain the following left exact complex of abelian groups:
\begin{equation}\label{equation K-linear structure first res}
    0\rightarrow \Hom_{LH(\mathcal{E})}(V,W)\rightarrow \Hom_{LH(\mathcal{E})}(I(P_0),W) \rightarrow \Hom_{LH(\mathcal{E})}(I(P_1),W).
\end{equation}
By construction of $LH(\mathcal{E})$, there are objects $X,Y\in \mathcal{E}$ which fit in a short exact sequence:
\begin{equation*}
    0\rightarrow I(X)\rightarrow I(Y)\rightarrow W\rightarrow 0.
\end{equation*}
As the $I(P_i)$ are projective in $LH(\mathcal{E})$, we obtain the following commutative diagram :
\begin{equation}\label{equation K-linear structure first res 2}
\begin{tikzcd}
0 \arrow[r] & {\Hom_{\mathcal{E}}(P_0,X)} \arrow[r] \arrow[d] & {\Hom_{\mathcal{E}}(P_0,Y)} \arrow[r] \arrow[d] & {\Hom_{LH(\mathcal{E})}(I(P_0),W)} \arrow[r] \arrow[d] & 0 \\
0 \arrow[r] & {\Hom_{\mathcal{E}}(P_1,X)} \arrow[r]           & {\Hom_{\mathcal{E}}(P_1,Y)} \arrow[r]           & {\Hom_{LH(\mathcal{E})}(I(P_1),W)} \arrow[r]           & 0
\end{tikzcd}
\end{equation}
where the rows are short exact sequences of abelian groups. As $\mathcal{E}$ is $K$-linear, the first two terms in each
row are $K$-vector spaces, the first map in each row is $K$-linear, and the first two vertical maps are also $K$-linear. Thus, there are unique $K$-vector space structures on $\Hom_{LH(\mathcal{E})}(I(P_i),W)$ for $i=0,1$ making the above commutative diagram a commutative diagram of $K$-vector spaces. Using this structure, it follows that there is a unique $K$-linear structure on $\Hom_{LH(\mathcal{E})}(V,W)$ making $(\ref{equation K-linear structure first res})$ a left exact sequence of $K$-vector spaces. Thus, we obtain a unique $K$-linear structure on $LH(\mathcal{E})$.
\end{proof}
By \cite[Lemma 3.7]{bode2021operations}, the categories $\Mod_{\Indban}(\mathscr{A})$ are elementary. Hence, we obtain a canonical $K$-linear structure on $LH(\Mod_{\Indban}(\mathscr{A}))=\Mod_{LH(\widehat{\mathcal{B}}c_K)}(I(\mathscr{A}))$, which we denote by:
\begin{equation*}
    \Hom_{I(\mathscr{A})}(-,-):LH(\Mod_{\Indban}(\mathscr{A}))\times LH(\Mod_{\Indban}(\mathscr{A}))\rightarrow \operatorname{Vect}_K.
\end{equation*}
As usual, this $K$-linear structure extends to the categories of Ind-Banach modules over a sheaf of Ind-Banach algebras on a rigid space $X$. We may use this structure to define the next \emph{Hom} functor:
\begin{defi}\label{defi external sheaf of hom}
Let $X$ be a rigid space and $\mathscr{A}$ be a sheaf of Ind-Banach algebras on $X$. We define the (external) sheaf of homomorphisms functor of $\mathscr{A}$ as follows:
\begin{equation*}
    \mathcal{H}om_{\mathscr{A}}(-,-):\Mod_{\Indban}(\mathscr{A})^{\op}\times \Mod_{\Indban}(\mathscr{A})\rightarrow \operatorname{Shv}(X,\operatorname{Vect}_K),
\end{equation*}
for $\mathcal{M},\mathcal{N}\in \Mod_{\Indban}(\mathscr{A})$,  and every admissible open $U\subset X$ we have:
\begin{equation*}
 \Gamma(U,\mathcal{H}om_{\mathscr{A}}(\mathcal{M},\mathcal{N}))= \operatorname{Hom}_{\mathscr{A}_{\vert U}}(\mathcal{M}_{\vert U},\mathcal{N}_{\vert U})),
\end{equation*}
and restriction maps given by restriction of morphisms of sheaves. More generally, if $\mathscr{B}$ is a sheaf of algebras in $\operatorname{Shv}(X,LH(\widehat{\mathcal{B}}c_K))$, we may repeat this construction to define a functor:
\begin{equation*}
    \mathcal{H}om_{\mathscr{B}}(-,-): \Mod_{LH(\widehat{\mathcal{B}}c_K)}(\mathscr{B})\times \Mod_{LH(\widehat{\mathcal{B}}c_K)}(\mathscr{B})\rightarrow \operatorname{Shv}(X,\operatorname{Vect}_K). 
\end{equation*}
\end{defi}
If $X=\Sp(K)$, we obtain $\mathcal{H}om_{\mathscr{A}}(-,-)=\Hom_{\mathscr{A}}(-,-)$, so we recover the functor defined above. Furthermore, given a sheaf of Ind-Banach algebras $\mathscr{A}$, we have the following identity of functors:
\begin{equation}\label{equation external sheaf of hom}
    \mathcal{H}om_{\mathscr{A}}(-,-)=\mathcal{H}om_{I(\mathscr{A})}(I(-),I(-)).
\end{equation}
As usual, the fact that $\Mod_{\Indban}(\mathscr{A})$ does not have enough injectives makes is unclear if $\mathcal{H}om_{\mathscr{A}}(-,-)$ admits a derived functor. However, we can use the fact that $\Mod_{LH(\widehat{\mathcal{B}}c_K)}(I(\mathscr{A}))$ is a Grothendieck abelian category, together with identity (\ref{equation external sheaf of hom}) to define:
\begin{equation*}
    R\mathcal{H}om_{\mathscr{A}}(-,-):=R\mathcal{H}om_{I(\mathscr{A})}(-,-):\operatorname{D}(\mathscr{A})^{\op}\times \operatorname{D}(\mathscr{A}) \rightarrow \operatorname{D}(\operatorname{Shv}(X,\operatorname{Vect}_K)).
\end{equation*}
Similarly, we define the following derived functor:
\begin{equation*}
    R\Hom_{\mathscr{A}}(-,-):=R\Hom_{I(\mathscr{A})}(-,-):\operatorname{D}(\mathscr{A})^{\op}\times \operatorname{D}(\mathscr{A}) \rightarrow \operatorname{D}(\operatorname{Vect}_K).
\end{equation*}
The interest of these \emph{Hom} functors stems from the following observation:\newline
Choose $\mathcal{M},\mathcal{N}\in \Mod_{\Indban}(\mathscr{A})$. As $\Mod_{LH(\widehat{\mathcal{B}}c_K)}(I(\mathscr{A}))$ is a Grothendieck abelian category, it follows by \cite[Tag 06XP]{stacks-project} that for every $n\geq 0$ the \emph{Ext} group:
\begin{equation*}
\operatorname{Ext}_{\mathscr{A}}^n(\mathcal{M},\mathcal{N}):=R^n\Hom_{\mathscr{A}}(\mathcal{M},\mathcal{N}),
\end{equation*}
parametrizes  the isomorphism classes of $n$-th degree Yoneda's extensions of $I(\mathcal{N})$ by $I(\mathcal{M})$ as $I(\mathscr{A})$-modules.
This perspective will be further explored in Section \ref{Section Hochschild cohomology and extension I}.
\section{Stein Spaces}\label{Chapter Stein spaces}
As mentioned in the introduction, the goal of this paper is giving explicit applications of our theory of Hochschild (co)-homology of $\wideparen{\D}$-modules, and justifying its usefulness as a tool in the study of rigid analytic spaces and $p$-adic geometric representation theory. Let $X$ be a smooth rigid analytic space. By construction, the Hochschild cohomology spaces $\operatorname{HH}^{\bullet}(\wideparen{\D}_X)$ are geometric invariants of $X$. As with any geometric invariant, it is of the utmost importance to determine a class of spaces for which it is best behaved, and show that this class is wide enough to contain interesting examples.\\

In our case, this role will be played by the smooth Stein spaces. These spaces were first introduced into the rigid analytic setting by R. Kiehl in \cite{kiehl1967theorem}. Roughly speaking, a Stein space is a rigid analytic space $X$ equipped with a strictly increasing affinoid cover:
\begin{equation*}
   X_0\subset \cdots \subset X_{n}\subset X_{n+1}\subset \cdots \subset X,
\end{equation*}
satisfying that the identity $\OX_X(X)=\varprojlim_n \OX_X(X_n)$ makes $\OX_X(X)$ a nuclear Fréchet space. From the algebraic point of view, these spaces are specially well-behaved. Namely, the category of coherent $\OX_X$-modules $\operatorname{Coh}(\OX_X)$ behaves very much like the category of coherent modules on an affinoid space. In particular, coherent modules have trivial sheaf cohomology, and each module is completely determined by its global sections. If $X$ is smooth and satisfies some mild technical conditions, we will show that these properties also hold for the category of co-admissible $\wideparen{\D}_X$-modules. Furthermore, an even stronger statement holds. Namely, we will show below that $\wideparen{\D}_X(X)$ is a Fréchet-Stein algebra, and that the category of co-admissible $\wideparen{\D}_X$-modules is equivalent to the category of co-admissible $\wideparen{\D}_X(X)$-modules. From the analytic point of view, Stein spaces are also specially well-behaved. To wit, we will show that the global sections of objects in $\operatorname{Coh}(\OX_X)$ and $\mathcal{C}(\wideparen{\D}_X)$ are canonically nuclear Fréchet spaces. Furthermore, as will be explored in Chapter \ref{Chapter applications of HC}, the de Rham complex of $X$ is  canonically a strict complex of nuclear Fréchet spaces. This fact will be used in later sections to obtain explicit interpretations of the cohomology groups of $\operatorname{HH}^{\bullet}(\wideparen{\D}_X)$.\\

For use in later sections, we also include an account of the properties of finite group actions on Stein spaces. Namely, if $G$ is a finite group acting on a Stein space $X$, we will show that the fixed point locus of the action $X^G$ is always a Stein space. Similarly, we will show that the categorical quotient always $X/G$ exists as a rigid space, and it is a Stein space.\\

Finally, let us mention that the class of smooth Stein spaces includes many interesting examples. For instance, the rigid analytification of a smooth affine $K$-variety is a smooth Stein space. Thus, every projective smooth rigid space is covered by finitely many smooth Stein spaces. Furthermore, the $p$-adic symmetric spaces covered by P. Schneider and U. Stuhler in \cite{schneider1991cohomology} are smooth Stein spaces.

\subsection{Definition and basic properties}\label{section group act on Stein spaces}
We start the chapter by giving an introduction to the theory of $p$-adic Stein spaces. In particular, we will recall some of their geometric properties and study the category of coherent modules on a Stein space. Let us start by recalling the definition: 
\begin{defi}\label{defi Stein spaces}
We define the following objects:
\begin{enumerate}[label=(\roman*)]
    \item Let $V\subset U$ be an open immersion of affinoid spaces. We say $V$ is relatively compact in $U$, and write $V \Subset U$, if there is a set of affinoid generators $f_1,\cdots, f_n\in \OX_U(U)$ such that:
    \begin{equation*}
        V\subset \{ x\in U \textnormal{ } \vert \textnormal{ } \vert f_i(x)\vert < 1, \textnormal{ for }1\leq i\leq n\}.
    \end{equation*}
    \item A rigid space $X$ is called a quasi-Stein space if it admits an admissible affinoid cover $(U_i)_{i\in\mathbb{N}}$ such that $U_{i}\subset U_{i+1}$ is a Weierstrass subdomain. 
    \item We say $X$ is a Stein space if, in addition, we have $U_i\Subset U_{i+1}$ for all $i\geq 0$. We call $(U_i)_{i\in\mathbb{N}}$ a Stein covering of $X$, and write $X=\varinjlim_iU_i$.
    \end{enumerate}
\end{defi}
Notice that any Stein space is partially proper. In particular, it is separated. 
\begin{obs}
We point out that there are more general notions of Stein space available in the literature (cf. \cite{non-archSteinspaces}). These notions are centered around cohomological properties of coherent sheaves on $X$. Our focus, however, is on the analytical properties of the global sections, and of the transition maps of the Stein covering. Hence, we will stick to the classical definition stated above.
\end{obs}
Let $X$ be a Stein space. It follows from the definition that given a Stein covering $X=\varinjlim_nX_n$ then for each $n\geq 1$ there are affinoid generators $f_1,\cdots,f_{r_n}\in \OX_X(X_n)$, and some real number $\epsilon \in \sqrt{\vert K^*\vert}$, satisfying that  $0< \epsilon <1$ and such that we have:
\begin{equation}\label{equation special Stein covering}
    X_{n-1}\subseteq X_n\left(\epsilon^{-1} f_1,\cdots, \epsilon^{-1} f_{r_n}\right).
\end{equation}
If the identity holds, we recover the definition of a Stein space given by R. Kiehl in \cite[Definition 2.3]{kiehl1967theorem}.\\

Our interest in Stein spaces stems from the fact that they satisfy algebraic properties which are very similar to those of affinoid spaces, but have much better analytic properties. For instance, if $X$ is a smooth Stein space, then 
the complex $\Omega_{X/K}^{\bullet}(X):=\Gamma(X,\Omega_{X/K}^{\bullet})$ is a strict complex of Fréchet $K$-vector spaces, and its cohomology calculates the de Rham cohomology of $X$. This particular property will be exploited in the next sections. We will now give an overview of some of the most relevant features of Stein spaces. Recall the following definitions:
\begin{defi}[{\cite[Lemma 5.2]{bode2021operations}\cite[Definition 1.1]{kiehl analysis}}]
 A bounded map of Banach spaces $f:V\rightarrow W$ is called strictly completely continuous morphism (\emph{scc} morphism) if there is a sequence of morphisms:
 \begin{equation*}
     f_n:V\rightarrow W,
 \end{equation*}
such that $f=\varinjlim_{n\geq 0} f_n$ in operator norm and such that $f_n(V)$ is finite-dimensional for all $n\geq 0$.
\end{defi}
\begin{defi}[{\cite[Definition 5.3]{bode2021operations}}]
Let $V$ be a Fréchet space. We say $V$ is a nuclear Fréchet space if there is a Fréchet space $\varprojlim_n V_n$ satisfying the following:
\begin{enumerate}[label=(\roman*)]
    \item $V=\varprojlim_{n}V_n$ as Fréchet spaces.
    \item The image of $V$ in $V_n$ is dense for each $n\geq 0$.
    \item For each $n\geq 1$ there is a Banach $K$-vector space $F_n$ and a strict surjection:
    \begin{equation*}
        F_n\rightarrow V_n,
    \end{equation*}
    such that the composition $F_n\rightarrow V_n\rightarrow V_{n-1}$ is $\emph{scc}$.
\end{enumerate}
\end{defi}

Let us start our study of Stein spaces with the following lemma:
\begin{Lemma}\label{Lemma basic properties of Stein spaces}
Let $X=\varinjlim_nX_n$ be a Stein space. The following hold:
\begin{enumerate}[label=(\roman*)]
    \item $\Gamma(X,\OX_X)=\varprojlim_n \Gamma(X_n,\OX_X)$ is a Fréchet-Stein presentation.
    \item $\Gamma(X,\OX_X)$ is a nuclear Fréchet $K$-space. 
    \item Any closed analytic subspace $Y\rightarrow X$ is a Stein space.
    \item  If $Y$ is contained in an affinoid space  then $\operatorname{dim}(Y)=0$. 
\end{enumerate}
\end{Lemma}
\begin{proof}
Statement $(i)$ follows directly from the definitions. For statement $(ii)$, we may assume for simplicity that the Stein covering $X=\varinjlim_n X_n$ satisfies the conditions in equation $(\ref{equation special Stein covering})$. Fix some $n\geq 1$, and consider the surjection $\mathbb{B}^{r_n}_K\rightarrow X_{n}$ given by $t_i\mapsto f_i$. We have a pullback diagram:
\begin{equation*}
\begin{tikzcd}
X_n \arrow[r]               & \mathbb{B}^{r_n}_K                                                                     \\
X_{n-1} \arrow[r] \arrow[u] & {\mathbb{B}^{r_n}_K\left(\epsilon^{-1}t_1,\cdots \epsilon^{-1}t_{r_n} \right)} \arrow[u]
\end{tikzcd}
\end{equation*}
such that the horizontal maps are closed immersions, and the vertical maps correspond to immersions of Weierstrass subdomains. It suffices to show that the associated map:
\begin{equation*}
    K\langle t_1,\cdots, t_{r_n} \rangle\rightarrow K\langle \epsilon^{-1}t_1,\cdots, \epsilon^{-1}t_{r_n} \rangle,
\end{equation*}
is \emph{scc}. If $\epsilon \in \vert K^*\vert$, then this holds by \cite[pp. 63]{bode2021operations}. In the general case, we may choose a finite extension $L$ of $K$ such that $\epsilon \in \vert L^*\vert$. In this case, we have the following pushout diagram:
\begin{equation*}
\begin{tikzcd}
{K\langle t_1,\cdots, t_{r_n} \rangle} \arrow[r] \arrow[d]                 & {L\langle t_1,\cdots, t_{r_n} \rangle} \arrow[d]                 \\
{K\langle \epsilon^{-1}t_1,\cdots, \epsilon^{-1}t_{r_n} \rangle} \arrow[r] & {L\langle \epsilon^{-1}t_1,\cdots, \epsilon^{-1}t_{r_n} \rangle}
\end{tikzcd}
\end{equation*}
such that the horizontal maps are finite, and the rightmost vertical map is \emph{scc}. Again, the arguments in \cite[pp. 63]{bode2021operations} show that 
$K\langle t_1,\cdots, t_{r_n} \rangle\rightarrow K\langle \epsilon^{-1}t_1,\cdots, \epsilon^{-1}t_{r_n} \rangle$ is \emph{scc}. Therefore, $\Gamma(X,\OX_X)$ is a nuclear Fréchet $K$-space, as we wanted to show.\\

Let $Y\rightarrow X$ be a closed immersion, and let $\mathcal{I}_Y$ be the associated sheaf of ideals. As $\mathcal{I}_Y$ is a coherent sheaf on $X$, it follows that 
$\mathcal{I}_Y(X)$ is a co-admissible $\OX_X(X)$-module. Thus, we have:
\begin{equation*}
    \OX_Y(Y)=\OX_X(X)/\mathcal{I}_Y(X)=\varprojlim_n \OX_X(X_n)/\mathcal{I}_Y(X_n).
\end{equation*}
Letting $Y_n=\Sp(\OX_X(X_n)/\mathcal{I}_Y(X_n))$, it follows that $Y=\varinjlim_n Y_n$ is a Stein covering of $Y$. Assume that $Y$ is contained in an affinoid subdomain. Then $Y$ itself is an affinoid space. Hence,  $\OX_Y(Y)$
is a Banach $K$-algebra, and a nuclear Fréchet space.  Notice that nuclear Fréchet spaces are reflexive by \cite[Corollary 19.3]{schneider2013nonarchimedean}. However, by \cite[Proposition 11.1]{schneider2013nonarchimedean} a Banach $K$-vector space is reflexive if and only if it is finite-dimensional. Hence, $\operatorname{dim}(Y)=0$, as we wanted to show.
\end{proof}
In particular, if $X$ is a smooth Stein space with an action of a finite group $G$, this implies that the fixed point locus $X^G$ is also a Stein space. Furthermore, by \cite[Proposition 2.1.9]{p-adicCheralg} it follows that $X^G$ is also smooth. We will now analyze the category of coherent sheaves on a Stein space:
\begin{prop}\label{prop coherent modules on a Stein sapce}
Let $X=\varinjlim_n X_n$ be a Stein space, and let $\operatorname{Coh}(X)$ be the category of coherent sheaves on $X$. The following hold:
\begin{enumerate}[label=(\roman*)]
    \item Every $\mathcal{M}\in \operatorname{Coh}(X)$ satisfies $\operatorname{H}^i(X,\mathcal{M})=0$ for $i\geq 1$.
    \item There is an equivalence of abelian categories $\Gamma(X,-):\operatorname{Coh}(X)\rightarrow \mathcal{C}(\OX_X(X))$.
    \item For each $\mathcal{M}\in \mathcal{C}(\OX_X(X))$, let $\mathcal{M}_n=\OX_X(X_n)\otimes_{\OX_X(X)}\mathcal{M}$. The presentation $\mathcal{M}=\varprojlim_n \mathcal{M}_n$ makes $\mathcal{M}$ a nuclear Fréchet $K$-vector space.
\end{enumerate}
\end{prop}
\begin{proof}
The first two statements are shown in \cite{non-archSteinspaces}. The last follows by $(ii)$ in Lemma \ref{Lemma basic properties of Stein spaces}.
\end{proof}
\begin{coro}\label{coro sections of vector bundles are loc finite proj}
Let $X$ be a Stein space and $\mathscr{V}$ be a vector bundle on $X$. Then $\mathscr{V}(X)$ is a locally finite-free $\OX_X(X)$-module.
\end{coro}
\begin{proof}
As $\mathscr{V}$ is a vector bundle, the internal \emph{Hom} functor $\mathcal{H}om_{\OX_X}(\mathscr{V},-)$ is exact. In particular, its derived functor satisfies $R\mathcal{H}om_{\OX_X}(\mathscr{V},-)=\mathcal{H}om_{\OX_X}(\mathscr{V},-)$. On the other hand, by Proposition \ref{prop coherent modules on a Stein sapce}, the global sections functor $\Gamma(X,-)$ is also exact. Thus, we have:
\begin{equation*}
    R\operatorname{Hom}_{\OX_X(X)}(\mathscr{V}(X),-)=R\Gamma(X,R\mathcal{H}om_{\OX_X}(\mathscr{V},-))=\Gamma(X,\mathcal{H}om_{\OX_X}(\mathscr{V},-))=\Hom_{\OX_X(X)}(\mathscr{V}(X),-).
\end{equation*}
In particular, $R^1\operatorname{Hom}_{\OX_X(X)}(\mathscr{V}(X),\mathscr{V}(X))=0$. Consider a short exact sequence of $\OX_X(X)$-modules:
\begin{equation*}
    0\rightarrow I\rightarrow \OX_X(X)^J\rightarrow \mathscr{V}(X)\rightarrow 0.
\end{equation*}
Applying $R\operatorname{Hom}_{\OX_X(X)}(\mathscr{V}(X),-)$ to this sequence yields a short exact sequence:
\begin{equation*}
    0\rightarrow \operatorname{Hom}_{\OX_X(X)}(\mathscr{V}(X),I)\rightarrow \operatorname{Hom}_{\OX_X(X)}(\mathscr{V}(X),\OX_X(X)^J)\rightarrow \operatorname{Hom}_{\OX_X(X)}(\mathscr{V}(X),\mathscr{V}(X))\rightarrow 0.
\end{equation*}
Thus, the surjection $\OX_X(X)^J\rightarrow \mathscr{V}(X)$ has a split, and it follows that $\mathscr{V}(X)$ is a projective $\OX_X(X)$-module. As the rank of $\mathscr{V}(X)$ can be calculated at the stalks, and  $\mathscr{V}$ is a vector bundle, it follows that $\mathscr{V}(X)$ is locally finite-free, as wanted.
\end{proof}
For some of the developments below, it will be convenient to show that smooth Stein spaces admit covers by Stein spaces in which certain vector bundles are trivial. Our next goal is showing that such covers always exist. Let us start with the following lemma:
\begin{Lemma}\label{Lemma prod of Stein is Stein}
Let $X=\varinjlim_nX_n$, and $Y=\varinjlim_nY_n$ be two Stein spaces. Then $X\times Y=\varinjlim_n X_n\times Y_n$ is also a Stein space. Furthermore, there is a canonical isomorphism of Fréchet spaces:
\begin{equation*}
    \OX_{X\times Y}(X\times Y)=\OX_X(X)\widehat{\otimes}_K\OX_Y(Y)=\varprojlim\OX_X(X_n)\widehat{\otimes}_K\OX_Y(Y_n),
\end{equation*}
where $\widehat{\otimes}_K$ denotes the complete projective tensor product of locally convex spaces.
\end{Lemma}
\begin{proof}
It is clear that $X\times Y=\varinjlim_n X_n\times Y_n$ is an admissible cover of $X\times Y$ by affinoid subdomains. We just need to show that it is a Stein cover. Clearly, for every $n\geq 0$ we have that:
\begin{equation*}
    X_n\times Y_n\subset X_{n+1}\times Y_{n+1},
\end{equation*}
is a Weierstrass immersion. Furthermore, by \cite[Lemma 6.3.7]{bosch2014lectures},
$X_{n}\Subset X_{n+1}$, and $Y_n\Subset Y_{n+1}$ imply $X_n\times Y_n\Subset X_{n+1}\times Y_{n+1}$. Thus, $X\times Y=\varinjlim_n X_n\times Y_n$ is a Stein covering. The second part of the lemma is a classical statement in $p$-adic functional analysis. It can be found in \cite[pp. 107]{schneider2013nonarchimedean}.
\end{proof}
\begin{prop}\label{prop Zariski covers of Stein spaces are Stein}
Let $X$ be a Stein space and choose a holomorphic function $f\in \OX_X(X)$. In this situation, the Zariski open subspace:
\begin{equation*}
    X(f^{-1})=\bigcup_{n\geq 0}X\left(\frac{\pi^n}{f} \right)=\{x\in X \textnormal{ }\vert \textnormal{ } \vert f(x)\vert > 0 \},
\end{equation*}
is a Stein space.
\end{prop}
\begin{proof}
By Lemma  \ref{Lemma prod of Stein is Stein}, the product $X\times \mathbb{A}^1_K$ is a Stein space. Let $\mathcal{I}$ be the sheaf of ideals generated by $tf-1\in \OX_X(X)\widehat{\otimes}_K\OX_{\mathbb{A}^1_K}(\mathbb{A}^1_K)$, and $Y\subset X\times \mathbb{A}^1_K/\mathcal{I}$ be the closed subvariety associated to $\mathcal{I}$. There is a canonical isomorphism $X(f^{-1})\rightarrow Y$, which shows that $X(f^{-1})$ is a closed subvariety of $X\times \mathbb{A}^1_K$. Thus, it follows by statement $(iii)$ in Lemma \ref{Lemma basic properties of Stein spaces} that $X(f^{-1})$ is a Stein space.
\end{proof}
\begin{coro}\label{coro locally étale and stein}
 Let $X$ be a smooth Stein space. There is a finite covering $\{ U_i\}_{i=1}^n$  of $X$ such that each $U_i$ is a smooth Stein space with an étale map $U_i\rightarrow \mathbb{A}^r_K$.
\end{coro}
\begin{proof}
As $X$ is smooth, $\mathcal{T}_{X/K}$ is a locally free sheaf of rank $r=\dim(X)$. By Corollary \ref{coro sections of vector bundles are loc finite proj}, there is a family of global sections $\{f_i\}_{i\in I}\in \OX_X(X)$ such that 
the spaces $\{\Spec(\OX_X(X)[f_i^{-1}])\}_{i\in I}$ form a Zariski cover of $\Spec(\OX_X(X))$, and such that the module:
\begin{equation}\label{equation finite free locally on Stein}
    \mathcal{T}_{X/K}(X)[f_i^{-1}]\cong \bigoplus^r\OX_X(X)[f_i^{-1}]
\end{equation}
is a finite-free $\OX_X(X)[f_i^{-1}]$-module for each $i\in I$. As $\Spec(\OX_X(X))$ is quasi-compact, we can assume that $I$ is a finite index set. Thus, we have that $\{X(f_i^{-1})\}_{i\in I}$ is a finite Zariski cover of  $X$, and by Proposition \ref{prop Zariski covers of Stein spaces are Stein}, each of the $X(f_i^{-1})$ is a smooth Stein space. On the other hand, for each $i\in I$ there is a canonical factorization:
\begin{equation*}
    \OX_X(X)\rightarrow \OX_X(X)[f_i^{-1}]\rightarrow \OX_X(X(f_i^{-1})).
\end{equation*}
Thus, equation $(\ref{equation finite free locally on Stein})$ shows that  $\mathcal{T}_{X(f_i^{-1})/K}$  is globally free of rank $r$, and this implies that there are étale maps $X(f_i^{-1})\rightarrow \mathbb{A}^r_K$, as we wanted to show. 
\end{proof}
\subsection{\texorpdfstring{Co-admissible  $\wideparen{\D}$-modules on smooth Stein spaces}{}}\label{section co-ad D mod stein spaces}
Let $X=\varinjlim X_n$ be a smooth Stein space. The goal of this section is studying the properties of co-admissible $\wideparen{\D}_X$-modules on $X$. Following the path laid down in previous sections, we will show that the aforementioned category presents the same algebraic behavior as the category of co-admissible $\wideparen{\D}$-modules on an affinoid space, but satisfies better analytic properties. Namely, the global sections of co-admissible modules are canonically nuclear Fréchet spaces. We start by showing the corresponding results for coherent $\OX_X$-modules. First, notice that by \cite[Theorem 6.4]{bode2021operations}, we may regard the abelian category $\operatorname{Coh}(X)$ as a full abelian subcategory of the quasi-abelian category $\Mod_{\Indban}(\OX_X)$, and we will do so without further notice. Let us start with the following proposition:
\begin{prop}\label{prop acyclicity of coherent modules on Stein spaces}
    Let $X$ be a Stein space and $\mathcal{M}\in \operatorname{Coh}(\OX_X)$. Then 
   $\mathcal{M}$ is acyclic for the functor: 
   \begin{equation*}
       \Gamma(X,-):\Mod_{\Indban}(\OX_X)\rightarrow \Mod_{\Indban}(\OX_X(X)).
   \end{equation*}
\end{prop}
\begin{proof}
Let $\mathcal{M}\in \operatorname{Coh}(\OX_X)$. We need to show that $\mathcal{M}$ is \v{C}ech-acyclic. Every admissible cover of $X$ admits a refinement $\mathcal{V}=\{V_{n,i}\}_{n\geq 0}^{1\leq i \leq i_n}$, satisfying that for every $n\geq 0$, the family:
\begin{equation*}
    \mathcal{V}_n=\{V_{m,i}, \textnormal{ } m\leq n, \, 1\leq i\leq i_m\},
\end{equation*}
is a finite admissible cover of $X_n$. Hence, it suffices to show that the \v{C}ech complex $ C(\mathcal{V},\mathcal{M})^{\bullet}$ is a strict and exact complex of Banach spaces for any cover satisfying said properties.\\

Fix a cover $\mathcal{V}$ as above. Notice that $\mathcal{V}_n\subset \mathcal{V}_{n+1}$ for each $n$, and that the \v{C}ech complexes:
\begin{equation*}
    C(\mathcal{V}_n,\mathcal{M}_{\vert X_n})^{\bullet}:=\left(0\rightarrow \Gamma(X_n,\mathcal{M})\rightarrow \prod_{m\leq n,\textnormal{ } 1 \leq j\leq i_m} \Gamma(V_{m,j},\mathcal{M})\rightrightarrows \cdots  \right),
\end{equation*}
are exact complexes of Banach spaces for every $n\geq 0$. Hence, they are strict and exact complexes in $\Indban$.
Furthermore, the inclusions $\mathcal{V}_n\subset \mathcal{V}_{n+1}$ induce maps:
\begin{equation*}
    C(\mathcal{V}_{n+1},\mathcal{M}_{\vert X_{n+1}})^{\bullet}\rightarrow  C(\mathcal{V}_n,\mathcal{M}_{\vert X_n})^{\bullet},
\end{equation*}
and we have $C(\mathcal{V},\mathcal{M})^{\bullet}=\varprojlim_n C(\mathcal{V}_n,\mathcal{M}_{\vert X_n})^{\bullet}$.\\ 

As each of the $C(\mathcal{V}_n,\mathcal{M}_{\vert X_n})^{\bullet}$ is exact, it follows by \cite[Corollary 5.2]{bode2021operations} that it suffices to show that for each $m\geq 0$ the inverse systems:
\begin{equation*}
    C(\mathcal{V},\mathcal{M})^{m}=\varprojlim_n C(\mathcal{V}_n,\mathcal{M}_{\vert X_n})^{m},
\end{equation*}
are pre-nuclear inverse systems of Banach spaces with dense image. By construction, the maps $C(\mathcal{V}_{n+1},\mathcal{M}_{\vert X_{n+1}})^{m}\rightarrow  C(\mathcal{V}_n,\mathcal{M}_{\vert X_n})^{m}$ are split epimorphisms for $m> 0$. Hence, they clearly form a pre-nuclear system of Banach spaces. For $m=0$, the inverse system becomes:
\begin{equation*}
    \Gamma(X,\mathcal{M})=\varprojlim_n \Gamma(X_n,\mathcal{M}),
\end{equation*}
 which is nuclear by statement $(iii)$ in Proposition \ref{prop coherent modules on a Stein sapce}. Hence, $C(\mathcal{V},\mathcal{M})^{\bullet}$ is strict and exact.
\end{proof}
In particular, the sheaves $\Omega_{X/K}^n$ are $\Gamma(X,-)$-acyclic for every $n\geq 0$. This will be used in the following chapter to compare $\operatorname{HH}^{\bullet}(\wideparen{\D}_X)$ with $\operatorname{H}^{\bullet}_{\operatorname{dR}}(X)$. 
\begin{obs}
For the rest of the section, we assume that $K$ is either discretely valued or algebraically closed, and that $X$ admits an étale map $X\rightarrow \mathbb{A}^r_K$.
\end{obs}
Our goal now is extending this result to co-admissible $\wideparen{\D}_X$-modules. It will be convenient to start by showcasing some of the algebraic properties of $\mathcal{C}(\wideparen{\D}_X)$. Let us start with the following theorem:
\begin{teo}\label{teo global sections of co-admissible modules on Stein spaces}
There is a family of $\mathcal{R}$-modules:
\begin{equation*}
    \{\mathcal{T}_n\}_{n\geq 0},
\end{equation*}
and a strictly increasing sequence of non-negative integers $\{m(n) \}_{n\geq 0}$ satisfying the following properties for each $n\geq 0$:
\begin{enumerate}[label=(\roman*)]
    \item $\mathcal{T}_n\subset \mathcal{T}_{X/K}(X_n)$.
    \item We have an identity of $A^{\circ}_n$modules: $\mathcal{T}_n=A^{\circ}_n\otimes_{A^{\circ}_{n+1}}\mathcal{T}_{n+1}$.
     \item  $\pi^{m(n)}\mathcal{T}_n\subset $ is a free $(\mathcal{R},A^{\circ}_n)$-lie lattice of $\mathcal{T}_{X/K}(X_n)$. 
\end{enumerate}
In this situation, let $\widehat{\D}_{n,m}=\widehat{U}(\pi^m \mathcal{T}_n)_K$. Then the following hold:
\begin{enumerate}[label=(\roman*)]
    \item For each $n\geq 0$, the inverse limit $\wideparen{\D}_X(X_n)=\varprojlim_{m\geq m(n)} \widehat{\D}_{n,m}$ is a Fréchet-Stein presentation.
    \item The inverse limit $\wideparen{\D}_X(X)=\varprojlim_i\wideparen{\D}_X(X_i)=\varprojlim_n \widehat{\D}_{n,m(n)}$ is a Fréchet-Stein presentation.
\end{enumerate}
Furthermore, every co-admissible $\wideparen{\D}_X(X)$-module is a nuclear Fréchet space.
\end{teo}
\begin{proof}
If $K$ is discretely valued, the theorem is a special instance of \cite[Proposition 4.3.2]{p-adicCatO}, and \cite[Corollary 4.3.4]{p-adicCatO}. The general case follows by the same argument together with  the flatness results in  \cite[Theorem 4.1.11]{ardakovequivariant} and \cite[Theorem 4.2.8]{ardakovequivariant}.
\end{proof}
\begin{obs}
It is precisely at this point where we need that $K$ is either discretely valued or algebraically closed. Indeed, for arbitrary complete non-archimedean extensions of $\mathbb{Q}_p$ the algebras of power bounded elements $A^{\circ}_n$ as above are not necessarily  topologically finitely presented.
\end{obs}
For the rest of this section, we fix a family $\{\mathcal{T}_n\}_{n\geq 0}$, and a strictly increasing sequence of non-negative integers $\{m(n) \}_{n\geq 0}$ as above. As a consequence of this theorem, we obtain the following:
\begin{coro}\label{coro maps to stein covering are c-flat}
For each $n\geq 0$ the canonical map $\wideparen{\D}_X(X)\rightarrow \wideparen{\D}_X(X_n)$ is $c$-flat.   
\end{coro}
\begin{proof}
In light of the previous theorem, it suffices to show that for each $n\geq 0$ and each $i\geq n$    the map $\widehat{\D}_{i,m(i)}\rightarrow \widehat{\D}_{n,m(i)}$ is flat. However, by construction, we have:
\begin{equation*}
    \pi^{m(i)}\mathcal{T}_n=A_n^{\circ}\otimes_{A_i^{\circ}}\pi^{m(i)}\mathcal{T}_i.
\end{equation*}
Hence, the result  is a direct consequence of \cite[Theorem 4.2.8]{ardakovequivariant}.
\end{proof}
We point out that the existence of such presentations depends on our assumption that $X$ admits an étale map $X\rightarrow \mathbb{A}^r_K$. However, by Corollary \ref{coro locally étale and stein}, every smooth Stein space can be covered by finitely many smooth Stein spaces satisfying this assumption. Thus, imposing the existence of an étale map to $\mathbb{A}^r_K$ is not a very restrictive condition. With this theorem at hand, we are ready to start characterizing the category of co-admissible $\wideparen{\D}_X$-modules on $X$. Let us start with the following:
\begin{Lemma}\label{lemma co-admis on Stein spaces are determined by global sections}
Let $\mathcal{M}\in \mathcal{C}(\wideparen{\D}_X)$, and $\mathcal{N}\in \Mod_{\Indban}(\wideparen{\D}_X)$. We have an identification:
    \begin{equation*}
        \Hom_{\wideparen{\D}_X}(\mathcal{M},\mathcal{N})=\Hom_{\wideparen{\D}_X(X)}(\Gamma(X,\mathcal{M}),\Gamma(X,\mathcal{N})).
    \end{equation*}    
\end{Lemma}
\begin{proof}
Let $\mathcal{M}\in \mathcal{C}(\wideparen{\D}_X)$, and $\mathcal{N}\in \Mod_{\Indban}(\wideparen{\D}_X)$.  First, as each $X_n$ is an affinoid space, it follows by a slight modification of  \cite[Proposition 6.2.12]{HochDmod} that we have an identification:
\begin{equation}\label{equation identity global sections}
    \Hom_{\wideparen{\D}_{X_n}}(\mathcal{M}_{\vert X_n},\mathcal{N}_{\vert X_n})=\Hom_{\wideparen{\D}_X(X_n)}(\Gamma(X_n,\mathcal{M}),\Gamma(X_n,\mathcal{N})).
\end{equation}
For each $n\geq 0$, and each $i\geq 0$, let $\mathcal{M}_{n,i}:=\widehat{\D}_{n,i}\overrightarrow{\otimes}_{\wideparen{\D}_X(X)}\mathcal{M}(X_n)$. As $\mathcal{M}$ is a sheaf of Ind-Banach spaces, we have the following identities:
\begin{equation*}
    \mathcal{M}(X)=\varprojlim_n\mathcal{M}(X_n)=\varprojlim_n\varprojlim_{i\geq m(n)}\mathcal{M}_{n,i}=\varprojlim_{n}\mathcal{M}_{n,n(m)}.
\end{equation*}
Thus, $\mathcal{M}(X)$ is a co-admissible $\wideparen{\D}_X(X)$-module. Chose a morphism of Ind-Banach $\wideparen{\D}_X(X)$-modules $f:\mathcal{M}(X)\rightarrow \mathcal{N}(X)$. By \cite[ Proposition 6.2.9]{HochDmod}, for each $n\geq 0$ there is a commutative diagram:
\begin{equation*}
\begin{tikzcd}
\mathcal{M}(X) \arrow[r, "f"] \arrow[d]                               & \mathcal{N}(X) \arrow[d] \\
\wideparen{\D}_X(X_n)\overrightarrow{\otimes}_{\wideparen{\D}_X(X)}\mathcal{M}(X) \arrow[r] & \mathcal{N}(X_n)        
\end{tikzcd}
\end{equation*}
such that the lower map is uniquely determined by $f:\mathcal{M}(X)\rightarrow \mathcal{N}(X)$. In virtue of identity (\ref{equation identity global sections}), it suffices to show that we have an identification of Ind-Banach $\wideparen{\D}_X(X_n)$-modules:
\begin{equation*}
    \mathcal{M}(X_n)=\wideparen{\D}_X(X_n)\overrightarrow{\otimes}_{\wideparen{\D}_X(X)}\mathcal{M}(X).
\end{equation*}
By \cite[Proposition 5.33]{bode2021operations}, and Corollary \ref{coro maps to stein covering are c-flat}, we have the following identity in $LH(\widehat{\mathcal{B}}c_K)$:
\begin{equation*}
I(\wideparen{\D}_X(X_n)\overrightarrow{\otimes}_{\wideparen{\D}_X(X)}\mathcal{M}(X))=I((\wideparen{\D}_X(X_n)\wideparen{\otimes}_{\wideparen{\D}_X(X)}\mathcal{M}(X))^b),
\end{equation*}
where $\wideparen{\otimes}$ denotes the co-admissible tensor product from \cite[Section 7]{ardakov2019}. Hence, it suffices to show that we have an identification of Fréchet spaces:
\begin{equation*}
    \mathcal{M}(X_n)=\wideparen{\D}_X(X_n)\wideparen{\otimes}_{\wideparen{\D}_X(X)}\mathcal{M}(X).
\end{equation*}
As $\mathcal{M}(X)$ is a co-admissible $\wideparen{\D}_X(X)$-module, we have the following identities of Ind-Banach spaces:
\begin{equation*}
    \mathcal{M}(X)=\varprojlim_n \mathcal{M}(X_n)=\varprojlim_n\varprojlim_i \mathcal{M}_{n,i}.
\end{equation*}
On the other hand, by definition of the co-admissible tensor product, we have the following identities:
\begin{equation*}  \wideparen{\D}_X(X_n)\wideparen{\otimes}_{\wideparen{\D}_X(X)}\mathcal{M}(X)=\varprojlim_{i\geq m(n)}\widehat{\D}_{n,i}\overrightarrow{\otimes}_{\widehat{\D}_{i,i}}\mathcal{M}_{i,i}=\varprojlim_{i\geq n}\mathcal{M}_{n,i}=\mathcal{M}(X_n),
\end{equation*}
which is precisely the identity we wanted to show.
\end{proof}
We may use this lemma to show the following:
\begin{prop}\label{prop characterization of co-admissible modules on a smooth Stein space}
There is an equivalence of abelian categories: 
\begin{equation*}
    \Gamma(X,-):\mathcal{C}(\wideparen{\D}_X)\leftrightarrows \mathcal{C}(\wideparen{\D}_X(X)): \operatorname{Loc}(-),
\end{equation*}
where $\operatorname{Loc}(-)$ is defined by sending a co-admissible $\wideparen{\D}_X(X)$-module $\mathcal{M}$ to the unique co-admissible $\wideparen{\D}_X$-module satisfying the following identity:
\begin{equation*}
    \operatorname{Loc}(\mathcal{M})(X_n)=\wideparen{\D}_X(X_n)\overrightarrow{\otimes}_{\wideparen{\D}_X(X)}\mathcal{M}.
\end{equation*}
\end{prop}
\begin{proof}
The fact that $\Gamma(X,-)$ is fully faithful was shown in the previous lemma. Thus, we only need to show that it is essentially surjective. On the other hand, we also saw in the previous lemma that for every $\mathcal{M}\in\mathcal{C}(\wideparen{\D}_X(X))$, and every $n\geq 0$ the Ind-Banach $\wideparen{\D}_X(X_n)$-module $\wideparen{\D}_X(X_n)\overrightarrow{\otimes}_{\wideparen{\D}_X(X)}\mathcal{M}$  is co-admissible. Hence, by \cite[Theorem 6.4]{bode2021operations}, it extends to a co-admissible 
$\wideparen{\D}_{X_n}$-module. Furthermore, as co-admissible $\wideparen{\D}$-modules  on affinoid spaces are completely determined by their global sections, it is clear that $\operatorname{Loc}(\mathcal{M})$ is a sheaf of co-admissible $\wideparen{\D}_X$-modules. We just need to show that:
\begin{equation*}
    \mathcal{M}=\varprojlim_n \wideparen{\D}_X(X_n)\overrightarrow{\otimes}_{\wideparen{\D}_X(X)}\mathcal{M}.
\end{equation*}
However, using the notation from the proof of Lemma \ref{lemma co-admis on Stein spaces are determined by global sections} we have:
\begin{equation*}
    \varprojlim_n \wideparen{\D}_X(X_n)\overrightarrow{\otimes}_{\wideparen{\D}_X(X)}\mathcal{M}=\varprojlim_n\varprojlim_{m\geq m(n)}\mathcal{M}_{n,m}=\varprojlim_{n}\mathcal{M}_{n,m(n)}=\mathcal{M},
\end{equation*}
as we wanted to show.
\end{proof}
This allows us to obtain a version of \cite[Corollary 5.15]{bode2021operations} for Stein spaces:
\begin{coro}\label{coro map from co-admissible modules to complete born modules is an immersion}
We have the following commutative diagram:
\begin{equation*}
\begin{tikzcd}
\mathcal{C}(\wideparen{\D}_X) \arrow[r, "{\Gamma(X,-)}"] \arrow[d, "\operatorname{diss}\circ(-)^b"'] & \mathcal{C}(\wideparen{\D}_X(X)) \arrow[d, "\operatorname{diss}\circ(-)^b"] \\
\Mod_{\Indban}(\wideparen{\D}_X) \arrow[r, "{\Gamma(X,-)}"]         & \Mod_{\Indban}(\wideparen{\D}_X(X))       
\end{tikzcd}
\end{equation*}
where the upper horizontal map is an equivalence of abelian categories and the vertical maps are exact and fully faithful.    
\end{coro}
Thus, from now on, we will identify $\mathcal{C}(\wideparen{\D}_X(X))$ with its essential image inside $\Mod_{\Indban}(\wideparen{\D}_X(X))$.\\

Now that we have an accurate description of the category of co-admissible $\wideparen{\D}_X$-modules on $X$, we can start analyzing its behavior with respect to the global sections functor. As usual, the first step is understanding the situation in the affinoid case:
\begin{Lemma}\label{Lemma acyclicity of co-admissible modules on affinoid spaces}
Let $X$ be a smooth affinoid space with an étale map $X\rightarrow \mathbb{A}^r_K$ and $\mathcal{M}\in \mathcal{C}(\wideparen{\D}_X)$. Then 
   $\mathcal{M}$ is acyclic with respect to the functor $\Gamma(X,-):\Mod_{\Indban}(\wideparen{\D}_X)\rightarrow \Mod_{\Indban}(\wideparen{\D}_X(X))$.    
\end{Lemma}
\begin{proof}
As $X$ is quasi-compact, it suffices to show that $\mathcal{M}$ has trivial higher \v{C}ech cohomology with respect to finite affinoid covers. Let $\mathcal{V}=\{ V_i\}_{i=1}^m$ be such a cover.  We need to show that the augmented \v{C}ech complex:
\begin{equation*}
    \mathcal{C}(\mathcal{V},\mathcal{M})^{\bullet}=\left(0\rightarrow \Gamma(X,\mathcal{M})\rightarrow \prod_{i=1}^m \Gamma(V_i,\mathcal{M})\rightarrow \cdots\rightarrow \Gamma(\cap_{i=1}^mV_i,\mathcal{M})\rightarrow 0\right),
\end{equation*}
is a strict exact complex of Ind-Banach spaces. However, each object in $\mathcal{C}(\mathcal{V},\mathcal{M})^{\bullet}$ is the image of a complete bornological space under the dissection functor. Hence, it suffices to show that the complex is strict exact in the category of complete bornological spaces.\\

For each $V_{i}\in\mathcal{V}$, let $A_{i}=\OX_X(V_{i})$, and set $A=\OX_X(X)$. As $\mathcal{T}_{X/K}(X)$ is free as an $A$-module and $\mathcal{V}$ is a finite cover of $X$, we may apply \cite[Proposition 5.1]{ardakov2019} and  \cite[Proposition 7.6]{ardakov2019} to choose adequate smooth affine formal models of $A$, and  of the
$A_{i}$, so that we have Fréchet-Stein presentations:
\begin{equation*}
    \wideparen{\D}_X(V_{i})=\varprojlim_r\widehat{\D}_{r,i},\, \wideparen{\D}_X(X)=\varprojlim_r\widehat{\D}_r,
\end{equation*}
satisfying that the complex:
\begin{equation}\label{equation auxiliary complex in acyclycity of co-admissible modules}
    \mathcal{C}{r}^{\bullet}:= \left(0\rightarrow \widehat{\D}_{r}\widehat{\otimes}_{\wideparen{\D}_X(X)}\Gamma(X,\mathcal{M})\rightarrow \prod_{i=1}^m\widehat{\D}_{r,i}\widehat{\otimes}_{\wideparen{\D}_X(V_{i})}\Gamma(V_{i},\mathcal{M}) \rightarrow\cdots\right),
\end{equation}
is a strict and exact complex of complete bornological spaces for every $r\geq 0$. Notice that the aforementioned results only show that $(\mathcal{C}_{r}^{\bullet})^t$ is a strict exact complex in $LCS_K$. However, as this is a complex of Banach spaces, strict exactness is preserved under $(-)^b$ by \cite[Lemma 4.4]{bode2021operations}. Furthermore, by definition of co-admissible module, we have $\mathcal{C}(\mathcal{V},\mathcal{M})^{\bullet}=\varprojlim_r\mathcal{C}_{r}^{\bullet}$.\\

As each of the  $\mathcal{C}_{r}^{\bullet}$ is strict and exact, we may apply  \cite[Corollary 5.2]{bode2021operations} to conclude that showing strict exactness of $\mathcal{C}(\mathcal{V},\mathcal{M})^{\bullet}$ is reduced to showing that for each $m\geq 0$, the inverse limit:
\begin{equation*}
    \mathcal{C}(\mathcal{V},\mathcal{M})^{m}=\varprojlim_r\mathcal{C}_{r}^{m},
\end{equation*}
is a pre-nuclear inverse system of Banach spaces. By construction, for every affinoid subdomain $V_{i}\in \mathcal{V}$, we have an isomorphism:
\begin{equation*}
\Gamma(V_{i},\mathcal{M})=\varprojlim_r\widehat{\D}_{r,i}\widehat{\otimes}_{\wideparen{\D}_X(V_{i})}\Gamma(V_{i},\mathcal{M}).
\end{equation*}
By \cite[Proposition 5.5]{bode2021operations} such a presentation makes $\Gamma(V_{i},\mathcal{M})$ a nuclear $A_{i}$-module. In particular, the associated inverse system is a pre-nuclear system of Banach spaces. As this argument holds for $X$ and finite intersections of elements in $\mathcal{V}$, we conclude that $\mathcal{C}(\mathcal{V},\mathcal{M})^{\bullet}$ is a strict and exact complex of complete bornological spaces, as wanted.
\end{proof}
We are finally ready to show the main theorem of the section:
\begin{prop}\label{prop acyclicity of co-admissible modules on Stein spaces}
    Let $X$ be a smooth Stein space with an étale map $X\rightarrow \mathbb{A}^r_K$. Let $\mathcal{M}\in \mathcal{C}(\wideparen{\D}_X)$, then 
   $\mathcal{M}$ is acyclic with respect to the functor:
   \begin{equation*}
       \Gamma(X,-):\Mod_{\Indban}(\wideparen{\D}_X)\rightarrow \Mod_{\Indban}(\wideparen{\D}_X(X)).
   \end{equation*}

\end{prop}
\begin{proof}
Choose a family of covers $\cup_{n\geq 0}\mathcal{V}_n$  as in the proof of Proposition \ref{prop acyclicity of coherent modules on Stein spaces}. By Lemma \ref{Lemma acyclicity of co-admissible modules on affinoid spaces}, it follows that the \v{C}ech complex $\mathcal{C}(\mathcal{V}_n,\mathcal{M}_{\vert X_n})^{\bullet}$ is a strict and exact complex of complete bornological spaces for every $n\geq 0$. Furthermore, the inclusions $\mathcal{V}_n\subset \mathcal{V}_{n+1}$ induce an isomorphism:
\begin{equation*}
    \mathcal{C}(\mathcal{V},\mathcal{M})^{\bullet}=\varprojlim\mathcal{C}(\mathcal{V}_n,\mathcal{M}_{\vert X_n})^{\bullet}.
\end{equation*}
Hence, we need to show that for each $m\geq 0$, we have:
\begin{equation*}
    R^1\varprojlim_n\mathcal{C}(\mathcal{V}_n,\mathcal{M}_{\vert X_n})^{m}=0.
\end{equation*}
 As in the proof of Proposition \ref{prop acyclicity of coherent modules on Stein spaces}, for $m\geq 1$ the maps in the inverse system
$\varprojlim_n\mathcal{C}(\mathcal{V}_n,\mathcal{M}_{\vert X_n})^{m}$ are split surjections, hence $R^1\varprojlim_n\mathcal{C}(\mathcal{V}_n,\mathcal{M}_{\vert X_n})^{m}=0$ for all $m\geq 1$. Thus, we only need to show:
\begin{equation*}
    R^1\varprojlim\Gamma(X_n,\mathcal{M})=0.
\end{equation*}
However, applying the notation from Lemma \ref{lemma co-admis on Stein spaces are determined by global sections} to $\mathcal{M}(X)$, we have:
\begin{equation*}
    \varprojlim\Gamma(X_n,\mathcal{M})=\varprojlim_{n}\mathcal{M}_{n,m(n)}.
\end{equation*}
By Theorem \ref{teo global sections of co-admissible modules on Stein spaces} this presentation makes $\mathcal{M}(X)$ a nuclear Fréchet space. Thus, we have:
\begin{equation*}
    R^1\varprojlim\Gamma(X_n,\mathcal{M})=0,
\end{equation*}
by \cite[Theorem 5.26]{bode2021operations}. As we wanted to show.
\end{proof}
\begin{obs}\label{obs version for E modules}
Notice that if $X=\varinjlim_nX_n$ is a Stein space with an étale map $X\rightarrow \mathbb{A}^r_K$, then $X^2:=\varinjlim_nX_n\times X_n$ is also a Stein space with an étale map $X^2\rightarrow \mathbb{A}^{2r}_K$. In particular, the results above hold for co-admissible $\wideparen{\D}_{X^2}$-modules. Therefore, the analogous results hold for co-admissible $\wideparen{E}_X$-modules as well, via the equivalences of quasi-abelian categories:
\begin{equation*}
    \operatorname{S}:\Mod_{\Indban}(\wideparen{\D}_{X^2})\leftrightarrows \Mod_{\Indban}(\wideparen{E}_X):\operatorname{S}^{-1}.
\end{equation*}

\end{obs}
\section{Applications of Hochschild cohomology}\label{Chapter applications of HC}
After introducing the class of spaces we are interested in, the next step is studying their behavior with respect to Hochschild cohomology. 
Let $X$ be a smooth Stein space with an étale map $X\rightarrow \mathbb{A}^r_K$. In this chapter, we will compare the Hochschild cohomology complex $\operatorname{HH}^{\bullet}(\wideparen{\D}_X)$ with a bunch of other complexes arising from different cohomology theories. As an upshot, we obtain numerous explicit interpretations for the Hochschild cohomology groups of $\wideparen{\D}_X$. Let us now give an overview of the contents of the chapter: We start by  recalling a result by E. Grosse-Klönne in \cite{grosse2004rham}, which states that we have a canonical identification of complexes:
\begin{equation*}
   \operatorname{H}^{\bullet}_{\operatorname{dR}}(X)=\Gamma(X,\Omega_{X/K}^{\bullet}),
\end{equation*}
which makes $\operatorname{H}^{\bullet}_{\operatorname{dR}}(X)$ a strict complex of nuclear Fréchet spaces. This result, together with our previous calculations, shows that there is a strict quasi-isomorphism  $\operatorname{HH}^{\bullet}(\wideparen{\D}_X)\rightarrow \operatorname{H}^{\bullet}_{\operatorname{dR}}(X)^b$. In particular, $\operatorname{HH}^{\bullet}(\wideparen{\D}_X)$ is a strict complex of (bornologifications of) nuclear Fréchet spaces.  We then use this fact to show a Künneth formula for Hochschild cohomology:
\begin{equation*}
    \operatorname{HH}^{\bullet}(\wideparen{\D}_{X\times Y})=\operatorname{HH}^{\bullet}(\wideparen{\D}_{X})\widehat{\otimes}_K^{\mathbb{L}}\operatorname{HH}^{\bullet}(\wideparen{\D}_{Y}),
\end{equation*}
where $Y$ is another smooth Stein space. Next, we use the Spencer resolution:
\begin{equation*}
 0\rightarrow \wideparen{\D}_X\overrightarrow{\otimes}_{\OX_X}\wedge^n\mathcal{T}_{X/K}\rightarrow \cdots \rightarrow \wideparen{\D}_X\rightarrow \OX_X\rightarrow 0,
\end{equation*}
to express $\operatorname{HH}^{\bullet}(\wideparen{\D}_X)$ as a derived internal Hom functor. In particular, we show that the following identity holds in $\operatorname{D}(\widehat{\mathcal{B}}c_K)$ :
\begin{equation*}
   \operatorname{HH}^{\bullet}(\wideparen{\D}_X)=R\underline{\Hom}_{\wideparen{\D}_X(X)}(\OX_X(X),\OX_X(X)).
\end{equation*}
Following this line of thought, we then use the discussion on external sheaves of homomorphisms from Section \ref{section external hom functors} to deduce the following identities in $\operatorname{D}(\operatorname{Vect}_K)$:
\begin{equation*}
   \operatorname{Forget}(J\operatorname{HH}^{\bullet}(\wideparen{\D}_X))=R\Hom_{\wideparen{\D}_X}(\OX_X,\OX_X)=R\Hom_{\wideparen{\D}_X(X)}(\OX_X(X),\OX_X(X)). 
\end{equation*}
Notice that this shows that the vector spaces underlying the cohomology groups of $\operatorname{HH}^{\bullet}(\wideparen{\D}_X)$ parameterize the Yoneda's extensions of $I(\OX_X(X))$ by $I(\OX_X(X))$ in $\Mod_{LH(\widehat{\mathcal{B}}c_K)}(I(\wideparen{\D}_X(X)))$. Similarly, this also shows that the isomorphism classes of Yoneda extensions of $I(\OX_X)$ by $I(\OX_X)$
in $\Mod_{LH(\widehat{\mathcal{B}}c_K)}(I(\wideparen{\D}_X))$ are completely determined by their global sections, and that every Yoneda extension of the global sections is induced by an extension at the level of sheaves. We remark that this type of phenomena resembles the behavior of quasi-coherent $\D$-modules on smooth affine $K$-varieties.\\

One of the most interesting features of smooth Stein spaces is that many algebraic properties of co-admissible $\wideparen{\D}$-modules can already be detected at the level of global sections. This stems from the equivalence of abelian categories obtained in Proposition \ref{prop characterization of co-admissible modules on a smooth Stein space}:
\begin{equation*}
    \Gamma(X,-):\mathcal{C}(\wideparen{\D}_X)\leftrightarrows \mathcal{C}(\wideparen{\D}_X(X)):\operatorname{Loc}(-). 
\end{equation*}
Unfortunately, these abelian categories do not contain injective objects, so it is not possible to extend this equivalence to the derived level. Nonetheless, in analogy with the algebraic setting, we would like to show that the Hochschild cohomology of $\wideparen{\D}_X$ can be calculated at the global sections. Namely, we want to show that there is an identification in $\operatorname{D}(\widehat{\mathcal{B}}c_K)$:
\begin{equation}\label{equation HC at global sections}
    \operatorname{HH}^{\bullet}(\wideparen{\D}_X)=\operatorname{HH}^{\bullet}(\wideparen{\D}_X(X)).
\end{equation}
In order to do so, we resort to the sheaf of bi-enveloping algebras $\wideparen{E}_X$ on $X^2$, and show that there is a co-admissible 
$\wideparen{E}_X$-module $\mathcal{S}_X$ such that there is a strict exact complex:
\begin{equation*}
     E^{\bullet}_{\wideparen{\D}_X}:= \left(0\rightarrow P^{2r}\rightarrow \cdots\rightarrow P^0\rightarrow \mathcal{S}_X\oplus \Delta_*\wideparen{\D}_X\rightarrow 0\right),
\end{equation*}
where each $P^i$ is a finite-projective $\wideparen{E}_X$-module and $P^0=\wideparen{E}_X$. The identification in (\ref{equation HC at global sections}) then follows naturally from the existence of the resolution $E^{\bullet}_{\wideparen{\D}_X}$. As a consequence, we obtain an isomorphism $\operatorname{HH}^{\bullet}(\wideparen{\D}_X(X))=\operatorname{H}_{\operatorname{dR}}^{\bullet}(X)^b$. This isomorphism will be instrumental to our approach to the deformation theory of $\wideparen{\D}_X(X)$ in Chapter \ref{Chapter deformation theory}.
\subsection{Comparison with de Rham cohomology}\label{Section dR cohomology}
In this section, we will use the results from \cite{grosse2004rham} to study the Hochschild cohomology groups of $\wideparen{\D}_X$ for a smooth Stein space $X$.  Until the end of the section, we fix a smooth Stein space $X=\varinjlim_nX_n$ admitting an étale map $X\rightarrow \mathbb{A}^r_K$. Recall that, as we saw above, this is always the case locally.\\

As shown in Proposition  \ref{prop acyclicity of coherent modules on Stein spaces}, the sheaves of differentials $\Omega_{X/K}^{n}$ are acyclic for the functor:
\begin{equation*}
    \Gamma(X,-):\operatorname{Shv}(X,\Indban)\rightarrow \Indban.
\end{equation*}
Thus, we have the following identities in $\operatorname{D}(\widehat{\mathcal{B}}c_K)$:
\begin{equation*}
    \operatorname{HH}^{\bullet}(\wideparen{\D}_X)=R\Gamma(X,\Omega_{X/K}^{\bullet})=\Gamma(X,\Omega_{X/K}^{\bullet}).
\end{equation*}
Notice that this is not enough to conclude that $\Gamma(X,\Omega_{X/K}^{\bullet})$ is a strict complex of bornological spaces. In particular, the Hochschild cohomology spaces need not be bornological spaces. The reason behind working with smooth Stein spaces stems from the following proposition: 
\begin{prop}[{\cite[Corollary 3.2]{grosse2004rham}}]\label{prop elmar's result on strictness of de Rham complex}
Let $X$ be a smooth Stein space. Then the complex:
\begin{equation*}
   \operatorname{H}_{\operatorname{dR}}^{\bullet}(X)= \left(0\rightarrow \OX_X(X) \rightarrow \Omega_{X/K}^1(X)\rightarrow \cdots \rightarrow \Omega_{X/K}^{\operatorname{dim}(X)}(X) \right),
\end{equation*}
is a strict complex of nuclear Fréchet spaces. Thus, $\operatorname{H}_{\operatorname{dR}}^n(X)$ is a nuclear Fréchet space for $n\geq 0$.
\end{prop}
Notice that the proposition above does not imply that the complex $\Omega_{X/K}^{\bullet}$ is a strict complex in $\operatorname{Shv}(X,\Indban)$, only that $\Gamma(X,\Omega_{X/K}^{\bullet})$ is a strict complex in $\Indban$. Furthermore, $\Omega_{X/K}^{\bullet}$ will in general not be a strict complex of sheaves. For instance, $\mathbb{A}^1_K$ is a smooth Stein space, but $\Omega^{\bullet}_{\mathbb{A}^1_K/K}$ is not a strict complex of sheaves. Indeed, if $\Omega^{\bullet}_{\mathbb{A}^1_K/K}$ was strict, then the complex:
\begin{equation*}
    \Gamma(\mathbb{B}^1_K,\Omega^{\bullet}_{\mathbb{A}^1_K/K})=\Omega^{\bullet}_{\mathbb{B}^1_K/K}(\mathbb{B}^1_K),
\end{equation*}
would be a strict complex in $\Indban$. However, as seen in the introduction, this is false.\\

Recall from \cite[pp. 93,  Proposition 3]{houzel2006seminaire} that there is an adjunction:
\begin{equation*}
    (-)^t:\mathcal{B}c_K\leftrightarrows LCS_K:(-)^b.
\end{equation*}
 By the previous proposition, we may regard $\operatorname{H}_{\operatorname{dR}}^{\bullet}(X)$ as a complex of Fréchet spaces, and we let $\operatorname{H}_{\operatorname{dR}}^{\bullet}(X)^b$ be the associated complex of complete bornological spaces. We can use these tools to show an improved version of \cite[Corollary 8.1.12]{HochDmod}:
\begin{teo}\label{teo hochschild cohomology groups as de rham cohomology groups}
Let $X$ be a smooth Stein space. We have the following identities in $\operatorname{D}(\widehat{\mathcal{B}}c_K):$
\begin{equation*}
    \operatorname{HH}^{\bullet}(\wideparen{\D}_X)=R\Gamma(X,\Omega_{X/K}^{\bullet})=\Gamma(X,\Omega_{X/K}^{\bullet})=\operatorname{H}_{\operatorname{dR}}^{\bullet}(X)^b.
\end{equation*}
Thus, $\operatorname{HH}^{\bullet}(\wideparen{\D}_X)$ is a strict complex, and $\operatorname{HH}^{n}(\wideparen{\D}_X)=\operatorname{H}_{\operatorname{dR}}^{n}(X)^b$ is a nuclear Fréchet space for $n\geq 0$.
\end{teo}
\begin{proof}
 The first three identities have already been shown. For the last identity, regard $\operatorname{H}_{\operatorname{dR}}^{\bullet}(X)$
 as a complex in $LCS_K$ with the topology constructed in Proposition \ref{prop elmar's result on strictness of de Rham complex}. By that proposition, all objects in $\operatorname{H}_{\operatorname{dR}}^{\bullet}(X)$ are nuclear Fréchet spaces. As $\operatorname{H}_{\operatorname{dR}}^{\bullet}(X)$ is a strict complex of Fréchet spaces, all images are closed. Thus, by \cite[Proposition 19.4]{schneider2013nonarchimedean}
 all kernels, images and cohomology groups of $\operatorname{H}_{\operatorname{dR}}^{\bullet}(X)$ are nuclear Fréchet spaces. Hence, by \cite[Proposition 5.12]{bode2021operations} the complex:
 \begin{equation*}
     \operatorname{H}_{\operatorname{dR}}^{\bullet}(X)^b=\Gamma(X,\Omega_{X/K}^{\bullet}),
 \end{equation*}
 is a strict exact complex of complete bornological spaces. 
\end{proof}
As a first application of this theorem, we will calculate the Hochschild cohomology of $\wideparen{\D}_{X\times Y}$, where $X$ and $Y$ are smooth Stein spaces. In order to do this, we will now show that the Künneth formula holds for the de Rham cohomology of products of smooth Stein spaces. Notice that this result was already claimed in 
\cite[Proposition 3.3]{grosse2004rham}. However, the arguments used there only work if one of the spaces has finite de Rham cohomology. Let us start with the following lemma: 
\begin{Lemma}\label{Lemma junneth formula for complexes}
    Let $X$ and $Y$ be smooth Stein spaces. Then the following complex is strict:
 \begin{equation*}
    \operatorname{H}_{\operatorname{dR}}^{\bullet}(X)^b\widehat{\otimes}_K^{\mathbb{L}}\operatorname{H}_{\operatorname{dR}}^{\bullet}(Y)^b. 
 \end{equation*}   
In particular, for each $n\geq 0$ we have the following identity in $\widehat{\mathcal{B}}c_K$:
    \begin{equation*}
        \operatorname{H}^n(\operatorname{H}_{\operatorname{dR}}^{\bullet}(X)^b\widehat{\otimes}_K^{\mathbb{L}}\operatorname{H}_{\operatorname{dR}}^{\bullet}(Y)^b)=\bigoplus_{r+s=n}\operatorname{H}_{\operatorname{dR}}^{r}(X)^b\widehat{\otimes}_K\operatorname{H}_{\operatorname{dR}}^{s}(Y)^b.
    \end{equation*}
\end{Lemma}
\begin{proof}
By the above, $\operatorname{H}_{\operatorname{dR}}^{\bullet}(X)^b$ and $\operatorname{H}_{\operatorname{dR}}^{\bullet}(Y)^b$ are strict complexes of Fréchet spaces. Thus, it follows by \cite[Proposition 4.22]{bode2021operations} and \cite[Corollary 4.24]{bode2021operations} that $I(\operatorname{H}_{\operatorname{dR}}^{\bullet}(X)^b)$ and $I(\operatorname{H}_{\operatorname{dR}}^{\bullet}(Y)^b)$ are complexes of $\widetilde{\otimes}_K$-flat modules such that the image of any differential is $\widetilde{\otimes}_K$-flat. Furthermore, we have:
\begin{equation*}
    \operatorname{H}_{\operatorname{dR}}^{\bullet}(X)^b\widehat{\otimes}_K^{\mathbb{L}}\operatorname{H}_{\operatorname{dR}}^{\bullet}(Y)^b=I(\operatorname{H}_{\operatorname{dR}}^{\bullet}(X)^b)\widetilde{\otimes}_K^{\mathbb{L}}I(\operatorname{H}_{\operatorname{dR}}^{\bullet}(Y)^b).
\end{equation*}
Therefore, we may use the fact that \cite[Theorem 3.63]{weibel1994introduction} holds in any closed symmetric monoidal Grothendieck category to get:
\begin{equation*}
\operatorname{H}^n(\operatorname{H}_{\operatorname{dR}}^{\bullet}(X)^b\widehat{\otimes}_K^{\mathbb{L}}\operatorname{H}_{\operatorname{dR}}^{\bullet}(Y)^b)=\bigoplus_{r+s=n}I(\operatorname{H}_{\operatorname{dR}}^{r}(X)^b)\widetilde{\otimes}_KI(\operatorname{H}_{\operatorname{dR}}^{s}(Y)^b).   
\end{equation*}
Furthermore, as all de Rham cohomology groups are Fréchet spaces, they are metrizable. Thus, by  \cite[Proposition 4.25]{bode2021operations}
we have:
\begin{equation*}
    \bigoplus_{r+s=n}I(\operatorname{H}_{\operatorname{dR}}^{r}(X)^b)\widetilde{\otimes}_KI(\operatorname{H}_{\operatorname{dR}}^{s}(Y)^b)=I(\bigoplus_{r+s=n}\operatorname{H}_{\operatorname{dR}}^{r}(X)^b\widehat{\otimes}_K\operatorname{H}_{\operatorname{dR}}^{s}(Y)^b).
\end{equation*}

Hence, all cohomology groups of $\operatorname{H}_{\operatorname{dR}}^{\bullet}(X)^b\widehat{\otimes}_K^{\mathbb{L}}\operatorname{H}_{\operatorname{dR}}^{\bullet}(Y)^b$ are objects of  the essential image of the functor $I:\widehat{\mathcal{B}}c_K\rightarrow LH(\widehat{\mathcal{B}}c_K)$. Therefore,  $\operatorname{H}_{\operatorname{dR}}^{\bullet}(X)^b\widehat{\otimes}_K^{\mathbb{L}}\operatorname{H}_{\operatorname{dR}}^{\bullet}(Y)^b$ is a strict complex, as we wanted.
\end{proof}
\begin{teo}\label{teo künneth formula de Rham cohomology}
    Let $X$ and $Y$ be smooth Stein spaces. We have the following identity in $\operatorname{D}(\widehat{\mathcal{B}}c_K):$
    \begin{equation*}
      \operatorname{H}_{\operatorname{dR}}^{\bullet}(X\times Y)^b=\operatorname{H}_{\operatorname{dR}}^{\bullet}(X)^b\widehat{\otimes}_K^{\mathbb{L}}\operatorname{H}_{\operatorname{dR}}^{\bullet}(Y)^b.  
    \end{equation*}
\end{teo}
In particular, for any $n\in \mathbb{Z}$, we get the following Künneth formula in $\widehat{\mathcal{B}}c_K$:
\begin{equation*}
    \operatorname{H}_{\operatorname{dR}}^{n}(X\times Y)^b=\bigoplus_{r+s=n}\operatorname{H}_{\operatorname{dR}}^{r}(X)^b\widehat{\otimes}_K\operatorname{H}_{\operatorname{dR}}^{s}(Y)^b.
\end{equation*}
\begin{proof}
As $X$ and $Y$ are smooth Stein spaces, $X\times Y$ is also a smooth Stein space by Lemma \ref{Lemma prod of Stein is Stein}.  We need to show that the canonical morphism of chain complexes of complete bornological spaces:
\begin{equation*}
    \operatorname{H}_{\operatorname{dR}}^{\bullet}(X)^b\widehat{\otimes}_K^{\mathbb{L}}\operatorname{H}_{\operatorname{dR}}^{\bullet}(Y)^b\rightarrow\operatorname{H}_{\operatorname{dR}}^{\bullet}(X\times Y)^b,
\end{equation*}
is an isomorphism in $\operatorname{D}(\widehat{\mathcal{B}}c_K)$. By Lemma \ref{Lemma junneth formula for complexes}, this amounts to showing that for $n\in\mathbb{Z}$ we have:
\begin{equation*}
    \operatorname{H}_{\operatorname{dR}}^{n}(X\times Y)^b=\bigoplus_{r+s=n}\operatorname{H}_{\operatorname{dR}}^{r}(X)^b\widehat{\otimes}_K\operatorname{H}_{\operatorname{dR}}^{s}(Y)^b.
\end{equation*}
Fix some $n\in \mathbb{Z}$. By the proof of \cite[Corollary 3.2]{grosse2004rham}, we have an admissible cover $\{U_i\}_{i\geq 0}$ of $X$ satisfying that $U_i\subset U_{i+1}$ for all $i\geq 0$, and $U_i$ is a smooth Stein space with finite dimensional de Rham cohomology for every $i\geq 0$. We claim that such a cover satisfies the following identity $\widehat{\mathcal{B}}c_K$:
\begin{equation*}
    \operatorname{H}_{\operatorname{dR}}^{n}(X)^b=\varprojlim_i\operatorname{H}_{\operatorname{dR}}^{n}(U_i)^b,
\end{equation*}
so that we have a strict short exact sequence:
\begin{equation}\label{equation short exact sequence inverse limit de Rham cohomology}
    0\rightarrow \operatorname{H}_{\operatorname{dR}}^{n}(X)^b \rightarrow \prod_{i\geq 0}\operatorname{H}_{\operatorname{dR}}^{n}(U_i)^b\rightarrow \prod_{i\geq 0}\operatorname{H}_{\operatorname{dR}}^{n}(U_i)^b\rightarrow 0.
\end{equation}
Indeed, as $\operatorname{H}_{\operatorname{dR}}^{n}(U_i)$ is finite dimensional for all $i\geq 0$, it follows by \cite[Chapter 2 Corollary 5]{schneider1991cohomology} that:
\begin{equation*}
    R^1\varprojlim \operatorname{H}_{\operatorname{dR}}^{n}(U_i)=0.
\end{equation*}
Hence, we have the following short exact sequence in $LCS_K$:
\begin{equation*}
    0\rightarrow \operatorname{H}_{\operatorname{dR}}^{n}(X) \rightarrow \prod_{i\geq 0}\operatorname{H}_{\operatorname{dR}}^{n}(U_i)\rightarrow \prod_{i\geq 0}\operatorname{H}_{\operatorname{dR}}^{n}(U_i)\rightarrow 0.
\end{equation*}
The last two terms of the sequence are countable products of nuclear Fréchet spaces, hence nuclear Fréchet spaces by \cite[Proposition 19.7]{schneider2013nonarchimedean}. Thus, it follows by \cite[Proposition 5.12]{bode2021operations} that $(\ref{equation short exact sequence inverse limit de Rham cohomology})$ is a strict short exact sequence of complete bornological spaces.\\

Let $\{V_i\}_{i\geq 0}$ be an admissible covering of $Y$ satisfying the same properties as $\{U_i\}_{i\geq 0}$. Then by the Künneth formula \cite[proposition 3.3]{grosse2004rham}, it follows that $\{U_i\times V_i\}_{i\geq 0}$ is an admissible covering of $X\times Y$ satisfying the same properties. In particular, we have: 
\begin{equation*}
\operatorname{H}_{\operatorname{dR}}^{n}(X\times Y)^b=\varprojlim_{i}\operatorname{H}_{\operatorname{dR}}^{n}(U_i\times V_i)^b=\bigoplus_{r+s=n}\varprojlim_{i}\left(\operatorname{H}_{\operatorname{dR}}^{r}(U_i)^b\widehat{\otimes}_K\operatorname{H}_{\operatorname{dR}}^{s}(V_i)^b \right).    
\end{equation*}
Fix $r,s\geq 0$ such that $r+s=n$. Then we have:
\begin{equation*}
    \varprojlim_{i}\left(\operatorname{H}_{\operatorname{dR}}^{r}(U_i)^b\widehat{\otimes}_K\operatorname{H}_{\operatorname{dR}}^{s}(V_i)^b \right)=\varprojlim_{j}\varprojlim_i\left(\operatorname{H}_{\operatorname{dR}}^{r}(U_i)^b\widehat{\otimes}_K\operatorname{H}_{\operatorname{dR}}^{s}(V_j)^b \right).
\end{equation*}
For each $j\geq 0$, the space $\operatorname{H}_{\operatorname{dR}}^{s}(V_j)^b$ is finite dimensional, and $\prod_{i\geq 0}\operatorname{H}_{\operatorname{dR}}^{r}(U_i)^b$ is the product of a countable family of finite dimensional spaces. Thus, by \cite[Corollary 5.2]{bode2021operations}, we have:
\begin{equation*}
    \left(\prod_{i\geq 0}\operatorname{H}_{\operatorname{dR}}^{r}(U_i)^b\right)\widehat{\otimes}_K\operatorname{H}_{\operatorname{dR}}^{s}(V_j)^b=\prod_{i\geq 0}\left(\operatorname{H}_{\operatorname{dR}}^{r}(U_i)^b\widehat{\otimes}_K\operatorname{H}_{\operatorname{dR}}^{s}(V_j)^b\right).
\end{equation*}
As $-\widehat{\otimes}_K\operatorname{H}_{\operatorname{dR}}^{s}(V_j)^b$ is strongly exact, the fact that $(\ref{equation short exact sequence inverse limit de Rham cohomology})$ is strict exact implies that we have a strict short exact sequence:
\begin{equation*}
     0\rightarrow \operatorname{H}_{\operatorname{dR}}^{r}(X)^b\widehat{\otimes}_K\operatorname{H}_{\operatorname{dR}}^{s}(V_j)^b \rightarrow \prod_{i\geq 0}\left(\operatorname{H}_{\operatorname{dR}}^{r}(U_i)^b\widehat{\otimes}_K\operatorname{H}_{\operatorname{dR}}^{s}(V_j)^b\right)\rightarrow \prod_{i\geq 0}\left(\operatorname{H}_{\operatorname{dR}}^{r}(U_i)^b\widehat{\otimes}_K\operatorname{H}_{\operatorname{dR}}^{s}(V_j)^b\right)\rightarrow 0.
\end{equation*}
In particular, we arrive at the following identity in $\widehat{\mathcal{B}}c_K$:
\begin{equation*}
    \varprojlim_i\left(\operatorname{H}_{\operatorname{dR}}^{r}(U_i)^b\widehat{\otimes}_K\operatorname{H}_{\operatorname{dR}}^{s}(V_j)^b \right)=\operatorname{H}_{\operatorname{dR}}^{r}(X)^b\widehat{\otimes}_K\operatorname{H}_{\operatorname{dR}}^{s}(V_j)^b.
\end{equation*}
Applying the same argument above, and using the fact that $\operatorname{H}_{\operatorname{dR}}^{r}(X)$ is a Fréchet space, we get:
\begin{equation*}
\varprojlim_j\varprojlim_{i}\left(\operatorname{H}_{\operatorname{dR}}^{r}(U_i)^b\widehat{\otimes}_K\operatorname{H}_{\operatorname{dR}}^{s}(V_j)^b \right)=\varprojlim_j\operatorname{H}_{\operatorname{dR}}^{r}(X)^b\widehat{\otimes}_K\operatorname{H}_{\operatorname{dR}}^{s}(V_j)^b=\operatorname{H}_{\operatorname{dR}}^{r}(X)^b\widehat{\otimes}_K\operatorname{H}_{\operatorname{dR}}^{s}(Y)^b.
\end{equation*}
Hence, we get $\operatorname{H}_{\operatorname{dR}}^{n}(X\times Y)^b=\bigoplus_{r+s=n}\operatorname{H}_{\operatorname{dR}}^{r}(X)^b\widehat{\otimes}_K\operatorname{H}_{\operatorname{dR}}^{s}(Y)^b$, as we wanted to show.
\end{proof}
As a consequence, we get a Künneth formula for Hochschild cohomology:
\begin{coro}\label{coro kunneth formula for Hochschild cohomology}
Let $X$ and $Y$ be smooth Stein spaces. The following identity holds in $\operatorname{D}(\widehat{\mathcal{B}}c_K):$
\begin{equation*}
\operatorname{HH}^{\bullet}(\wideparen{\D}_{X\times Y})=\operatorname{HH}^{\bullet}(\wideparen{\D}_{X})\widehat{\otimes}^{\mathbb{L}}_K\operatorname{HH}^{\bullet}(\wideparen{\D}_{Y}).    
\end{equation*}
In particular, for any $n\in \mathbb{Z}$ we have the following Künneth formula in $\widehat{\mathcal{B}}c_K:$
\begin{equation*}
    \operatorname{HH}^{n}(\wideparen{\D}_{X\times Y})=\bigoplus_{r+s=n}\operatorname{HH}^{r}(\wideparen{\D}_{X})\widehat{\otimes}_K\operatorname{HH}^{s}(\wideparen{\D}_{Y}).
\end{equation*}
\end{coro}
\begin{proof}
We have the following isomorphisms:
\begin{equation*}
 \operatorname{HH}^{n}(\wideparen{\D}_{X\times Y})=\operatorname{H}_{\operatorname{dR}}^{\bullet}(X\times Y)^b=\operatorname{H}_{\operatorname{dR}}^{\bullet}(X)^b\widehat{\otimes}_K^{\mathbb{L}}\operatorname{H}_{\operatorname{dR}}^{\bullet}(Y)^b=\operatorname{HH}^{n}(\wideparen{\D}_{X})\widehat{\otimes}_K^{\mathbb{L}}\operatorname{HH}^{n}(\wideparen{\D}_{Y}),
\end{equation*}
 where the first and third identities follow from Theorem \ref{teo hochschild cohomology groups as de rham cohomology groups}, and the second one is Theorem \ref{teo künneth formula de Rham cohomology}. The second part of the theorem follows from the corresponding Künneth formula for the de Rham cohomology of a product of smooth Stein spaces.
\end{proof}
\subsection{Hochschild cohomology and extensions I}\label{Section Hochschild cohomology and extension I}
As mentioned above, to final goal of this chapter is highlighting the relevance of Hochschild cohomology in the study of smooth rigid analytic spaces, and in particular in the study of smooth Stein spaces. In the previous section, we established a direct comparison between $\operatorname{HH}^{\bullet}(\wideparen{\D}_X)$ and $\operatorname{H}_{\operatorname{dR}}^{\bullet}(X)$, which showcases the geometric significant of Hochschild cohomology. In this section, we follow a more algebraic approach, by showing how $\operatorname{HH}^{\bullet}(\wideparen{\D}_X)$ can be related to the Ext functors arising as the derived functors of the different types of \emph{Hom} functors.
As before, we fix a smooth Stein space $X=\varinjlim_n X_n$ with an étale map $X\rightarrow \mathbb{A}^r_K$.
\begin{obs}
For the rest of the section, we assume $K$ is discretely valued or algebraically closed.
\end{obs}
Let us start by showing an \emph{inner} version of Lemma \ref{lemma co-admis on Stein spaces are determined by global sections}:
\begin{Lemma}\label{Lemma global sections inner hom functor}
Let $\mathcal{M}\in \mathcal{C}(\wideparen{\D}_X)$, and $\mathcal{N}\in \Mod_{\Indban}(\wideparen{\D}_X)$ . The following identity holds:
\begin{equation*}
    \Gamma(X,\underline{\mathcal{H}om}_{\wideparen{\D}_X}(\mathcal{M},\mathcal{N}))=\underline{\Hom}_{\wideparen{\D}_X(X)}(\mathcal{M}(X),\mathcal{N}(X)).
\end{equation*}
\end{Lemma}
\begin{proof}
By definition of  $\underline{\mathcal{H}om}_{\wideparen{\D}_X}(-,-)$, we have the following identity:
\begin{equation*}
  \Gamma(X,\underline{\mathcal{H}om}_{\wideparen{\D}_X}(\mathcal{M},\mathcal{N}))=\operatorname{Eq}\left(\prod_{U\subset X}\underline{\Hom}_{\wideparen{\D}_X(U)}(\mathcal{M}(U),\mathcal{N}(U))\rightrightarrows \prod_{V\subset U}\underline{\Hom}_{\wideparen{\D}_X(U)}(\mathcal{M}(U),\mathcal{N}(V))\right),  
\end{equation*}
where $U$ (resp. $V$) runs over all the affinoid subspaces of $X$ (resp. all affinoid subdomains of $U$). The maps in the equalizer are constructed as follows: Let $U\subset X$ be an affinoid open subspace, and $V\subset U$ be an affinoid subdomain. The first of the maps in the equalizer is  the product of the maps:
\begin{equation*}
    \underline{\Hom}_{\wideparen{\D}_X(U)}(\mathcal{M}(U),\mathcal{N}(U))\rightarrow \underline{\Hom}_{\wideparen{\D}_X(U)}(\mathcal{M}(U),\mathcal{N}(V)),
\end{equation*}
induced by pushforward along the restriction map $\mathcal{N}(U)\rightarrow \mathcal{N}(V)$. The second map is obtained via products of maps of the form:
\begin{equation*}
    \underline{\Hom}_{\wideparen{\D}_X(V)}(\mathcal{M}(V),\mathcal{N}(V))\rightarrow \underline{\Hom}_{\wideparen{\D}_X(U)}(\mathcal{M}(U),\mathcal{N}(V)),
\end{equation*}
induced by pullback along the map $\mathcal{M}(U)\rightarrow \mathcal{M}(V)$. As $X=\varinjlim X_n$ is a Stein space, we may simplify this equalizer as follows:
\begin{equation}\label{equation 1 global sections of inner hom}  
\operatorname{Eq}\left(\prod_{n\geq 1}\underline{\Hom}_{\wideparen{\D}_X(X_n)}(\mathcal{M}(X_n),\mathcal{N}(X_n))\rightrightarrows \prod_{n\geq 1}\underline{\Hom}_{\wideparen{\D}_X(X_n)}(\mathcal{M}(X_n),\mathcal{N}(X_{n-1}))\right).
\end{equation}
As $\mathcal{M}$ is a co-admissible $\wideparen{\D}_X$-module, it follows by Proposition \ref{prop characterization of co-admissible modules on a smooth Stein space}, that for every $n\geq 0$ we have:
\begin{equation*}
    \mathcal{M}(X_n)=\wideparen{\D}_X(X_n)\overrightarrow{\otimes}_{\wideparen{\D}_X(X)}\mathcal{M}(X).
\end{equation*}
In particular, a quick Yoneda argument shows that for every $n\geq 0$ we have the following identities:
\begin{align*}
  \underline{\Hom}_{\wideparen{\D}_X(X_n)}(\mathcal{M}(X_n),\mathcal{N}(X_n))=\underline{\Hom}_{\wideparen{\D}_X(X)}(\mathcal{M}(X),\mathcal{N}(X_n)),\\ \underline{\Hom}_{\wideparen{\D}_X(X_n)}(\mathcal{M}(X_n),\mathcal{N}(X_{n-1}))= \underline{\Hom}_{\wideparen{\D}_X(X)}(\mathcal{M}(X),\mathcal{N}(X_{n-1}))
\end{align*}
Hence, the equalizer in $(\ref{equation 1 global sections of inner hom})$ can be rewritten as follows:
\begin{align*}
  \operatorname{Eq}\left(\prod_{n\geq 1}\underline{\Hom}_{\wideparen{\D}_X(X)}(\mathcal{M}(X),\mathcal{N}(X_n))\rightrightarrows \prod_{n\geq 1}\underline{\Hom}_{\wideparen{\D}_X(X)}(\mathcal{M}(X),\mathcal{N}(X_{n-1}))\right)\\
  =\varprojlim_n\underline{\Hom}_{\wideparen{\D}_X(X)}(\mathcal{M}(X),\mathcal{N}(X_n))\\=\underline{\Hom}_{\wideparen{\D}_X(X)}(\mathcal{M}(X),\varprojlim_n\mathcal{N}(X_n))\\
  =\underline{\Hom}_{\wideparen{\D}_X(X)}(\mathcal{M}(X),\mathcal{N}(X)),  
\end{align*}
and this is precisely the identity we wanted to show.
\end{proof}
As $\OX_X$ is a co-admissible $\wideparen{\D}_X$-module, we have the following identity of functors:
\begin{equation}\label{equation composition of differnet internal homs}
\Gamma(X,\underline{\mathcal{H}om}_{\wideparen{\D}_X}(\OX_X,-))=\underline{\Hom}_{\wideparen{\D}_X(X)}(\OX_X(X),\Gamma(X,-)),
\end{equation}
and we wish to study its derived functor. Recall from \cite[Section 6.3]{bode2021operations} that the Spencer resolution $S^{\bullet}$ gives a strict resolution of $\OX_X$ by finite locally free $\wideparen{\D}_X$-modules. Namely, we have a complex:
\begin{equation*}
    S^{\bullet}:=\left(0\rightarrow \wideparen{\D}_X\overrightarrow{\otimes}_{\OX_X}\wedge^n\mathcal{T}_{X/K}\rightarrow \cdots \rightarrow \wideparen{\D}_X \rightarrow 0\right),
\end{equation*}
which is quasi-isomorphic to $\OX_X$ in $\operatorname{D}(\wideparen{\D}_X)$. Furthermore, as we have an étale map $X\rightarrow \mathbb{A}^r_K$, the terms in $S^{\bullet}$ are finite free $\wideparen{\D}_X$-modules. Let $\mathcal{M}\in \mathcal{C}(\wideparen{\D}_X)$ be a co-admissible $\wideparen{\D}_X$-module. In this situation, we have the following chain of identities:
\begin{equation*}
    R\left(\Gamma(X,\underline{\mathcal{H}om}_{\wideparen{\D}_X}(\OX_X,\mathcal{M}))\right)=R\Gamma(X,R\underline{\mathcal{H}om}_{\wideparen{\D}_X}(\OX_X,\mathcal{M}))=R\Gamma(X,\underline{\mathcal{H}om}_{\wideparen{\D}_X}(S^{\bullet},\mathcal{M})),
\end{equation*}
where the first identity is \cite[Proposition 3.3.12]{HochDmod}, and the second identity follows from the fact that $S^{\bullet}$ is a resolution of $\OX_X$ by finite free $\wideparen{\D}_X$-modules. Furthermore, the objects of the complex $\underline{\mathcal{H}om}_{\wideparen{\D}_X}(S^{\bullet},\mathcal{M})$ are finite direct sums of $\mathcal{M}$. In particular, they are co-admissible $\wideparen{\D}_X$-modules. Thus, the objects of $\underline{\mathcal{H}om}_{\wideparen{\D}_X}(S^{\bullet},\mathcal{M})$ are $\Gamma(X,-)$-acyclic by Proposition \ref{prop acyclicity of co-admissible modules on Stein spaces}. Hence, we have:
\begin{equation*}
R\Gamma(X,\underline{\mathcal{H}om}_{\wideparen{\D}_X}(S^{\bullet},\mathcal{M}))=\Gamma(X,\underline{\mathcal{H}om}_{\wideparen{\D}_X}(S^{\bullet},\mathcal{M}))=\underline{\Hom}_{\wideparen{\D}_X(X)}(\Gamma(X,S^{\bullet}),\Gamma(X,\mathcal{M})).    
\end{equation*}
On the other hand, as the map $S^{\bullet}\rightarrow \OX_X$ is an isomorphism in $\operatorname{D}(\wideparen{\D}_X)$ between complexes of $\Gamma(X,-)$-acyclic modules, it follows that $\Gamma(X,S^{\bullet})$ is a strict resolution of $\OX_X(X)$ by finite-free $\wideparen{\D}_X(X)$-modules. Therefore, we have the following identities in $\operatorname{D}(\widehat{\mathcal{B}}c_K)$:
\begin{equation*}
 R\underline{\Hom}_{\wideparen{\D}_X(X)}(\OX_X(X),\Gamma(X,\mathcal{M}))= \underline{\Hom}_{\wideparen{\D}_X(X)}(\Gamma(X,S^{\bullet}),\Gamma(X,\mathcal{M}))=R\Gamma(X,R\underline{\mathcal{H}om}_{\wideparen{\D}_X}(\OX_X,\mathcal{M})).
\end{equation*}
We may summarize the previous discussion into the following proposition:
\begin{defi}[{\cite[Definition 8.1.3.]{HochDmod}}]
For any $\mathcal{M}\in\operatorname{D}(\wideparen{\D}_X)$, the de Rham complex of $\mathcal{M}$ is defined as the following complex in $\operatorname{D}(\operatorname{Shv}(X,\Indban))$:
\begin{equation*}
    \operatorname{dR}(\mathcal{M}):= R\underline{\mathcal{H}om}_{\wideparen{\D}_X}(\OX_{X},\mathcal{M}).
\end{equation*}    
\end{defi}
\begin{prop}\label{prop hochschild cohomology as ext of structure sheaf as D module part 1}
Let $X$ be a Stein space with an étale map $X\rightarrow \mathbb{A}^r_K$. For any co-admissible module $\mathcal{M}\in \mathcal{C}(\wideparen{\D}_X)$ we have the following identity in $\operatorname{D}(\widehat{\mathcal{B}}c_K):$
\begin{equation*}
    R\Gamma(X,\operatorname{dR}(\mathcal{M}))=R\underline{\Hom}_{\wideparen{\D}_X(X)}(\OX_X(X),\mathcal{M}(X)).
\end{equation*}
\end{prop}
Next, we want to investigate the extent to which this discussion holds for the external \emph{Hom} functors $\mathcal{H}om_{\wideparen{\D}_X}(-,-)$, and $\Hom_{\wideparen{\D}_X}(-,-)$ from Definition \ref{defi external sheaf of hom}.  As we will see below, these functors behave exceptionally well when applied to co-admissible $\wideparen{\D}_X$-modules on a smooth Stein space. Furthermore, as we mentioned at the end of Section \ref{section external hom functors}, the derived functor $R\Hom_{\wideparen{\D}_X}(-,-)$ can be used to parameterize Yoneda's extensions. Our next goal is using this interpretation to obtain a new application of the Hochschild cohomology groups of $\wideparen{\D}_X$. Let us start with the following lemma: 
\begin{Lemma}\label{lemma external hom Stein co-adm 1}
Let $X$ be a Stein space or an affinoid space, with an étale map $X\rightarrow \mathbb{A}^r_K$.\newline Let $\mathcal{M}\in \mathcal{C}(\wideparen{\D}_X)$ and $\mathcal{N}\in\Mod_{LH(\widehat{\mathcal{B}}c_K)}(I(\wideparen{\D}_X))$. We have the following identity of $K$-vector spaces:
\begin{equation*}
    \Hom_{I(\wideparen{\D}_X)}(I(\mathcal{M}),\mathcal{N})=\Hom_{I(\wideparen{\D}_X(X))}(I(\mathcal{M}(X)),\mathcal{N}(X)).
\end{equation*}
\end{Lemma}
\begin{proof}
This is analogous to the proof of Lemma \ref{lemma co-admis on Stein spaces are determined by global sections}.
\end{proof}
Many of the arguments leading up to Proposition \ref{prop hochschild cohomology as ext of structure sheaf as D module part 1} are a consequence of the fact that locally finite-free $I(\wideparen{\D}_X)$-modules are acyclic for $\underline{\mathcal{H}om}_{I(\wideparen{\D}_X)}(-,-)$, and this allows us to use the Spencer resolution to calculate all derived functors involved. In order to study the situation for the external sheaves of homomorphisms, we need to show that a similar result holds. Let us start with the following:
\begin{Lemma}\label{lemma forgetful functor}
Let $\mathcal{V}\in LH(\widehat{\mathcal{B}}c_K)$. Then for any $n\geq 0$  we have:
\begin{equation*}
    \Hom_{LH(\widehat{\mathcal{B}}c_K)}(\bigoplus_{i=1}^nI(K),\mathcal{V})=\bigoplus_{i=1}^n\operatorname{Forget}(J(\mathcal{V})).
\end{equation*}
\end{Lemma}
\begin{proof}
For simplicity, we may assume $n=1$. By construction of the left heart, there are $V_1,V_2\in\widehat{\mathcal{B}}c_K$ satisfying that we have the following short exact sequence in $LH(\widehat{\mathcal{B}}c_K)$:
\begin{equation*}
    0\rightarrow I(V_1)\rightarrow I(V_2)\rightarrow\mathcal{V}\rightarrow 0.
\end{equation*}
As $I(K)$ is projective in $LH(\widehat{\mathcal{B}}c_K)$, the functor $\Hom_{LH(\widehat{\mathcal{B}}c_K)}(I(K),-)$ is exact. Thus, we have the following exact sequence of $K$-vector spaces:
\begin{equation*}
    0\rightarrow \Hom_{LH(\widehat{\mathcal{B}}c_K)}(I(K),I(V_1))\rightarrow \Hom_{LH(\widehat{\mathcal{B}}c_K)}(I(K),I(V_2))\rightarrow \Hom_{LH(\widehat{\mathcal{B}}c_K)}(I(K),\mathcal{V})\rightarrow 0.
\end{equation*}
As $J$ and $\operatorname{Forget}(-)$ are exact and commute with finite direct sums, this short exact sequence shows that it suffices to show the claim for $\mathcal{V}=I(W)$ for some $W\in \widehat{\mathcal{B}}c_K$. However, as the functor $I:\widehat{\mathcal{B}}c_K\rightarrow LH(\widehat{\mathcal{B}}c_K)$ is fully faithful, we have:
\begin{equation*}
   \Hom_{LH(\widehat{\mathcal{B}}c_K)}(I(K),I(W))=\Hom_{\widehat{\mathcal{B}}c_K}(K,W)=\Hom_{\mathcal{B}c_K}(J(K),J(W))=\operatorname{Forget}(J(W)), 
\end{equation*}
which is the identity we wanted to show. Hence, the lemma holds.
\end{proof}
Next, we want to generalize this lemma to  the level of sheaves. Recall the version for sheaves of $J$, and $\operatorname{Forget}(-)$ given in Definition \ref{defi J and Forget for sheaves}. We can use the previous lemma to show the following:
\begin{Lemma}\label{Lemma finite free are acyclic for external hom}
The following identity of functors holds:
\begin{equation*}
    \mathcal{H}om_{I(\wideparen{\D}_X)}(I(\wideparen{\D}_X),-)=\operatorname{Forget}(J(-)):\Mod_{LH(\widehat{B}c_K)}(I(\wideparen{\D}_X))\rightarrow \operatorname{Shv}(X,\operatorname{Vect}_K).
\end{equation*}
In particular, $\mathcal{H}om_{I(\wideparen{\D}_X)}(I(\wideparen{\D}_X),-)$ is an exact functor.
\end{Lemma}
\begin{proof}
Choose a module $\mathcal{M}\in \Mod_{LH(\widehat{B}c_K)}(I(\wideparen{\D}_X))$. By Lemma \ref{lemma external hom Stein co-adm 1}, $\mathcal{H}om_{I(\wideparen{\D}_X)}(I(\wideparen{\D}_X),\mathcal{M})$ is a sheaf of $K$-vector spaces given by the following expression for each affinoid open subspace $U\subset X$:
\begin{equation*}
   \Gamma(U,\mathcal{H}om_{I(\wideparen{\D}_X)}(I(\wideparen{\D}_X),\mathcal{M}))=\Hom_{I(\wideparen{\D}_X(U))}(I(\wideparen{\D}_X(U)),\mathcal{M}(U))=\Hom_{LH(\widehat{\mathcal{B}}c_K)}(I(K),\mathcal{M}(U)). 
\end{equation*}
Hence, by Lemma \ref{lemma forgetful functor} we have
the following identity of functors:
 \begin{equation*}
     \mathcal{H}om_{I(\wideparen{\D}_X)}(I(\wideparen{\D}_X),-)=\operatorname{Forget}(J\underline{\mathcal{H}om}_{I(\wideparen{\D}_X)}(I(\wideparen{\D}_X),-)),
 \end{equation*}
Furthermore, by \cite[Lemma 8.1.10]{HochDmod} the composition:
\begin{equation*}
    \operatorname{Forget}(-)\circ J:\operatorname{Shv}(X,LH(\widehat{\mathcal{B}}c_K))\rightarrow \operatorname{Shv}(X,\operatorname{Vect}_K),
\end{equation*}
is exact. Thus, $\mathcal{H}om_{I(\wideparen{\D}_X)}(I(\wideparen{\D}_X),-)$ is exact, as we wanted to show.
\end{proof}
As before, this result has important consequences for the derived external \emph{Hom} functors:
\begin{Lemma}\label{Lemma different ext hom functors}
Let $\mathcal{I}\in \Mod_{LH(\widehat{\mathcal{B}}c_K)}(I(\wideparen{\D}_X))$  be an injective object. Then $\mathcal{H}om_{I(\wideparen{\D}_X)}(I(\OX_X),\mathcal{I})$ is a $\Gamma(X,-)$-acyclic sheaf of $K$-vector spaces. In particular, the following identity holds:
\begin{equation*}
    R\Hom_{\wideparen{\D}_X}(\OX_X,-)=R\Gamma(X,R\mathcal{H}om_{\wideparen{\D}_X}(\OX_X,-)):\operatorname{D}(\wideparen{\D}_X)^+\rightarrow \operatorname{D}(\operatorname{Vect}_K)^+.
\end{equation*}
\end{Lemma}
\begin{proof}
As $\mathcal{I}$ is injective, we can use the Spencer resolution to obtain the following exact sequence:
\begin{multline*}
    0\rightarrow \mathcal{H}om_{I(\wideparen{\D}_X)}(I(\OX_X),\mathcal{I}) \rightarrow \mathcal{H}om_{I(\wideparen{\D}_X)}(I(\wideparen{\D}_X),\mathcal{I})  \rightarrow \\
    \cdots \rightarrow \mathcal{H}om_{I(\wideparen{\D}_X)}(I(\wideparen{\D}_X\overrightarrow{\otimes}_{\OX_X}\wedge^n\mathcal{T}_{X/K}),\mathcal{I})\rightarrow 0.
\end{multline*}
An application of Lemma \ref{Lemma finite free are acyclic for external hom} shows that this sequence can be rewritten as follows:
\begin{multline*}
    0\rightarrow \mathcal{H}om_{I(\wideparen{\D}_X)}(I(\OX_X),\mathcal{I}) \rightarrow \operatorname{Forget}(J\underline{\mathcal{H}om}_{I(\wideparen{\D}_X)}(I(\wideparen{\D}_X),\mathcal{I}))  \rightarrow \\
    \cdots \rightarrow \operatorname{Forget}(J\underline{\mathcal{H}om}_{I(\wideparen{\D}_X)}(I(\wideparen{\D}_X\overrightarrow{\otimes}_{\OX_X}\wedge^n\mathcal{T}_{X/K}),\mathcal{I}))\rightarrow 0.
\end{multline*}
Recall that by \cite[Lemma 8.1.10]{HochDmod}, the functors $J$ and $\operatorname{Forget}(-)$ are exact. Thus, it follows that we have an isomorphism of sheaves:
\begin{equation*}
    \mathcal{H}om_{I(\wideparen{\D}_X)}(I(\OX_X),\mathcal{I})=\operatorname{Forget}(J\underline{\mathcal{H}om}_{I(\wideparen{\D}_X)}(I(\OX_X),\mathcal{I})).
\end{equation*}
As $\mathcal{I}$ is injective, it follows by \cite[Lemma 3.3.13]{HochDmod} that $\underline{\mathcal{H}om}_{I(\wideparen{\D}_X)}(I(\OX_X),\mathcal{I})$ is injective. Thus, it is $\Gamma(X,-)$-acyclic. But then it follows by \cite[Proposition 8.1.11]{HochDmod} that $\mathcal{H}om_{I(\wideparen{\D}_X)}(I(\OX_X),\mathcal{I})$ is also $\Gamma(X,-)$-acyclic, as we wanted to show. For the second part of the lemma, notice that we have restricted the derived functors to the bounded below derived category. Hence, the equation follows from the canonical identification:
\begin{equation*}
\Hom_{\wideparen{\D}_X}(-,-)=\Gamma(X,\mathcal{H}om_{\wideparen{\D}_X}(-,-)),    
\end{equation*}
together with the first part of the lemma.
\end{proof}
These lemmas provide us with enough technical tools to show the following theorem:
\begin{teo}\label{teo hochschild cohomology as ext of structure sheaf as D module}
Let $X$ be a Stein space with an étale map $X\rightarrow \mathbb{A}^r_K$. For any co-admissible module $\mathcal{M}\in \mathcal{C}(\wideparen{\D}_X)$ we have the following identity in $\operatorname{D}(\widehat{\mathcal{B}}c_K):$
\begin{equation}\label{equation HH and inner hom}
    R\Gamma(X,\operatorname{dR}(\mathcal{M}))=R\underline{\Hom}_{\wideparen{\D}_X(X)}(\OX_X(X),\mathcal{M}(X)).
\end{equation}
Furthermore, we have the following identities in $\operatorname{D}(\operatorname{Vect}_K)$:
\begin{equation}\label{equation HH and ext hom}
    R\Hom_{\wideparen{\D}_X}(\OX_X,\mathcal{M})=R\Gamma(X,R\mathcal{H}om_{\wideparen{\D}_X}(\OX_X,\mathcal{M}))=R\Hom_{\wideparen{\D}_X(X)}(\OX_X(X),\mathcal{M}(X)).
\end{equation}
In particular, the following holds: 
\begin{enumerate}[label=(\roman*)]
    \item $\operatorname{HH}^{\bullet}(\wideparen{\D}_X)=R\underline{\Hom}_{\wideparen{\D}_X(X)}(\OX_X(X),\OX_X(X))$ as objects in $\operatorname{D}(\widehat{\mathcal{B}}c_K)$.
    \item We have the following identities in $\operatorname{D}(\operatorname{Vect}_K)$:
    \begin{equation*}
        \operatorname{Forget}(J\operatorname{HH}^{\bullet}(\wideparen{\D}_X))=R\Hom_{\wideparen{\D}_X}(\OX_X,\OX_X)=R\Hom_{\wideparen{\D}_X(X)}(\OX_X(X),\OX_X(X)).
    \end{equation*}
\end{enumerate}
\end{teo}
\begin{proof}
Equation (\ref{equation HH and inner hom}) is a special case of Proposition \ref{prop hochschild cohomology as ext of structure sheaf as D module part 1}. Furthermore, our calculations of Hochschild cohomology show that $\operatorname{HH}^{\bullet}(\wideparen{\D}_X)=R\Gamma(X,\operatorname{dR}(\OX_X))$. Thus, statement $(i)$ holds as well.\newline
Next, we want to show (\ref{equation HH and ext hom}). Notice that the first identity was already shown in Lemma \ref{Lemma different ext hom functors}. The second identity is a direct calculation using the Spencer resolution. Namely, we have the following:
\begin{multline*}
    R\Hom_{\wideparen{\D}_X(X)}(\OX_X(X),\mathcal{M}(X))=\Hom_{\wideparen{\D}_X(X)}(\Gamma(X,S^{\bullet}),\mathcal{M}(X))=\Gamma(X,\mathcal{H}om_{\wideparen{\D}_X}(S^{\bullet},\mathcal{M}))\\
    =\Gamma(X,R\mathcal{H}om_{\wideparen{\D}_X}(\OX_X,\mathcal{M}))=R\Gamma(X,R\mathcal{H}om_{\wideparen{\D}_X}(\OX_X,\mathcal{M})).
\end{multline*}
In order to show $(ii)$,  we proceed by a direct computation:
\begin{flalign*}
 R\Hom_{\wideparen{\D}_X(X)}(\OX_X(X),\OX_X(X))=&&
 \end{flalign*}
 \begin{flalign*}   
 &&\left(0\rightarrow \Hom_{\OX_X(X)}(\OX_X(X),\OX_X(X))\rightarrow \cdots \rightarrow \Hom_{\OX_X(X)}(\wedge^n\mathcal{T}_{X/K},\OX_X(X)) \rightarrow 0  \right)\\
 &&=\left(\operatorname{Forget}(J\OX_X(X))\rightarrow \cdots     \rightarrow \operatorname{Forget}(J\Omega_{X/K}^n(X))\right)\\
 &&=\operatorname{Forget}(J\Gamma(X,\Omega^{\bullet}_{X/K}))\\
 &&=\operatorname{Forget}(J\operatorname{HH}^{\bullet}(\wideparen{\D}_X)),
\end{flalign*}
where in the second identity we are using the fact that the functor:
\begin{equation*}
    \operatorname{Coh}(\OX_X)\rightarrow \Mod_{\Indban}(\OX_X),
\end{equation*}
is fully faithful (\emph{cf.} \cite[Theorem 6.4]{bode2021operations}), together with Lemma \ref{lemma forgetful functor}. Hence, the theorem holds.
\end{proof}
We remark that this result was already known in the context of smooth affine $K$-varieties. See, for example, \cite[Proposition 2.1]{etingof2004cherednik}. However, our result adds an extra layer of depth, as the spaces involved are complete bornological spaces. As mentioned in the end of Section \ref{section external hom functors}, this  theorem allows us to regard  $\operatorname{HH}^{n}(\wideparen{\D}_X)$ as the parameter space of isomorphism classes of the $n$-th Yoneda extensions of $I(\OX_X(X))$ by $I(\OX_X(X))$ in $\Mod_{LH(\widehat{\mathcal{B}}c_K)}(I(\wideparen{\D}_X(X)))$. Furthermore, Theorem \ref{teo hochschild cohomology as ext of structure sheaf as D module} also shows that every such extension  induces a unique 
extension of $I(\OX_X)$ by $I(\OX_X)$ in $\Mod_{LH(\widehat{\mathcal{B}}c_K)}(I(\wideparen{\D}_X))$.\\

In the case of a first order Yoneda extension, this correspondence can be obtained explicitly. Namely, 
by Theorem \ref{teo global sections of co-admissible modules on Stein spaces} there is a Fréchet-Stein presentation
$\wideparen{\D}_X(X)=\varprojlim_m \widehat{\D}_m$. Let $\mathcal{M}$ be a co-admissible  $\wideparen{\D}_X(X)$-module, and set $\mathcal{M}_m=\widehat{\D}_m\otimes_{\wideparen{\D}_X(X)}\mathcal{M}$, so that we have $\mathcal{M}=\varprojlim_m\mathcal{M}_m$. Assume we have the following degree one Yoneda extension in $\Mod_{LH(\widehat{\mathcal{B}}c_K)}(I(\wideparen{\D}_X))$:
\begin{equation*}
    0\rightarrow I(\mathcal{M})\rightarrow \mathcal{N}\rightarrow I(\OX_X(X))\rightarrow 0.
\end{equation*}
For every $m\geq 0$, we have the following short exact sequence:
\begin{equation*}
    0\rightarrow I(\mathcal{M}_m)\rightarrow I(\widehat{\D}_{m})\widetilde{\otimes}_{I(\wideparen{\D}_X(X))}\mathcal{N}\rightarrow I(\OX_X(X))\rightarrow 0.
\end{equation*}
This sequence is split as a sequence of $I(\OX_X(X))$-modules. In particular, the $I(\widehat{\D}_{m})$-module:
\begin{equation*}
\mathcal{N}_m:=I(\widehat{\D}_{m})\widetilde{\otimes}_{I(\wideparen{\D}_X(X))}\mathcal{N},
\end{equation*}
is the image under $I:\widehat{\mathcal{B}}c_K\rightarrow LH(\widehat{\mathcal{B}}c_K)$ of a Banach space. Furthermore, as $\widehat{\D}_{m}$ is noetherian, it follows that $\mathcal{N}_m$ is a finitely generated $I(\widehat{\D}_{m})$-module. As a consequence of the snake Lemma, we have an isomorphism:
\begin{equation*}
    \mathcal{N}\rightarrow \varprojlim_m \mathcal{N}_m,
\end{equation*}
and then it follows by \cite[Corollary 5.40]{bode2021operations}
that $\mathcal{N}$ is the image under $I$ of a co-admissible module, which we will also denote by $\mathcal{N}$ for simplicity. Thus, we obtain a degree one Yoneda extension:
\begin{equation*}
    0\rightarrow \operatorname{Loc}(\mathcal{M})\rightarrow \operatorname{Loc}(\mathcal{N})\rightarrow \OX_X\rightarrow 0,
\end{equation*}
and it follows by Proposition \ref{prop characterization of co-admissible modules on a smooth Stein space} that this process establishes a bijection between the two sets of isomorphism classes of first order Yoneda extensions.\\

In the general case, the situation is not so simple. In particular, a $n$-th degree Yoneda extension of $I(\OX_X(X))$ by $I(\mathcal{M})$ is a complex of the following form:
\begin{equation*}
 0\rightarrow I(\mathcal{M})\rightarrow \mathcal{N}_{n-1} \rightarrow \cdots \rightarrow \mathcal{N}_0\rightarrow I(\OX_X(X))\rightarrow 0,   
\end{equation*}
and it is not clear that the $\mathcal{N}_i$ are co-admissible modules. In fact, it is not even clear that they are objects in 
$\Mod_{\Indban}(\wideparen{\D}_X)$. Hence, the identification of the parameter spaces for isomorphism classes of $n$-th degree Yoneda extensions:
\begin{equation*}
   \operatorname{Ext}_{\wideparen{\D}_X}^n(\OX_X,\operatorname{Loc}(\mathcal{M}))=\operatorname{Ext}_{\wideparen{\D}_X(X)}^n(\OX_X(X),\mathcal{M}), 
\end{equation*}
is only accessible via the means of homological algebra. This showcases an algebraic application of Hochschild cohomology to the study of $\wideparen{\D}_X$-modules. Furthermore, this type of behavior resembles the properties of coherent modules in the theory of $\D$-modules on smooth algebraic $K$-varieties, and reinforces the idea that co-admissible modules are the right $p$-adic analog of coherent $\wideparen{\D}_X$-modules.
\subsection{Hochschild cohomology and extensions II}\label{Section HC and ext 2}
Recall from Section \ref{Section background HochDmod} that given a rigid space $X$ and a sheaf of Ind-Banach algebras $\mathscr{A}$, we define the sheaf of enveloping algebras of $\mathscr{A}$ as the following sheaf of Ind-Banach algebras on $X$:
\begin{equation*}
\mathscr{A}^e:=\mathscr{A}\overrightarrow{\otimes}_K\mathscr{A}^{\op}.
\end{equation*}
If $X=\Sp(K)$, we will call $\mathscr{A}^e$  the enveloping algebra of $\mathscr{A}$. Let $X=\varinjlim_n X_n$ be a smooth Stein space admitting an étale map $X\rightarrow \mathbb{A}^r_K$, and let $\wideparen{E}_X$ be the sheaf of bi-enveloping algebras of $X$. Recall that this is a sheaf of Ind-Banach algebras on $X^2:=X\times_{\Sp(K)}X$, and it follows by the contents of \cite[Section 5.4]{HochDmod} that we have the following identity of Ind-Banach algebras: 
\begin{equation*} \wideparen{\D}_X(X)^e:=\wideparen{\D}_X(X)\overrightarrow{\otimes}_K\wideparen{\D}_X(X)^{\op}=\wideparen{E}_X(X^2).
\end{equation*}
 We remark that this notation is also used in a slightly different context. In particular, if $R$ is a commutative algebra and $A$ is a $R$-algebra, we define its enveloping algebra as:
\begin{equation*}
    A^e=A\otimes_RA^{\op}.
\end{equation*}
As Ind-Banach spaces do not have an underlying $K$-vector space, there is no confusion possible between these two concepts. Hence, we will use this terminology for both of them, and the type of object we are working with will make clear which one of the two we are referring to.\\ 

Notice that the key to the results of the previous section is the fact that we have a resolution of $\OX_X$ by locally finite-free $\wideparen{\D}_X$-modules. Indeed, the existence of the Spencer resolution allowed us to give explicit calculations of several derived functors, and we used this to show a bunch of identities in the corresponding derived categories. Consider the diagonal morphism:
\begin{equation*}
    \Delta:X\rightarrow X^2.
\end{equation*}
As shown in \cite{HochDmod}, the pushforward $\Delta_*\wideparen{\D}_X$ is canonically a co-admissible $\wideparen{E}_X$-module, and we have an identification of derived functors:
\begin{equation*}
    R\underline{\mathcal{H}om}_{\wideparen{\D}_X^e}(\wideparen{\D}_X,\wideparen{\D}_X)=\Delta^{-1}R\underline{\mathcal{H}om}_{\wideparen{E}_X}(\Delta_*\wideparen{\D}_X,\Delta_*\wideparen{\D}_X).
\end{equation*}
The goal of this section is obtaining an interpretation of the Hochschild cohomology groups $\operatorname{HH}^{\bullet}(\wideparen{\D}_X)$ as isomorphism classes of Yoneda extensions of $I(\Delta_*\wideparen{\D}_X)$ by $I(\Delta_*\wideparen{\D}_X)$. The idea is trying to replicate the arguments we used in the proof of Theorem \ref{teo hochschild cohomology as ext of structure sheaf as D module}. Ideally, we would want to show that there is a resolution of $\Delta_*\wideparen{\D}_X$ by finite-projective $\wideparen{E}_X$-modules. Unfortunately, we were not able to show that this holds in general. However, we will show a slightly weaker statement, which will be enough for our purposes. Notice that $\Delta_*\wideparen{\D}_X$ is a co-admissible $\wideparen{E}_X$-module. Hence, by Proposition \ref{prop characterization of co-admissible modules on a smooth Stein space}, it is enough to show that  $\wideparen{\D}_{X}(X)$ admits a resolution by finite-projective $\wideparen{\D}_{X}(X)^e$-modules. Notice that $\wideparen{\D}_{X}(X)$ is a finitely generated $\wideparen{\D}_{X}(X)^e$-module. However, as $\wideparen{\D}_{X}(X)^e$ is not noetherian, it is not immediate that the canonical epimorphism:
\begin{equation*}
    \wideparen{\D}_{X}(X)^e\rightarrow\wideparen{\D}_{X}(X),
\end{equation*}
extends to a finite-projective resolution. In order to solve this issue, we will first construct the resolution in a setting where the relevant rings are noetherian, and then obtain our desired resolution via a series of pullbacks.\\

A distinguished class of smooth Stein spaces is given by the rigid-analytification of smooth affine varieties (cf. \cite[Section 5.4]{bosch2014lectures}). Given a smooth algebraic $K$-variety $X$, we let $X^{\operatorname{an}}$ be its rigid-analytification. The properties of the rigid analytic GAGA functor can be used to study the properties of $\wideparen{\D}_{X^{\operatorname{an}}}$-modules using $\D_X$-modules, where some of the considerations are simpler. This approach was initiated in  \cite[Section 4]{p-adicCatO}, and it is the strategy we will follow here. For simplicity in the notation, we will write $\mathbb{A}^r_K$ to denote the algebraic variety and the associated rigid variety. The context will make clear which one of the two we are referring to.\\

Let us start with the following lemma:
\begin{Lemma}\label{Lemma flatness of map from diff op to completed diff op}
Assume $K$ is discretely valued and let $X$ be a smooth algebraic $K$-variety with free tangent sheaf. In this situation the canonical map:
  \begin{equation*}
      \D_X(X)^e\rightarrow \wideparen{\D}_{X^{\operatorname{an}}}(X^{\operatorname{an}})^e,
  \end{equation*}
is faithfully flat and has dense image.
\end{Lemma}
\begin{proof}
As $X$ has free tangent sheaf, it follows that $X^{\operatorname{an}}$ also has free tangent sheaf. Thus, we may use the isomorphisms:
\begin{equation*}
    \D_X(X)\rightarrow \D_X(X)^{\op}, \quad \wideparen{\D}_{X^{\operatorname{an}}}(X^{\operatorname{an}})\rightarrow \wideparen{\D}_{X^{\operatorname{an}}}(X^{\operatorname{an}})^{\op},
\end{equation*}
from \cite[Proposition 6.2.1]{HochDmod} to reduce the problem to showing that the canonical map:
\begin{equation*}
    \D_X(X)^2:=\D_X(X)\otimes_K\D_X(X)\rightarrow \wideparen{\D}_{X^{2,\operatorname{an}}}(X^{2,\operatorname{an}}),
\end{equation*}
is faithfully flat and has dense image. This was shown in \cite[Corollary 4.4.10]{p-adicCatO}.
\end{proof}
We may use this lemma to take the first step towards our desired resolution: 
\begin{Lemma}\label{Lemma resolution of differential operators on affine spaces}
$\Delta_*\wideparen{\D}_{\mathbb{A}^r_K}$ admits a resolution of length $2r$ by finite-projective $\wideparen{E}_{\mathbb{A}^r_K}$-modules.    
\end{Lemma}
\begin{proof}
In order to simplify notation and to avoid ambiguities, we will write $\D_{\mathbb{A}^r_K}(\mathbb{A}^r_K)$ to denote the filtered $K$-algebra $U(\operatorname{Der}_K(K[t_1,\cdots,t_r]))$. That is, $\D_{\mathbb{A}^r_K}(\mathbb{A}^r_K)$ is the algebra of differential operators associated to the scheme structure of $\mathbb{A}^r_K$. By Proposition \ref{prop characterization of co-admissible modules on a smooth Stein space}, it is enough to show that  $\wideparen{\D}_{\mathbb{A}^r_K}(\mathbb{A}^r_K)$ admits a resolution by finite-projective $\wideparen{\D}_{\mathbb{A}^r_K}(\mathbb{A}^r_K)^e$-modules. Assume first that $K$ is discretely valued. Notice that we have a canonical surjection:
\begin{equation*}
    \D_{\mathbb{A}^r_K}(\mathbb{A}^r_K)^e\rightarrow \D_{\mathbb{A}^r_K}(\mathbb{A}^r_K).
\end{equation*}
In particular, $\D_{\mathbb{A}^r_K}(\mathbb{A}^r_K)$ is finitely generated as a $\D_{\mathbb{A}^r_K}(\mathbb{A}^r_K)^e$-module. As mentioned above, the fact that $X$ has free tangent sheaf implies that we have an isomorphism of $K$-algebras:
\begin{equation*}
 \D_{\mathbb{A}^{2r}_K}(\mathbb{A}^{2r}_K)\rightarrow   \D_{\mathbb{A}^r_K}(\mathbb{A}^r_K)^e.
\end{equation*}
Hence, we may apply \cite[Theorem 2.6.11]{hotta2007d} to conclude that there is a resolution:
\begin{equation*}
    P^{\bullet}:=\left(0\rightarrow P^{2r}\rightarrow \cdots \rightarrow P^{0}\rightarrow 0 \right),
\end{equation*}
of $\D_{\mathbb{A}^r_K}(\mathbb{A}^r_K)$ such that each $P^i$
is a  finite-projective $\D_{\mathbb{A}^r_K}(\mathbb{A}^r_K)^e$-module. By Lemma \ref{Lemma flatness of map from diff op to completed diff op}, the map $\D_{\mathbb{A}^r_K}(\mathbb{A}^r_K)^e\rightarrow \wideparen{\D}_{\mathbb{A}^r_K}(\mathbb{A}^r_K)^e$  is  faithfully flat with dense image. Hence, the complex:
\begin{equation*}
    \wideparen{P}^{\bullet}:=\wideparen{\D}_{\mathbb{A}^r_K}(\mathbb{A}^r_K)^e\otimes_{\D_{\mathbb{A}^r_K}(\mathbb{A}^r_K)^e}P^{\bullet},
\end{equation*}
is  a resolution by finite-projective $\wideparen{\D}_{\mathbb{A}^r_K}(\mathbb{A}^r_K)^e$-modules of length $2r$ of the following $\wideparen{\D}_{\mathbb{A}^r_K}(\mathbb{A}^r_K)^e$-module:
\begin{equation*}
    D:=\wideparen{\D}_{\mathbb{A}^r_K}(\mathbb{A}^r_K)^e\otimes_{\D_{\mathbb{A}^r_K}(\mathbb{A}^r_K)^e}\D_{\mathbb{A}^r_K}(\mathbb{A}^r_K).
\end{equation*}
By the extension-restriction adjunction, we get a $\wideparen{\D}_{\mathbb{A}^r_K}(\mathbb{A}^r_K)^e$-linear map:
\begin{equation*}
    D\rightarrow \wideparen{\D}_{\mathbb{A}^r_K}(\mathbb{A}^r_K),
\end{equation*}
and we need to show it is an isomorphism. Consider the following family of closed disks indexed by non-negative integers $n\geq 0$:
\begin{equation*}
    \mathbb{B}^{r,n}_{K}=\Sp(K\langle \pi^nt_1,\cdots,\pi^nt_r\rangle).
\end{equation*}
Then $\mathbb{A}^r_K=\varinjlim_n\mathbb{B}^{r,n}_{K}$ is a Stein covering of $\mathbb{A}^r_K$ as a rigid variety. For each $n\geq 0$, let:
\begin{equation*}
    A_n:=K\langle \pi^nt_1,\cdots,\pi^nt_r\rangle,
\end{equation*}
and consider the family of affine  formal models:
\begin{equation*}
   \mathcal{A}_n:=\mathcal{R}\langle \pi^nt_1,\cdots,\pi^nt_r \rangle\subset A_n, 
\end{equation*}
which satisfy that the $\mathcal{A}_n$-module $\mathcal{T}_n=\bigoplus_{i=1}^n\mathcal{A}_n\partial_i$ is a finite-free $\mathcal{A}_n$-Lie lattice of $\mathcal{T}_{\mathbb{A}^r_K/K}(\mathbb{B}^{r,n}_{K})$.\\
As $D$ and $\wideparen{\D}_{\mathbb{A}^r_K}(\mathbb{A}^r_K)$  are co-admissible $\wideparen{\D}_{\mathbb{A}^r_K}(\mathbb{A}^r_K)^e$-modules, it is enough to show that the map:
\begin{equation*}
 \varphi_n:\widehat{U}(\pi^n\mathcal{T}_n)_K^e\otimes_{\D_{\mathbb{A}^r_K}(\mathbb{A}^r_K)^e}\D_{\mathbb{A}^r_K}(\mathbb{A}^r_K)  \rightarrow \widehat{U}(\pi^n\mathcal{T}_n)_K,
\end{equation*}
is an $\widehat{U}(\pi^n\mathcal{T}^n)^e$-linear isomorphism for every $n\geq 0$.  We will only show the case $n=0$, the rest is analogous. Consider the algebras $A=K[t_1,\cdots,t_r]$ and $\mathcal{A}=\mathcal{R}[t_1,\cdots,t_r]$. Notice that $\mathcal{A}$ is a finitely generated $\mathcal{R}$-algebra satisfying that:
\begin{equation*}
   A=K\otimes_{\mathcal{R}}\mathcal{A}, \quad  \mathcal{A}_0=\varprojlim\mathcal{R}[t_1,\cdots,t_r]/\pi^m.
\end{equation*}
Furthermore, the following identities hold:
\begin{equation*}
\mathcal{T}=\bigoplus_{i=1}^n\mathcal{A}\partial_i=\operatorname{Der}_{\mathcal{R}}(\mathcal{A}), \quad   K\otimes_{\mathcal{R}}\mathcal{T}=\operatorname{Der}_{K}(A).
\end{equation*}
Thus, we have $\D_{\mathbb{A}^r_K}(\mathbb{A}^r_K)=K\otimes_{\mathcal{R}}U(\mathcal{T})$, and $\mathcal{T}_0=\mathcal{A}_0\otimes_{\mathcal{A}}\mathcal{T}$. Therefore, there is a map:
\begin{equation*}
    \widetilde{\varphi}:\left(\widehat{U}(\mathcal{T}_0)\widehat{\otimes}_{\mathcal{R}}\widehat{U}(\mathcal{T}_0)^{\op}\right)\otimes_{U(\mathcal{T})^e}U(\mathcal{T})  \rightarrow \widehat{U}(\mathcal{T}_0),
\end{equation*}
satisfying $\varphi=\operatorname{Id}_K\otimes \widetilde{\varphi}$. Hence, it is enough to show that $\widetilde{\varphi}$ is an isomorphism. As both sides are $\pi$-adically complete, it suffices to show that the map is an isomorphism after taking quotients by high enough powers of $\pi$. Chose some $m\geq 0$. Then we have the following identities of $\mathcal{R}$-modules:
\begin{equation*}
  \widehat{U}(\mathcal{T}_0)/\pi^m= U(\mathcal{T}_0)/\pi^m=\left(\mathcal{A}_0\otimes_{\mathcal{A}}U(\mathcal{T})\right)\otimes_{\mathcal{R}}\mathcal{R}/\pi^m=U(\mathcal{T})/\pi^m,
\end{equation*}
where the second identity follows from the PBW Theorem for Lie-Rinehart algebras. On the other hand, we have the following chain of identities of $\mathcal{R}$-modules:
\begin{align*}
 \left( \left(\widehat{U}(\mathcal{T}_0)\widehat{\otimes}_{\mathcal{R}}\widehat{U}(\mathcal{T}_0)^{\op}\right)\otimes_{U(\mathcal{T})^e}U(\mathcal{T})\right)\otimes_{\mathcal{R}}\mathcal{R}/\pi^m=\left(\widehat{U}(\mathcal{T}_0)\widehat{\otimes}_{\mathcal{R}}\widehat{U}(\mathcal{T}_0)^{\op}\right)/\pi^m\otimes_{\left(U(\mathcal{T})/\pi^m\right)^e}U(\mathcal{T})/\pi^m\\
=\left(U(\mathcal{T}_0)/\pi^m\right)^e\otimes_{\left(U(\mathcal{T})/\pi^m\right)^e}U(\mathcal{T})/\pi^m\\
=U(\mathcal{T})/\pi^m.
\end{align*}
 Thus, $\widetilde{\varphi}$ is an isomorphism of $\mathcal{R}$-algebras, as we wanted to show. For general $K$, the result follows by taking a complete discretely valued subfield $Q\subset K$ and then using the identity:
 \begin{equation*}
     \wideparen{\D}_{\mathbb{A}^r_Q}(\mathbb{A}^r_Q)=Q\widehat{\otimes}_K\wideparen{\D}_{\mathbb{A}^r_K}(\mathbb{A}^r_K),
 \end{equation*}
 together with the fact that $Q\widehat{\otimes}_K-$ preserves acyclic complexes of Fréchet $K$-vector spaces.
\end{proof}
Finally, we can use these lemmas to obtained our fabled resolution:
\begin{prop}\label{prop resolution of differential operators on smooth Stein spaces}
Let $X$ be a smooth and separated rigid space with an étale map $f:X\rightarrow \mathbb{A}^r_K$. There is a co-admissible $\wideparen{E}_X$-module $\mathcal{S}_X$ such that there is a strict exact complex:
\begin{equation*}
   E^{\bullet}_{\wideparen{\D}_X}:= \left(0\rightarrow P^{2r}\rightarrow \cdots\rightarrow P^0\rightarrow \mathcal{S}_X\oplus \Delta_*\wideparen{\D}_X\rightarrow 0\right),
\end{equation*} 
where each $P^i$ is a finite-projective $\wideparen{E}_X$-module and $P^0=\wideparen{E}_X$.
\end{prop}
\begin{proof}
By the side-switching Theorem for co-admissible bimodules \cite[Theorem 6.3.6]{HochDmod}, together with the fact that $\operatorname{S}^{-1}(\Delta_*\wideparen{\D}_X)=\Delta_+\OX_X$, it is enough to show that there is a co-admissible $\wideparen{\D}_{X^2}$-module  $\mathcal{M}$ such that $\mathcal{M}\oplus\Delta_+\OX_X$ admits a resolution by finite-projective $\wideparen{\D}_{X^2}$-modules. By Lemma \ref{Lemma resolution of differential operators on affine spaces}, there is a strict exact complex of finite-projective $\wideparen{\D}_{\mathbb{A}^{2r}_K}$-modules $P^{\bullet}$ of the following form:
\begin{equation*}
    P^{\bullet}:=\left(0\rightarrow P^{2r}\rightarrow \cdots\rightarrow \wideparen{\D}_{\mathbb{A}^{2r}_K}\rightarrow \Delta_*\OX_{\mathbb{A}^{r}_K}\rightarrow 0\right).
\end{equation*}
As étale maps are stable under base-change and composition, the product $f\times f:X^2\rightarrow \mathbb{A}^{2r}_K$ is also an étale map. We claim that the extraordinary inverse image $\left(f\times f\right)^!(P^{\bullet})$ yields the desired complex.  First, notice that as $f\times f:X^2\rightarrow \mathbb{A}^{2r}_K$ is étale, we have a canonical isomorphism $\wideparen{\D}_{X^2\rightarrow \mathbb{A}^{2r}_K}\cong \wideparen{\D}_{X^2}$. Furthermore, by \cite[Theorem 9.11]{bode2021operations}  for any co-admissible $\wideparen{\D}_{\mathbb{A}^{2r}_K}$-module $\mathcal{N}$ we have:
\begin{equation*}
\left(f\times f\right)^!\mathcal{N}=\wideparen{\D}_{X^2}\overrightarrow{\otimes}_{\left(f\times f\right)^{-1}\wideparen{\D}_{\mathbb{A}^{2r}_K}}\left(f\times f\right)^{-1}\mathcal{N}=\OX_{X^2}\overrightarrow{\otimes}_{\left(f\times f\right)^{-1}\OX_{\mathbb{A}^{2r}_K}}\left(f\times f\right)^{-1}\mathcal{N}.    
\end{equation*}
As $P^{\bullet}$ is an exact complex of co-admissible modules, it follows that $\left(f\times f\right)^!(P^{\bullet})$ is also an exact complex of co-admissible modules. Hence, we have:
\begin{equation*}
    \left(f\times f\right)^!(P^{\bullet})=\left(0\rightarrow \left(f\times f\right)^!P^{2r}\rightarrow \cdots\rightarrow \wideparen{\D}_{X^2}\rightarrow \left(f\times f\right)^!\Delta_*\OX_{\mathbb{A}^{r}_K}\rightarrow 0\right).
\end{equation*}
Consider the following fiber product:
\begin{equation*}
    Y:=X^2\times_{\mathbb{A}^{2r}_K}\mathbb{A}^{r}_K.
\end{equation*}
By properties of the pullback, the projection $\alpha:Y\rightarrow X^2$ is a closed immersion, and the projection $\beta:Y\rightarrow \mathbb{A}^{r}_K$ is an étale map. Furthermore, the maps $\Delta:X\rightarrow X^2$, and $f:X\rightarrow \mathbb{A}^{r}_K$, induce a map:
\begin{equation*}
    X\rightarrow Y,
\end{equation*}
which is an étale closed immersion. Hence, $X$ is a union of connected components of $Y$. We let $Z=Y\setminus X$ be the union of the other connected components of $Y$, and set $\mathcal{M}=\alpha_+\OX_{Z}$.\newline
As above, we have:
\begin{equation*}
\left(f\times f\right)^!\Delta_+\OX_{\mathbb{A}^{r}_K}=\OX_{X^2}\overrightarrow{\otimes}_{\left(f\times f\right)^{-1}\OX_{\mathbb{A}^{2r}_K}}\left(f\times f\right)^{-1}\Delta_+\OX_{\mathbb{A}^r_K}.    
\end{equation*}
In particular, as $\Delta_+\OX_{\mathbb{A}^r_K}$ has support contained in the diagonal of
$\mathbb{A}^{2r}_K$, the co-admissible  $\wideparen{\D}_{X^2}$-module $\left(f\times f\right)^!\Delta_+\OX_{\mathbb{A}^r_K}$ has support contained in $Y$. By Kashiwara's equivalence, we have:
\begin{equation*}
  \left(f\times f\right)^!\Delta_+\OX_{\mathbb{A}^r_K}=\alpha_+\alpha^! \left(f\times f\right)^!\Delta_+\OX_{\mathbb{A}^r_K}=\alpha_+\beta^!\OX_{\mathbb{A}^r_K}=\alpha_+\OX_{Y}=\alpha_+\OX_{Z}\oplus \Delta_+\OX_{X}.
\end{equation*}
Thus, $\left(f\times f\right)^!(P^{\bullet})$ is a resolution of $\mathcal{M}\oplus\Delta_+\OX_{X}$ by finite-projective $\wideparen{\D}_{X^2}$-modules, as  wanted.
\end{proof}
\begin{coro}\label{coro global dim}
 Let $X=\Sp(A)$ be an affinoid space and $\mathcal{A}\subset A$ be an affine formal model of $A$. Assume $\mathcal{T}_{X/K}$ admits a free $(\mathcal{R},\mathcal{A})$-Lie lattice $\mathcal{L}$ satisfying the following two properties:
 \begin{equation*}
     \mathcal{L}(\mathcal{A})\subset \pi\mathcal{A},\quad [\mathcal{L},\mathcal{L}]\subset \pi^2\mathcal{L}.
 \end{equation*}
 In this situation, the following inequality holds for each $n\geq 0$:
 \begin{equation*}
     \operatorname{gl.dim.}(\widehat{U}(\pi^n\mathcal{L})_K)\leq 2\operatorname{dim}(X).
 \end{equation*}
 Furthermore, if $M\in \Mod_{\Indban}(\widehat{U}(\pi^n\mathcal{L})_K)$ satisfies that it is projective as an Ind-Banach space, then $M$ admits a projective resolution  of $\widehat{U}(\pi^n\mathcal{L})_K$-modules of length $2\operatorname{dim}(X)$.
\end{coro}
\begin{proof}
For simplicity, let $B_n:=\widehat{U}(\pi^n\mathcal{L})_K$, and $B:=\wideparen{\D}_X(X)$. In virtue of \cite[Theorem 4.1.11]{ardakovequivariant}, our assumptions on $\mathcal{L}$ imply that we have a Fréchet-Stein presentation $B=\varprojlim_nB_n$. Similarly, \cite[Corollary 6.2.6]{HochDmod}  implies that $B^e=\varprojlim_nB_n^e$ is also a Fréchet-Stein presentation. In this situation, $B$ is a co-admissible $B^e$-module, and we have identities:
\begin{equation*}
    B_n=B_n^e\widehat{\otimes}_{B^e}B,
\end{equation*}
for each $n\geq 0$ (see \cite[Section 7]{HochDmod} for details). Let $E^{\bullet}_{\wideparen{\D}_X}$ be the complex obtained in Proposition \ref{prop resolution of differential operators on smooth Stein spaces} and fix some $n\geq 0$. As $E^{\bullet}_{\wideparen{\D}_X}$ is a strict and exact complex of co-admissible $B^e$-modules, it remains strict exact after applying the functor $B_n^e\widehat{\otimes}_{B^e}-$. Thus, the fact that $B_n^e$ is noetherian and $B_n$ is finite implies that $B_n$ admits a resolution of length $2\operatorname{dim}(X)$ by finite-projective $B_n^e$-modules.\\
Let $P^{\bullet}$ be such a resolution. Notice that the algebra $B_n$ is a Banach space which has a dense countably generated vector subspace. Hence, by  \cite[Proposition 10.4]{schneider2013nonarchimedean}, and \cite[Lemma 4.16]{bode2021operations} it is a projective object in $\Indban$. In particular, the algebra $B_n^e:=B_n\widehat{\otimes}_KB_n^{\op}$ is a projective object viewed as a right $B_n$-module. Therefore, every finite-projective $B_n^e$-module is a fortiori a projective right $B_n$-module. It follows that $P^{\bullet}$
is a resolution of $B_n$ by projective $B_n$-modules. Thus, for any finitely generated left $B_n$-module $M$ we have the following identities in the derived category $\operatorname{D}(B_n)$:
\begin{equation*}
M=B_n\widehat{\otimes}^{\mathbb{L}}_{B_n}M =P^{\bullet}\widehat{\otimes}_{B_n}M.
\end{equation*}
In particular, $P^{\bullet}\widehat{\otimes}_{B_n}M$ is a resolution of $B_n$-modules of $M$ of length $2\operatorname{dim}(X)$. However, using the same arguments as above, $M$ is a projective object in $\Indban$, and therefore the left $B_n$-module $B_n^e\widehat{\otimes}_{B_n}M=   B_n\widehat{\otimes}_KM$ is also projective. Hence, if $P$ is a finite-projective $B_n^e$-module, then $P\widehat{\otimes}_{B_n}M$ is a projective left $B_n$-module. We have thus shown that $P^{\bullet}\widehat{\otimes}_{B_n}M$ is a projective resolution of $M$ of length $2\operatorname{dim}(X)$. As $B_n$ is noetherian and $M$ is finitely generated, this implies that $M$ admits a resolution by finite-projective $B_n$-modules of length $2\operatorname{dim}(X)$. Hence, we get $ \operatorname{gl.dim.}(B_n)\leq 2\operatorname{dim}(X)$, as we wanted to show. The same arguments as above show the claim at the end of the corollary.
\end{proof}
\begin{obs}
For the rest of the section we assume $K$ is discretely valued or algebraically closed.
\end{obs}
Even though the resolution we obtained is slightly weaker than the result we had in the previous section, it is still enough to deduce the results we are looking for: 
\begin{teo}\label{teo hochschild cohomology in terms of extendions in DXe 1}
Let $X$ be a Stein space with an étale map $X\rightarrow \mathbb{A}^r_K$. For any co-admissible module $\mathcal{M}\in \mathcal{C}(\wideparen{E}_X)$ we have the following identity in $\operatorname{D}(\widehat{\mathcal{B}}c_K)$:
\begin{equation*}
    R\Gamma(X^2,R\underline{\mathcal{H}om}_{\wideparen{E}_X}(\Delta_*\wideparen{\D}_X,\mathcal{M}))=R\underline{\Hom}_{\wideparen{\D}_X(X)^e}(\wideparen{\D}_X(X),\Gamma(X^2,\mathcal{M})).
\end{equation*}
In particular, we have the following identity: 
\begin{equation}\label{equation Hochcoh determined at global sections}
    \operatorname{HH}^{\bullet}(\wideparen{\D}_X)=R\underline{\Hom}_{\wideparen{\D}_X(X)^e}(\wideparen{\D}_X(X),\wideparen{\D}_X(X))=:\operatorname{HH}^{\bullet}(\wideparen{\D}_X(X)).
\end{equation}
\end{teo}
\begin{proof}
Let $\mathcal{N}$ be a co-admissible  $\wideparen{E}_X$-module. As $X^2$ is a smooth Stein space with an étale map $X^2\rightarrow \mathbb{A}^{2r}_K$, it follows by Lemma \ref{Lemma global sections inner hom functor} that we have an identification of functors: 
\begin{equation*}
    \Gamma(X^2,\underline{\mathcal{H}om}_{\wideparen{E}_X}(\mathcal{N},-))=\underline{\Hom}_{\wideparen{\D}_X(X)^e}(\Gamma(X^2,\mathcal{N}),\Gamma(X^2,-)).
\end{equation*}
Thus, for each $\mathcal{M}\in \mathcal{C}(\wideparen{E}_X)$ we have a natural map in the homotopy category $\operatorname{K}(LH(\widehat{\mathcal{B}}c_K))$: 
\begin{equation*}
   \psi_{\mathcal{N},\mathcal{M}}: R\underline{\Hom}_{\wideparen{\D}_X(X)^e}(\Gamma(X^2,\mathcal{N}),\Gamma(X^2,\mathcal{M}))\rightarrow R\Gamma(X^2,R\underline{\mathcal{H}om}_{\wideparen{E}_X}(\mathcal{N},\mathcal{M})).
\end{equation*}
We need to show that this is a quasi-isomorphism for $\mathcal{N}=\Delta_*\wideparen{\D}_X$ and every $\mathcal{M}\in \mathcal{C}(\wideparen{E}_X)$.\\

Until the end of the proof, let $\mathcal{T}=\mathcal{S}_X\oplus \Delta_*\wideparen{\D}_X$. By Proposition \ref{prop resolution of differential operators on smooth Stein spaces}, there is a complex of finite-projective $\wideparen{E}_X$-modules $P^{\bullet}$ with a quasi-isomorphism $P^{\bullet}\rightarrow \mathcal{T}$. Notice that both complexes are complexes of co-admissible $\wideparen{E}_X$-modules. Hence, this quasi-isomorphism is automatically strict by \cite[Corollary 5.15]{bode2021operations}. Thus, we have the following chain of identities in $\operatorname{D}(\widehat{\mathcal{B}}c_K)$:
\begin{align*}
 R\Gamma(X^2,R\underline{\mathcal{H}om}_{\wideparen{E}_X}(\mathcal{T},\mathcal{M}))=R\Gamma(X^2,\underline{\mathcal{H}om}_{\wideparen{E}_X}(P^{\bullet},\mathcal{M}))=\Gamma(X^2,\underline{\mathcal{H}om}_{\wideparen{E}_X}(P^{\bullet},\mathcal{M}))\\
 =\underline{\Hom}_{\wideparen{\D}_X(X)^e}(\Gamma(X^2,P^{\bullet}),\Gamma(X^2,\mathcal{M}))\\=R\underline{\Hom}_{\wideparen{\D}_X(X)^e}(\Gamma(X^2,\mathcal{T}),\Gamma(X^2,\mathcal{M})),
\end{align*}
where we make extensive use of the fact that co-admissible modules are acyclic for the global sections functor, and finite-projective $\wideparen{E}_X$-modules are acyclic for the $\underline{\mathcal{H}om}_{\wideparen{E}_X}(-,\mathcal{M})$-functor. Thus, it follows that $\psi_{\mathcal{T},\mathcal{M}}$ is a quasi-isomorphism for every $\mathcal{M}\in \mathcal{C}(\wideparen{E}_X)$.\\

Notice that both the global sections functor and its associated derived functor commute with finite direct sums. Similarly, both $\underline{\mathcal{H}om}_{\wideparen{E}_X}(-,-)$, and $\underline{\Hom}_{\wideparen{\D}_X(X)^e}(-,-)$ and its derived functors commute with finite direct sums on the first variable. As 
$\mathcal{T}=\mathcal{S}_X\oplus \Delta_*\wideparen{\D}_X$, it follows that:
\begin{equation*} \psi_{\mathcal{T},\mathcal{M}}=\psi_{\mathcal{S},\mathcal{M}}\oplus \psi_{\Delta_*\wideparen{\D}_X,\mathcal{M}},
\end{equation*}
for every $\mathcal{M}\in \mathcal{C}(\wideparen{E}_X)$.  This implies that $\psi_{\Delta_*\wideparen{\D}_X,\mathcal{M}}$ is also a quasi-isomorphism, as wanted.\\
In order to show (\ref{equation Hochcoh determined at global sections}), we consider the following identities in $\operatorname{D}(\widehat{\mathcal{B}}c_K)$:
\begin{align*}
  \operatorname{HH}^{\bullet}(\wideparen{\D}_X):=R\Gamma(X,R\underline{\mathcal{H}om}_{\wideparen{\D}^e_X}(\wideparen{\D}_X,\wideparen{\D}_X)) =R\Gamma(X,\Delta^{-1}R\underline{\mathcal{H}om}_{\wideparen{E}_X}(\Delta_*\wideparen{\D}_X,\Delta_*\wideparen{\D}_X))\\
  =R\Gamma(X^2,R\underline{\mathcal{H}om}_{\wideparen{E}_X}(\Delta_*\wideparen{\D}_X,\Delta_*\wideparen{\D}_X))\\=R\underline{\Hom}_{\wideparen{\D}_X(X)^e}(\wideparen{\D}_X(X),\wideparen{\D}_X(X))\\
  =:\operatorname{HH}^{\bullet}(\wideparen{\D}_X(X)),
\end{align*}
where the second identity is \cite[Proposition 3.5.5]{HochDmod}, the third one follows by \cite[Lemma 3.5.3]{HochDmod}, and the fourth one is a special case of the first part of the theorem.
\end{proof}
Hence, in the case of Stein spaces with free tangent sheaf, we deduce that the Hochschild cohomology of $\wideparen{\D}_X$  is completely determined by the Hochschild cohomology of its global sections $\wideparen{\D}_X(X)$.\newline
As in the previous section, we can also give a version for the external sheaves of homomorphisms:
\begin{teo}
Let $X$ be a Stein space with an étale map $X\rightarrow \mathbb{A}^r_K$. For any co-admissible module $\mathcal{M}\in \mathcal{C}(\wideparen{E}_X)$ we have the following identities in $\operatorname{D}(\operatorname{Vect}_K)$:
\begin{equation*}
    R\Hom_{\wideparen{E}_X}(\Delta_*\wideparen{\D}_X,\mathcal{M})=R\Gamma(X^2,R\mathcal{H}om_{\wideparen{E}_X}(\Delta_*\wideparen{\D}_X,\mathcal{M}))=R\Hom_{\wideparen{\D}_X(X)^e}(\wideparen{\D}_X(X),\Gamma(X^2,\mathcal{M})).
\end{equation*}
Hence, we have the following identities: 
\begin{equation}\label{equation extensions are determined on global sections}
   \operatorname{Forget}(J\operatorname{HH}^{\bullet}(\wideparen{\D}_X))=R\Hom_{\wideparen{\D}_X(X)^e}(\wideparen{\D}_X(X),\wideparen{\D}_X(X))=R\Hom_{\wideparen{E}_X}(\Delta_*\wideparen{\D}_X,\Delta_*\wideparen{\D}_X).
\end{equation}    
\end{teo}
\begin{proof}
Lemmas \ref{lemma external hom Stein co-adm 1}, and \ref{Lemma different ext hom functors} hold for finite-projective $\wideparen{E}_X$-modules. Thus, the identities:
\begin{equation*}
     R\Hom_{\wideparen{E}_X}(\Delta_*\wideparen{\D}_X,\mathcal{M})=R\Gamma(X^2,R\mathcal{H}om_{\wideparen{E}_X}(\Delta_*\wideparen{\D}_X,\mathcal{M}))=R\Hom_{\wideparen{\D}_X(X)^e}(\wideparen{\D}_X(X),\Gamma(X^2,\mathcal{M})),
\end{equation*}
can be deduced using the arguments from the proof of Theorem \ref{teo hochschild cohomology in terms of extendions in DXe 1}. Hence, we only have to show that (\ref{equation extensions are determined on global sections}) holds. The second identity in (\ref{equation extensions are determined on global sections})  is a special case of the first part of the theorem. For the first one, consider the following identities in $\operatorname{D}(\operatorname{Vect}_K)$:
\begin{flalign*}
 R\Hom_{\wideparen{\D}_X(X)^e}(\mathcal{S}_X(X^2)\oplus \wideparen{\D}_X(X),\wideparen{\D}_X(X))=\Hom_{\wideparen{\D}_X(X)^e}(\Gamma(X^2,E^{\bullet}_{\wideparen{\D}_X}),\wideparen{\D}_X(X))&&
 \end{flalign*}
 \begin{flalign*}   
 &&=\left(0\rightarrow \Hom_{\wideparen{\D}_X(X)^e}(\wideparen{\D}_X(X)^e,\wideparen{\D}_X(X))\rightarrow \cdots \rightarrow \Hom_{\wideparen{\D}_X(X)^e}(\bigoplus_{i=1}^{n_{2r}}\wideparen{\D}_X(X)^e,\wideparen{\D}_X(X)) \rightarrow 0  \right)\\
&&=\left(0\rightarrow \Hom_{\Indban}(K,\wideparen{\D}_X(X))\rightarrow \cdots \rightarrow \Hom_{\Indban}(\bigoplus_{i=1}^{n_{2r}}K,\wideparen{\D}_X(X)) \rightarrow 0  \right)\\ 
&&=\left(0\rightarrow \operatorname{Forget}(JI(\wideparen{\D}_X(X)))\rightarrow \cdots \rightarrow \bigoplus_{i=1}^{n_{2r}}\operatorname{Forget}(JI(\wideparen{\D}_X(X))) \rightarrow 0  \right)\\
&&=\operatorname{Forget}( J\underline{\Hom}_{\wideparen{\D}_X(X)^e}(\Gamma(X^2,E^{\bullet}_{\wideparen{\D}_X}),\wideparen{\D}_X(X)))\\
&&=\operatorname{Forget}( JR\underline{\Hom}_{\wideparen{\D}_X(X)^e}(\mathcal{S}_X(X^2)\oplus \wideparen{\D}_X(X),\wideparen{\D}_X(X))),
\end{flalign*}
where we have made extensive use of Lemma \ref{lemma forgetful functor}. Notice that, by construction, $\wideparen{\D}_X(X)^e$ is the image under the dissection functor of a Fréchet algebra. Regard for a moment $\wideparen{\D}_X(X)^e$ as an algebra in $\widehat{\mathcal{B}}c_K$. Then we have the following identity of functors $\Mod_{\widehat{\mathcal{B}}c_K}(\wideparen{\D}_X(X)^e)\rightarrow \operatorname{Vect}_K$:
\begin{equation*}
    \Hom_{I(\wideparen{\D}_X(X)^e)}(I(-),I(\wideparen{\D}_X(X)))=\operatorname{Forget}(J\underline{\Hom}_{I(\wideparen{\D}_X(X)^e)}(I(-),I(\wideparen{\D}_X(X)))).
\end{equation*}
Hence, we can use the fact that $\Mod_{\widehat{\mathcal{B}}c_K}(\wideparen{\D}_X(X))$ has enough projective objects, and that all functors in the above equation commute with finite direct sums to obtain  the following identities in $\operatorname{D}(\operatorname{Vect}_K)$:
\begin{equation*}
   R\Hom_{\wideparen{\D}_X(X)^e}(\wideparen{\D}_X(X),\wideparen{\D}_X(X))=\operatorname{Forget}( JR\underline{\Hom}_{\wideparen{\D}_X(X)^e}( \wideparen{\D}_X(X),\wideparen{\D}_X(X)))=\operatorname{Forget}(J\operatorname{HH}^{\bullet}(\wideparen{\D}_X)),
\end{equation*}
and this is precisely what we wanted to show. Thus, the theorem holds.
\end{proof}
\begin{obs}
We remark that the results thus far also hold if $X$ is an affinoid space with free tangent sheaf and $K$ is an arbitrary complete non-archimedean extension of $\mathbb{Q}_p$.
\end{obs}
In conclusion, we have been able to obtain a direct relation between the Hochschild cohomology groups and the Yoneda extensions of $I(\wideparen{\D}_X(X))^e$-modules of $I(\wideparen{\D}_X(X))$
by $I(\wideparen{\D}_X(X))$. Furthermore, the isomorphism classes of such extensions are in one-to-one correspondence with the extensions classes at the level of sheaves. This relation will be further explored in  Chapter \ref{Chapter deformation theory}, where we will study the deformation theory of $\wideparen{\D}_X(X)$.\\

We finish the section by analyzing the case of rigid-analytifications of smooth algebraic varieties:
\begin{coro}\label{coro hochschild cohomology of analytification of smooth affine varieties}
Let $X$ be a smooth affine algebraic variety with an étale map $X\rightarrow \mathbb{A}^r_K$. Then $X^{\operatorname{an}}$ is a Smooth Stein space with an étale map $X^{\operatorname{an}}\rightarrow \mathbb{A}^r_K$, and we have:
\begin{equation*}
    \operatorname{Forget}(J\operatorname{HH}^{\bullet}(\wideparen{\D}_{X^{\operatorname{an}}}))=\operatorname{H}_{\operatorname{dR}}^{\bullet}(X^{\operatorname{an}})=
    \operatorname{H}_{\operatorname{dR}}^{\bullet}(X)= 
    R\Hom_{\D_X(X)^e}(\D_X(X),\D_X(X))=:\operatorname{HH}^{\bullet}(\D_{X}(X)).
\end{equation*}
In particular, the Hochschild complex $\operatorname{HH}^{\bullet}(\wideparen{\D}_{X^{\operatorname{an}}})$ has finite-dimensional cohomology. 
\end{coro}
\begin{proof}
The first identity was shown in Theorem \ref{teo hochschild cohomology groups as de rham cohomology groups}, the second one is \cite[Theorem 3.2]{grosse2004rham}, and the third one is \cite[Proposition 2.3]{etingof2004cherednik}. The last identity is just the definition of the Hochschild cohomology of an abstract $K$-algebra (see, for example \cite[Remark 1.1.14]{witherspoon2019hochschild}). The fact that the de Rham cohomology groups of $X$ are finite dimensional is a classical result.  It can be found in \cite{Monskyfiniteness}.
\end{proof}
Combining this theorem with \cite[Proposition 2.1]{etingof2004cherednik}, we get the following identity in $\operatorname{D}(\operatorname{Vect}_K)$:
\begin{equation*}
R\Hom_{\wideparen{\D}_{X^{\operatorname{an}}}(X^{\operatorname{an}})}(\OX_{X^{\operatorname{an}}}(X^{\operatorname{an}}),\OX_{X^{\operatorname{an}}}(X^{\operatorname{an}}))=R\Hom_{\D_X(X)}(\OX_{X}(X),\OX_{X}(X)).
\end{equation*}
Thus, these \emph{Ext} functors are preserved by the rigid-analytic GAGA functor.
\subsection{Lie-Rinehart cohomology}\label{Secion LR cohomology}
We finish the chapter by giving a final interpretation of the Hochschild cohomology of $\wideparen{\D}_X$. Recall the definition of Lie-Rinehart cohomology: 
\begin{defi}[{\cite[Section 4]{rinehart1963differential}}]
Let $A$  be a $K$-algebra. For a $(K,A)$-Lie algebra $L$, and a $L$-representation $M$, we define the Lie-Rinehart cohomology of $L$ with coefficients in $M$ as:
\begin{equation*}
  \operatorname{H}^{\bullet}_A(L,M)=R\operatorname{Hom}_{U(L)}(A,M).  
\end{equation*}   
\end{defi}
For $A=K$, we recover the notion of Lie algebra cohomology (cf. \cite[Section 7]{weibel1994introduction}). If $X$ is a smooth rigid analytic space, then $\mathcal{T}_{X/K}(X)$ is canonically a Lie algebroid on $\OX_X$. In other words, $\mathcal{T}_{X/K}$ is a finite locally-free $\OX_X$-module which is also a sheaf of $(K,\OX_X)$-Lie algebras. Hence, it make sense to consider the following cohomology groups:
\begin{equation*}
  \operatorname{H}^{\bullet}_{\OX_X}(\mathcal{T}_{X/K},\mathcal{M}):=\operatorname{H}^{\bullet}_{\OX_X(X)}(\mathcal{T}_{X/K}(X),\mathcal{M}(X))=R\Hom_{\D_X(X)}(\OX_X(X),\mathcal{M}(X)), 
\end{equation*}
where $\mathcal{M}\in \mathcal{C}(\wideparen{\D}_X(X))$. We may relate this cohomology theory to $\operatorname{HH}^{\bullet}(\wideparen{\D}_X)$ as follows:
\begin{prop}
 Let $X$ be a Stein space with an étale map $X\rightarrow \mathbb{A}^r_K$. For any co-admissible module $\mathcal{M}\in \mathcal{C}(\wideparen{\D}_X)$ we have the following identity in $\operatorname{D}(\operatorname{Vect}_K)$:
 \begin{equation*}
     R\Hom_{\wideparen{\D}_X(X)}(\OX_X(X),\mathcal{M}(X))=\operatorname{H}^{\bullet}_{\OX_X}(\mathcal{T}_{X/K},\mathcal{M}).
 \end{equation*}
 In particular, we have $\operatorname{Forget}(J\operatorname{HH}^{\bullet}(\wideparen{\D}_X))=\operatorname{H}^{\bullet}_{dR}(X)=\operatorname{H}^{\bullet}_{\OX_X}(\mathcal{T}_{X/K},\OX_X)$.
\end{prop}
\begin{proof}
By Theorem  \ref{teo hochschild cohomology as ext of structure sheaf as D module} and its proof, we have:
\begin{equation*}
  R\Hom_{\wideparen{\D}_X(X)}(\OX_X(X),\mathcal{M}(X))=\Hom_{\wideparen{\D}_X(X)}(\Gamma(X,S^{\bullet}),\mathcal{M}(X)), 
\end{equation*}
where $S^{\bullet}$ is the Spencer resolution. By construction,  $S^{\bullet}$ is a $\wideparen{\D}_X$-module generalization of the  Chevalley-Eilenberg complex (cf. \cite[Section 4]{rinehart1963differential}), which exists for every smooth Lie-Rinehart algebra. In the case of $\mathcal{T}_{X/K}(X)$, the Chevalley-Eilenberg complex has the following form:
\begin{equation*}
C^{\bullet}:=\left(0\rightarrow \D_X(X)\otimes_{\OX_X(X)}\wedge^n\mathcal{T}_{X/K}(X)\rightarrow \cdots \rightarrow \D_X(X) \rightarrow 0\right),
\end{equation*}
and is a resolution of $\OX_X(X)$ by finite-free $\D_X(X)$-modules. In this situation, we have:
\begin{flalign*}
\operatorname{H}^{\bullet}_{\OX_X}(\mathcal{T}_{X/K},\OX_X):= R\Hom_{\D_X(X)}(\OX_X(X),\mathcal{M}(X))=\Hom_{\D_X(X)}(C^{\bullet},\mathcal{M}(X))&&
\end{flalign*}
\begin{flalign*}
&&=\left(0\rightarrow \Hom_{\OX_X(X)}(\OX_X(X),\mathcal{M}(X))\rightarrow \cdots \rightarrow \Hom_{\OX_X(X)}(\wedge^n\mathcal{T}_{X/K},\mathcal{M}(X)) \rightarrow 0  \right)\\
&&=\Hom_{\wideparen{\D}_X(X)}(\Gamma(X,S^{\bullet}),\mathcal{M}(X)).
\end{flalign*}
Thus, the proposition holds.
\end{proof}
We point out that this corollary is similar to the results of \cite{schmidt2012stableflatnessnonarchimedeanhyperenveloping} in the context of Arens-Michael envelops of finite dimensional reductive $K$-Lie algebras. In particular, check \cite[Corollary 2.2]{schmidt2012stableflatnessnonarchimedeanhyperenveloping}.\\

The previous results show that the complex $\operatorname{H}^{\bullet}_{\OX_X}(\mathcal{T}_{X/K},\OX_X)$ has a natural structure as a strict complex of nuclear Fréchet spaces. Hence, it would be interesting to adapt the results of \cite{kosmeijer2024hochschildcohomologylierinehartalgebras} to the setting of smooth Stein spaces. We believe that some of the requirements of finite generation on the algebras could be bypassed by analytic methods. We do not pursue this study here.
\section{Duality results}\label{section duality results}
Let $X$ be either an affinoid space or a $p$-adic Stein space, and assume $X$ has free tangent sheaf. The goal of this section is showing that the Hochschild (co)-homology of $\wideparen{\D}_X$-bimodules satisfies a version of Poincaré duality. Namely, we will show that for any complex $\mathscr{M}\in \operatorname{D}(\wideparen{E}_X)$ there is a canonical isomorphism in $\operatorname{D}(\operatorname{Shv}(X^2,\Indban))$:
\begin{equation*}
    R\underline{\mathcal{H}om}_{\wideparen{E}_X}(\Delta_*\wideparen{\D}_X,\mathscr{M})=\Delta_*\wideparen{\D}_X\overrightarrow{\otimes}^L_{\wideparen{E}_X}\mathscr{M}[-2\operatorname{dim}(X)].
\end{equation*}
Similarly, for any complex $\mathscr{M}\in \operatorname{D}(\wideparen{\D}_X(X)^e)$ there is a canonical isomorphism in $\operatorname{D}(\widehat{\mathcal{B}}c_K)$:
\begin{equation*}
    R\underline{\Hom}_{\wideparen{\D}_X(X)^e}(\wideparen{\D}_X(X),\mathscr{M})=\wideparen{\D}_X(X)\widehat{\otimes}^L_{\wideparen{\D}_X(X)^e}\mathscr{M}[-2\operatorname{dim}(X)].
\end{equation*}
This type of result is often times called Van den Bergh duality, as it was first obtained in the setting of associative algebras by M. Van den Bergh in \cite{vdBduality}. Notice that one of the most important features of this result is that it does not impose any finiteness conditions on $\mathscr{M}$. In particular, we do not require that $\mathscr{M}$ is a $\mathcal{C}$-complex.  In latter stages of the section, we will use Van den Bergh duality to translate the results of Section \ref{Section HC and ext 2} to Hochschild homology. In particular, we will show that there is a canonical isomorphism in $\operatorname{D}(\widehat{\mathcal{B}}c_K)$:
\begin{equation*}
    \operatorname{HH}_{\bullet}(\wideparen{\D}_X)=\wideparen{\D}_X(X)\widehat{\otimes}_{\wideparen{\D}_X(X)^e}^{\mathbb{L}}\wideparen{\D}_X(X)=:\operatorname{HH}_{\bullet}(\wideparen{\D}_X(X)).
\end{equation*}
Thus obtaining an interpretation of $\operatorname{HH}_{\bullet}(\wideparen{\D}_X)$ in terms of Tor functors, and calculating   $\operatorname{HH}_{\bullet}(\wideparen{\D}_X(X))$. 
\subsection{\texorpdfstring{Review of operations of $\wideparen{\D}_X$-modules}{}}\label{subsection review operations}
Let $X$ be a smooth rigid analytic space. We will devote this section to reviewing some basic operations on the category of $\wideparen{\D}_X$-modules. Most of the constructions in this section are due to A. Bode, and a detailed description of the properties of these operations can be found in his papers \cite{bode2021operations} and \cite{bodeauslander}. Let us start by introducing the duality functor:
\begin{defi}
We define the duality functor as the following functor:
\begin{align*}
    &\mathbb{D}:\operatorname{D}(\wideparen{\D}_X)\rightarrow \operatorname{D}(\wideparen{\D}_X)^{\op},\\    
    \mathcal{M}\mapsto \mathbb{D}(\mathcal{M})=&R\underline{\mathcal{H}om}_{\wideparen{\D}_X}(\mathcal{M},\wideparen{\D}_X)\overrightarrow{\otimes}^{\mathbb{L}}_{\OX_X}\Omega_X^{-1} [\operatorname{dim}(X)],   
\end{align*} 
where $-\overrightarrow{\otimes}^{\mathbb{L}}_{\OX_X}\Omega_X^{-1}:\operatorname{D}(\wideparen{\D}_X^{\op})\rightarrow \operatorname{D}(\wideparen{\D}_X)$ is the side-changing operator from \cite[Section 6.3]{bode2021operations}.
\end{defi}
The duality functor is an extremely important part of the theory of co-admissible $\wideparen{\D}_X$-modules, as it is a key ingredient in the $p$-adic analytic version of holonomic $\D$-modules. Namely, the weakly holonomic $\wideparen{\D}_X$-modules defined in \cite{Ardakov_Bode_Wadsley_2021}. A co-admissible $\wideparen{\D}_X$-module  $\mathcal{M}$ is called weakly holonomic if its dual complex $\operatorname{D}(\mathcal{M})$ is a co-admissible $\wideparen{\D}_X$-module. That is, $\mathbb{D}(\mathcal{M})$ is concentrated in degree $0$, and its only non-trivial cohomology group is a co-admissible $\wideparen{\D}_X$-module. Recall as well that a co-admissible $\wideparen{\D}_X$-module is called an integrable connection if it is coherent as an $\OX_X$-module. To us, the most important results regarding these two concepts is the following proposition:
\begin{prop}\label{prop duality of integrable connections}
Let $\mathcal{M}$ be an integrable connection on $X$. Then $\mathcal{M}$ is weakly co-admissible. Furthermore, there is a canonical identification of co-admissible $\wideparen{\D}_X$-modules:
\begin{equation*}
    \mathbb{D}(\mathcal{M})=\underline{\mathcal{H}om}_{\OX_X}(\mathcal{M},\OX_X).
\end{equation*}
In particular, we have $\mathbb{D}(\OX_X)=\OX_X$.    
\end{prop}
\begin{proof}
This is shown in \cite[Proposition 7.3]{bode2021operations}.
\end{proof}
Let now $f:X\rightarrow Y$ be a morphism of smooth rigid analytic spaces. So far, we have only used two operations associated to this morphism. Namely, the 
extraordinary inverse image functor:
\begin{gather*}
    f^!:\operatorname{D}(\wideparen{\D}_Y)\rightarrow \operatorname{D}(\wideparen{\D}_X),\\
    \mathcal{M}\mapsto f^!\mathcal{M}=\OX_X\overrightarrow{\otimes}^{\mathbb{L}}_{f^{-1}\OX_Y}f^{-1}\mathcal{M}[\operatorname{dim}(X)-\operatorname{dim}(Y)]
\end{gather*}
and the direct image functor:
\begin{gather*}
    f_+:\operatorname{D}(\wideparen{\D}_X)\rightarrow \operatorname{D}(\wideparen{\D}_Y).\\
    \mathcal{M}\mapsto Rf_*(\wideparen{\D}_{Y\leftarrow X}\overrightarrow{\otimes}^{\mathbb{L}}_{\wideparen{\D}_X}\mathcal{M})
\end{gather*}
Using the duality functor, we can obtain new versions of these morphisms:
\begin{defi}
Let $f:X\rightarrow Y$ be a morphism of smooth rigid analytic spaces. We define the following functors:
\begin{enumerate}[label=(\roman*)]
    \item The extraordinary direct image functor: $f_!:=\mathbb{D}\circ f_+ \circ \mathbb{D}:\operatorname{D}(\wideparen{\D}_X)\rightarrow \operatorname{D}(\wideparen{\D}_Y)$.
    \item The inverse image functor: $f^+:=\mathbb{D}\circ f^! \circ\mathbb{D}:\operatorname{D}(\wideparen{\D}_Y)\rightarrow \operatorname{D}(\wideparen{\D}_X)$.
\end{enumerate}
\end{defi}
For us, the most important feature of these operations is  the following proposition:
\begin{prop}\label{prop equality direct image and extraordinary direct image}
Let $f:X\rightarrow Y$ be a closed immersion of smooth rigid analytic spaces, and let $\mathcal{M}\in \mathcal{C}(\wideparen{\D}_X)$ be weakly holonomic. Then there is a canonical isomorphism of $\wideparen{\D}_Y$-modules:
\begin{equation*}
    f_+\mathcal{M}=f_!\mathcal{M}.
\end{equation*}
In particular, we have $\mathbb{D}f_+\mathcal{M}=f_+\mathbb{D}\mathcal{M}$.
\end{prop}
\begin{proof}
 As the claim is purely local, we may start by assuming that $X$ and $Y$ are affinoid. Let $\mathcal{M}$ be a weakly holonomic $\wideparen{\D}_X$-module. We claim that the image under $f_+$ of a weakly holonomic $\wideparen{\D}_X$-module is a weakly holonomic $\wideparen{\D}_Y$-module. In particular, $f_!\mathcal{M}$ is a co-admissible $\wideparen{\D}_Y$-module. Indeed this is \cite[Lemma 7.1]{Ardakov_Bode_Wadsley_2021} in the case where $K$ is discretely valued, and follows by the contents of \cite[Section 5.2]{bodeauslander}, and in particular \cite[Lemma 5.10]{bodeauslander}, in the general case.

 Let $\mathcal{C}_X(\wideparen{\D}_Y)$ be the category if co-admissible $\wideparen{\D}_Y$-modules supported on $X$, and let $\mathcal{N}\in \mathcal{C}_X(\wideparen{\D}_Y)$. By the version of Kashiwara's equivalence given in \cite[Proposition 8.1.1]{HochDmod} together with \cite[Theorem 1.3]{bodeauslander}, for every co-admissible $\wideparen{\D}_Y$-module $\mathcal{N}$ supported on $X$ we have:
 \begin{equation*}
     \underline{\mathcal{H}om}_{\wideparen{\D}_Y}(f_+\mathcal{M},\mathcal{N})=f_*\underline{\mathcal{H}om}_{\wideparen{\D}_Y}(\mathcal{M},f^!\mathcal{N})=\underline{\mathcal{H}om}_{\wideparen{\D}_Y}(f_!\mathcal{M},\mathcal{N}).
 \end{equation*}
Hence, the global sections of these sheaves are isomorphic. As  both $f_+\mathcal{M}$ and $f_!\mathcal{M}$ are co-admissible, it follows by Lemma \ref{Lemma global sections inner hom functor} that there is an isomorphism of complete bornological spaces:
\begin{equation*}
    \underline{\Hom}_{\wideparen{\D}_Y(Y)}(f_+\mathcal{M}(Y),\mathcal{N}(Y))=\underline{\Hom}_{\wideparen{\D}_Y(Y)}(f_!\mathcal{M}(Y),\mathcal{N}(Y)),
\end{equation*}
Thus, it follows that there is an isomorphism of functors:
\begin{equation*}
   \Hom_{\wideparen{\D}_Y}(f_+\mathcal{M},-)=\Hom_{\wideparen{\D}_Y}(f_!\mathcal{M},-):\mathcal{C}_X(\wideparen{\D}_Y)\rightarrow \operatorname{Set},
\end{equation*}
and it follows by Yoneda's lemma that $f_+\mathcal{M}=f_!\mathcal{M}$, as we wanted to show.
\end{proof}
\begin{coro}\label{coro we need for vdb duality}
Let $X$ be a smooth and separated rigid analytic space. Then the following isomorphism holds in $\operatorname{D}(\wideparen{E}_X^{\op})$:
\begin{equation*}
    R\underline{\mathcal{H}om}_{\wideparen{E}_X}(\Delta_*\wideparen{\D}_X,\wideparen{E}_X)=\Delta_*\wideparen{\D}_X[-2\operatorname{dim}(X)]
\end{equation*}
\end{coro}
\begin{proof}
We start by noticing that by the side-switching theorem \cite[Proposition 6.1.16]{HochDmod} and the previous proposition we have the following identities in $\operatorname{D}(\operatorname{Shv}(X^2,\Indban))$:
\begin{equation*}
    R\underline{\mathcal{H}om}_{\wideparen{E}_X}(\Delta_*\wideparen{\D}_X,\wideparen{E}_X)=R\underline{\mathcal{H}om}_{\wideparen{\D}_{X^2}}(\Delta_+\OX_X,\wideparen{\D}_{X^2})=\Delta_+\OX_X\overrightarrow{\otimes}_{\OX_{X^2}}^{\mathbb{L}}\Omega_{X^2}[-2\operatorname{dim}(X)],
\end{equation*}
where the last identity follows by Proposition \ref{prop duality of integrable connections}. Thus, we have shown that $R\underline{\mathcal{H}om}_{\wideparen{E}_X}(\Delta_*\wideparen{\D}_X,\wideparen{E}_X)$ is a complex in $\operatorname{D}(\wideparen{E}_X^{\op})$ concentrated in degree $2\operatorname{dim}(X)$. In order to see that this is isomorphic as a right $\wideparen{E}_X$-module to $\Delta_*\wideparen{\D}_X[-2\operatorname{dim}(X)]$, we  again apply the side-changing isomorphism for right modules, as in  \cite[pp. 59]{HochDmod}.
\end{proof}
Using the results from Section \ref{Section background HochDmod}, we can show the following variant:
\begin{coro}\label{coro calabi-yau property of DX}
Assume $X$ is an affinoid space or $X$ is a Stein space and $K$ is discretely valued or algebraically closed. If $X$ has free tangent sheaf, then the following isomorphism holds in $\operatorname{D}(\wideparen{\D}_X(X)^e)$:
\begin{equation*}
    R\underline{\Hom}_{\wideparen{\D}_X(X)^e}(\wideparen{\D}_X(X),\wideparen{\D}_X(X)^e)=\wideparen{\D}_X(X)[-2\operatorname{dim}(X)].
\end{equation*}
\end{coro}
\begin{proof}
Under our assumption of freeness of the tangent sheaf, the corollary follows by the previous proposition, together with Theorem \ref{teo hochschild cohomology in terms of extendions in DXe 1}.
\end{proof}
\subsection{Van den Bergh duality}\label{section vdb}
After the calculations performed in the previous section, we are almost ready to show Van den Bergh duality. Until further notice, we let $X$ be a separated rigid space, and assume that it admits an étale map $X\rightarrow \mathbb{A}^r_K$. As shown in Proposition \ref{prop resolution of differential operators on smooth Stein spaces}, this implies that there is a co-admissible $\wideparen{E}_X$-module $\mathcal{S}_X$ such that there is a strict exact sequence:
\begin{equation*}
    E^{\bullet}_{\wideparen{\D}_X}:= \left(0\rightarrow P^{2r}\rightarrow \cdots\rightarrow P^0\rightarrow \mathcal{S}_X\oplus \Delta_*\wideparen{\D}_X\rightarrow 0\right),
\end{equation*}
where each of the $P^i$ is a finite-projective $\wideparen{E}_X$-module and $P^1=\wideparen{E}_X$. The existence of this complex will be instrumental in our proof of Van den Bergh duality. Let us start with the following lemma:
\begin{Lemma}\label{lemma vdb duality}
Let $\mathscr{M}\in \operatorname{D}(\wideparen{E}_X)$. Then there is a canonical isomorphism in $\operatorname{D}(\operatorname{Shv}(X^2,\Indban)$:
\begin{equation*}
    R\underline{\mathcal{H}om}_{\wideparen{E}_X}(\Delta_*\wideparen{\D}_X,\mathscr{M})=R\underline{\mathcal{H}om}_{\wideparen{E}_X}(\Delta_*\wideparen{\D}_X,\wideparen{E}_X)\overrightarrow{\otimes}^{\mathbb{L}}_{\wideparen{E}_X}\mathscr{M}.
\end{equation*}
Furthermore, if $X$ is affinoid or $X$ is Stein and $K$ is either discretely valued or algebraically closed, then for any $\mathscr{M}\in \operatorname{D}(\wideparen{\D}_X(X)^e)$  there is a canonical isomorphism in $\operatorname{D}(\widehat{\mathcal{B}}c_K)$:
\begin{equation*}
    R\underline{\Hom}_{\wideparen{\D}_X(X)^e}(\wideparen{\D}_X(X),\mathscr{M})=R\underline{\Hom}_{\wideparen{\D}_X(X)^e}(\wideparen{\D}_X(X),\wideparen{\D}_X(X)^e)\overrightarrow{\otimes}^{\mathbb{L}}_{\wideparen{\D}_X(X)^e}\mathscr{M}.
    \end{equation*}
\end{Lemma}
\begin{proof}
We only show the first case, the second one is analogous. First, notice that for $\mathcal{M}\in \operatorname{D}(\wideparen{E}_X)$ there is always a natural transformation:
\begin{equation*}
   \varphi(-): R\underline{\mathcal{H}om}_{\wideparen{E}_X}(-,\wideparen{E}_X)\overrightarrow{\otimes}^{\mathbb{L}}_{\wideparen{E}_X}\mathscr{M}\rightarrow  R\underline{\mathcal{H}om}_{\wideparen{E}_X}(-,\mathscr{M}).
\end{equation*}
We need to show that this is an isomorphism for $\Delta_*\wideparen{\D}_X$. As the functors commute with finite direct sums, it suffices to show this for $\mathcal{N}:=\mathcal{S}_X\oplus\Delta_*\wideparen{\D}_X$. Consider the following strict exact sequence:
\begin{equation*}
    0\rightarrow \mathcal{K} \rightarrow \wideparen{E}_X\rightarrow \mathcal{N}\rightarrow 0,
\end{equation*}
where $\wideparen{E}_X\rightarrow \mathcal{N}$ is the first morphism in $E^{\bullet}_{\wideparen{\D}_X}$. The natural transformation $\varphi$ induces a morphism of distinguished triangles:
\begin{equation*}
\begin{tikzcd}
{R\underline{\mathcal{H}om}_{\wideparen{E}_X}(\mathcal{N},\wideparen{E}_X)\overrightarrow{\otimes}^{\mathbb{L}}_{\wideparen{E}_X}\mathscr{M}} \arrow[r] \arrow[d]                                                             & {R\underline{\mathcal{H}om}_{\wideparen{E}_X}(\wideparen{E}_X,\wideparen{E}_X)\overrightarrow{\otimes}^{\mathbb{L}}_{\wideparen{E}_X}\mathscr{M}} \arrow[r] \arrow[d]                                                             & { R\underline{\mathcal{H}om}_{\wideparen{E}_X}(\mathcal{K},\wideparen{E}_X)\overrightarrow{\otimes}^{\mathbb{L}}_{\wideparen{E}_X}\mathscr{M}}\arrow[d]                                                             \\
{R\underline{\mathcal{H}om}_{\wideparen{E}_X}(\mathcal{N},\mathscr{M})} \arrow[r] & {R\underline{\mathcal{H}om}_{\wideparen{E}_X}(\wideparen{E}_X,\mathscr{M})} \arrow[r] & {R\underline{\mathcal{H}om}_{\wideparen{E}_X}(\mathcal{K},\mathscr{M})}
\end{tikzcd}
\end{equation*}
and the middle map is always a quasi-isomorphism. Thus, it suffices to show that $\varphi(\mathcal{K})$ is a quasi-isomorphism. However, by construction, $\mathcal{K}$ admits a strict resolution by finite-projective $\wideparen{E}_X$-modules of length $2\operatorname{dim(X)}-1$. Hence, we can repeat this process inductively to show that $\varphi(\mathcal{K})$ is a quasi-isomorphism, and therefore $\varphi(\mathcal{N})$ is a quasi-isomorphism, as we wanted to show.
\end{proof}
We remark that if $\mathscr{M}$ is a co-admissible $\wideparen{E}_X$-module, then the previous result is a special instance of \cite[Proposition 6.19]{bodeauslander}. The relevance of this lemma is that it enables us to also handle the case when $\mathscr{M}$ is not co-admissible. The importance of this stems from the fact that the deformation theory of a co-admissible module is parameterized  by the Hochschild cohomology with values in a bimodule which is far from co-admissible, and so we need homological tools to also deal with this fact. See \cite[Section 3.9]{deformationsalg} for a discussion of this in the case of associative $K$-algebras. After all this work has been done, the proof of Van den Bergh duality reduces to a simple formal computation:
\begin{teo}\label{teo Van den Bergh duality}
Let $X$ be a smooth and separated rigid space with an étale map $X\rightarrow \mathbb{A}^r_K$. For any $\mathscr{M}\in \operatorname{D}(\wideparen{E}_X)$ there is a canonical isomorphism in $\operatorname{D}(\operatorname{Shv}(X^2,\Indban)$:
\begin{equation*}
    R\underline{\mathcal{H}om}_{\wideparen{E}_X}(\Delta_*\wideparen{\D}_X,\mathscr{M})=\Delta_*\wideparen{\D}_X\overrightarrow{\otimes}^{\mathbb{L}}_{\wideparen{E}_X}\mathscr{M}[-2\operatorname{dim}(X)].
\end{equation*}
Furthermore, if $X$ is affinoid or $X$ is Stein and $K$ is either discretely valued or algebraically closed, then for any $\mathscr{M}\in \operatorname{D}(\wideparen{\D}_X(X)^e)$  there is a canonical isomorphism in $\operatorname{D}(\widehat{\mathcal{B}}c_K)$:
\begin{equation*}
    \operatorname{H}^{\bullet}(\wideparen{\D}_X(X),\mathscr{M})=\operatorname{H}_{\bullet}(\wideparen{\D}_X(X),\mathscr{M})[-2\operatorname{dim}(X)].
    \end{equation*} 
In particular, for each $n\geq 0$ there is an isomorphism in $LH(\widehat{\mathcal{B}}c_K)$:
\begin{equation*}
    \operatorname{HH}^{n}(\wideparen{\D}_X(X),\mathscr{M})=\operatorname{HH}_{2\operatorname{dim}(X)-n}(\wideparen{\D}_X(X),\mathscr{M})
\end{equation*}
\end{teo}
\begin{proof}
Again, we only deal with the first case, the second is analogous. Let $\mathscr{M}\in \operatorname{D}(\wideparen{E}_X)$. Then by Corollary \ref{coro we need for vdb duality}, and  Lemma \ref{lemma vdb duality} we have the following isomorphisms in $\operatorname{D}(\operatorname{Shv}(X^2,\Indban))$:
\begin{equation*}
R\underline{\mathcal{H}om}_{\wideparen{E}_X}(\Delta_*\wideparen{\D}_X,\mathscr{M})=R\underline{\mathcal{H}om}_{\wideparen{E}_X}(\Delta_*\wideparen{\D}_X,\wideparen{E}_X)\overrightarrow{\otimes}^{\mathbb{L}}_{\wideparen{E}_X}\mathscr{M}= \Delta_*\wideparen{\D}_X\overrightarrow{\otimes}^{\mathbb{L}}_{\wideparen{E}_X}\mathscr{M}[-2\operatorname{dim}(X)],   
\end{equation*}
and this is precisely the identity we wanted to show. The last part follows from the identities:
\begin{equation*}
    \operatorname{HH}^{n}(\wideparen{\D}_X(X),\mathscr{M})=\operatorname{H}^{n-2\operatorname{dim}(X)}(\operatorname{H}_{\bullet}(\wideparen{\D}_X(X),\mathscr{M}))=:\operatorname{HH}_{2\operatorname{dim}(X)-n}(\wideparen{\D}_X(X),\mathscr{M}),
\end{equation*}
where the last identity follows from our conventions regarding the degree of Hochschild homology.
\end{proof}
Before moving on, we would like to put the results of this section into perspective. In particular let $X$ be either affinoid or Stein, and assume it has free tangent sheaf. The two key ingredients for the proof of Van den Bergh duality of  $\wideparen{\D}_X(X)$ are the existence of the strict exact complex:
\begin{equation*}
    E^{\bullet}_{\wideparen{\D}_X}:= \left(0\rightarrow P^{2r}\rightarrow \cdots\rightarrow P^0\rightarrow \mathcal{S}_X\oplus \Delta_*\wideparen{\D}_X\rightarrow 0\right),
\end{equation*}
and the identity of graded $\wideparen{\D}_X(X)^e$-modules:
\begin{equation*}
    R\underline{\Hom}_{\wideparen{\D}_X(X)^e}(\wideparen{\D}_X(X),\wideparen{\D}_X(X)^e)=\wideparen{\D}_X(X)[-2\operatorname{dim}(X)].
\end{equation*}
In the setting of associative $K$-algebras, these two properties can be encapsulated as follows:
\begin{defi}\label{defi Calabi-Yau algebra}
Let $A$ be an associative $K$-algebra. We make the following definitions:
\begin{enumerate}[label=(\roman*)]
    \item $A$ is homologically smooth if it admits a finite resolution by finite projective $A^e$-modules.
    \item $A$ is a Calabi-Yau algebra of dimension $d$ if it is homologically smooth and there is an isomorphism of graded $A^e$-modules:
    \begin{equation*}
        R\Hom_{A^e}(A,A^e)\cong A[-d].
    \end{equation*}
\end{enumerate}
\end{defi}
Calabi-Yau algebras were introduced by V. Ginzburg in \cite{ginzcalabi-yau} in the context of differential graded algebras as a non-commutative generalization of the algebras of regular functions on affine Calabi-Yau varieties. These algebras are a central object in the deformation theory of associative algebras and have multiple applications to non-commutative geometry. See, for example \cite{deformationsalg} for an overview of the subject. In the context of Hochschild (co)-homology of associative algebras, one of the most important features of Calabi-Yau algebras is that they always satisfy Van den Bergh duality. In particular, the proofs given above carry over to the setting of bimodules over an associative algebra without significant changes. Notice that the definition of homologically smooth algebra can be reformulated in more categorical terms. Let us start by recalling the following definitions:
\begin{defi}
Let $\mathcal{A}$ be an additive category with arbitrary direct sums. We say an object $X$ in $\mathcal{A}$ is compact if for every direct sum $\bigoplus_{i\in I}Y_i$ the canonical map:
\begin{equation*}
    \bigoplus_{i\in I}\Hom_{\mathcal{A}}(X,V_i)\rightarrow \Hom_{\mathcal{A}}(X,\bigoplus_{i\in I}V_i),
\end{equation*}
is an isomorphism of sets.
\end{defi}
Let $\mathcal{A}$ be a Grothendieck abelian category. Then its derived category $\operatorname{D}(\mathcal{A})$ is an additive category with direct sums. Hence, we can study the compact objects in $\operatorname{D}(\mathcal{A})$. Furthermore, in most situations these objects are completely determined by the compact objects in $\mathcal{A}$:
\begin{Lemma}[{\cite[Tag 0948]{stacks-project}}]\label{lemma compact objects derived category}
Let $\mathcal{A}$ be a Grothendieck abelian category, and let  $\mathcal{S}$ be a set of objects of $\mathcal{A}$ satisfying the following conditions:
\begin{enumerate}[label=(\roman*)]
    \item Every object of $\mathcal{A}$ is a quotient of a direct sum of elements of $\mathcal{S}$.
    \item Every object of $\mathcal{S}$ is compact.
\end{enumerate}
Then every compact object of $\operatorname{D}(\mathcal{A})$ is a direct summand in $\operatorname{D}(\mathcal{A})$ of a finite complex of finite direct sums of elements of $\mathcal{S}$.
\end{Lemma}
Thus, if $A$ is an associative $K$-algebra and $A^e$
is its enveloping algebra, then the fact that finite-projective $A^e$-modules form a system of compact projective generators of $\Mod(A^e)$ implies that $A$ is homologically smooth if and only if $A$ is a compact object of $D(A^e)$.\\

In the bornological setting the situation is much more subtle. Indeed, let $\mathscr{A}$ be a complete bornological algebra, and assume for simplicity that $\mathscr{A}$ is Fréchet. In this situation, the Grothendieck abelian category $\Mod_{LH(\widehat{\mathcal{B}}c_K)}(I(\mathscr{A})^e)$ is no longer generated by finite-projective $I(\mathscr{A}^e)$-modules. Instead, it admits a  system of compact projective generators of the form $I(\mathscr{A})^e\widetilde{\otimes}_KI(V)$, where $V$ is a Banach $K$-vector space admitting an orthonormal basis. In this situation, we may again invoke \cite[Tag 0948]{stacks-project}, and it follows that $\mathscr{A}$ is a compact object in $\operatorname{D}(\mathscr{A}^e)$ if and only if $\mathscr{A}$ is a direct summand in $\operatorname{D}(\mathscr{A}^e)$ of a finite complex of finite direct sums of elements  of the form 
$I(\mathscr{A})^e\widetilde{\otimes}_KI(V)$. However, the natural map:
\begin{equation*}
  R\underline{\mathcal{H}om}_{I(\mathscr{A})^e}(I(\mathscr{A})^e\widetilde{\otimes}_KI(V),I(\mathscr{A})^e)\overrightarrow{\otimes}^{\mathbb{L}}_{I(\mathscr{A})^e}\mathscr{M}\rightarrow  R\underline{\mathcal{H}om}_{I(\mathscr{A})^e}(I(\mathscr{A})^e\widetilde{\otimes}_KI(V),\mathscr{M}),
\end{equation*}
is not a quasi-isomorphism for general $\mathscr{M}$ unless $V$ is a finite-dimensional $K$-vector space. Hence, $\mathscr{A}$ may not satisfy Van den Bergh duality even if $\mathscr{A}$ is a compact object in $\operatorname{D}(\mathscr{A}^e)$, and there is some non-negative integer $d\geq 0$ such that:
\begin{equation}\label{equation defi of calabi yau}
    R\underline{\Hom}_{\mathscr{A}^e}(\mathscr{A},\mathscr{A}^e)\cong \mathscr{A}[-d].
\end{equation}
In other words, one cannot directly translate the notion of a Calabi-Yau algebra to the complete bornological setting and expect to retain the good homological properties. Nevertheless, our results thus far can be rephrased by saying that $\wideparen{\D}_X(X)$ is a complete bornological version of a Calabi-Yau algebra. It would perhaps be of interest to study if the techniques from the theory of Calabi-Yau algebras can be applied to the study of the algebras $\wideparen{\D}_X(X)$ and their deformation theory. 
\begin{obs*}
Notice that if $\mathscr{A}^e$ is noetherian and Banach (or nuclear Fréchet), then the fact that $\mathscr{A}$ is a finite $\mathscr{A}^e$-module shows that $\mathscr{A}$ is a compact object in $\operatorname{D}(\mathscr{A}^e)$
if and only if $\mathscr{A}$ admits a finite strict resolution by finite-projective $\mathscr{A}^e$-modules. In this situation, if (\ref{equation defi of calabi yau}) holds for some $d\geq 0$, then $\mathscr{A}$ satisfies Van den Bergh duality.    
\end{obs*}
\subsection{Hochschild homology}\label{section HHom}
We end the section by establishing an analog of Theorem \ref{teo hochschild cohomology in terms of extendions in DXe 1} for Hochschild homology of co-admissible $\wideparen{E}_X$-modules:
\begin{teo}\label{teo hochschild homology in terms of Tor}
Assume $X$ is affinoid or $X$ is Stein and $K$ is either discretely valued or algebraically closed.    If $X$ admits an étale map $X\rightarrow \mathbb{A}^r_K$, then for any co-admissible module $\mathcal{M}\in \mathcal{C}(\wideparen{E}_X)$ we have the following identity in $\operatorname{D}(\widehat{\mathcal{B}}c_K)$:
\begin{equation*}   R\Gamma(X^2,\Delta_*\wideparen{\D}_X\overrightarrow{\otimes}_{\wideparen{E}_X}^{\mathbb{L}}\mathcal{M})=\wideparen{\D}_X(X)\widehat{\otimes}_{\wideparen{\D}_X(X)^e}^{\mathbb{L}}\Gamma(X^2,\mathcal{M}).
\end{equation*}
In particular, we have the following identity: 
\begin{equation}\label{equation Hochho determined at global sections}
    \operatorname{HH}_{\bullet}(\wideparen{\D}_X)=R\Gamma(X,\Omega_{X/K}^{\bullet})[2\operatorname{dim}(X)]=\wideparen{\D}_X(X)\widehat{\otimes}_{\wideparen{\D}_X(X)^e}^{\mathbb{L}}\wideparen{\D}_X(X)=:\operatorname{HH}_{\bullet}(\wideparen{\D}_X(X)).
\end{equation}
\end{teo}
\begin{proof}
This follows by applying the two different forms of Van den Bergh duality in Theorem \ref{teo Van den Bergh duality}, together with Theorem \ref{teo hochschild cohomology in terms of extendions in DXe 1}. Indeed, let $\mathcal{M}\in \mathcal{C}(\wideparen{E}_X)$, we have the following identities in $\operatorname{D}(\widehat{\mathcal{B}}c_K)$:
\begin{multline*}
 R\Gamma(X^2,\Delta_*\wideparen{\D}_X\overrightarrow{\otimes}_{\wideparen{E}_X}^{\mathbb{L}}\mathcal{M})=R\Gamma(X^2,R\underline{\mathcal{H}om}_{\wideparen{E}_X}(\Delta_*\wideparen{\D}_X,\mathscr{M}))[2\operatorname{dim}(X)]\\
 =R\underline{\Hom}_{\wideparen{\D}_X(X)^e}(\wideparen{\D}_X(X),\Gamma(X,\mathcal{M}))[2\operatorname{dim}(X)]=\wideparen{\D}_X(X)\widehat{\otimes}_{\wideparen{\D}_X(X)^e}^{\mathbb{L}}\Gamma(X^2,\mathcal{M}),   
\end{multline*}
and this is precisely the identity we wanted to show. The second part follows by the fact that $\Delta_*\wideparen{\D}_X$ is a co-admissible $\wideparen{E}_X$-module, together with Theorem \ref{teo HochDmod main teo}.
\end{proof}
Thus, we obtain a new interpretation of $\operatorname{HH}_{\bullet}(\wideparen{\D}_X)$ in terms of Tor functors. Furthermore, in both the cases we are interested in, the complex $\Omega_{X/K}^{\bullet}$ is a complex of $\Gamma(X,-)$-acyclic sheaves. Thus, we can obtain the following consequences for $\operatorname{HH}_{\bullet}(\wideparen{\D}_X)$:
\begin{coro}
Assume $X$ is affinoid or $X$ is Stein and $K$ is either discretely valued or algebraically closed.    If $X$ admits an étale map $X\rightarrow \mathbb{A}^r_K$, then we have that $\operatorname{HH}_{n}(\wideparen{\D}_X(X))\neq 0$ implies $\operatorname{dim}(X)\leq n\leq 2\operatorname{dim}(X)$. In particular, we have:
 \begin{equation*}
     \operatorname{HH}_{0}(\wideparen{\D}_X(X))=\operatorname{H}^0(\wideparen{\D}_X(X)\widehat{\otimes}_{\wideparen{\D}_X(X)^e}^{\mathbb{L}}\wideparen{\D}_X(X))=\wideparen{\D}_X(X)\widehat{\otimes}_{\wideparen{\D}_X(X)^e}\wideparen{\D}_X(X)=0.
 \end{equation*}
Furthermore, in the Stein  case we have the following identity in $\operatorname{D}(\widehat{\mathcal{B}}c_K)$:
\begin{equation*}
\wideparen{\D}_X(X)\widehat{\otimes}_{\wideparen{\D}_X(X)^e}^{\mathbb{L}}\wideparen{\D}_X(X)=\operatorname{H}^{\bullet}_{\operatorname{dR}}(X)^b[2\operatorname{dim}(X)].
\end{equation*}
\end{coro}
\begin{proof}
 As mentioned above, in both cases $\Omega_{X/K}^{\bullet}$ is a complex of acyclic sheaves of length $\operatorname{dim}(X)$. In particular, $R\Gamma(X,\Omega_{X/K}^{\bullet})$ can  only have non-trivial cohomology in degrees contained between $0$ and $\operatorname{dim}(X)$. Thus, the previous theorem shows that $\operatorname{HH}_{\bullet}(\wideparen{\D}_X(X))$ is also concentrated in degrees between $-2\operatorname{dim}(X)$ and $-\operatorname{dim}(X)$. The second part of the corollary follows by Theorem \ref{teo hochschild cohomology groups as de rham cohomology groups}. 
\end{proof}
\section{Deformation Theory}\label{Chapter deformation theory}
The application of Hochschild cohomology to the study of deformation theory of associative algebras was started by M. Gerstenhaber in \cite{deformations}. In this chapter, we will translate some of his results to the realm of complete bornological algebras. The main obstacle in doing this is that the classical Hochschild cohomology of associative algebras is based upon explicit calculations with the bar complex, whereas our approach is purely homological in nature. In the setting of associative $K$-algebras, every object appearing in the bar resolution of an algebra is projective, and therefore both notions agree. However, this is no longer the case for complete bornological algebras.\\

In order to bridge this gap, we will establish some comparison results between our notion of Hochschild cohomology, and the one emerging from the bar resolution. In particular, we will build a first quadrant spectral sequence connecting these two cohomology theories, and use it to obtain some explicit calculations. We point out that a first approach to Hochschild cohomology using explicit complexes was first conveyed by J.L. Taylor in \cite{taylor1972homology}. In particular, Taylor develops a theory of relative homological algebra for locally convex algebras based upon the use of resolutions which are split-exact as complexes of locally convex spaces. While Taylor's version of homological algebra lacks the flexibility of the machinery of quasi-abelian categories, it still has proven to be a powerful tool, and has found several applications in $p$-adic representation theory in recent years. For instance, in J. Kohlhaase's work on the cohomology of locally analytic representations \cite{cohrep}, and in T. Schmidt's work on stable flatness of non-archimedean hyperenveloping algebras \cite{schmidt2012stableflatnessnonarchimedeanhyperenveloping}. While Taylor's approach lacks some of the good homological properties 
that our definition of Hochschild cohomology enjoys, an advantage of working with the bar resolution is that there are explicit interpretations of the Hochschild cohomology groups in low degrees. Furthermore, these interpretations turn out to be analytic versions of the classical Hochschild cohomology groups of an associative algebra.\\

Once the basic theory has been established, we will focus on the particular case of the algebra of infinite order differential operators. In particular, if $X$ is a smooth Stein space equipped with an étale map $X\rightarrow \mathbb{A}^r_K$, we will use the Hochschild cohomology spaces $\operatorname{HH}^{\bullet}(\wideparen{\D}_X)$ to study certain algebraic invariants of $\wideparen{\D}_X(X)$. Namely, we will focus on the center, the space of (bounded) outer derivations, and the space of infinitesimal deformations of $\wideparen{\D}_X(X)$. Furthermore, if $X$ has finite-dimensional de Rham cohomology, we will be able to obtain the following explicit isomorphisms:
\begin{equation*}
 \operatorname{H}^0_{\operatorname{dR}}(X)^b=\operatorname{Z}(\wideparen{\D}_X(X)), \,  \operatorname{H}^1_{\operatorname{dR}}(X)^b=\Outder(\wideparen{\D}_X(X)), \, \operatorname{H}^2_{\operatorname{dR}}(X)^b=\mathscr{E}xt(\wideparen{\D}_X(X)),
\end{equation*}
which are a $p$-adic analytic version of the classical identities for smooth affine $K$-varieties.\newline
Along the way, we will see how the algebras of infinite order twisted differential operators studied in \cite{p-adicCheralg} can be used to construct explicit representatives for each isomorphism class of infinitesimal deformation of $\wideparen{\D}_X(X)$. Furthermore, this construction is analogous to the algebraic case. 
\subsection{Taylor's homological algebra}\label{taylor HA}
Let $\mathscr{A}$ be a complete bornological algebra, and
\begin{equation*}
    \mathscr{A}^e:=\mathscr{A}\widehat{\otimes}_K\mathscr{A}^{\op},
\end{equation*}
be its enveloping algebra. Recall from \cite[Section 2]{taylor1972homology} the bar resolution of $\mathscr{A}$ as a $\mathscr{A}^e$-module:
\begin{equation}\label{equation bar resolution}
 B(\mathscr{A})^{\bullet}:=  \left(\cdots\rightarrow \mathscr{A}^e\widehat{\otimes}_K\mathscr{A}\rightarrow\mathscr{A}^e\right).
\end{equation}
This is a strict complex of complete bornological $\mathscr{A}^e$-modules with a canonical quasi-isomorphism:
\begin{equation*}
    B(\mathscr{A})^{\bullet}\rightarrow \mathscr{A}.
\end{equation*}
The augmented complex  $B(\mathscr{A})^{\bullet}\rightarrow \mathscr{A}$ is split-exact in $\widehat{\mathcal{B}}c_K$. Furthermore, regarding $B(\mathscr{A})^{\bullet}$ as a complex of complete bornological $\mathscr{A}$-modules via the morphism of complete bornological algebras $\mathscr{A}\rightarrow \mathscr{A}^e$, it can be shown that $B(\mathscr{A})^{\bullet}\rightarrow \mathscr{A}$ is split-exact as a complex of $\mathscr{A}$-modules. Let $\mathscr{M}$ be a complete bornological $\mathscr{A}^e$-module. Applying $\underline{\Hom}_{\mathscr{A}^e}(-,\mathscr{M})$ to $B(\mathscr{A})^{\bullet}$, and using the extension-restriction adjunction, we obtain the following chain complex:
\begin{equation*}
   \mathcal{L}(\mathscr{A},\mathscr{M})^{\bullet}:= \left(0\rightarrow \mathscr{M}\xrightarrow[]{\delta_0^{\mathscr{M}}} \underline{\Hom}_{\widehat{\mathcal{B}}c_K}(\mathscr{A},\mathscr{M})\xrightarrow[]{\delta_1^{\mathscr{M}}} \underline{\Hom}_{\widehat{\mathcal{B}}c_K}(\widehat{\otimes}_K^2\mathscr{A},\mathscr{M})\xrightarrow[]{\delta_2^{\mathscr{M}}} \cdots\right),
\end{equation*}
where in general we have:
\begin{equation*}
    \mathcal{L}(\mathscr{A},\mathscr{M})^{n}=\underline{\Hom}_{\widehat{\mathcal{B}}c_K}(\widehat{\otimes}_K^n\mathscr{A},\mathscr{M}),
\end{equation*}
and the differentials are given by the following formula:
\begin{multline}\label{equation differentials in complex computing Taylor cohomology}
   \delta_n^{\mathscr{M}}f(a_0\otimes\cdots\otimes a_n)=a_0f(a_1\otimes\cdots\otimes a_n) +\sum_{i=1}^n(-1)^if(a_0\otimes\cdots a_{i-1}a_i\otimes \cdots a_n)\\
   +(-1)^{n+1}f(a_0\otimes\cdots a_{n-1})a_n.
\end{multline}
Whenever there is no confusion possible with the algebra with respect to which we are regarding $\mathscr{M}$ as a module, we will omit it from the notation, and simply write:
\begin{equation*}
    \mathcal{L}(\mathscr{M})^{\bullet}:=\mathcal{L}(\mathscr{A},\mathscr{M})^{\bullet}.
\end{equation*}
If $\mathscr{M}=\mathscr{A}$, we will write $\delta_n:=\delta_n^\mathscr{A}$ for each $n\geq 0$. We can use these complexes to define the following cohomology groups:
\begin{defi}[{\cite[Definition 2.4]{taylor1972homology}}]\label{defi TH cohomology groups}
The Taylor-Hochschild cohomology spaces of $\mathscr{A}$ with coefficients in $\mathscr{M}$ are the following bornological spaces:
\begin{equation*}
\operatorname{HH}_{\operatorname{T}}^n(\mathscr{A},\mathscr{M})=\operatorname{Coker}\left(J(\mathcal{L}(\mathscr{A},\mathscr{M})^{n-1})\rightarrow J(\operatorname{Ker}(\delta_n^{\mathscr{M}}))\right).
\end{equation*}
As before, if the complete bornological algebra $\mathscr{A}$ is clear, we will drop it from the notation and write $\operatorname{HH}_{\operatorname{T}}^n(\mathscr{M}):=\operatorname{HH}_{\operatorname{T}}^n(\mathscr{A},\mathscr{M})$. For simplicity, we will call these spaces the \emph{TH} cohomology spaces of $\mathscr{A}$ with coefficients in $\mathscr{M}$. If $\mathscr{M}=\mathscr{A}$, we will call the $\operatorname{HH}_{\operatorname{T}}^n(\mathscr{A})$ the \emph{TH} cohomology spaces of $\mathscr{A}$. 
\end{defi}
Before moving on, let us make a few comments on the definition. First, choose some $n\geq 0$. Notice that before calculating the cokernel defining $\operatorname{HH}_{\operatorname{T}}^n(\mathscr{M})$, we apply the functor $J:\widehat{\mathcal{B}}c_K\rightarrow \mathcal{B}c_K$. This is done because cokernels in $\widehat{\mathcal{B}}c_K$ do not commute with the forgetful functor:
\begin{equation*}
    \operatorname{Forget}(-):\widehat{\mathcal{B}}c_K\rightarrow \operatorname{Vect}_K.
\end{equation*}
In particular, consider the short exact sequence in $\widehat{\mathcal{B}}c_K$:
\begin{equation*}
    0\rightarrow \operatorname{Im}(\delta_{n-1}^\mathscr{M})\rightarrow \operatorname{Ker}(\delta_n^\mathscr{M})\rightarrow \operatorname{Coker}\left(\operatorname{Im}(\delta_{n-1}^\mathscr{M})\rightarrow \operatorname{Ker}(\delta_n^\mathscr{M})\right)\rightarrow 0. 
\end{equation*}
If the morphism $\delta_{n-1}^\mathscr{M}:\mathcal{L}(\mathscr{A})^{n-1}\rightarrow \mathcal{L}(\mathscr{A})^{n}$ fails to be strict, then it could happen that $\operatorname{Im}(\delta_{n-1}^\mathscr{M})$ is not closed in $\operatorname{Ker}(\delta_n^\mathscr{M})$. In this situation, the vector space underlying $\operatorname{Coker}(\operatorname{Im}(\delta_{n-1}^\mathscr{M})\rightarrow \operatorname{Ker}(\delta_n^\mathscr{M}))$ is not isomorphic to $\operatorname{Ker}(\delta_n^\mathscr{M})/\operatorname{Im}(\delta_{n-1}^\mathscr{M})$, but rather to $\operatorname{Ker}(\delta_n^\mathscr{M})/\widehat{\operatorname{Im}(\delta_{n-1}^\mathscr{M})}$. As our main reason for introducing the TH cohomology groups is that they allow for explicit calculations, it would not make sense for us to introduce this completion process into the picture. Thus, 
we define the cohomology groups as objects of $\mathcal{B}c_K$ rather than of $\widehat{\mathcal{B}}c_K$. We remark that the bornology of $\operatorname{HH}_{\operatorname{T}}^n(\mathscr{M})$ can be poorly behaved. For instance, $\operatorname{HH}_{\operatorname{T}}^n(\mathscr{M})$ may not be separated. We can, however, show the following:
\begin{prop}\label{prop 0th TC is Fréchet}
$\operatorname{HH}_{\operatorname{T}}^0(\mathscr{M})$ is a closed subspace of $\mathscr{M}$. Thus, it is a complete bornological space.
\end{prop}
\begin{proof}
Follows at once from the fact that $\underline{\Hom}_{\widehat{\mathcal{B}}c_K}(\mathscr{A},\mathscr{M})$ is separated and $\mathscr{M}$ is complete.
\end{proof}
For the rest of the TH cohomology groups, their bornological properties need to be analyzed on a case-to-case basis. We will carry out this study for $\wideparen{\D}_X(X)$ in Section \ref{TH groups of D}.\\

Notice that we have not made any use of quasi-abelian categories in the definition. One reason for this is that working with the underlying $K$-vector spaces allows us to obtain the explicit interpretations of the TH cohomology groups. In fact, these interpretations are bornological versions of the ones for the Hochschild cohomology groups of associative algebras.  This is a logical outcome of the fact that the bar resolution (\ref{equation bar resolution}) is an bornological version of the classical bar complex. The second reason is a  historical one, as the contents of \cite{taylor1972homology} predate the notion of quasi-abelian category.\\

As advertised above, one of the main features of the TH cohomology spaces is that we have explicit interpretations of the spaces in low degrees. Our next goal is describing those interpretations.
\begin{defi}
Let $\mathscr{A}$  be a complete bornological algebra and $\mathscr{M}$ be a complete bornological bimodule. A bounded derivation of $\mathscr{A}$ with coefficients in $\mathscr{M}$ is a bounded map $D:\mathscr{A}\rightarrow \mathscr{M}$ satisfying the Leibnitz rule. That is, for $x,y\in \mathscr{A}$ we have:
\begin{equation*}
    D(xy)=xD(y)+D(x)y.
\end{equation*}
We let $\Der(\mathscr{A},\mathscr{M})$ be the $K$-vector space of all bounded derivations of $\mathscr{A}$ with coefficients in $\mathscr{M}$, and regard it as a (complete) bornological space as a subspace of $\underline{\Hom}_{\widehat{\mathcal{B}}c_K}(\mathscr{A},\mathscr{M})$.  The space of inner derivations of $\mathscr{A}$ with coefficients in $\mathscr{M}$ is the following $K$-vector space:
     \begin{equation*}
         \Inn(\mathscr{A}):= \{f\in \Der_K(\mathscr{A},\mathscr{M})\, \vert \, \textnormal{ there is }m\in \mathscr{M} \textnormal{ such that } f(a)=am - ma,\textnormal{ for all } a\in \mathscr{A} \}.
     \end{equation*}
We regard $\Inn(\mathscr{A},\mathscr{M})$ as bornological space with bornology induced by $\underline{\Hom}_{\widehat{\mathcal{B}}c_K}(\mathscr{A},\mathscr{M})$.
\end{defi}
With these concepts at hand, we can give a tangible interpretation of the TH cohomology groups:
\begin{prop}[{\cite[Remark I.3.3]{helemskii2012homology}}]\label{prop interpretation of Taylor cohomology groups for n=0,1}
The following hold:
\begin{enumerate}[label=(\roman*)]
    \item $\operatorname{HH}_{\operatorname{T}}^0(\mathscr{A},\mathscr{M})=\operatorname{Z}(\mathscr{A},\mathscr{M}):=\{m\in \mathscr{M} \textnormal{ }\vert \textnormal{ }am=ma, \textnormal{ for all }a\in \mathscr{A}\}$.
    \item $\Der_K(\mathscr{A},\mathscr{M})=\operatorname{Ker}(\delta_1^{\mathscr{M}})$, and $\Inn(\mathscr{A},\mathscr{M})=\operatorname{Im}(\delta_0^\mathscr{M})$.
    \item $\operatorname{HH}_{\operatorname{T}}^1(\mathscr{A})=\Outder(\mathscr{A},\mathscr{M}):=\Der_K(\mathscr{A},\mathscr{M})/\Inn(\mathscr{A},\mathscr{M})$.
\end{enumerate}
\end{prop}
\begin{proof}
The proof is a straightforward calculation using the formulas in equation (\ref{equation differentials in complex computing Taylor cohomology}).
\end{proof}
The interpretation of $\operatorname{HH}_{\operatorname{T}}^2(\mathscr{A})$ is more involved, and requires us to introduce some concepts from deformation theory. We borrow the following discussion from  \cite[Chapter I]{helemskii2012homology}:
\begin{defi}[{\cite[Chapter I Definition 1.6]{helemskii2012homology}}]
Let $\mathscr{A}$ be a bornological algebra and $\mathscr{M}$ be a $\mathscr{A}$-bimodule. A singular extension of $\mathscr{A}$ by $\mathscr{M}$ is a triple $(\mathscr{B},i,\sigma)$, where $\mathscr{B}$ is a bornological algebra, $i:\mathscr{M}\rightarrow \mathscr{B}$ is the inclusion of a square-zero ideal, and $\sigma:\mathscr{B}\rightarrow \mathscr{A}$  is a  morphism of  bornological algebras such that there is a strict short exact sequence:
\begin{equation*}
    0\rightarrow \mathscr{M}\xrightarrow[]{i}\mathscr{B}\xrightarrow[]{\sigma}\mathscr{A}\rightarrow 0,
\end{equation*}
and such that $\sigma:\mathscr{B}\rightarrow \mathscr{A}$ admits a bounded $K$-linear split $\rho:\mathscr{A}\rightarrow \mathscr{B}$. If $\mathscr{M}=\mathscr{A}$, we will call the triple $(\mathscr{B},i,\sigma)$ a singular extension of $\mathscr{A}$.
\end{defi}
\begin{obs}
We make the following remarks to the definition:
\begin{enumerate}[label=(\roman*)]
    \item These extensions are also called Hochschild extensions or infinitesimal deformations in the literature.
    \item Notice that $\mathscr{B}$ is complete if and only if $\mathscr{A}$  and $\mathscr{M}$ are complete.
\end{enumerate}
\end{obs}
Until further notice, we let $\mathscr{A}$ be a complete bornological algebra and $\mathscr{M}$ be a complete bornological $\mathscr{A}$-bimodule. As singular extensions are split in $\widehat{\mathcal{B}}c_K$, we will often identify $\mathscr{M}$ with its image $i(\mathscr{M})$. Two singular extensions $(\mathscr{B},i,\sigma)$, $(\mathscr{B}',i',\sigma')$ are said to be equivalent if there is a bounded morphism of $K$-algebras $f:\mathscr{B}\rightarrow \mathscr{B}'$ such that the following diagram is commutative:
\begin{equation*}
\begin{tikzcd}
0 \arrow[r] & \mathscr{M} \arrow[r, "i"] \arrow[d, "="] & \mathscr{B} \arrow[r, "\sigma"] \arrow[d, "f"] & \mathscr{A} \arrow[d, "="] \arrow[r] & 0 \\
0 \arrow[r] & \mathscr{M} \arrow[r, "i'"]               & \mathscr{B}' \arrow[r, "\sigma'"]              & \mathscr{A} \arrow[r]                & 0
\end{tikzcd} 
\end{equation*}
The trivial extension of $\mathscr{A}$ by $\mathscr{M}$ is the one defined by the product:
\begin{equation}\label{equation multiplication trivial extension}
\left(\mathscr{A}\oplus\mathscr{M}\right)\otimes_K\left(\mathscr{A}\oplus\mathscr{M}\right)\rightarrow \mathscr{A}\oplus \mathscr{M},\, ((a_1,m_1),(a_2,m_2))\mapsto (a_1a_2,a_1m_2 + a_2m_1).
\end{equation}
For reasons that will be explained below, we will denote the trivial singular extension of $\mathscr{A}$ by $\mathscr{M}$ as    $\mathscr{A}_{\mathscr{M}}^{\epsilon}$. If, in addition, we have $\mathscr{M}=\mathscr{A}$, then we will write $\mathscr{A}^{\epsilon}:=\mathscr{A}^{\epsilon}_{\mathscr{A}}$.
\begin{defi}
We define $\mathscr{E}xt(\mathscr{A},\mathscr{M})$ as the pointed set of equivalence classes of singular extensions of $\mathscr{A}$ by $\mathscr{M}$, with distinguished point given by the trivial extension $\mathscr{A}_{\mathscr{M}}^{\epsilon}$. As usual, if $\mathscr{M}=\mathscr{A}$, we will simply write $\mathscr{E}xt(\mathscr{A})$.
\end{defi}

 Given a singular extension $(\mathscr{B},i,\sigma)$ with splitting $\rho:\mathscr{A}\rightarrow \mathscr{B}$, we define the operator:
 \begin{equation*}
\omega(\rho):\mathscr{A}\widehat{\otimes}_K\mathscr{A}\rightarrow \mathscr{M},\, (a,b)\mapsto \omega(\rho)(a\otimes b)=\rho(a)\rho(b)-\rho(ab).
 \end{equation*}
It follows by a straightforward calculation that $\delta^{\mathscr{M}}_2(\omega(\rho))=0$. In particular, $\omega(\rho)$ determines a cohomology class in $\operatorname{HH}_{\operatorname{T}}^2(\mathscr{A},\mathscr{M})$. Furthermore, it can be shown that this cohomology class only depends on the equivalence class of the singular extension $(\mathscr{B},i,\sigma)$, and not on the choice of splitting. Hence, we get a map:
\begin{equation*}
  \omega:\mathscr{E}xt(\mathscr{A},\mathscr{M})\rightarrow \operatorname{HH}_{\operatorname{T}}^2(\mathscr{A},\mathscr{M}), \,\omega([(\mathscr{B},i,\sigma)])=[\omega(\rho)],   
\end{equation*}
where $\rho$ is any bounded $K$-split of $(\mathscr{B},i,\sigma)$. This map satisfies the following property:
\begin{teo}[{\cite[Theorem I.1.10]{helemskii2012homology}}]\label{teo interpretation of Taylor cohomology groups for n=2}
The map $\omega:\mathscr{E}xt(\mathscr{A},\mathscr{M})\rightarrow \operatorname{HH}_{\operatorname{T}}^2(\mathscr{A},\mathscr{M})$ is a bijection sending the distinguished point of $\mathscr{E}xt(\mathscr{A},\mathscr{M})$ to zero.
\end{teo}
The previous bijection allows us to regard $\mathscr{E}xt(\mathscr{A},\mathscr{M})$ as a bornological vector space, and we will do so in the sequel without further mention. Furthermore, we make the following remark:
\begin{obs}
  We remark that the contents of this section can be adapted to non-complete bornological algebras by using the non-complete projective tensor product $\otimes_K$  in the definition of the Bar complex given in (\ref{equation bar resolution}). Furthermore, Proposition \ref{prop interpretation of Taylor cohomology groups for n=0,1}, and Theorem \ref{teo interpretation of Taylor cohomology groups for n=2} also hold in this setting.  
\end{obs}
The next logical step is establishing a comparison between the TH cohomology groups defined in this section and the Hochschild cohomology groups we have been studying thus far. However, before we do that,  we would like to put forward a slightly different perspective of the contents of this section: Let $K^{\epsilon}:=K[t]/t^2$ be the ring of dual numbers of $K$. This is an affinoid algebra, whose associated affinoid space satisfies the following property: Let $X$ be a rigid $K$-variety, then we have:
\begin{equation*}
    \Hom(\Sp(K^{\epsilon}), X)= \{(x,v), \, \textnormal{ where }x\in X(K), \textnormal{ and } v\in\mathcal{T}_{X/K,x}\otimes_{\OX_{X,x}}K(x) \}.
\end{equation*}
In other words, $\Sp(K^{\epsilon})$ co-represents the functor sending a rigid space $X$ to a pair $(x,v)$, where $x$ is a $K$-point in $X$, and $v$ is a vector in the tangent space of $x$. This space (along with its \emph{higher order versions}, the rings $K[t]/t^n$) plays a prominent role in deformation theory and, as we will see below, we can reformulate most of the concepts explained before in terms of $K^{\epsilon}$-algebras.\\

Notice that there is a canonical complete bornological algebra structure on $\mathscr{A}^{\epsilon}:=\mathscr{A}\otimes_KK[t]/t^2$ given by the tensor product of algebras. Furthermore, this product is given by the formula (\ref{equation multiplication trivial extension}) specialized to the case $\mathscr{M}=\mathscr{A}$. This is the reason why we denote the trivial singular extension of $\mathscr{A}$ by $\mathscr{A}^{\epsilon}$.  On the other hand, notice that a singular extension of $\mathscr{A}$ by $\mathscr{A}$ is by definition a complete bornological algebra structure on  $\mathscr{B}=\mathscr{A}\otimes_KK^{\epsilon}$ satisfying that we have an identification of complete bornological algebras:
\begin{equation*}
    \mathscr{A}= \mathscr{B}/(t)=\mathscr{B}\otimes_{K^{\epsilon}}K.
\end{equation*}
This fact justifies the use of the term infinitesimal deformation, and links our theory with the classical deformation theory of associative $K$-algebras, as developed in \cite{deformations}. Similarly, given two infinitesimal deformations $\mathscr{B}$, and $\mathscr{C}$, an equivalence between the associated singular extensions is nothing but an isomorphism of complete bornological $K^{\epsilon}$-algebras:
\begin{equation*}
    \varphi:\mathscr{B}\rightarrow \mathscr{C},
\end{equation*}
which is the identity after reducing modulo $t$.  As a consequence of Theorem \ref{teo interpretation of Taylor cohomology groups for n=2}, it follows that every such isomorphism is of the form:
\begin{equation*}
    \varphi: \mathscr{B}\rightarrow \mathscr{C},\, x+ty\mapsto x+ t( d(x)+y),
\end{equation*}
where $x,y\in\mathscr{A}$, and $d:\mathscr{A}\rightarrow \mathscr{A}$ is an element in $\mathcal{L}(\mathscr{A})^1$. Let us define the following objects:
\begin{defi}
 Let $\mathscr{B}$ be an infinitesimal deformation of $\mathscr{A}$. An infinitesimal automorphism of $\mathscr{B}$ is an equivalence of singular extensions:
 \begin{equation*}
     \varphi: \mathscr{B}\rightarrow \mathscr{B}.
 \end{equation*}
We let $\operatorname{Aut}(\mathscr{B})_{\operatorname{Inf}}$ be the group of infinitesimal automorphisms of $\mathscr{B}$.    
\end{defi}
We may use the above to classify the infinitesimal automorphisms of deformations of $\mathscr{A}$:
\begin{coro}\label{coro infinitesimal deformations}
Let $\mathscr{B}$ be an infinitesimal deformation of $\mathscr{A}$. There is a group isomorphism:
\begin{equation*}
    \Der(\mathscr{A})\rightarrow \operatorname{Aut}(\mathscr{B})_{\operatorname{Inf}}, \, d\mapsto  (x+ty\mapsto x+t(d(x)+y)).
\end{equation*}
In particular, $\operatorname{Aut}(\mathscr{B})_{\operatorname{Inf}}$ is an abelian group.
\end{coro}
\begin{proof}
 This is a direct consequence of the proof of \cite[Theorem I.1.10]{helemskii2012homology}.
\end{proof}
As in the classical case, it is possible to show that the space $\operatorname{HH}_{\operatorname{T}}^3(\mathscr{A})$ is the space of obstructions to extending an infinitesimal deformation $\mathscr{B}$ to a second order deformation (\emph{i.e.} a complete bornological algebra structure on $\mathscr{A}\otimes _KK[t]/(t^3)$ which restricts to $\mathscr{B}$ after quotienting by $t^2$). However, as we will see below, this space is harder to relate to our notion of Hochschild cohomology, so we will not give the details here. Let us also mention that, in the setting of Banach algebras, $\operatorname{HH}_{\operatorname{T}}^3(\mathscr{A})$ has also been used to study the perturbation theory of $\mathscr{A}$. See \cite[I.2.2]{helemskii2012homology} for a complete discussion. Finally, we point out that there is no known interpretation of $\operatorname{HH}_{\operatorname{T}}^n(\mathscr{A})$ for $n\geq 4$.\\

Before moving on, let us introduce a homological variant of the previous discussion. First, notice that there is a canonical isomorphism of complete bornological algebras:
\begin{equation*}
    (\mathscr{A}^e)^{\op}=(\mathscr{A}\widehat{\otimes}_K\mathscr{A}^{\op})^{\op}=\mathscr{A}^{\op}\widehat{\otimes}_K\mathscr{A}\cong \mathscr{A}^e.
\end{equation*}
In particular, the quasi-abelian categories of left and right complete bornological $\mathscr{A}^e$-modules are canonically isomorphic. In particular, we can regard the bar complex $B(\mathscr{A})^{\bullet}$ as a complex of right $\mathscr{A}^e$-modules. Thus, for any complete bornological $\mathscr{A}^e$-module $\mathscr{M}$, we can apply the functor $-\widehat{\otimes}_{\mathscr{A}^e}\mathscr{M}$ to the Bar complex to obtain the following chain complex:
\begin{equation*}
    \mathcal{T}(\mathscr{A},\mathscr{M})^{\bullet}:=\left(   
 \cdots \xrightarrow[]{\delta^{-2}_{\mathscr{M}}}\widehat{\otimes}_K^2\mathscr{A}\widehat{\otimes}_K\mathscr{M}\xrightarrow[]{\delta^{-1}_{\mathscr{M}}}\mathscr{A}\widehat{\otimes}_K\mathscr{M}  \xrightarrow[]{\delta^0_{\mathscr{M}}}\mathscr{M}   \right),
\end{equation*}
where, in general, we have:
\begin{equation*}
    \mathcal{T}(\mathscr{A},\mathscr{M})^{-n}=\widehat{\otimes}_K^n\mathscr{A}\widehat{\otimes}_K\mathscr{M},
\end{equation*}
and the differentials are given by the following formula:
\begin{multline}
    \delta^{-n}_{\mathscr{M}}(a_1\otimes \cdots \otimes a_{n}\otimes m)=   a_2\otimes \cdots\otimes a_n \otimes ma_1+\sum_{i=1}^{n-1}(-1)^i  a_1\otimes \cdots \otimes a_ia_{i+1}\otimes \cdots\otimes a_n \otimes m   \\
    + (-1)^n  a_1\otimes a_2\otimes\cdots\otimes a_{n-1}\otimes a_nm.
\end{multline}
As before, whenever there is no confusion with our choice of $\mathscr{A}$, we will simply write:
\begin{equation*}
    \mathcal{T}(\mathscr{M})^{\bullet}:=\mathcal{T}(\mathscr{A},\mathscr{M})^{\bullet},
\end{equation*}
and if $\mathscr{M}=\mathscr{A}$, we will write $\delta^{-n}:=\delta^{-n}_\mathscr{A}$ for $n\leq 0$. We can now define the following homology groups:
\begin{defi}[{\cite[Definition 2.4]{taylor1972homology}}]\label{defi TH homology groups}
The Taylor-Hochschild homology spaces of $\mathscr{A}$ with coefficients in $\mathscr{M}$ are the following bornological spaces:
\begin{equation*}
\operatorname{HH}^{\operatorname{T}}_n(\mathscr{A},\mathscr{M})=\operatorname{Coker}\left(J(\mathcal{T}(\mathscr{A},\mathscr{M})^{-n-1})\rightarrow J(\operatorname{Ker}(\delta^{-n}_{\mathscr{M}}))\right).
\end{equation*}
As before, if the complete bornological algebra $\mathscr{A}$ is clear, we will drop it from the notation and write $\operatorname{HH}^{\operatorname{T}}_n(\mathscr{M}):=\operatorname{HH}^{\operatorname{T}}_n(\mathscr{A},\mathscr{M})$. For simplicity, we will call these spaces the \emph{TH} homology spaces of $\mathscr{A}$ with coefficients in $\mathscr{M}$. If $\mathscr{M}=\mathscr{A}$, we will call the $\operatorname{HH}^{\operatorname{T}}_n(\mathscr{A})$ the \emph{TH} homology spaces of $\mathscr{A}$. 
\end{defi}
Unfortunately, the TH homology spaces of $\mathscr{A}$ with coefficients in $\mathscr{M}$ do not admit a description as explicit as the TH cohomology groups. However, as we will see, they present better behavior with respect to homological algebra, and their relation with the Hochschild homology groups $\operatorname{HH}^{\bullet}(\mathscr{A},\mathscr{M})$ is much less convoluted than in the cohomological setting.
\subsection{Comparison with Hochschild (co)-homology}\label{comparison HC TH}
In order to compare the TH cohomology groups with the Hochschild cohomology spaces studied in previous sections, we need to translate these complexes into $LH(\widehat{\mathcal{B}}c_K)$. Until the end of the section, we let $\mathscr{A}$ be a Fréchet $K$-algebra.\\
   
First, notice that every object in $B(\mathscr{A})^{\bullet}$ is a finite complete tensor product of copies of $\mathscr{A}$. As $\mathscr{A}$ is Fréchet, it follows by \cite[Proposition 4.25]{bode2021operations} that:
\begin{equation*}
I(\widehat{\otimes}^n_K\mathscr{A})=\widetilde{\otimes}_K^nI(\mathscr{A}).
\end{equation*}
By the same reasoning, it also follows that we have:
\begin{equation*}
    I(\mathscr{A}^e)=I(\mathscr{A}\widehat{\otimes}_K\mathscr{A}^{\op})=I(\mathscr{A})\widetilde{\otimes}_KI(\mathscr{A})^{\op}=I(\mathscr{A})^e.
\end{equation*}
Furthermore, $\widehat{\mathcal{B}}c_K$ has enough flat projectives stable under $\widehat{\otimes}_K$  (\emph{cf.} \cite[Corollary 4.17]{bode2021operations}). Hence, it follows by \cite[Proposition 1.5.3]{schneiders1999quasi} that we have an identification of functors:
\begin{equation*}
    I(\underline{\Hom}_{\widehat{\mathcal{B}}c_K}(-,-))=\underline{\Hom}_{LH(\widehat{\mathcal{B}}c_K)}(I(-),I(-)).
\end{equation*}
Let $\mathscr{M}$ be a complete bornological $\mathscr{A}^e$-module. We may use these facts to obtain the following identity of chain complexes in $\operatorname{D}(\widehat{\mathcal{B}}c_K)$:
\begin{equation*}
\mathcal{L}_I(\mathscr{A},\mathscr{M})^{\bullet}:=I(\mathcal{L}(\mathscr{A},\mathscr{M})^{\bullet})=\left( 0\rightarrow I(\mathscr{M})\xrightarrow[]{I(\delta_0^\mathscr{M})} \underline{\Hom}_{LH(\widehat{B}c_K)}(I(\mathscr{A}),I(\mathscr{M}))\xrightarrow[]{I(\delta_1^\mathscr{M})} \cdots     \right),
\end{equation*}
where in general we have $\mathcal{L}_I(\mathscr{A},\mathscr{M})^{n}=\underline{\Hom}_{LH(\widehat{B}c_K)}(\widetilde{\otimes}_K^nI(\mathscr{A}),I(\mathscr{M}))$ for every $n\geq 0$.\\

We can now start introducing Hochschild cohomology into the picture. In particular, by definition of Hochschild cohomology of a complete bornological algebra, we have:
\begin{equation*}
    \operatorname{HH}^{\bullet}(\mathscr{A},\mathscr{M}):=R\underline{\Hom}_{\mathscr{A}^e}(\mathscr{A},\mathscr{M}).
\end{equation*}
As shown above, the bar resolution $B(\mathscr{A})^{\bullet}$
is a strict exact resolution of $\mathscr{A}$ as a $\mathscr{A}^e$-module. Hence, we get the following identity:
\begin{equation*}
    \operatorname{HH}^{\bullet}(\mathscr{A},\mathscr{M})=R\underline{\Hom}_{\mathscr{A}^e}(B(\mathscr{A})^{\bullet},\mathscr{M}).
\end{equation*}
We may use the fact that $LH(\widehat{\mathcal{B}}c_K)$ is a Grothendieck abelian category to give an explicit calculation of this complex. Namely, choose an injective resolution $I(\mathscr{M})\rightarrow\mathcal{I}^{\bullet}$. Then we have the following:
\begin{equation}\label{equation TH SS}
 \operatorname{HH}^{\bullet}(\mathscr{A},\mathscr{M})=\operatorname{Tot}_{\pi}\left(\underline{\Hom}^{\bullet,\bullet}_{I(\mathscr{A})^e}(I(B(\mathscr{A})^{\bullet}),\mathcal{I}^{\bullet})\right).  
\end{equation}
As a consequence of this discussion, we arrive at the following theorem:
\begin{teo}\label{teo existence TH SS}
There is a first quadrant spectral sequence in $LH(\widehat{\mathcal{B}}c_K)$:
\begin{equation*}
E_1^{p,q}=\underline{\operatorname{Ext}}^q_{LH(\widehat{\mathcal{B}}c_K)}(\widetilde{\otimes}_K^pI(\mathscr{A}),I(\mathscr{M}))    \Rightarrow \operatorname{HH}^{p+q}(\mathscr{A},\mathscr{M}), 
\end{equation*}
which we call the Taylor-Hochschild spectral sequence. Furthermore, we have an identification:
\begin{equation*}
    E_1^{\bullet,0}=\mathcal{L}_I(\mathscr{A},\mathscr{M})^{\bullet},
\end{equation*}
and we have $E_1^{0,q}=0$ for all $q\geq 1$.
\end{teo}
\begin{proof}
The existence of the spectral sequence follows from equation (\ref{equation TH SS}). For the rest, notice that:
\begin{equation*}
    E_1^{0,\bullet}=R\underline{\Hom}_{\widehat{\mathcal{B}}c_K}(K,\mathscr{M})=\mathscr{M},
\end{equation*}
and that $E_1^{\bullet,0}=\mathcal{L}_I(\mathscr{A},\mathscr{M})^{\bullet}$ by definition of the sequence.
\end{proof}
\begin{coro}\label{coro H1 borno implies delta0 strict}
Let $\mathscr{M}\in \Mod_{\widehat{\mathcal{B}}c_K}(\mathscr{A}^e)$. The following are equivalent:
\begin{enumerate}[label=(\roman*)]
    \item $\operatorname{HH}^1(\mathscr{A},\mathscr{M})$ is an object in $\widehat{\mathcal{B}}c_K$
    \item The map $\mathscr{M}\xrightarrow[]{\delta_0^\mathscr{M}}\Hom_{\widehat{\mathcal{B}}c_K}(\mathscr{A},\mathscr{M})$ is strict.
\end{enumerate}
\end{coro}
\begin{proof}
Let $\mathscr{M}\in \Mod_{\widehat{\mathcal{B}}c_K}(\mathscr{A}^e)$. By construction of the Taylor-Hochschild spectral sequence, we have $E_1^{1,0}=0$. Hence, there is an isomorphism $\operatorname{HH}^1(\mathscr{A},\mathscr{M})\rightarrow \operatorname{H}^1(\mathcal{L}_I(\mathscr{A},\mathscr{M})^{\bullet})$. By the properties of the canonical morphism $I:\widehat{\mathcal{B}}c_L\rightarrow LH(\widehat{\mathcal{B}}c_K)$, it follows that the map $\delta_0^{\mathscr{M}}$ is strict if and only if $\operatorname{H}^1(\mathcal{L}_I(\mathscr{A},\mathscr{M})^{\bullet})$ is a complete bornological space. Hence, the result is clear.
\end{proof}
In the general case, the complex $\operatorname{HH}^{\bullet}(\mathscr{A},\mathscr{M})$ is not necessarily strict. Hence, it is difficult to use the Taylor-Hochschild spectral sequence to deduce properties of the TH cohomology groups of $\mathscr{M}$. However, if $\operatorname{HH}^{\bullet}(\mathscr{A},\mathscr{M})$ is strict, then the situation is not as dire. The sequence also behaves exceedingly well in the case of Banach algebras:
\begin{coro}
Assume $\mathscr{A}$ is a Banach algebra and one of the two following conditions hold:
\begin{enumerate}[label=(\roman*)]
    \item $K$ is discretely valued.
    \item $\mathscr{A}$ admits a dense subspace which is countably generated.
\end{enumerate}
Then we have the following identification in $\operatorname{D}(\widehat{\mathcal{B}}c_K)$:
\begin{equation*}
    \operatorname{HH}^{\bullet}(\mathscr{A},\mathscr{M})=\mathcal{L}_I(\mathscr{A},\mathscr{M})^{\bullet}.
\end{equation*}
\end{coro}
\begin{proof}
If any of the two conditions above hold, then it follows by \cite[Lemma 4.16]{bode2021operations} and \cite[Proposition 10.4]{schneider2013nonarchimedean} that $\widetilde{\otimes}_K^pI(\mathscr{A})$ is a projective object in $LH(\widehat{\mathcal{B}}c_K)$ for each $p\geq 0$. In this situation, the spaces $\underline{\operatorname{Ext}}^q_{LH(\widehat{\mathcal{B}}c_K)}(\widetilde{\otimes}_K^pI(\mathscr{A}),I(\mathscr{M}))$ are trivial for all $q\geq 1$, and the sequence degenerates.
\end{proof}
Let us now study the TH homology groups.  Recall the dissection functor $\operatorname{diss}:\widehat{\mathcal{B}}c_K\rightarrow \Indban$ studied in Section \ref{section Ind-Ban spaces}. By assumption, the algebra $\mathscr{A}$ is (the bornologification of) a Fréchet space. Thus, it follows by Proposition \ref{prop comparison of tensor products metrizable spaces} that for every $n\geq 0$ we have:
\begin{equation}\label{equation diss of bar resolution}
  \operatorname{diss}(B(\mathscr{A})^{n})= \overrightarrow{\otimes}_K^n\operatorname{diss}(\mathscr{A})\overrightarrow{\otimes}_K\operatorname{diss}(\mathscr{A})^e.
\end{equation}
Thus, the complex $\operatorname{diss}(B(\mathscr{A})^{\bullet})$ is a strict exact resolution of $\operatorname{diss}(\mathscr{A})$ as a $\operatorname{diss}(\mathscr{A})^e$-module. Furthermore, as $\overrightarrow{\otimes}_K$ is exact in $\Indban$, it follows by (\ref{equation diss of bar resolution}) that $\operatorname{diss}(B(\mathscr{A})^{\bullet})$ is in fact a strict exact flat resolution of $\operatorname{diss}(\mathscr{A})$ as a $\operatorname{diss}(\mathscr{A})^e$-module. Hence, we have the following proposition:
\begin{prop}\label{prop relation between TH homology and HH}
Consider the following complex in $LH(\widehat{\mathcal{B}}c_K)$:
\begin{equation*}
    B^{\bullet}(I(\mathscr{A})):=I(B^{\bullet}(\mathscr{A}))=\left(\cdots\rightarrow I(\mathscr{A})^e\widetilde{\otimes}_KI(\mathscr{A})\rightarrow I(\mathscr{A})^e\right).
\end{equation*}
Then $B^{\bullet}(I(\mathscr{A}))$ is a flat resolution of $I(\mathscr{A})$. Thus, for each $\mathscr{N}\in \Mod_{LH(\widehat{\mathcal{B}}c_K)}(I(\mathscr{A}))$ we have:
\begin{equation*}
    I(\mathscr{A})\widetilde{\otimes}^{\mathbb{L}}_{I(\mathscr{A})^e}\mathscr{N}=\left( \cdots\rightarrow \widetilde{\otimes}_K^2I(\mathscr{A})\widetilde{\otimes}_K\mathscr{N}  \rightarrow I(\mathscr{A})\widetilde{\otimes}_K\mathscr{N}\rightarrow  \mathscr{N}\right).
\end{equation*}
Furthermore, if $\mathscr{M}\in \Mod_{\widehat{\mathcal{B}}c_K}(\mathscr{A}^e)$ is the bornologification of a metrizable locally convex space, then we have the following identification in $\operatorname{D}(\widehat{\mathcal{B}}c_K)$:
\begin{equation*}
    \operatorname{HH}_{\bullet}(\mathscr{A},\mathscr{M})=I(\mathcal{T}(\mathscr{A},\mathscr{M})^{\bullet}).
\end{equation*}
\end{prop}
\begin{proof}
By  Proposition \ref{prop extension of closed structure}, the functor $I:\Indban\rightarrow LH(\widehat{\mathcal{B}}c_K)$ is strong symmetric monoidal, and by \cite[Proposition 4.21]{bode2021operations} every object in the essential image of $I$ is flat with respect to $\widetilde{\otimes}_K$. Hence, it follows by the previous discussion that  $B^{\bullet}(I(\mathscr{A}))$ is a flat resolution of $I(\mathscr{A})$ as a $I(\mathscr{A})^e$-module. On the other hand, for every object $V\in LH(\widehat{\mathcal{B}}c_K)$ and every $\mathscr{N}\in \Mod_{LH(\widehat{\mathcal{B}}c_K)}(I(\mathscr{A}))$ we have:
\begin{equation*}
    \left(V\widetilde{\otimes}_KI(\mathscr{A})^e\right)\widetilde{\otimes}_{I(\mathscr{A})^e}\mathscr{N}=V\widetilde{\otimes}_K\mathscr{N}.
\end{equation*}
Thus, the first part of the proposition holds. Let $\mathscr{M}\in \Mod_{\widehat{\mathcal{B}}c_K}(\mathscr{A}^e)$ be the bornologification of a metrizable locally convex space. In this situation, we have the following identities:
\begin{multline*}
    \operatorname{HH}^{\bullet}(\mathscr{A},\mathscr{M}):=\mathscr{A}\widehat{\otimes}^{\mathbb{L}}_{\mathscr{A}^e}\mathscr{M}=I(\mathscr{A})\widetilde{\otimes}^{\mathbb{L}}_{I(\mathscr{A})^e}I(\mathscr{M})\\
    =\left( \cdots\rightarrow \widetilde{\otimes}_K^2I(\mathscr{A})\widetilde{\otimes}_KI(\mathscr{M})  \rightarrow I(\mathscr{A})\widetilde{\otimes}_KI(\mathscr{M})\rightarrow  I(\mathscr{M})\right)=I(\mathcal{T}(\mathscr{A},\mathscr{M})^{\bullet}),
\end{multline*}
where the last identity follows by Propositions \ref{prop extension of closed structure} and \ref{prop comparison of tensor products metrizable spaces}.
\end{proof}
In most of the cases we are interested in, both $\mathscr{A}$ and $\mathscr{A}^e$ are Fréchet-Stein algebras, and  $\mathscr{M}$ is a co-admissible $\mathscr{A}^e$-module. Hence, $\mathscr{M}$ is a Fréchet space. In particular, it is metrizable, and therefore the conditions of the previous proposition are satisfied. 
\subsection{\texorpdfstring{The TH cohomology groups of $\wideparen{\D}_X(X)$}{}}\label{TH groups of D} Assume $K$ is either discretely valued or algebraically closed, and let $X$ be a smooth Stein space with an étale map $X\rightarrow \mathbb{A}^r_K$. In this section we will apply the previous theory to study the Hochschild cohomology of $\wideparen{\D}_X(X)$. As shown above, $\wideparen{\D}_X(X)$ is canonically a nuclear Fréchet algebra, so we may calculate its TH cohomology groups $\operatorname{HH}_{\operatorname{T}}^{\bullet}(\wideparen{\D}_X(X))$. Furthermore, as shown in Theorems \ref{teo hochschild cohomology groups as de rham cohomology groups}, and \ref{teo hochschild cohomology in terms of extendions in DXe 1}, we have the following identification in $\operatorname{D}(\widehat{\mathcal{B}}c_K)$:
\begin{equation*}
   \operatorname{HH}^{\bullet}(\wideparen{\D}_X(X))= \operatorname{HH}^{\bullet}(\wideparen{\D}_X)=\operatorname{H}^{\bullet}_{\operatorname{dR}}(X)^b,
\end{equation*}
which shows that $\operatorname{HH}^{\bullet}(\wideparen{\D}_X(X))$ is (the bornologification of) a strict complex of nuclear Fréchet spaces. Hence, we are in the best possible situation, and we may start using the Taylor-Hochschild spectral sequence  to obtain calculations of the TH cohomology groups:
\begin{prop}\label{prop computation 0th taylor cohomology group}
We have the following identities in $\widehat{\mathcal{B}}c_K$:
\begin{equation*}
  \operatorname{HH}^{0}(\wideparen{\D}_X)=\operatorname{H}_{\operatorname{dR}}^{0}(X)^b=\operatorname{Z}(\wideparen{\D}_X(X))=\operatorname{HH}_{\operatorname{T}}^0(\wideparen{\D}_X(X)).
\end{equation*}
\end{prop}
\begin{proof}
Notice that by the formulation of the Taylor-Hochschild spectral sequence, we have:
\begin{equation*}
    I(\operatorname{HH}^{0}(\wideparen{\D}_X))=\operatorname{H}^0\left(\mathcal{L}_I(\wideparen{\D}_X(X))^{\bullet}\right)=\operatorname{Ker}(I(\delta_0)).
\end{equation*}
As $I$ is a right adjoint, it preserves kernels. Hence, we have:
\begin{equation*}  
\operatorname{H}^0\left(\mathcal{L}_I(\wideparen{\D}_X(X))^{\bullet}\right)=\operatorname{Ker}(I(\delta_0))=I(\operatorname{Ker}(\delta_0))=
I(\operatorname{HH}_{\operatorname{T}}^0(\wideparen{\D}_X(X))). 
\end{equation*}
Thus, we have $\operatorname{HH}^{0}(\wideparen{\D}_X)=\operatorname{HH}_{\operatorname{T}}^0(\wideparen{\D}_X(X))$, as wanted. The third identity is Proposition \ref{prop interpretation of Taylor cohomology groups for n=0,1}. 
\end{proof}
Working a bit more, we can get a similar result for $\operatorname{HH}_{\operatorname{T}}^1(\wideparen{\D}_X(X))$. Let us start with the following:
\begin{prop}\label{prop comparison map first cohomology group}

 The following hold:
 \begin{enumerate}[label=(\roman*)]
     \item The map $\wideparen{\D}_X(X)\xrightarrow[]{\delta_0} \underline{\Hom}_{\widehat{\mathcal{B}}c_K}(\wideparen{\D}_X(X),\wideparen{\D}_X(X))$ is strict with closed image in $\widehat{\mathcal{B}}c_K$. Furthermore, there is a decomposition of complete bornological spaces:
     \begin{equation*}
         \wideparen{\D}_X(X)=\operatorname{Z}(\wideparen{\D}_X(X))\oplus \operatorname{Im}(\delta_0).
     \end{equation*}
     \item  $\operatorname{HH}^{1}(\wideparen{\D}_X)=\operatorname{H}^1\left(\mathcal{L}(\wideparen{\D}_X(X))^{\bullet}\right)=\operatorname{HH}_{\operatorname{T}}^1(\wideparen{\D}_X(X)).$ 
 \end{enumerate}
\end{prop}
\begin{proof}
By our previous calculations we have $\operatorname{HH}^1(\wideparen{\D}_X)=\operatorname{H}^1_{\operatorname{dR}}(X)^b$. As the latter is a complete bornological space, it follows by Corollary \ref{coro H1 borno implies delta0 strict} that $\delta_0$ is strict. In particular, $\operatorname{Im}(\delta_0)$ is a closed subspace of 
$\underline{\Hom}_{\widehat{\mathcal{B}}c_K}(\wideparen{\D}_X(X),\wideparen{\D}_X(X))$. Thus, we 
 have identifications:
\begin{equation*}
  \operatorname{HH}^1(\wideparen{\D}_X)= \operatorname{H}^1\left(\mathcal{L}(\wideparen{\D}_X(X))^{\bullet}\right)=\operatorname{HH}_{\operatorname{T}}^1(\wideparen{\D}_X(X)), 
\end{equation*}
and statement $(ii)$ holds. For the last part of claim $(i)$, choose a decomposition $X=\cup_{i\in I} X_i$, where each $X_i$ is a connected component of $X$. Notice that $I$ is always countable by definition of Stein space. Furthermore, by statement $(iii)$ in Lemma \ref{Lemma basic properties of Stein spaces}, each $X_i$ is a Stein space with free tangent sheaf. As $\wideparen{\D}_X$ is a sheaf on $X$, we have the following identity of complete bornological algebras:
\begin{equation*}
    \wideparen{\D}_X(X)=\prod_{i\in I}\wideparen{\D}_X(X_i).
\end{equation*}
By definition of the center of an algebra, we have the following identity of bornological spaces:
\begin{equation*}
    \operatorname{HH}_{\operatorname{T}}^0(\wideparen{\D}_X(X))=\prod_{i\in I}\operatorname{HH}_{\operatorname{T}}^0(\wideparen{\D}_X(X_i))=\prod_{i\in I}\operatorname{H}_{\operatorname{dR}}^0(X_i)^b=\prod_{i\in I}K,
\end{equation*}
where the last identity follows because each $X_i$ is connected. As each $\wideparen{\D}_X(X_i)$ is a Fréchet space, and $K$ is finite-dimensional, we can use \cite[Proposition 10.5]{schneider2013nonarchimedean} to conclude that each of the inclusions:
\begin{equation*}
    0\rightarrow K\rightarrow \wideparen{\D}_X(X_i),
\end{equation*}
has a bounded split. Hence, we obtain the following decomposition in $\widehat{\mathcal{B}}c_K$:
\begin{equation*}
\wideparen{\D}_X(X)=\operatorname{Z}(\wideparen{\D}_X(X))\oplus \operatorname{Im}(\delta_0),
\end{equation*}
which is the identity we wanted to show.
\end{proof}
We can draw a plethora of conclusions from this proposition:
\begin{coro}\label{coro dr and outer}
$\operatorname{H}_{\operatorname{dR}}^{1}(X)^b=\Outder(\wideparen{\D}_X(X))$ as complete bornological spaces. 
\end{coro}
\begin{proof}
Follows from the previous proposition, together with  Proposition \ref{prop interpretation of Taylor cohomology groups for n=0,1}.  
\end{proof}
\begin{coro}\label{coro Der is nf fd-case}
If  $\operatorname{H}_{\operatorname{dR}}^{1}(X)^b$ is finite dimensional then $\Der_K(\wideparen{\D}_X(X))$ is a nuclear Fréchet space.
\end{coro}
\begin{proof}
If $\operatorname{HH}_{\operatorname{T}}^1(\wideparen{\D}_X(X))$ is finite dimensional, then we have a decomposition:
\begin{equation*}
    \Der_K(\wideparen{\D}_X(X))=\operatorname{Im}(\delta_0)\oplus \operatorname{HH}_{\operatorname{T}}^1(\wideparen{\D}_X(X)).
\end{equation*}
By the proof of Proposition \ref{prop comparison map first cohomology group}, the space $\operatorname{Im}(\delta_0)$ is isomorphic to a closed subspace of $\wideparen{\D}_X(X)$, and therefore is a nuclear Fréchet space. Hence, $\Der_K(\wideparen{\D}_X(X))$ is also a nuclear Fréchet space.
\end{proof}
Even though the isomorphism obtained in Corollary \ref{coro dr and outer} is obtained via purely homological methods, it is possible to define an explicit $K$-linear map:
\begin{equation*}
    c_1:\operatorname{H}^1_{\operatorname{dR}}(X)^b\rightarrow \operatorname{HH}_{\operatorname{T}}^1(\wideparen{\D}_X(X)),
\end{equation*}
which is an isomorphism in some situations. This construction involves showing that any closed $1$-form $\lambda\in \Omega_{X/K}^{1,cl}(X)$  induces a bounded derivation on $\wideparen{\D}_X(X)$. In order to simplify the calculations, we will make use of Corollary \ref{coro infinitesimal deformations}, and instead show that $\lambda$ induces an infinitesimal automorphism of the trivial deformation $\wideparen{\D}_X(X)^{\epsilon}$. Hence, before defining the aforementioned map, we will study the structure of infinitesimal deformations of $\wideparen{\D}_X(X)$. Let us start with the following lemma:
\begin{Lemma}\label{Lemma comparison map second cohomology group}
We have the following short exact sequence in $LH(\widehat{\mathcal{B}}c_K)$:
     \begin{equation}\label{equation Lemma comparison map second cohomology group}
         0 \rightarrow E_3^{1,1}  \rightarrow I(\operatorname{HH}^{2}(\wideparen{\D}_X)) \rightarrow \operatorname{H}^2(\mathcal{L}_I(\wideparen{\D}_X(X))^{\bullet})  \rightarrow 0.
     \end{equation}
     In particular $E_3^{1,1}$ is in the image of $I:\widehat{\mathcal{B}}c_K\rightarrow LH(\widehat{\mathcal{B}}c_K)$. Thus, it follows that the morphism:
     \begin{equation*}
         \delta_1:\underline{\Hom}_{\widehat{\mathcal{B}}c_K}(\wideparen{\D}_X(X),\wideparen{\D}_X(X))\rightarrow \underline{\Hom}_{\widehat{\mathcal{B}}c_K}(\widehat{\otimes}_K^2\wideparen{\D}_X(X),\wideparen{\D}_X(X)),
     \end{equation*}
is strict if and only if   $E_3^{1,1}\rightarrow I(\operatorname{HH}^{2}(\wideparen{\D}_X))$ is the image under $I$ of a  strict monomorphism. Furthermore, if this holds, then the following identity also holds:
\begin{equation*}
    \operatorname{H}^2(\mathcal{L}(\wideparen{\D}_X(X))^{\bullet})=\operatorname{HH}_{\operatorname{T}}^2(\wideparen{\D}_X(X)).
\end{equation*}
\end{Lemma}
\begin{proof}
The Hochschild-Taylor spectral sequence is a first quadrant spectral sequence with $E_1^{0,q}=0$ for all $q\geq 1$. Hence, the existence of the short exact sequence $(\ref{equation Lemma comparison map second cohomology group})$ follows from standard facts on spectral sequences (\emph{cf}. \cite[Chapter 5]{weibel1994introduction}). By \cite[Proposition 1.2.29]{schneiders1999quasi}, the essential image of $I:\widehat{\mathcal{B}}c_K\rightarrow LH(\widehat{\mathcal{B}}c_K)$ is closed under sub-objects. Hence,
$E^{1,1}_3$ is in the essential image of $I$.  Choose some object $F\in\widehat{\mathcal{B}}c_K$ such that $I(F)=E_3^{1,1}$. As $I$ is fully faithful, there is a unique map $f:F\rightarrow \operatorname{HH}^{2}(\wideparen{\D}_X)$ such that its image under $I$ is the boundary map $E_3^{1,1}\rightarrow I(\operatorname{HH}^{2}(\wideparen{\D}_X))$. Assume this map is strict. As $I$ preserves strict short exact sequences, this implies that $\operatorname{H}^2(\mathcal{L}_I(\wideparen{\D}_X(X))^{\bullet})$
is also in the image of $I$. By definition, we 
 have:
 \begin{equation*}
     \mathcal{L}_I(\wideparen{\D}_X(X))^{\bullet}=I(\mathcal{L}(\wideparen{\D}_X(X))^{\bullet}).
 \end{equation*}
 Hence, the map $\delta_1$ must also be strict. The converse is analogous. For the last part of the lemma, recall that the functor $J:\widehat{\mathcal{B}}c_K\rightarrow \mathcal{B}c_K$ preserves strict short exact sequences. Thus, the fact that $\delta_1$ is strict implies that we have:
 \begin{multline*} \operatorname{HH}_{\operatorname{T}}^2(\wideparen{\D}_X(X)):=\operatorname{Coker}\left(J(\mathcal{L}(\wideparen{\D}_X(X))^{1})\rightarrow J(\operatorname{Ker}(\delta_2))\right)=J\left(\operatorname{Coker}\left(\mathcal{L}(\wideparen{\D}_X(X))^{1}\rightarrow \operatorname{Ker}(\delta_2)\right) \right)\\
 =J(\operatorname{H}^2(\mathcal{L}(\wideparen{\D}_X(X))^{\bullet})).
 \end{multline*}
Thus, $\operatorname{HH}_{\operatorname{T}}^2(\wideparen{\D}_X(X))$ is a complete bornological space, and we have $\operatorname{HH}_{\operatorname{T}}^2(\wideparen{\D}_X(X))=\operatorname{H}^2(\mathcal{L}(\wideparen{\D}_X(X))^{\bullet})$.
\end{proof}
Ideally, we would like to show that the map $I(\operatorname{HH}^{2}(\wideparen{\D}_X)) \rightarrow \operatorname{H}^2(\mathcal{L}_I(\wideparen{\D}_X(X))^{\bullet})$ is an isomorphism. Equivalently, we would like to show that the term $E_3^{1,1}$ is trivial. This is hard to do directly, as the explicit expression of $E_3^{1,1}$ is rather complicated, and we have almost no understanding of the higher Ext groups $\underline{\operatorname{Ext}}_{\widehat{\mathcal{B}}c_K}(\widehat{\otimes}_K^n\wideparen{\D}_X(X),\wideparen{\D}_X(X))$. However, there are some situations in which we can show that this holds. Namely, when the second de Rham cohomology group $\operatorname{H}_{\operatorname{dR}}^2(X)$ is finite-dimensional.\\

Indeed, the strategy is constructing a $K$-linear monomorphism:
\begin{equation*}
    \operatorname{HH}^{2}(\wideparen{\D}_X)\xrightarrow[]{c_2}\operatorname{HH}_{\operatorname{T}}^2(\wideparen{\D}_X(X)).
\end{equation*}
As before, we have the following identifications of complete bornological spaces: 
\begin{equation*}
    \operatorname{HH}^{2}(\wideparen{\D}_X)=\operatorname{H}_{\operatorname{dR}}^{2}(X)^b=\operatorname{H}^2\left(\Gamma(X,\Omega_{X/K}^{\bullet})\right).
\end{equation*}
Thus, every cohomology class $\beta\in \operatorname{HH}^{2}(\wideparen{\D}_X)$ is represented by a
closed $2$-form $\omega_{\beta} \in \Omega^{2,cl}_{X/K}(X)$. \newline
By definition of a $2$-form, $\omega_{\beta}$
induces a $\OX_X(X)$-bilinear map:
\begin{equation*}
    \omega_{\beta}:\mathcal{T}_{X/K}(X)\oplus\mathcal{T}_{X/K}(X)\rightarrow \OX_X(X).
\end{equation*}
As $\omega_{\beta}$ is closed, it follows by the contents of \cite[Section 2.2]{p-adicCheralg} that this induces a $(K,\OX_X(X))$-Lie algebra structure on the following $\OX_X(X)$-module:
\begin{equation*}
    \mathcal{A}_{\omega_{\beta}}(X):=\OX_X(X)\bigoplus \mathcal{T}_{X/K}(X),
\end{equation*}
where the bracket of two elements $(f,v),(g,w)\in\mathcal{A}_{\omega_{\beta}}(X)$ is given by the following formula:
\begin{equation}\label{bracket of tdo}
    [(f,v),(g,w)]=(v(g)-w(f)+\omega_{\beta}(v,w),[v,w]),
\end{equation}
and the anchor map is the projection onto the second factor. A Lie algebra  such as the one described above is called an Atiyah algebra or Picard algebroid on $X$.\\

By \cite[Definition 2.3.1]{p-adicCheralg}, we can use the Atiyah algebra $\mathcal{A}_{\omega_{\beta}}(X)$ to obtain an algebra of twisted differential operators $\D_{\omega_{\beta}}(X)$ on $X$. These algebras are generalizations of $\D_X(X)$ which carry a canonical filtration called the filtration by order of differential operators. Furthermore, it can be shown that their isomorphism class as filtered $K$-algebras is completely determined by the cohomology class of $\omega_{\beta}$ in $\operatorname{H}^2_{\operatorname{dR}}(X)$ (\emph{cf}. \cite[Corollary 2.2.14]{p-adicCheralg}). We will make extensive use of this property in our construction of the map $c_2:\operatorname{HH}^{2}(\wideparen{\D}_X)\rightarrow\operatorname{HH}_{\operatorname{T}}^2(\wideparen{\D}_X(X)).$\\

By construction, the algebra of twisted differential operators $\D_{\omega_{\beta}}(X)$ arises as a quotient of the universal enveloping algebra of $\mathcal{A}_{\omega_{\beta}}(X)$. In particular, we can give a presentation of $\D_{\omega_{\beta}}(X)$ in terms of generators and relations. To simplify notation, for every $v\in \mathcal{T}_{X/K}(X)$ we let $\mathbb{L}_v$ be its image in $\operatorname{Sym}_K(\mathcal{A}_{\omega_{\beta}}(X))$. Then $\D_{\omega_{\beta}}(X)$ is the quotient of $\operatorname{Sym}_K(\mathcal{A}_{\omega_{\beta}}(X))$ by the   following relations: 
\begin{equation}\label{equation relations defining tdo}
\begin{gathered}
    f\cdot \mathbb{L}_v=\mathbb{L}_{fv}, \quad  f\cdot g=fg, \quad 1_{\OX_X(X)}=1_{\operatorname{Sym}_K(\mathcal{A}_{\omega}(X))},\\
    \quad [\mathbb{L}_v,\mathbb{L}_w] =\mathbb{L}_{[v,w]} +\omega_{\beta}(\mathbb{L}_v,\mathbb{L}_w), \quad [\mathbb{L}_v,f]=v(f),
\end{gathered}   
\end{equation}
where $f,g\in \OX_X(X)$, $v,w\in\mathcal{T}_{X/K}(X)$, and the symbol $\cdot$ represents the product in $\operatorname{Sym}_K(\mathcal{A}_{\omega_{\beta}}(X))$. Notice that if $\omega_{\beta}=0$ we recover the definition of $\D_X(X)$.\\

The strategy now is defining a  slight modification of this algebra, and use it to construct an infinitesimal deformation of $\D_X(X)$. Indeed, let $K^{\epsilon}=K[t]/(t^2)$ be the ring of dual numbers, and consider the following Fréchet-Stein algebra:
\begin{equation*}
    \OX_X(X)^{\epsilon}:=K^{\epsilon}\otimes_K\OX_X(X)=\OX_X[t]/(t^2),
\end{equation*}
which is the algebra of rigid functions on the non-smooth Stein space:
\begin{equation*}
    X^{\epsilon}:=X\times_{\Sp(K)}\Sp(K^{\epsilon}).
\end{equation*}
Similarly, we can consider the following $\OX_X(X)^{\epsilon}$-module:
\begin{equation*}
  \mathcal{T}_{X/K}(X)^{\epsilon}:= \OX_X(X)^{\epsilon}\otimes_{\OX_X(X)}\mathcal{T}_{X/K}(X)=K^{\epsilon}\otimes_K\mathcal{T}_{X/K}(X).
\end{equation*}
As $X^{\epsilon}$ is not smooth, its tangent sheaf is not free. Thus, $\mathcal{T}_{X/K}(X)^{\epsilon}$ is not isomorphic to $\mathcal{T}_{X^{\epsilon}/K}(X^{\epsilon})$.
\begin{defi}\label{defi deformations}
We define the following objects:
\begin{enumerate}[label=(\roman*)]\label{defi deformation of algebras}
    \item We define the following $\OX_X(X)^{\epsilon}$-module: $\mathcal{A}_{t\omega_{\beta}}(X)=\OX_X(X)^{\epsilon}\bigoplus\mathcal{T}_{X/K}(X)^{\epsilon}$.
    \item We define $\D_{t\omega_{\beta}}(X)$ as the quotient of  $\operatorname{Sym}_{K^{\epsilon}}(\mathcal{A}_{t\omega_{\beta}}(X))$ by the following relations:
    \begin{equation}\label{equation conditions defining deformation}
    \begin{gathered}
        f\mathbb{L}_v=\mathbb{L}_{fv},\quad  f\cdot g=fg, \quad 1_{\OX_X(X)}=1_{\operatorname{Sym}_{K^{\epsilon}}(\mathcal{A}_{t\omega}(X))},\\ [\mathbb{L}_v,\mathbb{L}_w]=\mathbb{L}_{[v,w]}+t\omega_{\beta}(v,w), \quad [\mathbb{L}_v,f]=v(f),
    \end{gathered}    
    \end{equation}    
\end{enumerate}
where we have $f,g\in \OX_X(X)$, and $v,w\in \mathcal{T}_{X/K}(X)$. Notice that  $\D_{t\omega}(X)$ is canonically a $K^{\epsilon}$-algebra.
\end{defi}
As mentioned above, the fact that $X^{\epsilon}$ is not smooth implies that $\mathcal{A}_{t\omega_{\beta}}(X)$ is not an Atiyah algebra on $X^{\epsilon}$. Thus, technically speaking, $\D_{t\omega_{\beta}}(X)$ is not an algebra of twisted differential operators on $X^{\epsilon}$. However, our explicit construction allows us to show that it satisfies similar properties:
\begin{Lemma}\label{lemma decomposition auxiliary algebra}
We have the following canonical decomposition of $\OX_X(X)$-modules:
\begin{equation*}
    \D_{t\omega_{\beta}}(X)=\D_{X}(X)\oplus t\D_{X}(X).
\end{equation*}
\end{Lemma}
\begin{proof}
Notice that, by definition, we have the following canonical decompositions of $\OX_X(X)$-modules:
\begin{align*}
    \OX_X(X)^{\epsilon}:=K^{\epsilon}\otimes_K\OX_X(X)=\OX_X(X)\oplus t\OX_X(X),\\
    \mathcal{T}_{X/K}(X)^{\epsilon}:=K^{\epsilon}\otimes_K\mathcal{T}_{X/K}(X)=\mathcal{T}_{X/K}(X)\oplus t\mathcal{T}_{X/K}(X)
\end{align*}
Hence, our decomposition follows by the fact that $t$ commutes with every element in  $\D_{t\omega_{\beta}}(X)$.
\end{proof}
Recall that the map $\D_X(X)\rightarrow \wideparen{\D}_X(X)$ endows $\D_X(X)$ with a canonical bornology. Thus, by the previous lemma, we obtain a canonical bornology on $\D_{t\omega_{\beta}}(X)$. Furthermore, it is easy to check that the product in $\D_{t\omega_{\beta}}(X)$ is bounded with respect to this bornology, and thus $\D_{t\omega_{\beta}}(X)$ is canonically a bornological algebra. Our interest in $\D_{t\omega_{\beta}}(X)$ stems from the following proposition:
\begin{prop}\label{prop construction of Hochschild extensions}
There is a map $\D_{t\omega_{\beta}}(X)\xrightarrow[]{\sigma_{\omega_{\beta}}}\D_X(X)$ such that there is a singular extension:
\begin{equation*}
    0\rightarrow \D_X(X)\rightarrow \D_{t\omega_{\beta}}(X)\xrightarrow[]{\sigma_{\omega_{\beta}}}\D_X(X)\rightarrow 0.
\end{equation*}
\end{prop}
\begin{proof}
Notice that reducing the relations in (\ref{equation conditions defining deformation}) modulo $t$ we obtain the relations in (\ref{equation relations defining tdo}). Thus, by Lemma \ref{lemma decomposition auxiliary algebra}, reducing $\D_{t\omega_{\beta}}(X)$ modulo $t$ yields the following short exact sequence:
\begin{equation*}
    0\rightarrow t\D_{X}(X)\rightarrow \D_{t\omega_{\beta}}(X)\xrightarrow[]{\sigma_{\omega_{\beta}}}\D_X(X)\rightarrow 0,
\end{equation*}
where $\sigma_{\omega_{\beta}}$ is a bounded morphism of bornological $K$-algebras.  The next step is showing that $\sigma_{\omega_{\beta}}$ has a bounded $K$-linear split. We do this by defining one explicitly. By assumption, there is an étale map $X\rightarrow \mathbb{A}^r_K$. Thus, by the PBW Theorem \cite[Proposition 2.3.2]{p-adicCheralg}, there are tangent fields $v_1,\cdots,v_r\in \mathcal{T}_{X/K}(X)$ such that we have an isomorphism of $\OX_X(X)$-modules: 
\begin{equation*}
   \D_X(X)=\bigoplus_{\underline{n}\in \mathbb{Z}^{\geq 0,r}}\OX_X(X)\mathbb{L}_{v_1}^{\underline{n}(1)}\cdots \mathbb{L}_{v_r}^{\underline{n}(r)}. 
\end{equation*}
This identity, together with Lemma \ref{lemma decomposition auxiliary algebra}, induces the following decomposition of $\OX_X(X)$-modules:
\begin{equation}\label{equation basis of deformation}
    \D_{t\omega_{\beta}}(X)=\left(\bigoplus_{\underline{n}\in \mathbb{Z}^{\geq 0,r}}\OX_X(X)\mathbb{L}_{v_1}^{\underline{n}(1)}\cdots \mathbb{L}_{v_r}^{\underline{n}(r)}\right)\bigoplus t\left(\bigoplus_{\underline{n}\in \mathbb{Z}^{\geq 0,r}}\OX_X(X)\mathbb{L}_{v_1}^{\underline{n}(1)}\cdots \mathbb{L}_{v_r}^{\underline{n}(r)} \right).
\end{equation}
Hence, we can define a $K$-linear split to $\sigma_{\omega_{\beta}}$ by the following formula:
\begin{equation*}
    \rho_{\omega_{\beta}}: \D_X(X)\rightarrow \D_{t\omega_{\beta}}(X), \, a_{\underline{n}}\mathbb{L}_{v_1}^{\underline{n}(1)}\cdots \mathbb{L}_{v_r}^{\underline{n}(r)}\mapsto a_{\underline{n}}\mathbb{L}_{v_1}^{\underline{n}(1)}\cdots \mathbb{L}_{v_r}^{\underline{n}(r)} \oplus 0.
\end{equation*}
Next, notice that as $t^2=0$ in $K^{\epsilon}$, it follows that $t\D_X(X)$ is a square-zero ideal in $\D_{t\omega_{\beta}}(X)$. Thus, it has a canonical structure as a bornological $\D_X(X)$-bimodule. All we need to do is showing that  $t\D_X(X)$ is isomorphic to $\D_X(X)$ as a bornological bimodule. It suffices to show that for each $1\leq i,j\leq r$ we have the following identity:
\begin{equation*}
    \rho_{\omega_{\beta}}(\mathbb{L}_{v_i})\cdot t\mathbb{L}_{v_j}=t\rho_{\omega_{\beta}}(\mathbb{L}_{v_i}\mathbb{L}_{v_j}).
\end{equation*}
If $i\leq j$, then the explicit basis of $\D_{t\omega_{\beta}}(X)$ as an $\OX_X(X)$-module showcased in (\ref{equation basis of deformation}) makes this automatic. Thus, we may freely assume that $i> j$. In this situation, our choice of vector fields $v_1,\cdots,v_r$ implies that $[\mathbb{L}_{v_i},\mathbb{L}_{v_j}]=0$ in $\D_X(X)$. Hence, we have:
\begin{equation*}
    t\rho_{\omega_{\beta}}(\mathbb{L}_{v_i}\mathbb{L}_{v_j})=t\rho_{\omega_{\beta}}(\mathbb{L}_{v_j}\mathbb{L}_{v_i})=t\mathbb{L}_{v_j}\mathbb{L}_{v_i}.
\end{equation*}
On the other hand, using the relations in (\ref{equation conditions defining deformation}), we have:
\begin{equation*}
    \rho_{\omega_{\beta}}(\mathbb{L}_{v_i})\cdot t\mathbb{L}_{v_j}=t(\mathbb{L}_{v_i}\mathbb{L}_{v_j})=t(\mathbb{L}_{v_j}\mathbb{L}_{v_i}+ t\omega_{\beta}(v_i,v_j))=t\mathbb{L}_{v_j}\mathbb{L}_{v_i}.
\end{equation*}
Thus, both expressions agree, and the proposition holds.
\end{proof}
Applying the bornological completion functor, we obtain the following Hochschild extension:
\begin{equation*}
\begin{tikzcd}
0 \arrow[r] & \wideparen{\D}_X(X) \arrow[r] & \wideparen{\D}_{t\omega_{\beta}}(X) \arrow[r, "\wideparen{\sigma}_{\omega_{\beta}}"] & \wideparen{\D}_X(X) \arrow[l, "\wideparen{\rho}_{\omega_{\beta}}"', bend right] \arrow[r] & 0.
\end{tikzcd}
\end{equation*}
Furthermore, the proof of Proposition \ref{prop construction of Hochschild extensions} shows that our choice of an ordered $\OX_X(X)$-basis $\{v_1,\cdots, v_r\}$ of $\mathcal{T}_{X/K}(X)$ induces a $K$-linear map:
\begin{equation*}
   c_2: \Omega_{X/K}^{2,cl}(X)\rightarrow \operatorname{Ker}(\delta_2),
\end{equation*}
defined by sending a closed $2$-form $\omega$ to the bounded $K$-linear morphism:
\begin{equation*}    c_2(\omega):\wideparen{\D}_X(X)\widehat{\otimes}_K\wideparen{\D}_X(X)\rightarrow \wideparen{\D}_X(X),\, (a,b)\mapsto c_2(\omega)(a\otimes b)=\wideparen{\rho}_{\omega}(a)\wideparen{\rho}_{\omega}(b)-\wideparen{\rho}_{\omega}(ab),
\end{equation*}
where $\rho_{\omega}$ is constructed as in Proposition \ref{prop construction of Hochschild extensions}. This map is rather hard to describe in general. However, as an upshot of the proof of Proposition \ref{prop construction of Hochschild extensions}, we get the following formula:
\begin{equation}
    c_2(\omega)(v_i\otimes v_j)= \begin{cases} 0 & \textnormal{ if $i\leq j$}\\  \omega(v_i,v_j) &  \textnormal{ if $i> j$} 
    \end{cases}
\end{equation}
Composing this map with the projection $\operatorname{Ker}(\delta_2)\rightarrow \operatorname{HH}_{\operatorname{T}}^2(\wideparen{\D}_X(X))$, we obtain a $K$-linear map:
\begin{equation*}
    c_2: \Omega_{X/K}^{2,cl}(X)\rightarrow \operatorname{HH}_{\operatorname{T}}^2(\wideparen{\D}_X(X)).
\end{equation*}
Notice that, a priory, $c_2: \Omega_{X/K}^{2,cl}(X)\rightarrow \operatorname{HH}_{\operatorname{T}}^2(\wideparen{\D}_X(X))$ depends on the choice of an ordered basis for $\mathcal{T}_{X/K}(X)$.  However, this map is in fact independent of such a choice. Indeed, let $\omega \in \Omega_{X/K}^{2,cl}(X)$ be a closed differential $2$-form. Notice that the construction of the algebra $\wideparen{\D}_{t\omega}(X)$ from Definition \ref{defi deformations} is completely independent of the choice of a basis of $\mathcal{T}_{X/K}(X)$. In particular, a different choice of ordered basis of $\mathcal{T}_{X/K}(X)$ does not alter the product of $\wideparen{\D}_{t\omega}(X)$, it only induces a different bounded split of the canonical projection:
\begin{equation*}
    \wideparen{\D}_{t\omega}(X)\rightarrow \wideparen{\D}_X(X).
\end{equation*}
Furthermore,  as shown in Proposition \ref{prop construction of Hochschild extensions}, this map is obtained by reducing $\wideparen{\D}_{t\omega}(X)$ modulo $t$. Hence, it is also independent of the basis of $\mathcal{T}_{X/K}(X)$. However, in this situation, Theorem \ref{teo interpretation of Taylor cohomology groups for n=2} shows that isomorphism class of the infinitesimal deformation $\wideparen{\D}_{t\omega}(X)$ in $\operatorname{HH}_{\operatorname{T}}^2(\wideparen{\D}_X(X))$ is independent of the choice of the split. Therefore, the map:
\begin{equation*}
   c_2:\Omega_{X/K}^{2,cl}(X)\rightarrow \operatorname{HH}_{\operatorname{T}}^2(\wideparen{\D}_X(X)),
\end{equation*}
is canonical. That is, it does not depend on the choice of an ordered basis of $\mathcal{T}_{X/K}(X)$.\\

Next, we want to show that $c_2$ descends to a $K$-linear monomorphism:
\begin{align*}
    c_2:\operatorname{H}_{\operatorname{dR}}^{2}&(X)
    ^b\rightarrow \operatorname{HH}_{\operatorname{T}}^2(\wideparen{\D}_X(X)).
\end{align*}
This is equivalent to showing that given a closed $2$-form $\omega\in\Omega_{X/K}^{2,cl}(X)$, its class $[\omega]\in \operatorname{H}^2_{\operatorname{dR}}(X)^b$ is trivial if and only if $\wideparen{\D}_{t\omega}(X)$ is a trivial Hochschild extension. Let us now show that this holds:
\begin{prop}
The Hochschild extension $\wideparen{\D}_{t\omega}(X)$ is trivial if and only if $[\omega]=0$ in $\operatorname{H}^2_{\operatorname{dR}}(X)^b$.
\end{prop}
\begin{proof}
By definition, the  extension $\wideparen{\D}_{t\omega}(X)$ is trivial if and only if there is a commutative diagram:
\begin{equation*}
\begin{tikzcd}
0 \arrow[r] & \wideparen{\D}_X(X) \arrow[r] \arrow[d, "="] & \wideparen{\D}_{t\omega}(X) \arrow[r] \arrow[d, "\varphi"] & \wideparen{\D}_X(X) \arrow[r] \arrow[d, "="] & 0 \\
0 \arrow[r] & \wideparen{\D}_X(X) \arrow[r]                & \wideparen{\D}_X(X)^{\epsilon} \arrow[r] & \wideparen{\D}_X(X) \arrow[r]                & 0
\end{tikzcd}
\end{equation*}
where $\varphi:\wideparen{\D}_{t\omega}(X)\rightarrow \wideparen{\D}_X(X)^{\epsilon}$ is an isomorphism of complete bornological algebras. Using the split:
\begin{equation*}
    \wideparen{\D}_{t\omega}(X)=\wideparen{\D}_X(X) \bigoplus t\wideparen{\D}_X(X),
\end{equation*}
commutativity of this diagram means that for $x,y\in\wideparen{\D}_X(X)$, we have the following identity:
\begin{equation*}
    \varphi(x+ty)=x +t(\eta(x)+ y),
\end{equation*}
where $\eta:\wideparen{\D}_X(X)\rightarrow \wideparen{\D}_X(X)$ is $K$-linear and bounded. Our goal is showing that $\omega$ is exact. Namely, we will show that there is a form $\beta:\mathcal{T}_{X/K}(X)\rightarrow \OX_X(X)$ such that for all $v,w\in \mathcal{T}_{X/K}(X)$ we have:
\begin{equation*}
    \omega(v,w)=v(\beta(w))-w(\beta(v))-\beta([v,w]).
\end{equation*}
Next, recall that  $\OX_X(X)$ has a canonical $\wideparen{\D}_X(X)$-module structure. As the deformation $\wideparen{\D}_X(X)^{\epsilon}$ is trivial, it follows that $\OX_X(X)^{\epsilon}$ has a canonical structure as a  $\wideparen{\D}_X(X)^{\epsilon}$-module. Namely, for a pair of elements $x+ty\in\wideparen{\D}_X(X)^{\epsilon}$, and $f+tg\in\OX_X(X)^{\epsilon}$, we have:
\begin{equation*}
    (x+ty)(f+tg)=x(f)+t(x(g)+y(f)).
\end{equation*}
We can use the morphism $\wideparen{\D}_{t\omega}(X)\rightarrow \wideparen{\D}_X(X)^{\epsilon}$ to obtain a $\wideparen{\D}_{t\omega}(X)$-module structure on $\OX_X(X)^{\epsilon}$. Explicitly, for $x+ty\in\wideparen{\D}_{t\omega}(X)^{\epsilon}$, and $f+tg\in\OX_X(X)^{\epsilon}$, we have:
\begin{equation*}
    (x+ty)(f+tg)=x(f)+t(x(g)+  y(f)+\eta(x)(f)).
\end{equation*}
In particular, the above formula induces a bounded $K$-linear map:
\begin{equation*}
    \eta(-):\wideparen{\D}_X(X)\rightarrow \End_{\widehat{\mathcal{B}}c_K}(\OX_X(X)).
\end{equation*}
In order to simplify the notation, given $x\in \wideparen{\D}_{X}(X)^{\epsilon}$,  we will denote the bounded $K$-linear endomorphism  $x(-):\OX_X(X)^{\epsilon}\rightarrow \OX_X(X)^{\epsilon}$ also by $x$. Choose $v,w\in\mathcal{T}_{X/K}(X)$. Then we have the following chain of identities in $\End_{\widehat{\mathcal{B}}c_K}(\OX_X(X)^{\epsilon})$:
\begin{align*}
\mathbb{L}_{[v,w]}+t\left(\eta(\mathbb{L}_{[v,w]}))  + \omega(v,w)\right) =\varphi([\mathbb{L}_v,\mathbb{L}_w])=[\varphi(\mathbb{L}_v),\varphi(\mathbb{L}_w)]=[\mathbb{L}_v +t\eta(\mathbb{L}_v), \mathbb{L}_w +t\eta(\mathbb{L}_w)]&\\
=\mathbb{L}_{[v,w]}+ t\left([\mathbb{L}_v,\eta (\mathbb{L}_w)] -[\mathbb{L}_w,\eta(\mathbb{L}_v)]\right)&.   
\end{align*}
Simplifying this expression, we arrive at the following equation in $\End_{\widehat{\mathcal{B}}c_K}(\OX_X(X))$:
\begin{equation}\label{equation triviality of extensions}
   \omega(v,w)=[\mathbb{L}_v,\eta (\mathbb{L}_w)] -[\mathbb{L}_w,\eta(\mathbb{L}_v)]-\eta(\mathbb{L}_{[v,w]}).
\end{equation}
Thus, it suffices to show that the restriction of $\eta(-)$ to $\mathcal{T}_{X/K}(X)$:
\begin{equation*}
    \eta(-):\mathcal{T}_{X/K}(X)\rightarrow \End_{\widehat{\mathcal{B}}c_K}(\OX_X(X)),
\end{equation*}
is a $1$-form. Equivalently, we need to show that  it is $\OX_X(X)$-linear and takes values in the subalgebra $\OX_X(X)\subset \End_{\widehat{\mathcal{B}}c_K}(\OX_X(X))$.\\

Assume that $\eta(f):\OX_X(X)\rightarrow \OX_X(X)$ is zero for all $f\in \OX_X(X)$. In this situation, for every $v\in \mathcal{T}_{X/K}(X)$ we have the following identity in $\End_{\widehat{\mathcal{B}}c_K}(\OX_X(X))$:
\begin{equation*}
\eta(f\mathbb{L}_v)=f\circ\eta(\mathbb{L}_v)+\eta(f)\circ\mathbb{L}_v=f\circ\eta(\mathbb{L}_v),
\end{equation*}
so the restriction of $\eta$ to $\mathcal{T}_{X/K}(X)$ is $\OX_X(X)$-linear. On the other hand, fix some $v\in \mathcal{T}_{X/K}(X)$. Then, using the canonical identification:
\begin{equation*}
    \OX_X(X)=\End_{\OX_X(X)}(\OX_X(X)),
\end{equation*}
it follows that  $\eta(\mathbb{L}_v)\in \OX_X(X)$ if and only if $[\eta(\mathbb{L}_v),f]=0$ for all $f\in \OX_X(X)$. We may repeat the calculations leading up to equation (\ref{equation triviality of extensions}) to obtain the following:
\begin{equation*}
    0=\eta(v(f))=\eta([\mathbb{L}_v,f])=[\mathbb{L}_v,\eta (f)] -[f,\eta(\mathbb{L}_v)]=[\eta(\mathbb{L}_v),f].
\end{equation*}
Thus, $\eta(\mathbb{L}_v)\in \OX_X(X)$, as we wanted to show. Unfortunately, in the general case it is not always true that $\eta(f)=0$ for all $f\in\OX_X(X)$. However, we will see that we can replace our representation in a way that the previous developments hold. In order to do this, we will use some techniques from classical deformation theory of associative algebras. In particular, until the end of the proof, the symbol $\operatorname{HH}^{\bullet}(-)$ will denote the Hochschild cohomology of associative algebras, and the \emph{Ext} functors are calculated in the corresponding categories of modules.\\

First, choose an ordered $\OX_X(X)$-basis of $\mathcal{T}_{X/K}(X)$. As shown above, this induces a split $\rho_{\omega}:\wideparen{\D}_X(X)\rightarrow \wideparen{\D}_{t\omega}(X)$. Notice that by construction $\rho_{\omega}$ is an algebra homomorphism when restricted to $\OX_X(X)$. In particular, the trivial deformation $\OX_X(X)^{\epsilon}$
is a subalgebra of $\wideparen{\D}_{t\omega}(X)$. Hence, restriction along the inclusion $\OX_X(X)^{\epsilon}\rightarrow \wideparen{\D}_{t\omega}(X)$
endows $\OX_X(X)^{\epsilon}$ with a new $\OX_X(X)^{\epsilon}$-module structure. Furthermore, using the formulas above, and the fact that this structure is $K^{\epsilon}$-linear, it follows that it is completely determined by the following formula for varying $f,g\in \OX_X(X)$:
\begin{equation*}
    f(g):=\varphi(f)(g)=(f+t\eta(f))(g)=fg+t\eta(f)(g).
\end{equation*}
Denote this new module structure by $\OX_X(X)^{\varphi}$. Notice that, reducing modulo $t$, we recover the usual action of $\OX_X(X)$ on itself. Thus, it follows that $\OX_X(X)^{\varphi}$ is an infinitesimal deformation of $\OX_X(X)$ as a module over itself. As shown in \cite[Proposition 3.9.3.]{deformationsalg}, the isomorphism classes of infinitesimal deformations of $\OX_X(X)$ as a module over itself are parameterized by the following space:
\begin{equation*}
\operatorname{HH}^1(\OX_X(X),\operatorname{End}_K(\OX_X(X)))=\operatorname{Ext}^1_{\OX_X(X)}(\OX_X(X),\OX_X(X))=0.
\end{equation*}
In particular, $\OX_X(X)^{\varphi}$ is induced by the following $2$-cocycle:
\begin{equation*}
    \OX_X(X)\rightarrow \End_K(\OX_X(X)),\, f\mapsto \eta(f):\OX_X(X)\rightarrow \OX(X).
\end{equation*}
The fact that its cohomology class in $\operatorname{HH}^1(\OX_X(X),\End_K(\OX_X(X)))$ is trivial implies that this 2-cocycle is contained in the image of the zeroth boundary map of the Hochschild complex:
\begin{equation*}
    \End_K(\OX_X(X))\rightarrow \Hom_K(\OX_X(X),\End_K(\OX_X(X))), \, \psi\mapsto (f\mapsto \psi(f-)-f\psi(-)).
\end{equation*}
Thus, there is some $K$-linear endomorphism $\psi:\OX_X(X)\rightarrow \OX_X(X)$ satisfying that:
\begin{equation*}
   \Psi: \OX_X(X)^{\varphi}\rightarrow \OX_X(X)^{\epsilon},\, x+ty\mapsto x+t(\psi(x)+y),
\end{equation*}
is an $\OX_X(X)^{\epsilon}$-linear isomorphism. Thus, we can use the isomorphism of $K$-algebras:
\begin{equation*}
    \End_K(\OX_X(X)^{\varphi})\rightarrow \End_K(\OX_X(X)^{\epsilon}), \, f\mapsto  \Psi\circ f\circ \Psi^{-1},
\end{equation*}
to obtain a $\wideparen{\D}_{t\omega}(X)$-module structure on $\OX_X(X)^{\epsilon}$. As $\Psi$ is $\OX_X(X)^{\epsilon}$-linear, the action of $\OX_X(X)^{\epsilon}$ on $\OX_X(X)^{\epsilon}$ obtained in this way is the canonical one. In particular, $\eta(f)=0$ for all $f\in \OX_X(X)$.
\end{proof}
\begin{coro}\label{coro existence of comparison map}
There is a $K$-linear monomorphism:
\begin{equation*}
    c_2:\operatorname{HH}^{2}(\wideparen{\D}_X)\rightarrow \operatorname{HH}^2_T(\wideparen{\D}_X(X)), \, [\omega]\mapsto [\wideparen{\D}_{t\omega}(X)].
\end{equation*}
\end{coro}
We remark that the previous result was first obtained in the setting of complex affine varieties by P. Etingof in \cite[Lemma 2.5]{etingof2004cherednik}, and that our proof closely follows his. We regard the fact that the algebraic proof can be adapted to the bornological setting as a hint at the idea that, once the proper homological tools have been established, the deformation theory of complete bornological algebras should behave like the classical deformation theory of associative $K$-algebras.\\

Finally, we can use these results to show the following theorem:
\begin{teo}\label{teo comparison map on degree 2}
Assume $\operatorname{H}_{\operatorname{dR}}^{2}(X)$ is finite dimensional. The following  isomorphisms hold:
\begin{equation*}
    \operatorname{HH}^{2}(\wideparen{\D}_X)=\operatorname{H}_{\operatorname{dR}}^{2}(X)^b=\mathscr{E}xt(\wideparen{\D}_X(X))=\operatorname{HH}_{\operatorname{T}}^2(\wideparen{\D}_X(X)).
\end{equation*}
Furthermore, $ c_2:\operatorname{HH}^{2}(\wideparen{\D}_X)\rightarrow \operatorname{HH}^2_T(\wideparen{\D}_X(X))$ is an isomorphism.
\end{teo}
\begin{proof}
If the conditions of the theorem hold, then $\operatorname{HH}^{2}(\wideparen{\D}_X)$ is a finite-dimensional $K$-vector space. Hence,  by Lemma \ref{Lemma comparison map second cohomology group} 
the space $\operatorname{HH}_{\operatorname{T}}^2(\wideparen{\D}_X(X))$ is also finite dimensional, and we have:
\begin{equation*}
    \operatorname{dim}(\operatorname{HH}^{2}(\wideparen{\D}_X))\geq \operatorname{HH}_{\operatorname{T}}^2(\wideparen{\D}_X(X)).
\end{equation*}
However, as $c_2:\operatorname{HH}^{2}(\wideparen{\D}_X)\rightarrow \operatorname{HH}^2_T(\wideparen{\D}_X(X))$ is injective, it follows that $c_2$ must be an isomorphism.
\end{proof}
As a consequence, we get the following two Corollaries:
\begin{coro}
Assume $\operatorname{H}_{\operatorname{dR}}^{2}(X)$ is finite dimensional. Then there is an isomorphism:
\begin{equation*}
    \mathscr{E}xt(\D_X(X))\rightarrow \mathscr{E}xt(\wideparen{\D}_X(X)).
\end{equation*}
\end{coro}
\begin{coro}\label{coro invartiants and analytification}
Let $X$ be a smooth affine $K$-variety with an étale map $X\rightarrow\mathbb{A}^r_K$.\newline The following identities hold:
\begin{enumerate}[label=(\roman*)]
    \item $\operatorname{Z}(\D_X(X))=\operatorname{Z}(\wideparen{\D}_{X^{\operatorname{an}}}(X^{\operatorname{an}}))$.
    \item  $\Outder(\D_X(X))=\Outder(\wideparen{\D}_{X^{\operatorname{an}}}(X^{\operatorname{an}}))$.
    \item $\mathscr{E}xt(\D_X(X))=\mathscr{E}xt(\wideparen{\D}_{X^{\operatorname{an}}}(X^{\operatorname{an}}))$.
\end{enumerate}
Furthermore, the space of bounded derivations $\Der_K(\wideparen{\D}_{X^{\operatorname{an}}}(X^{\operatorname{an}}))$ is a nuclear Fréchet space.
\end{coro}
\begin{proof}
This is a consequence of Corollary \ref{coro hochschild cohomology of analytification of smooth affine varieties} and the results of this section. 
\end{proof}
Thus, we have completely clarified the structure of infinitesimal deformations of $\wideparen{\D}_X(X)$, at least  in the case where $\operatorname{H}_{\operatorname{dR}}^{2}(X)$ is finite dimensional. For the general case, the arguments of Theorem \ref{teo comparison map on degree 2} do not hold, and a complete treatment is out of the scope of this paper. However, let us mention a possible approach. First, notice that  by the proof of \cite[Corollary 3.2]{grosse2004rham}, every smooth Stein space $X$ admits an admissible cover $\{U_i\}_{i\geq 0}$ satisfying that $U_i\subset U_{i+1}$ for all $i\geq 0$, and such that each $U_i$ is a smooth Stein space with finite-dimensional de Rham cohomology. Second, the construction of the comparison map $c_2$ from Corollary \ref{coro existence of comparison map} is functorial, and does not require that $\operatorname{H}_{\operatorname{dR}}^{2}(X)$ is finite-dimensional. Hence, a possible approach to the general case could arise from establishing a theory of infinitesimal deformations at the level of sheaves, and using that de Rham cohomology parametrizes such deformations in a cover. This type of study has been carried out in the setting of smooth algebraic $K$-varieties by A. Vitanov in \cite{vitanovdeformation}.\\

As we advertised at the beginning of the section, we will now use our understanding of the infinitesimal deformations of $\wideparen{\D}_X(X)$ to construct a $K$-linear comparison map:
\begin{equation*}
    c_1:\operatorname{H}^{1}_{\operatorname{dR}}(X)^b\xrightarrow[]{}\operatorname{HH}_{\operatorname{T}}^1(\wideparen{\D}_X(X)).
\end{equation*}
Let us start with the following proposition:
\begin{prop}\label{prop explicit derivations}
Every $\lambda \in \Omega_{X/K}^{1,cl}(X)$ induces a  unique bounded $K$-linear derivation:
\begin{equation*}
    c_1(\lambda):\wideparen{\D}_X(X)\rightarrow \wideparen{\D}_X(X),
\end{equation*}
defined for all $f\in \OX_X(X)$, and all $v\in\mathcal{T}_{X/K}(X)$ by the following formulas:
\begin{equation*}
    c_1(\lambda)(f)=0,\, c_1(\lambda)(v)=\lambda(v).
\end{equation*}
\end{prop}
\begin{proof}
Assume that the formulas given above make $c_1(\lambda)$ a $K$-linear derivation. Then  $c_1(\lambda)$ is bounded. As $\D_X(X)$ is dense in $\wideparen{\D}_X(X)$, 
it follows that $c_1(\lambda)$ is completely determined by its values on $\D_X(X)$. However, $\D_X(X)$ is generated as an algebra by $\OX_X(X)$ and $\mathcal{T}_{X/K}(X)$, so the formulas above uniquely determine $c_1(\lambda)$. Further, by definition of a derivation, for $x,y\in\D_X(X)$ we have:
\begin{equation*}
    c_1(\lambda)(xy)=xc_1(\lambda)(y)+c_1(\lambda)(x)y.
\end{equation*}
Thus, $c_1(\lambda)$ restricts to a bounded $K$-linear derivation:
\begin{equation*}
    c_1(\lambda):\D_X(X)\rightarrow \D_X(X).
\end{equation*}
Thus, we only need to show that the formulas above make $c_1(\lambda)$ a $K$-linear derivation on $\D_X(X)$. By Corollary \ref{coro infinitesimal deformations}, it suffices to construct an infinitesimal automorphism:
\begin{equation*}
   \varphi:\D_X(X)^{\epsilon}\rightarrow \D_X(X)^{\epsilon},
\end{equation*}
satisfying that for $f\in \OX_X(X)$, and $v\in\mathcal{T}_{X/K}(X)$ we have:
\begin{equation*}
    \varphi(f)=f,\, \varphi(\mathbb{L}_v)=\mathbb{L}_v+t\lambda(v).
\end{equation*}
The advantage of this formulation is that we have a presentation of $\D_X(X)^{\epsilon}$ in terms of generators and relations, which makes it easier to define algebra homomorphisms. Indeed, as discussed above,  
$\D_X(X)^{\epsilon}$ is the quotient of $\operatorname{Sym}_{K^{\epsilon}}(\mathcal{A}_{t0}(X))$
by the relations in equation (\ref{equation conditions defining deformation}).\newline
Consider the following $K^{\epsilon}$-linear  automorphism of $\mathcal{A}_{t0}(X)$:
\begin{equation*}
    \varphi:\mathcal{A}_{t0}(X)\rightarrow \mathcal{A}_{t0}(X), 
\end{equation*}
given by $\varphi(f)=f$, and $\varphi(\mathbb{L}_v)=\mathbb{L}_v+t\lambda(v)$ for $f\in \OX_X(X)$, and $v\in\mathcal{T}_{X/K}(X)$. By the properties of $\operatorname{Sym}_{K^{\epsilon}}(-)$, this extends to an automorphism of $K^{\epsilon}$-algebras $\varphi:\operatorname{Sym}_{K^{\epsilon}}(\mathcal{A}_{t0}(X))\rightarrow \operatorname{Sym}_{K^{\epsilon}}(\mathcal{A}_{t0}(X))$, so we may consider the composition:
\begin{equation*}
    \varphi:(\mathcal{A}_{t0}(X))\rightarrow \operatorname{Sym}_{K^{\epsilon}}(\mathcal{A}_{t0}(X))\rightarrow \D_X(X)^{\epsilon}.
\end{equation*}
Hence, $\varphi$ induces an infinitesimal automorphism of $\D_X(X)$ if and only if $\varphi$ preserves the relations in equation (\ref{equation conditions defining deformation}). This follows by straightforward calculations, so we only show this for the more complicated one. Indeed, we have the following identities in $\D_X(X)^{\epsilon}$:
\begin{multline*}
\varphi(\mathbb{L}_{[v,w]})-\varphi([\mathbb{L}_v,\mathbb{L}_w])=
\mathbb{L}_{[v,w]}+t\lambda([v,w])-[\mathbb{L}_v+t\lambda(v),\mathbb{L}_w+t\lambda(w)]\\
=t\left( \lambda([v,w])-v(\lambda(w))+w(\lambda(v))\right)=0,
\end{multline*}
where the last identity follows because $\lambda$ is closed. Hence, $\varphi:\D_X(X)^{\epsilon}\rightarrow \D_X(X)^{\epsilon}$ is an infinitesimal automorphism. Again, in virtue of Corollary \ref{coro infinitesimal deformations}, it follows that for $x,y\in \D_X(X)$ we have $\varphi(x+ty)=x+t(c_1(\lambda)(x)+y)$, where $c_1(\lambda):\D_X(X)\rightarrow \D_X(X)$ is a bounded $K$-linear derivation, and by construction we have $c_1(\lambda)(f)=0$ for $f\in \OX_X(X)$, and $c_1(\lambda)(v)=\lambda(v)$ for $v\in \mathcal{T}_{X/K}(X)$.
\end{proof}
Hence, we obtain a $K$-linear monomorphism:
\begin{equation*}
    \Omega_{X/K}^{1,cl}(X)\xrightarrow[]{c_1}\Der_K^b(\wideparen{\D}_X(X)).
\end{equation*}
Furthermore, if $\lambda=df$ for some $f\in \OX_X(X)$, then we have the following identity :
\begin{equation*}
    c_1(\lambda)(v)=\lambda(v)=df(v):=[v,f],
\end{equation*}
where $v\in \mathcal{T}_{X/K}(X)$. Thus, $c_1(\lambda)$ agrees with the inner derivation induced by $f\in \OX_X(X)$. Therefore, we have shown that we have the following commutative diagram of strict short exact sequences:
\begin{equation*}
    \begin{tikzcd}
0 \arrow[r] & \Omega_{X/K}^{1,ext}(X) \arrow[r] \arrow[d,"c_1"]                      & \Omega_{X/K}^{1,cl}(X) \arrow[r] \arrow[d, "c_1"]           & \operatorname{H}^1_{\operatorname{dR}}(X)^b \arrow[r] \arrow[d, "c_1"]      & 0 \\
0 \arrow[r] & \Inn(\wideparen{\D}_X(X)) \arrow[r] & \operatorname{\Der}_K^b(\wideparen{\D}_X(X)) \arrow[r] & \operatorname{HH}^1_T(\wideparen{\D}_X(X)) \arrow[r] & 0
\end{tikzcd}
\end{equation*}
where the first two vertical maps are monomorphisms. We may use this to show the following:
\begin{prop}\label{prop comparison map dimension 1}
There is a $K$-linear monomorphism:
\begin{equation*}
    c_1:\operatorname{H}_{\operatorname{dR}}^{1}(X)^b\rightarrow \operatorname{HH}^1_{\operatorname{T}}(\wideparen{\D}_X(X)), \, [\lambda]\mapsto [c_1(\lambda)].
\end{equation*}    
\end{prop}
\begin{proof}
In virtue of the previous commutative diagram, if suffices to show that given a closed $1$-form $\lambda$ that is not exact, the associated derivation $c_1(\lambda)$ is not an inner derivation. Assume that $c_1(\lambda)$ is inner. In this situation, there is an operator $x\in\wideparen{\D}_X(X)$ satisfying the following identity:
\begin{equation*}
    c_1(\lambda)=[x,-].
\end{equation*}
However, notice that, by construction, the derivation $c_1(\lambda)$ is $\OX_X(X)$-linear. Hence, the operator $x$ must commute with every element in $\OX_X(X)$. However, as shown in \cite[Proposition 7.5.2]{HochDmod}, this implies that $x\in\OX_X(X)$. This, in turn, implies that $\lambda$ is an exact form, a contradiction.
\end{proof}
As in the case of infinitesimal deformations, this result is most potent when applied to the case where $\operatorname{H}_{\operatorname{dR}}^{1}(X)^b$ is a finite-dimensional $K$-vector space. In this situation, we can show the following:
\begin{coro}\label{coro c1 iso fd case}
Assume $\operatorname{H}_{\operatorname{dR}}^{1}(X)^b$ is a finite-dimensional $K$-vector space. Then the map:
\begin{equation*}
    c_1:\operatorname{H}_{\operatorname{dR}}^{1}(X)^b\rightarrow \operatorname{HH}^1_{\operatorname{T}}(\wideparen{\D}_X(X)),
\end{equation*}
is a $K$-linear isomorphism.
\end{coro}
\begin{proof}
By Corollary \ref{coro dr and outer}, and Proposition \ref{prop comparison map dimension 1}, $c_1$ is and injective $K$-linear map between two finite-dimensional $K$-vector spaces of the same dimension. Hence, it must be an isomorphism.
\end{proof}
Along these lines, we can also show the following:
\begin{coro}
Assume $\operatorname{H}_{\operatorname{dR}}^{1}(X)^b$ is finite-dimensional.  Then there is an isomorphism:
\begin{equation*}
    \Outder(\D_X(X))\rightarrow \Outder(\wideparen{\D}_X(X)).
\end{equation*}
\end{coro}
\begin{proof}
By the construction above, the map $c_1:\Omega_{X/K}^{1,cl}(X)\rightarrow \operatorname{\Der}_K(\wideparen{\D}_X(X))$ factors as:
\begin{equation*}
    \Omega_{X/K}^{1,cl}(X)\rightarrow \operatorname{\Der}_K(\D_X(X)) \rightarrow \operatorname{\Der}_K(\wideparen{\D}_X(X)).
\end{equation*}
Thus, we get an induced map $\Outder(\D_X(X))\rightarrow \Outder(\wideparen{\D}_X(X))$, which must be a surjection because $c_1:\operatorname{H}^1_{\operatorname{dR}}(X)^b\rightarrow \Outder(\wideparen{\D}_X(X))$ is an isomorphism. We just need to show that this map is injective. As a consequence of the snake Lemma, it follows that $\Outder(\D_X(X))\rightarrow \Outder(\wideparen{\D}_X(X))$ is injective if every $f\in \wideparen{\D}_X(X)$ such that $fx-xf\in \D_X(X)$ for all $x\in \D_X(X)$ satisfies $f\in \D_X(X)$.\\

By assumption, there is an étale map $X\rightarrow \mathbb{A}^r_K$. Thus,  there are functions $x_1,\cdots,x_r\in \OX_X(X)$ with associated vector fields $\partial_{x_1},\cdots, \partial_{x_r}\in \mathcal{T}_{X/K}(X)$ such that:
\begin{equation*}
    \mathcal{T}_{X/K}(X)=\bigoplus_{i=1}^n\OX_X(X)\partial_{x_i}, \quad [\partial_{x_i},\partial_{x_j}]=0, \textnormal{ and } \partial_{x_i}(x_j)=\delta_{i,j}.
\end{equation*}
Notice that, under these conditions, we have the following identity for each $j\geq 0$, and each $1\leq i\leq n$:
\begin{equation}\label{equation trick with local coordinates}
    x_i\partial_{x_i}^{j+1}-\partial_{x_i}^{j+1}x_i=(j+1)\partial_{x_i}^j. 
\end{equation}
Choose $f\in\wideparen{\D}_X(X)$ satisfying the conditions above. Then $f$ has the following form:
\begin{equation*}
    f=\sum_{I\in \mathbb{N}^{r}}a_I\partial_{x_1}^{I(1)}\cdots \partial_{x_r}^{I(r)}.
\end{equation*}
Choose some $1\leq i \leq r$. By assumption, we have:
\begin{equation*}
   \partial_{x_i}f- f\partial_{x_i}=\sum_{I\in \mathbb{N}^{r}}\left(\partial_{x_i}a_I\partial_{x_1}^{I(1)}\cdots \partial_{x_r}^{I(r)}-a_I\partial_{x_1}^{I(1)}\cdots \partial_{x_r}^{I(r)}\partial_{x_i}\right)=\sum_{I\in \mathbb{N}^{r}}\partial_{x_i}(a_I)\partial_{x_1}^{I(1)}\cdots \partial_{x_r}^{I(r)}\partial_{x_i}.
\end{equation*}
As this must be contained in $\D_X(X)$, it follows that $\partial_{x_i}(a_I)=0$ for sufficiently big $I\in \mathbb{N}^r$. As we chose $i$ to be arbitrary, and $\partial_{x_1},\cdots,\partial_{x_r}$ are a basis of $\mathcal{T}_{X/K}(X)$, it follows that $a_I\in K$ for sufficiently big $I\in \mathbb{N}^r$. Again, choose some $1\leq i \leq r$. We may rewrite $f$ as follows:
\begin{equation}\label{equation decomposition}
    f=\sum_{n\geq 0}f_n\partial_{x_i}^n,
\end{equation}
where each of the $f_n$ does not contain any $\partial_{x_i}^n$. Furthermore, this decomposition is unique. Notice that by equation (\ref{equation trick with local coordinates}), we have $f_nx_i-x_if_n=0$ for all $n\geq 0$. Hence, we have:
\begin{equation*}
    x_if-fx_i=\sum_{n\geq 0}(n+1)f_{n+1}\partial_{x_i}^n.
\end{equation*}
Again, as this must be contained in $\D_X(X)$, it follows that $f_n=0$ for big enough $n\geq 0$. The fact that we choose an arbitrary $1\leq i \leq r$, and that the decomposition in (\ref{equation decomposition}) is unique implies that $a_I=0$ for sufficiently high $I$, as we wanted to show.
\end{proof}
Finally, let us mention that for $n\geq 3$ the situation is much more convoluted. For instance, the edge morphisms:
\begin{equation*}
    I(\operatorname{HH}^{n}(\wideparen{\D}_X)) \rightarrow \operatorname{H}^n(\mathcal{L}_I(\wideparen{\D}_X(X))^{\bullet}),
\end{equation*}
are not necessarily epimorphisms. Thus, even if all the de Rham cohomology groups are finite-dimensional, it is not clear if the maps $\delta_n$ are strict for $n\geq 3$. 
\subsection{\texorpdfstring{The TH homology groups of $\wideparen{\D}_X(X)$}{}}\label{section TH homology}
Keeping the assumptions of the previous section, the next step will be carrying out the previous discussion for the TH homology groups of $\wideparen{\D}_X(X)$. First, recall that, in virtue of Proposition \ref{prop relation between TH homology and HH}, the fact that $\wideparen{\D}_X(X)$ is a Fréchet space implies that we have the following identity in $\operatorname{D}(\widehat{\mathcal{B}}c_K)$:
\begin{equation*}
    \operatorname{HH}_{\bullet}(\wideparen{\D}_X(X))=\mathcal{T}(\wideparen{\D}_X(X))^{\bullet}.
\end{equation*}
Therefore, we may use the results of Section \ref{section duality results} to show the following proposition:
\begin{prop}\label{prop Th cohomology groups of DX}
The Taylor-Hochschild homology complex:
\begin{equation*}
    \mathcal{T}(\wideparen{\D}_X(X))^{\bullet}:=\left(   
 \cdots \xrightarrow[]{\delta^{-2}}\widehat{\otimes}_K^2\wideparen{\D}_X(X)\widehat{\otimes}_K\wideparen{\D}_X(X)\xrightarrow[]{\delta^{-1}}\wideparen{\D}_X(X)\widehat{\otimes}_K\wideparen{\D}_X(X)  \xrightarrow[]{\delta^0}\wideparen{\D}_X(X)   \right),
\end{equation*}
is a strict complex of nuclear Fréchet spaces. In particular, for each $n\geq 0$ there is an isomorphism:
\begin{equation*}
    \operatorname{HH}_{n}(\wideparen{\D}_X(X))=\operatorname{H}_{\operatorname{dR}}^{2\operatorname{dim}(X)-n}(X)^b=\operatorname{HH}_n^{\operatorname{T}}(\wideparen{\D}_X(X)),
\end{equation*}
and all the spaces are also nuclear Fréchet spaces.
\end{prop}
\begin{proof}
The fact that $\mathcal{T}(\wideparen{\D}_X(X))^{\bullet}$ is a complex of nuclear Fréchet spaces follows by the fact that $\wideparen{\D}_X(X)$ is a nuclear Fréchet space, together with the fact that being nuclear and Fréchet is preserved by $\widehat{\otimes}_K$ (\emph{cf.} \cite[Proposition 19.11]{schneider2013nonarchimedean}). As mentioned above, we have an identification:
\begin{equation*}
    \operatorname{HH}_{\bullet}(\wideparen{\D}_X(X))=\mathcal{T}(\wideparen{\D}_X(X))^{\bullet}
\end{equation*}
Thus, we may use this, together with Theorem \ref{teo hochschild homology in terms of Tor} to obtain the following identities in $\operatorname{D}(\widehat{\mathcal{B}}c_K)$:
\begin{equation*}
 \mathcal{T}(\wideparen{\D}_X(X))^{\bullet}=\operatorname{HH}_{\bullet}(\wideparen{\D}_X(X))=\operatorname{HH}_{\bullet}(\wideparen{\D}_X)=\operatorname{H}^{\bullet}_{\operatorname{dR}}(X)^b[2\operatorname{dim}(X)].  
\end{equation*}
As the rightmost complex is a strict complex, it follows that $\mathcal{T}(\wideparen{\D}_X(X))^{\bullet}$ is also strict. 
\end{proof}
Thus, as advertised above, the relation between the TH homology groups and the Hochschild homology groups of $\wideparen{\D}_X(X)$ is much simpler than in the cohomological case. Namely, they always agree. Even if the spaces $\operatorname{HH}_{\bullet}(\wideparen{\D}_X(X))$ do not have a direct interpretation in terms of deformation theory, the fact that they can be calculated using the complex $\mathcal{T}(\wideparen{\D}_X(X))^{\bullet}$ makes them easier to manipulate. In particular, many of the classical operations on Hochschild (co)-homology of associative algebras are defined in terms of explicit calculations using the complexes $\mathcal{L}(\wideparen{\D}_X(X))^{\bullet}$ and $\mathcal{T}(\wideparen{\D}_X(X))^{\bullet}$. For instance,
it is possible to define a cup product and a bilinear bracket on 
$\mathcal{L}(\wideparen{\D}_X(X))^{\bullet}$ endowing $\operatorname{HH}_{\operatorname{T}}^{\bullet}(\wideparen{\D}_X(X))$ with the structure of a Gerstenhaber algebra (see the monograph \cite[Chapter 1]{witherspoon2019hochschild} for the definitions). The fact that the complex $\mathcal{L}(\wideparen{\D}_X(X))^{\bullet}$ does not calculate $\operatorname{HH}^{\bullet}(\wideparen{\D}_X(X))$ implies that many of these constructions are not directly available for Hochschild cohomology. However, they are available in the homological setting and, in some situations, it is possible to use Van den Bergh duality to bypass the shortcomings of the TH cohomology complex $\mathcal{L}(\wideparen{\D}_X(X))^{\bullet}$.
\subsection{Some operations on Hochschild cohomology}\label{section operations}
As an instance of this, we remark that it is possible to define a bornological version of the Connes differential on Hochschild Homology:
\begin{equation*}
    \mathcal{B}:\operatorname{HH}_{\bullet}(\wideparen{\D}_X(X))\rightarrow \operatorname{HH}_{\bullet}(\wideparen{\D}_X(X))[-1].
\end{equation*}
Indeed, let $\mathscr{A}$ be a complete bornological algebra, and let $\mathcal{T}(\mathscr{A})^{\bullet}$ be its TH homology complex. The Connes differential will ultimately arise from a morphism of chain complexes of bornological spaces:
\begin{equation*}
    \mathcal{B}:\mathcal{T}(\mathscr{A})^{\bullet}\rightarrow \mathcal{T}(\mathscr{A})^{\bullet}[-1].
\end{equation*}
The operator $\mathcal{B}$  is defined via a composition of several different maps. First, for each $n\geq 0$ we define the following morphism of bornological spaces:
\begin{equation*}
    s_n:\mathcal{T}(\mathscr{A})^{-n}\rightarrow \mathcal{T}(\mathscr{A})^{-n-1}, \quad s_n(a_0\otimes \cdots \otimes a_n)\mapsto  1\otimes a_0\otimes \cdots \otimes a_n,
\end{equation*}
where $a_0,\cdots, a_n\in \mathscr{A}$. Next, notice that the group $\mathbb{Z}/(n+1)$ acts on $\mathcal{T}(\mathscr{A})^{-n}$ by signed cyclic permutation of the tensor factors. Namely, the action is induced by the following map:
\begin{equation*}
    t_n:\mathcal{T}(\mathscr{A})^{-n}\rightarrow \mathcal{T}(\mathscr{A})^{-n}, \quad t_n(a_0\otimes \cdots \otimes a_n)\mapsto  (-1)^n a_n\otimes a_0\otimes \cdots \otimes a_{n-1},
\end{equation*}
which is often called the cyclic operator on $\mathcal{T}(\mathscr{A})^{-n}$. We may use these actions to define a map:
\begin{equation*}
N_n:=1+t_n+\cdots+t_n^n:\mathcal{T}(\mathscr{A})^{-n}\rightarrow \mathcal{T}(\mathscr{A})^{-n},
\end{equation*}
which is called the norm operator on $\mathcal{T}(\mathscr{A})^{-n}$. Finally, we can condense all these maps into the definition of the Connes differential:
\begin{Lemma}
For each $n\geq 0$  consider the following bounded map:
\begin{equation*}
    \mathcal{B}_n:=(1-t_{n+1})s_nN_n:\mathcal{T}(\mathscr{A})^{-n}\rightarrow \mathcal{T}(\mathscr{A})^{-n-1}.
\end{equation*}
This map is given on simple tensors $a_0\otimes \cdots \otimes a_n\in \mathcal{T}(\mathscr{A})^{-n}$ by the following formula:
\begin{multline*}
    \mathcal{B}(a_0\otimes \cdots \otimes a_n)= \sum_{i=0}^n(-1)^{ni} 1\otimes a_i\otimes \cdots a_n\otimes a_{0}\otimes \cdots \otimes a_{i-1}\\
    - (-1)^{ni}1\otimes a_i\otimes 1\otimes a_{i+1}\otimes \cdots \otimes a_n\otimes a_{0}\otimes \cdots \otimes a_{i-1}.
\end{multline*}
Furthermore,  these maps induce a morphism of chain complexes of complete bornological spaces:
\begin{equation*}
    \mathcal{B}:\mathcal{T}(\mathscr{A})^{\bullet}\rightarrow \mathcal{T}(\mathscr{A})^{\bullet}[-1],
\end{equation*}
which we call the Connes differential.
\end{Lemma}
\begin{proof}
This is shown in \cite[Section 2.1.7]{lodaycyclic}.
\end{proof}
As $\mathcal{B}$ is a morphism of chain complexes, it induces a map on the Taylor-Hochschild homology:
\begin{defi}
We define the Connes differential on the \emph{TH} homology to be the map:
    \begin{equation*}
    \mathcal{B}:\operatorname{HH}^{\operatorname{T}}_{\bullet}(\mathscr{A})\rightarrow \operatorname{HH}^{\operatorname{T}}_{\bullet}(\mathscr{A})[-1].
\end{equation*}
\end{defi}
Back to our setting, we again let $X$ be a Stein space with free tangent sheaf, and let $\wideparen{\D}_X(X)$ be its algebra of $p$-adic differential operators. In this situation, it follows by Proposition \ref{prop Th cohomology groups of DX} that there is an isomorphism of complexes of nuclear Fréchet spaces $\operatorname{HH}_{\bullet}(\wideparen{\D}_X(X))\cong \operatorname{HH}^{\operatorname{T}}_{\bullet}(\wideparen{\D}_X(X))$. Hence, we obtain a version of the Connes differential for Hochschild homology:
\begin{equation*}
    \mathcal{B}:\operatorname{HH}_{\bullet}(\wideparen{\D}_X(X))\rightarrow \operatorname{HH}_{\bullet}(\wideparen{\D}_X(X))[-1].
\end{equation*}
\begin{obs}
 In fact, in virtue of Proposition \ref{prop relation between TH homology and HH} the Connes differential on Hochschild homology exists whenever $\mathscr{A}$ is a Fréchet algebra. However, for arbitrary $\mathscr{A}$, the \emph{TH} homology groups of $\mathscr{A}$ do not agree with its Hochschild homology groups.   
\end{obs}
Using the fact that $\wideparen{\D}_X(X)$ satisfies Van den Bergh duality, we have a quasi-isomorphism:
\begin{equation*}
 \operatorname{HH}^{\bullet}(\wideparen{\D}_X(X))\cong   \operatorname{HH}_{\bullet}(\wideparen{\D}_X(X))[2\operatorname{dim}(X)].
\end{equation*}
Thus, we obtain an operator on Hochschild cohomology:
\begin{equation*}
    \Delta:\operatorname{HH}^{\bullet}(\wideparen{\D}_X(X))\rightarrow \operatorname{HH}^{\bullet}(\wideparen{\D}_X(X))[-1],
\end{equation*}
defined as the unique morphism map making the following diagram commutative:
\begin{equation*}
\begin{tikzcd}
\operatorname{HH}^{\bullet}(\wideparen{\D}_X(X)) \arrow[r, "\Delta"] \arrow[d, "\cong"]                   & \operatorname{HH}^{\bullet}(\wideparen{\D}_X(X))[-1] \arrow[d, "\cong"]              \\
\operatorname{HH}_{\bullet}(\wideparen{\D}_X(X))[2\operatorname{dim}(X)] \arrow[r, "\mathcal{B}"] & \operatorname{HH}_{\bullet}(\wideparen{\D}_X(X))[2\operatorname{dim}(X)-1]
\end{tikzcd}
\end{equation*}
In particular, for each $n\geq 0$ we obtain a commutative diagram of complete bornological spaces:
\begin{equation*}
\begin{tikzcd}
\operatorname{HH}^{n}(\wideparen{\D}_X(X)) \arrow[r, "\Delta"] \arrow[d]                   & \operatorname{HH}^{n-1}(\wideparen{\D}_X(X)) \arrow[d]              \\
\operatorname{HH}_{2\operatorname{dim}(X)-n}(\wideparen{\D}_X(X)) \arrow[r, "\mathcal{B}"] & \operatorname{HH}_{2\operatorname{dim}(X)-n+1}(\wideparen{\D}_X(X))
\end{tikzcd}
\end{equation*}
In the setting of associative Calabi-Yau $K$-algebras (\emph{cf.} Definition \ref{defi Calabi-Yau algebra}), the operator $\Delta$ is called the Batalin-Vilkovisky operator on Hochschild cohomology, and it is studied in \cite[Section 9.3]{ginzcalabi-yau}. We remark that, if $A$ is a Calabi-Yau $K$-algebra, then the 
the Gerstenhaber bracket, the cup product, and the Batalin-Vilkovisky operator on $\operatorname{HH}^{\bullet}(A)$ are 
related via the following formula:
\begin{equation}\label{equation BV operator}
[\alpha,\beta]= \Delta (\alpha \smile \beta) -\Delta(\alpha)\smile \beta- (-1)^{\vert \alpha \vert}\alpha \smile \Delta(\beta),
\end{equation}
where $\alpha$ and $\beta$ are homogeneous elements in $\operatorname{HH}^{\bullet}(A)$. The Gerstenhaber bracket on $\operatorname{HH}^{\bullet}(A)$ is defined via an explicit formula at the level of chain complexes, so it is not possible to define it directly for complete bornological algebras. However, for any associative $K$-algebra, the cup product on $\operatorname{HH}^{\bullet}(A):=R\Hom_{A^e}(A,A)$ is the Yoneda product, and this can be defined at the bornological level.\\

Indeed, the Yoneda product can be extended naturally to any closed symmetric monoidal category (see \cite[Section 3.3]{bode2021operations} for the definitions).  Let us now give a description of this construction:
\begin{prop}\label{prop endomorphism algebra}
Let $\mathcal{C}$ be a closed symmetric monoidal category with tensor $\otimes$ and inner homomorphisms functor $H$.  Let $\mathscr{A}$ be a monoid in $\mathcal{C}$, and let $H_{\mathscr{A}}$ be the inner homomorphism functor relative to $\mathscr{A}$. Then for $M_1,M_2,M_3\in \Mod_{\mathcal{C}}(\mathscr{A})$ there is a canonical natural map:
\begin{equation*}
  H_{\mathscr{A}}(M_1,M_2)\otimes H_{\mathscr{A}}(M_2,M_3)\rightarrow H_{\mathscr{A}}(M_1,M_3), 
\end{equation*}
which we call the composition map. Furthermore, for any $M\in \Mod_{\mathcal{C}}(\mathscr{A})$, the object $H_{\mathscr{A}}(M,M)$ is also a monoid in $\mathcal{C}$, which we call the endomorphism algebra of $M$ (relative to $\mathscr{A}$).
\end{prop}
\begin{proof}
By the usual tensor-Hom adjunction, it is equivalent to define a map in $\mathcal{C}$:
\begin{equation*}
    H_{\mathscr{A}}(M_1,M_2)\rightarrow H(H_{\mathscr{A}}(M_2,M_3),H_{\mathscr{A}}(M_1,M_3))\cong H_{\mathscr{A}}(M_1\otimes H_{\mathscr{A}}(M_2,M_3),M_3).
\end{equation*}
Applying the same adjunction again, this is equivalent to giving an $\mathscr{A}$-linear map:
\begin{equation*}
    M_1\otimes H_{\mathscr{A}}(M_1,M_2)  \otimes H_{\mathscr{A}}(M_2,M_3)\rightarrow M_3.
\end{equation*}
The counit of the tensor-Hom adjunction is the evaluation map, which yields two $\mathscr{A}$-linear maps:
\begin{equation*}
    M_1\otimes H_{\mathscr{A}}(M_1,M_2)\rightarrow M_2, \quad M_2\otimes H_{\mathscr{A}}(M_2,M_3)\rightarrow M_3.
\end{equation*}
Combining these two maps, we obtain a composition:
\begin{equation*}
    M_1\otimes H_{\mathscr{A}}(M_1,M_2)  \otimes H_{\mathscr{A}}(M_2,M_3)\rightarrow M_2\otimes H_{\mathscr{A}}(M_2,M_3)\rightarrow M_3,
\end{equation*}
which is the map we wanted. Fix now some object $M\in \Mod_{\mathcal{C}}(\mathscr{A})$. We need to show that $H_{\mathscr{A}}(M,M)$ is a monoid in $\mathcal{C}$. First, we need to show that the obvious diagram:
\begin{equation*}
\begin{tikzcd}
{H_{\mathscr{A}}(M,M)\otimes H_{\mathscr{A}}(M,M)\otimes H_{\mathscr{A}}(M,M)} \arrow[d] \arrow[r] & {H_{\mathscr{A}}(M,M)\otimes H_{\mathscr{A}}(M,M)} \arrow[d] \\
{H_{\mathscr{A}}(M,M)\otimes H_{\mathscr{A}}(M,M)} \arrow[r]                                       & {H_{\mathscr{A}}(M,M)}                                      
\end{tikzcd}
\end{equation*}
is commutative. This is a straightforward calculation using the same arguments as above. Next, let $1_{\mathcal{C}}$ be the unit of the closed symmetric monoidal structure in $\mathcal{C}$. We need to show that there is an identity map $1_{\mathcal{C}}\rightarrow H_{\mathscr{A}}(M,M)$ satisfying that the composition:
\begin{equation}\label{equation unit morphism}
    1_{\mathcal{C}}\otimes H_{\mathscr{A}}(M,M) \rightarrow H_{\mathscr{A}}(M,M)\otimes H_{\mathscr{A}}(M,M)\rightarrow H_{\mathscr{A}}(M,M),
\end{equation}
is the identity map. We define the unit morphism as the unique map 
$1_{\mathcal{C}}\rightarrow H_{\mathscr{A}}(M,M)$ induced by the identity map $\operatorname{Id}:M\rightarrow M$ and the following isomorphisms:
\begin{equation*}
\Hom_{\mathcal{C}}(1_{\mathcal{C}},H_{\mathscr{A}}(M,M))=\Hom_{\Mod_{\mathcal{C}}(\mathscr{A})}(M\otimes 1_{\mathcal{C}}, M)=\Hom_{\Mod_{\mathcal{C}}(\mathscr{A})}(M, M).
\end{equation*}
With this definition, the fact that the composition in (\ref{equation unit morphism}) is the identity in $H_{\mathscr{A}}(M,M)$ follows by the arguments above, together with the fact that:
\begin{equation*}
 M\otimes 1_{\mathcal{C}}\rightarrow M\otimes H_{\mathscr{A}}(M,M)\rightarrow M,    
\end{equation*}
is the identity map of $M$.
\end{proof}
We may now specialize this construction to our setting to obtain the following corollary:
\begin{coro}
 Let $\mathscr{A}$ be a Fréchet algebra. There is a well defined cup product:
 \begin{equation*}
     \smile: \operatorname{HH}^{\bullet}(\mathscr{A})\widehat{\otimes}_K^{\mathbb{L}}\operatorname{HH}^{\bullet}(\mathscr{A})\rightarrow \operatorname{HH}^{\bullet}(\mathscr{A}),
 \end{equation*}
 making $\operatorname{HH}^{\bullet}(\mathscr{A})$ a monoid in $\operatorname{D}(\widehat{\mathcal{B}}c_K)$. Furthermore, if $X$ is a Stein space with free tangent sheaf, the cup product on $\operatorname{HH}^{\bullet}(\wideparen{\D}_X(X))$ descends to a cup product of graded complete bornological spaces:
 \begin{equation*}
     \smile: \bigoplus_{n\geq 0}\operatorname{HH}^{n}(\wideparen{\D}_X(X))\widehat{\otimes}_K\bigoplus_{n\geq 0}\operatorname{HH}^{n}(\wideparen{\D}_X(X))\rightarrow \bigoplus_{n\geq 0}\operatorname{HH}^{n}(\wideparen{\D}_X(X)).
 \end{equation*}
 In particular $\bigoplus_{n\geq 0}\operatorname{HH}^{n}(\wideparen{\D}_X(X))$ is a graded complete bornological algebra.
\end{coro}
\begin{proof}
First, recall that by  \cite[Theorem 14.4.8]{Kashiwara2006} and \cite[Corollary 4.17]{bode2021operations} the closed symmetric monoidal structure in $\widehat{\mathcal{B}}c_K$ given by $-\widehat{\otimes}_K-$ and $\underline{\Hom}_{\widehat{\mathcal{B}}c_K}(-,-)$ extends to a closed symmetric monoidal structure on $\operatorname{D}(\widehat{\mathcal{B}}c_K)$ given by $-\widehat{\otimes}^{\mathbb{L}}_K-$ and $R\underline{\Hom}_{\widehat{\mathcal{B}}c_K}(-,-)$. Next, as $\mathscr{A}$ is a Fréchet algebra, it follows that $\mathscr{A}^e$ is also a Fréchet algebra. In particular, $I(\mathscr{A}^e)=I(\mathscr{A})^e$ is a monoid in $LH(\widehat{\mathcal{B}}c_K)$  (\emph{cf.} Proposition \ref{prop comparison of tensor products metrizable spaces}). Furthermore, $I(\mathscr{A})^e$ is flat with respect to $\widetilde{\otimes}_K$. Hence, we have:
\begin{equation*}
I(\mathscr{A})^e\widetilde{\otimes}_K^{\mathbb{L}}I(\mathscr{A})^e=I(\mathscr{A})^e\widetilde{\otimes}_KI(\mathscr{A})^e,
\end{equation*}
and therefore $\mathscr{A}^e$ is a monoid in $\operatorname{D}(\widehat{\mathcal{B}}c_K)$. Similarly, $\mathscr{A}$ is a module over $\mathscr{A}^e$ in $\operatorname{D}(\widehat{\mathcal{B}}c_K)$. Thus, it follows by Proposition \ref{prop endomorphism algebra} that $\operatorname{HH}^{\bullet}(\mathscr{A}):=R\underline{\Hom}_{\mathscr{A}^e}(\mathscr{A},\mathscr{A})$ is a monoid in $\operatorname{D}(\widehat{\mathcal{B}}c_K)$.\\
Let now $X$ be a Stein space with free tangent sheaf. Then $\operatorname{HH}^{\bullet}(\wideparen{\D}_X(X))$ is a strict complex of nuclear Fréchet spaces. In particular, it is a complex of flat modules with flat cohomology groups. Hence, it follows that for each $n\geq 0$ we have:
\begin{equation*}
    \operatorname{H}^n\left(\operatorname{HH}^{\bullet}(\wideparen{\D}_X(X))\widehat{\otimes}_K^{\mathbb{L}}\operatorname{HH}^{\bullet}(\wideparen{\D}_X(X)) \right)=\bigoplus_{r+s=n}\operatorname{HH}^{r}(\wideparen{\D}_X(X))\widehat{\otimes}_K\operatorname{HH}^{s}(\wideparen{\D}_X(X)).
\end{equation*}
Hence, the product on $\operatorname{HH}^{\bullet}(\wideparen{\D}_X(X))$ descends to a morphism of graded complete bornological spaces:
\begin{equation*}
    \smile: \bigoplus_{n\geq 0}\operatorname{HH}^{n}(\wideparen{\D}_X(X))\widehat{\otimes}_K\bigoplus_{n\geq 0}\operatorname{HH}^{n}(\wideparen{\D}_X(X))\rightarrow \bigoplus_{n\geq 0}\operatorname{HH}^{n}(\wideparen{\D}_X(X)).
\end{equation*}
The fact that $\operatorname{HH}^{\bullet}(\wideparen{\D}_X(X))$ is a monoid  and the unit of the closed symmetric monoidal structure on $\operatorname{D}(\widehat{\mathcal{B}}c_K)$ is $K$
imply that $\bigoplus_{n\geq 0}\operatorname{HH}^{n}(\wideparen{\D}_X(X))$ is in fact a graded complete bornological algebra.
\end{proof}
Hence, we have shown that for any $n,m\geq 0$ there is a bounded cup product:
\begin{equation*}
    \smile:\operatorname{HH}^{n}(\wideparen{\D}_X(X))\widehat{\otimes}_K\operatorname{HH}^{m}(\wideparen{\D}_X(X))\rightarrow \operatorname{HH}^{n+m}(\wideparen{\D}_X(X)),
\end{equation*}
and a bounded  Batalin-Vilkovisky operator:
\begin{equation*}
    \Delta:\operatorname{HH}^{n}(\wideparen{\D}_X(X))\rightarrow \operatorname{HH}^{n-1}(\wideparen{\D}_X(X)).
\end{equation*}
Using these two operators, together with equation (\ref{equation BV operator}), we can define a bounded Gerstenhaber bracket on the graded complete bornological space $\bigoplus_{n\geq 0}\operatorname{HH}^{n}(\wideparen{\D}_X(X))$ via the formula:
\begin{equation*}
[\alpha,\beta]= \Delta (\alpha \smile \beta) -\Delta(\alpha)\smile \beta- (-1)^{\vert \alpha \vert}\alpha \smile \Delta(\beta),
\end{equation*}
for $\alpha$, and $\beta$ homogeneous elements in $\bigoplus_{n\geq 0}\operatorname{HH}^{n}(\wideparen{\D}_X(X))$. We will postpone studying the properties of this bracket to a later paper. However, the previous discussion should convince the reader that most of the operations from the classical theory of Hochschild cohomology of associative algebras can also be developed in the complete bornological setting. 
\subsection{Non-commutative differential forms I}\label{section ncdf1}
In order to generalize some of our previous results to the case where the de Rham cohomology is not finite, we will need to use a bornological version of the space of non-commutative differential forms. Our exposition is based on the constructions of  \cite{Cuntz-Quillen}, by J. Cuntz and D. Quillen. Until the end of the section, we let $K$ be a complete non-archimedean extension of $\mathbb{Q}_p$, $\mathscr{A}$ be a complete bornological algebra, and $\mathscr{A}^e$ be its enveloping algebra. 
\begin{defi}
The module of non-commutative differential forms of $\mathscr{A}$ is the following complete bornological $\mathscr{A}^e$-module:
\begin{equation*}
    \ncd(\mathscr{A}):= \operatorname{Ker}\left(\mathscr{A}^e\rightarrow \mathscr{A}  \right),
\end{equation*}
where $\mathscr{A}^e\rightarrow \mathscr{A}$ is the unique $\mathscr{A}^e$-linear bounded map defined on simple tensors by $x\otimes y\mapsto xy$.
\end{defi}
The idea is using $\ncd(\mathscr{A})$ to simplify some of the calculations of Hochschild cohomology. Our strategy is based upon the following lemma:
\begin{Lemma}\label{Lemma corepresentability of derivations}
Let $\mathscr{A}$ be a complete bornological algebra. Then the  functor:
\begin{equation*}
    \Der(\mathscr{A},-):\Mod_{\widehat{\mathcal{B}}c_K}(\mathscr{A}^e)\rightarrow  \widehat{\mathcal{B}}c_K,
\end{equation*}
is co-representable by $\ncd(\mathscr{A})$. In other words, there is a natural isomorphism of functors:
\begin{equation*}
    \underline{\Hom}_{\mathscr{A}^e}(\ncd(\mathscr{A}),-)\cong \operatorname{Der}(\mathscr{A},-).
\end{equation*}
\end{Lemma}
\begin{proof}
The proof is follows the lines of \cite[Proposition 3.1]{Cuntz-Quillen}. Let $\mathscr{M}$ be a complete bornological $\mathscr{A}^e$-module, and recall from Section \ref{taylor HA} the bar resolution of $\mathscr{A}$:
\begin{equation*}
 B(\mathscr{A})^{\bullet}:=  \left(\cdots\rightarrow \mathscr{A}^e\widehat{\otimes}_K\mathscr{A}\rightarrow\mathscr{A}^e\right).
\end{equation*}
By Proposition \ref{prop interpretation of Taylor cohomology groups for n=0,1} the space $\Der(\mathscr{A},\mathscr{M})$ is the kernel of the following map:
\begin{equation*}
    \delta_{1}^\mathscr{M}:\underline{\Hom}_{\mathscr{A}^e}(\mathscr{A}^e\widehat{\otimes}_K\mathscr{A},\mathscr{M})\rightarrow \underline{\Hom}_{\mathscr{A}^e}(\mathscr{A}^e\widehat{\otimes}_K\mathscr{A}\widehat{\otimes}_K\mathscr{A},\mathscr{M}).
\end{equation*}
Recall that the bar resolution $B(\mathscr{A})^{\bullet}$ is split as a complex of complete bornological $\mathscr{A}^e$-modules. In particular, it is strict exact. As $\underline{\Hom}_{\mathscr{A}^e}(-,\mathscr{M})$  is left exact, we get a chain of isomorphisms
\begin{equation*}
\Der(\mathscr{A},\mathscr{M})=\underline{\Hom}_{\mathscr{A}^e}(\operatorname{Coker}\left(\mathscr{A}^e\widehat{\otimes}_K\mathscr{A}\widehat{\otimes}_K\mathscr{A}\rightarrow
 \mathscr{A}^e\widehat{\otimes}_K\mathscr{A}\right),\mathscr{M})=\underline{\Hom}_{\mathscr{A}^e}(\ncd(\mathscr{A}),\mathscr{M}),
\end{equation*}
and this is precisely the identity we wanted to show.
\end{proof}
Thus, $\ncd(\mathscr{A})$ enables us to pass from derivations to homomorphisms, and the latter are much easier to study. Let us mention that for each $n\geq 2$ there are non-commutative analogs of the module of $n$-forms on a commutative algebra. These are defined via the following construction:
\begin{equation*}
    \Omega_{\operatorname{nc}}^n(\mathscr{A}):= \widehat{\otimes}_{\mathscr{A}}^n\ncd(\mathscr{A}),
\end{equation*}
where the $\mathscr{A}^e$-module structure on $\Omega_{\operatorname{nc}}^n(\mathscr{A})$ is given by the action of $\mathscr{A}$ on the leftmost factor, and the action of $\mathscr{A}^{\op}$ on the rightmost factor. As shown in \cite{Cuntz-Quillen},  in the case of associative $K$-algebras these modules are intimately related to the Hochschild cohomology of $\mathscr{A}$. In particular, it can be shown that $\Omega_{\operatorname{nc}}^n(\mathscr{A})$ co-represents 
the functor of $n$-reduced Hochschild cocycles, and it is possible to define a DGA structure on the direct sum:
\begin{equation*}
    \Omega_{\operatorname{nc}}^{\bullet}(\mathscr{A}):=\bigoplus_{n\geq 0}\Omega_{\operatorname{nc}}^n(\mathscr{A}).
\end{equation*}
 Although these constructions also make sense in the bornological setting, we will not pursue this study here. The reason behind this is that for $n\geq 2$ the module $\Omega_{\operatorname{nc}}^n(\mathscr{A})$ is too big, and thus the functor $\underline{\Hom}_{\mathscr{A}^e}(\Omega_{\operatorname{nc}}^n(\mathscr{A}),-)$ does not simplify matters.

\subsection{\texorpdfstring{Non-commutative differential forms II}{}}\label{section ncdf2}
As mentioned above, the goal is using the module of non-commutative differential forms to obtain more information on the Hochschild cohomology of the ring of complete $p$-adic differential operators on a smooth Stein space with free tangent sheaf. For future use, we will actually work in a slightly more general setting. Namely, until the end of the section, we fix a  (two-sided) Fréchet-Stein algebra $\mathscr{A}=\varprojlim_n \mathscr{A}_n$ satisfying the following properties:
\begin{enumerate}[label=(\roman*)]
    \item Every $\mathcal{M}\in \mathcal{C}(\mathscr{A})$ satisfies that the Fréchet-Stein presentation $\mathcal{M}=\varprojlim_n \mathcal{M}_n$ makes $\mathcal{M}$ a nuclear Fréchet space.
    \item Every $\mathcal{M}\in \mathcal{C}(\mathscr{A}^{\op})$ satisfies that the Fréchet-Stein presentation $\mathcal{M}=\varprojlim_n \mathcal{M}_n$ makes $\mathcal{M}$ a nuclear Fréchet space.
    \item The isomorphism of Fréchet algebras $\mathscr{A}^e=\varprojlim_n \mathscr{A}^e_n$ is a Fréchet-Stein presentation, and $\mathscr{A}^e$ also satisfies conditions $(i)$ and $(ii)$.
\end{enumerate}
Let us start the section with the following lemma:
\begin{Lemma}\label{lemma ncd is co-admissible}
The module of non-commutative differential forms $\ncd(\mathscr{A})$  is a co-admissible $\mathscr{A}^e$-module. In particular, there is a Fréchet-Stein presentation:
\begin{equation*}
   \ncd(\mathscr{A})=\varprojlim\ncd(\mathscr{A}_n),
\end{equation*}
and for each co-admissible $\mathscr{A}^e$-module $\mathcal{M}=\varprojlim_n \mathcal{M}_n$ there is a canonical isomorphism:
\begin{equation*}
    \Der(\mathscr{A},\mathcal{M})=\varprojlim \Der(\mathscr{A}_n,\mathcal{M}_n).
\end{equation*}
In particular, the space of bounded derivations $\Der(\mathscr{A},\mathcal{M})$ is a Fréchet space.
\end{Lemma}
\begin{proof}
As shown in \cite[Section 7]{HochDmod}, our assumptions on $\mathscr{A}$ imply that $\mathscr{A}$ is a co-admissible $\mathscr{A}^e$-module, and that for every $n\geq 0$ we have a canonical identification:
\begin{equation*}
    \mathscr{A}_n=\mathscr{A}^e_n\widehat{\otimes}_{\mathscr{A}^e}\mathscr{A}.
\end{equation*}
As the category of co-admissible modules of a Fréchet-Stein algebra is abelian, and the map $\mathscr{A}^e\rightarrow \mathscr{A}$ is $\mathscr{A}^e$-linear, this implies that its kernel $\ncd(\mathscr{A})$ is a co-admissible module. Furthermore, for every $n\geq 0$ we have a short exact sequence:
\begin{equation*}
    0\rightarrow   \mathscr{A}^e_n\widehat{\otimes}_{\mathscr{A}^e}\ncd(\mathscr{A})  \rightarrow \mathscr{A}^e_n\rightarrow \mathscr{A}_n\rightarrow 0,
\end{equation*}
from which it follows that $\mathscr{A}^e_n\widehat{\otimes}_{\mathscr{A}^e}\ncd(\mathscr{A})=\ncd(\mathscr{A}_n)$. Thus, we have $\ncd(\mathscr{A})=\varprojlim_n \ncd(\mathscr{A}_n)$, as we wanted to show. Let now $\mathcal{M}\in \mathcal{C}(\mathscr{A}^e)$ be a co-admissible module. Then by Lemma \ref{Lemma corepresentability of derivations} we have the following chain of identities of complete bornological modules:
\begin{equation*}
    \Der(\mathscr{A},\mathcal{M})=\underline{\Hom}_{\mathscr{A}^e}(\ncd(\mathscr{A}),\mathcal{M})=\varprojlim_n \underline{\Hom}_{\mathscr{A}_n^e}(\ncd(\mathscr{A}_n),\mathcal{M}_n)=\varprojlim_n \Der(\mathscr{A}_n,\mathcal{M}_n),
\end{equation*}
where the third identity follows by \cite[Lemma 7.2]{ardakov2019}. Thus, the lemma holds.
\end{proof}
With this lemma at hand, we can show the following proposition:
\begin{prop}\label{prop out as inverse limit}
Let $\mathcal{M}\in \mathcal{C}(\mathscr{A}^e)$ be a co-admissible module. The following hold:
\begin{enumerate}[label=(\roman*)]
    \item For each $n\geq 0$, we have $\operatorname{HH}^0(\mathscr{A}_n,\mathcal{M})=\operatorname{Z}(\mathscr{A}_n,\mathcal{M}_n)$, and this is a Banach space.
    \item There is an isomorphism of complete bornological spaces:
    \begin{equation*}   \operatorname{HH}^0(\mathscr{A},\mathcal{M})=\operatorname{Z}(\mathcal{M})=\varprojlim \operatorname{Z}(\mathscr{A}_n,\mathcal{M}_n)=\varprojlim \operatorname{HH}^0(\mathscr{A}_n,\mathcal{M}_n).
    \end{equation*}
    Furthermore, $\operatorname{HH}^0(\mathscr{A},\mathcal{M})$ is a nuclear Fréchet space.
    \item Assume that we have the following identity in $\operatorname{D}(\widehat{\mathcal{B}}c_K)$:
    \begin{equation*}
        R\varprojlim \operatorname{Z}(\mathscr{A}_n,\mathcal{M}_n)=\varprojlim \operatorname{Z}(\mathscr{A}_n,\mathcal{M}_n),
    \end{equation*}
    then there is an isomorphism of complete bornological spaces:
\begin{equation*}
    \mathcal{M}/\operatorname{Z}(\mathscr{A},\mathcal{M})=\varprojlim \left(\mathcal{M}_n/\operatorname{Z}(\mathscr{A}_n,\mathcal{M}_n)\right),
\end{equation*}
and this inverse limit makes $\mathcal{M}/\operatorname{Z}(\mathscr{A},\mathcal{M})$ a nuclear Fréchet space.
\end{enumerate}
Assume in addition that $\operatorname{HH}^1(\mathscr{A},\mathcal{M})$ is an object of $\widehat{\mathcal{B}}c_K$. Then the following hold:
\begin{enumerate}[resume,label=(\roman*)]
    \item For each $n\geq 0$, we have $\operatorname{HH}^1(\mathscr{A}_n,\mathcal{M}_n)=\Outder(\mathscr{A}_n,\mathcal{M}_n)$, and this is a Banach space.
    \item There is an isomorphism of complete bornological spaces:
    \begin{equation*}   \operatorname{HH}^1(\mathscr{A},\mathcal{M})=\Outder(\mathscr{A},\mathcal{M})=\varprojlim \Outder(\mathscr{A}_n,\mathcal{M}_n)=\varprojlim \operatorname{HH}^1(\mathscr{A}_n,\mathcal{M}_n),
    \end{equation*}
   and this inverse limit makes $\operatorname{HH}^1(\mathscr{A},\mathcal{M})$ a Fréchet space.
\end{enumerate}
\end{prop}
\begin{proof}
By definition, the Hochschild cohomology of $\mathcal{M}$   is given by the expression:
\begin{equation*}
    \operatorname{HH}^{\bullet}(\mathscr{A},\mathcal{M}):=R\underline{\Hom}_{\mathscr{A}^e}(\mathscr{A},\mathcal{M}).
\end{equation*}
In particular, we have the following chain of identities:
\begin{equation*}
    \operatorname{HH}^{0}(\mathscr{A},\mathcal{M})=\underline{\Hom}_{\mathscr{A}^e}(\mathscr{A},\mathcal{M})=\varprojlim \underline{\Hom}_{\mathscr{A}_n^e}(\mathscr{A}_n,\mathcal{M}_n)=\varprojlim \operatorname{HH}^{0}(\mathscr{A}_n,\mathcal{M}_n),
\end{equation*}
where, again, the second identity is a consequence of \cite[Lemma 7.2]{ardakov2019}. Furthermore, by 
Proposition \ref{prop interpretation of Taylor cohomology groups for n=0,1} we have $\underline{\Hom}_{\mathscr{A}^e}(\mathscr{A},-)=\operatorname{Z}(\mathscr{A},-)$. In particular, $\operatorname{HH}^0(\mathscr{A},\mathcal{M})\rightarrow \mathcal{M}$ is a closed immersion. By assumption, $\mathcal{M}$ is a nuclear Fréchet space. Hence, it follows by \cite[Proposition 19.4]{schneider2013nonarchimedean} that $\operatorname{HH}^0(\mathscr{A},\mathcal{M})$ is a nuclear Fréchet space as well. Thus, statements $(i)$ and $(ii)$ hold. For statement $(iii)$, notice that for any $n\geq 0$ we have the following strict short exact sequence of Banach spaces:
\begin{equation*}
    0 \rightarrow \operatorname{Z}(\mathscr{A}_n,\mathcal{M}_n)\rightarrow \mathcal{M}_n\rightarrow  \mathcal{M}_n/ \operatorname{Z}(\mathscr{A}_n,\mathcal{M}_n) \rightarrow 0.
\end{equation*}
By assumption, we have the following identity in $\operatorname{D}(\widehat{\mathcal{B}}c_K)$:
\begin{equation*}
    R\varprojlim \operatorname{Z}(\mathscr{A}_n,\mathcal{M}_n)=\varprojlim \operatorname{Z}(\mathscr{A}_n,\mathcal{M}_n).
\end{equation*}
In particular, there is a strict short exact sequence of nuclear Fréchet spaces:
\begin{equation*}
    0\rightarrow \operatorname{Z}(\mathscr{A},\mathcal{M})\rightarrow \mathcal{M}\rightarrow \mathcal{M}/\operatorname{Z}(\mathscr{A},\mathcal{M})\rightarrow 0,
\end{equation*}
and we have $\mathcal{M}/\operatorname{Z}(\mathscr{A},\mathcal{M})=\varprojlim_n \left(\mathcal{M}_n/\operatorname{Z}(\mathscr{A}_n,\mathcal{M}_n)\right)$. Hence, claim $(iii)$ also holds.\\

Assume now that $\operatorname{HH}^1(\mathscr{A},\mathcal{M})$ is an object in $\widehat{\mathcal{B}}c_K$. As shown in Corollary \ref{coro H1 borno implies delta0 strict}, this holds if and only if the sequence:
\begin{equation}\label{ses out as inverse limit}
 0\rightarrow   \mathcal{M}/\operatorname{Z}(\mathscr{A},\mathcal{M})\rightarrow \Der(\mathscr{A},\mathcal{M})\rightarrow  \Outder(\mathscr{A},\mathcal{M}) \rightarrow 0,
\end{equation}
is a strict short exact sequence in $\widehat{\mathcal{B}}c_K$, in which case we have $
\operatorname{HH}^1(\mathscr{A},\mathcal{M})\cong \Outder^b(\mathscr{A},\mathcal{M})$. By Lemma \ref{lemma ncd is co-admissible} we have an isomorphism of Fréchet spaces $\Der(\mathscr{A},\mathcal{M})=\varprojlim_n\Der(\mathscr{A}_n,\mathcal{M}_n)$. Hence, it follows that the sequence in (\ref{ses out as inverse limit}) arises as the inverse limit of the following short exact sequences of Banach spaces for each $n\geq 0$:
\begin{equation*}
    0\rightarrow   \mathcal{M}/\operatorname{Z}(\mathscr{A}_n,\mathcal{M}_n)\rightarrow \Der(\mathscr{A}_n,\mathcal{M}_n)\rightarrow  \Outder^b(\mathscr{A}_n,\mathcal{M}_n) \rightarrow 0.
\end{equation*}
Hence, claim $(iv)$ holds. By assumption, the presentation $\mathcal{M}=\varprojlim_n\mathcal{M}_n$ makes $\mathcal{M}$ a nuclear Fréchet space. Hence, we may again apply the contents of \cite[Section 5.3]{bode2021operations}
to deduce that:
\begin{equation*}
   R\varprojlim \left(\mathcal{M}_n/\operatorname{Z}(\mathscr{A}_n,\mathcal{M}_n)\right)=\varprojlim \left(\mathcal{M}_n/\operatorname{Z}(\mathscr{A}_n,\mathcal{M}_n)\right). 
\end{equation*}
Thus, it follows that $\Outder(\mathscr{A},\mathcal{M})=\varprojlim \Outder(\mathscr{A}_n,\mathcal{M}_n)$, and this shows $(v)$.
\end{proof}
\begin{coro}
Assume $K$ is either discretely valued or algebraically closed, and let $X=\varinjlim_n X_n$ be a smooth Stein space with free tangent sheaf. Consider a Fréchet-Stein presentation:
\begin{equation*}
    \wideparen{\D}_X(X)=\varprojlim_n \widehat{\D}_{n,m(n)},
\end{equation*}
as in Theorem \ref{teo global sections of co-admissible modules on Stein spaces}. Then for each $n\geq 0$ we have $\operatorname{Z}(\widehat{\D}_{n,m(n)})=\operatorname{H}^0_{\operatorname{dR}}(X_n)$. In particular, $\operatorname{Z}(\widehat{\D}_{n,m(n)})$  is a finite-dimensional $K$-vector space. Thus, we have the following identity in $\operatorname{D}(\widehat{\mathcal{B}}c_K)$:
\begin{equation*}
    \operatorname{Z}(\wideparen{\D}_X(X))=R\varprojlim_n \operatorname{Z}(\widehat{\D}_{n,m(n)}).
\end{equation*}
In this situation, the maps $\wideparen{\D}_X(X)\rightarrow \widehat{\D}_{n,m(n)}$ induce a natural isomorphism of Fréchet spaces:
\begin{equation*}
    \Outder(\wideparen{\D}_X(X))\rightarrow \varprojlim_n \Outder(\widehat{\D}_{n,m(n)}),
\end{equation*}
and this implies that the map:
\begin{equation*}
    c_1:\operatorname{H}_{\operatorname{dR}}^{1}(X)^b\rightarrow \Outder(\wideparen{\D}_X(X)),\, [\lambda]\mapsto [c_1(\lambda)]
\end{equation*}
is an isomorphism of nuclear Fréchet spaces.
\end{coro}
\begin{proof}
Using  Kashiwara's equivalence an the side-changing for bimodules developed in \cite{HochDmod}, it follows that for each $n\geq 0$ we have the following chain of identities:
\begin{equation*}
  \operatorname{Z}(\widehat{\D}_{n,m(n)})=\underline{\Hom}_{\widehat{\D}_{n,m(n)}^e}(\widehat{\D}_{n,m(n)},\widehat{\D}_{n,m(n)})=\underline{\Hom}_{\widehat{\D}_{n,m(n)}}(\OX_X(X_n),\OX_X(X_n))=\operatorname{H}^0_{\operatorname{dR}}(X_n),
\end{equation*}
and this is a finite-dimensional $K$-vector space. As $\operatorname{Z}(\wideparen{\D}_X(X))=R\varprojlim_n \operatorname{Z}(\widehat{\D}_{n,m(n)})$ is an inverse limit of finite-dimensional $K$-vector spaces, it follows by \cite[Section 2 Corollary 5]{schneider1991cohomology} that the sequence:
\begin{equation*}
  0\rightarrow  \operatorname{Z}(\wideparen{\D}_X(X))\rightarrow \prod_n \operatorname{Z}(\widehat{\D}_{n,m(n)})\rightarrow \prod_n \operatorname{Z}(\widehat{\D}_{n,m(n)})\rightarrow 0,
\end{equation*}
is a short exact sequence of $K$-vector spaces, and that all the maps are bounded. As $\prod_n \operatorname{Z}(\widehat{\D}_{n,m(n)})$ is a nuclear Fréchet space, it follows by \cite[Proposition 5.12]{bode2021operations} that the sequence is strict. Hence, by Proposition \ref{prop out as inverse limit} the maps $\wideparen{\D}_X(X)\rightarrow \widehat{\D}_{n,m(n)}$ induce a natural isomorphism of Fréchet spaces:
\begin{equation*}
    \Outder(\wideparen{\D}_X(X))\rightarrow \varprojlim_n \Outder(\widehat{\D}_{n,m(n)}).
\end{equation*}
Thus, the first part of the corollary holds. For the second part, we notice that by the proof of \cite[Corollary 3.2]{grosse2004rham}, we can consider an admissible cover $X=\cup_{m\geq 0}Y_m$ such that $Y_m\subset Y_{m+1}$ and each $Y_m$ is a smooth Stein space with finite-dimensional de Rham cohomology. In this situation, it follows by the discussion above and \emph{loc. cit.} that we have isomorphisms of nuclear Fréchet spaces:
\begin{equation*}
    \Outder(\wideparen{\D}_X(X))=\varprojlim_m \Outder(\wideparen{\D}_X(Y_m)), \textnormal{ and } \operatorname{H}_{\operatorname{dR}}^{1}(X)^b=\varprojlim_m \operatorname{H}_{\operatorname{dR}}^{1}(Y_m)^b.
\end{equation*}
Furthermore, the  identity $\Outder(\wideparen{\D}_X(X))=\varprojlim_m \Outder(\wideparen{\D}_X(Y_m))$ is induced by the restriction maps $\wideparen{\D}_X(X)\rightarrow \wideparen{\D}_X(Y_m)$.
By construction of $c_1:\operatorname{H}_{\operatorname{dR}}^{1}(X)^b\rightarrow \Outder(\wideparen{\D}_X(X))$, it follows that 
this map arises as the inverse limit of the maps:
\begin{equation*}
    c_1:\operatorname{H}_{\operatorname{dR}}^{1}(Y_m)^b\rightarrow \Outder(\wideparen{\D}_X(Y_m)),
\end{equation*}
for each $m\geq 0$. However, by assumption, $\operatorname{H}_{\operatorname{dR}}^{1}(Y_m)^b$ is finite-dimensional, so by Corollary \ref{coro c1 iso fd case} this is an isomorphism for each $m\geq 0$. Thus, the result holds.
\end{proof}
As an upshot, we obtain a generalization of Corollary \ref{coro Der is nf fd-case}:
\begin{defi}[{\cite[pp. 71  ]{schneider2013nonarchimedean}}]
 Let $V$ be a locally convex $K$-vector space. A $\mathcal{R}$-submodule $A\subset V$  is called compactoid (in $V$) if for any open lattice $U\subset V$ there are vectors $v_1,\cdots,v_n\in V$ such that:
 \begin{equation*}
     A\subset U+ \sum_{i=1}^n\mathcal{R}v_i.
 \end{equation*}
\end{defi}
\begin{coro}
Assume $K$ is discretely valued or algebraically closed, and let $X$ be a Stein space with free tangent sheaf. The space of bounded derivations $\Der(\wideparen{\D}_X(X))$ is a nuclear Fréchet space.   
\end{coro}
\begin{proof}
As the tangent space of $X$ is free, we have a short exact sequence of Fréchet spaces:
\begin{equation*}
    0\rightarrow \wideparen{\D}_X(X)/\operatorname{Z}(\wideparen{\D}_X(X))\rightarrow \Der(\wideparen{\D}_X(X))\rightarrow \Outder(\wideparen{\D}_X(X))\rightarrow 0.
\end{equation*}
By  \cite[Lemma 5.4]{bode2021operations} and \cite[Proposition 19.5]{schneider2013nonarchimedean}, it suffices to show  that for every Banach space $V$ every bounded map $f:\Der(\wideparen{\D}_X(X))\rightarrow V$ is compact. That is, that there is an open lattice $U\subset \Der(\wideparen{\D}_X(X))$ such that the closure $\overline{f(U)}$ is compactoid. As $W:=\wideparen{\D}_X(X)/\operatorname{Z}(\wideparen{\D}_X(X))$
is a nuclear Fréchet space, it follows by \emph{loc. cit.} that every bounded map from $W$ to a Banach space is compact. As the map $W\rightarrow \Der(\wideparen{\D}_X(X))$ is strict, this implies that there is an open lattice $U\subset \Der(\wideparen{\D}_X(X))$ such that the space $\overline{f(U\cap W)}$ is bounded and $c$-compact in $V$. Completing with respect to the norm induced by this lattice, we obtain a short exact sequence of Banach spaces:
\begin{equation*}
    0\rightarrow V_1\rightarrow V_2\rightarrow V_3\rightarrow 0,
\end{equation*}
satisfying that the map $f$ extends uniquely to a map $f':V_2\rightarrow V$, and such that the restriction of $f'$ to $V_1$ is compact. By the previous corollary, $\Outder(\wideparen{\D}_X(X))$ is an inverse limit of finite-dimensional spaces. Thus, it follows that $V_3$ is finite-dimensional. Hence, the map $f':V_2\rightarrow V$ is also compact, and therefore $\Der(\wideparen{\D}_X(X))$ is a nuclear Fréchet space, as we wanted to show.
\end{proof}
\bibliographystyle{plain}

Mathematisch-Naturwissenschaftliche Fakult\"at der Humboldt-Universit\"at zu Berlin, Rudower Chaussee 25, 12489 Berlin, Germany \newline
\textit{Email address: fpvmath@gmail.com}
\end{document}